\documentclass[10pt]{amsart}

 \newif\ifarxiv
 \arxivtrue
 \newif\ifshowupdate
 \showupdatefalse

 \usepackage{float}
 \usepackage[section]{placeins}

\usepackage{cite,graphicx,color}
\usepackage[letterpaper,margin=1.5in]{geometry}
\usepackage{microtype}
\usepackage{graphicx}
\usepackage{subfigure}
\usepackage{booktabs}
\usepackage{soul}

\usepackage[initials]{amsrefs}
\usepackage{amsaddr}

\title{Learning Optimal Flows for Non-Equilibrium Importance Sampling}
\author{Yu Cao}
\address{Courant Institute of Mathematical Sciences, New York University, New York, NY, 10012}
\email{yucaoyc@outlook.com}
\author{Eric Vanden-Eijnden}
\email{eve2@cims.nyu.edu}
\address{Courant Institute of Mathematical Sciences, New York University, New York, NY, 10012}
\date{\today}

\usepackage{hyperref}
 \usepackage{amsmath,amssymb}
\usepackage{mathrsfs}
\usepackage{amsthm}
\allowdisplaybreaks
\usepackage{thmtools}
\usepackage{csquotes}
\usepackage{verbatim}
\usepackage{enumerate}
\usepackage{dsfont}
\usepackage{bm}
\usepackage{enumitem}
\usepackage{afterpage}
\usepackage{xcolor}
\usepackage{mathtools}
\usepackage{physics}
\usepackage{wrapfig}

\usepackage{multirow, makecell}

\DeclarePairedDelimiter\floor{\lfloor}{\rfloor}

\newcommand{\propref}[1]{Proposition~\ref{#1}}
\newcommand{\lemref}[1]{Lemma~\ref{#1}}
\newcommand{\secref}[1]{Section~\ref{#1}}
\newcommand{\appref}[1]{Appendix~\ref{#1}}
\newcommand{\figref}[1]{Figure~\ref{#1}}

\newcommand{\defref}[1]{Definition~\ref{#1}}

\newcommand{\ie}{{i.e.}}
\newcommand{\eg}{{e.g.}}
\makeatletter
\newcommand*{\rom}[1]{\expandafter\@slowromancap\romannumeral #1@}
\makeatother
\newcommand{\wt}[1]{\widetilde{#1}}

\theoremstyle{plain}
\newtheorem{theorem}{Theorem}[section]
\newtheorem{proposition}[theorem]{Proposition}
\newtheorem{lemma}[theorem]{Lemma}

\theoremstyle{definition}
\newtheorem{definition}[theorem]{Definition}
\newtheorem{assumption}[theorem]{Assumption}
\theoremstyle{remark}
\newtheorem{remark}[theorem]{Remark}
\newtheorem{example}[theorem]{Example}

\newcommand\myeq[1]{\mathrel{\stackrel{\makebox[0pt]{\mbox{\normalfont\scriptsize #1}}}{=}}}
\newcommand\myge[1]{\mathrel{\stackrel{\makebox[0pt]{\mbox{\normalfont\tiny #1}}}{\ge}}}
\newcommand\myle[1]{\mathrel{\stackrel{\makebox[0pt]{\mbox{\normalfont\tiny #1}}}{\le}}}
\newcommand\myless[1]{\mathrel{\stackrel{\makebox[0pt]{\mbox{\normalfont\tiny #1}}}{<}}}

\newcommand{\ud}{\,\mathrm{d}}

\newcommand{\Natural}{\mathbb{N}}
\newcommand{\Int}{\mathbb{Z}}
\newcommand{\Real}{\mathbb{R}}
\newcommand{\ee}{\mathbb{E}}

\newcommand{\eps}{\epsilon}

\newcommand{\absbig}[1]{\big\lvert#1\big\rvert}
\newcommand{\Abs}[1]{\Big\lvert#1\Big\rvert}

\newcommand{\inner}[2]{\langle#1, #2\rangle}
\newcommand{\innerbig}[2]{\big\langle#1, #2\big\rangle}
\newcommand{\innerBig}[2]{\Big\langle#1, #2\Big\rangle}

\def\bigl{\mathopen\big}
\def\bigr{\mathclose\big}

\newcommand*\Laplace{\Delta}

\newcommand{\idmat}{\mathbb{I}}

\newcommand{\corrz}[2]{\bm{C}_{#1,#2}}

\newcommand{\corrg}[3]{\bm{V}_{#1,#2}^{(#3)}}
\newcommand{\partition}{\mathcal{Z}}
\newcommand{\dimn}{d}
\newcommand{\intreal}{\int_{-\infty}^{\infty}}

\newcommand{\accF}{\mathfrak{F}}
\newcommand{\fouriertrans}{\mathscr{F}}

\newcommand{\tm}{t_{-}}
\newcommand{\tp}{t_{+}}

\newcommand{\afin}{\mathcal{A}}

\newcommand{\newafin}{\afin_{\tm,\tp}}
\newcommand{\newainf}{\afin}

\newcommand{\bfin}{\mathscr{B}}
\newcommand{\binf}{\mathcal{B}}
\newcommand{\newafinarg}[2]{\afin_{#1,#2}}

\newcommand{\chrontime}{\mathcal{T}_{\leftarrow}}
\newcommand{\antichrontime}{\mathcal{T}_{\rightarrow}}

\newcommand{\dynb}{{\boldsymbol{b}}}

\newcommand{\AF}{\eta}
\newcommand{\mode}{m}

\renewcommand{\var}{\text{Var}}

\newcommand{\newvarfin}{\var_{\tm,\tp}}
\newcommand{\newmsecfin}{\mathcal{M}_{\tm,\tp}}
\newcommand{\newmsecfinarg}[2]{\mathcal{M}_{#1,#2}}
\newcommand{\newvarinf}{\var}
\newcommand{\newmsecinf}{\mathcal{M}}
\newcommand{\varmax}{\text{Var}^{(\max)}}

\newcommand{\derimsec}{\frac{\delta\newmsecfin(\dynb)}{\delta \dynb}}

\newcommand{\derimsecinf}{\frac{\delta\newmsecinf(\dynb)}{\delta \dynb}}

\newcommand{\derivarinf}{\frac{\delta\newvarinf(\dynb)}{\delta \dynb}}

\newcommand{\ftneis}{finite-time NEIS}

\newcommand{\itneis}{infinite-time NEIS}

\newcommand{\mani}[1]{\mathfrak{B}_{#1}}
\newcommand{\maniinf}{\mani{\infty}}
\newcommand{\basis}{\bm{e}}
\newcommand{\dom}{\Omega}
\newcommand{\effdom}{\mho}
\newcommand{\domsubset}{D}

\newcommand{\vectorzero}{\bm{0}}
\newcommand{\indi}{\chi}

\newcommand{\Rd}{\Real^{\dimn}}

\newcommand{\td}{\frac{\ud}{\ud t}}

\newcommand{\tube}{\mathfrak{T}}

\newcommand{\mapT}{\bm{T}}

\newcommand{\opSinf}{\bm{\mathcal{S}}^{\infty}}
\newcommand{\opGinf}{\mathcal{G}^{\infty}}
\newcommand{\opSinfarg}[2]{\opSinf_{#1}(#2)}
\newcommand{\opSinfargbig}[2]{\opSinf_{#1}\big(#2\big)}
\newcommand{\opGinfarg}[2]{\opGinf_{#1}(#2)}
\newcommand{\opGinfargbig}[2]{\opGinf_{#1}\big(#2\big)}

\newcommand{\fhit}[1]{\tau^{+}_{#1}}
\newcommand{\bhit}[1]{\tau^{-}_{#1}}
\newcommand{\diam}{\text{Diameter}}
\newcommand{\xst}{x^{\star}}

\definecolor{cadmiumgreen}{rgb}{0.0, 0.42, 0.24}

\newcommand{\state}[2]{\bm{X}_{#1}(#2)}
\newcommand{\statebig}[2]{\bm{X}_{#1}\big(#2\big)}

\newcommand{\stateidx}[3]{(\bm{X}_{#1}(#2))_{#3}}
\newcommand{\statesup}[3]{\bm{X}^{#3}_{#1}(#2)}
\newcommand{\statesupT}[3]{\wt{\bm{X}}^{#3}_{#1}(#2)}
\newcommand{\statesupbig}[3]{\bm{X}^{#3}_{#1}\big(#2\big)}

\newcommand{\stated}[2]{\bm{D}_{#1}(#2)}
\newcommand{\stateg}[2]{\bm{g}_{#1}(#2)}
\newcommand{\stategsup}[3]{\bm{g}^{#3}_{#1}(#2)}
\newcommand{\stateH}[2]{\bm{H}_{#1}(#2)}
\newcommand{\stateHsup}[3]{\bm{H}^{#3}_{#1}(#2)}
\newcommand{\stateL}[2]{\bm{L}_{#1}(#2)}
\newcommand{\stateLsup}[3]{\bm{L}^{#3}_{#1}(#2)}

\newcommand{\deristate}[2]{\bm{Y}_{#1}(#2)}
\newcommand{\deristatesup}[3]{\bm{Y}^{#3}_{#1}(#2)}

\newcommand{\deristatesupT}[3]{\wt{\bm{Y}}^{#3}_{#1}(#2)}

\newcommand{\stateD}[2]{\bm{Z}_{#1}(#2)}

\newcommand{\testfunc}[1]{\mathcal{F}^{(#1)}}
\newcommand{\deritestfunc}[1]{\mathcal{G}^{(#1)}}

\newcommand{\testfuncarg}[3]{\testfunc{#1}_{#2}(#3)}
\newcommand{\testfuncargbig}[3]{\testfunc{#1}_{#2}\big(#3\big)}
\newcommand{\testfuncepsarg}[4]{\testfunc{#1,#2}_{#3}(#4)}
\newcommand{\deritestfuncarg}[3]{\deritestfunc{#1}_{#2}(#3)}

\newcommand{\jacoarg}[2]{\mathcal{J}_{#1}(#2)}
\newcommand{\jacoargsup}[3]{\mathcal{J}^{#3}_{#1}(#2)}

\newcommand{\jacoargbig}[2]{\mathcal{J}_{#1}\big(#2\big)}
\newcommand{\jacomap}[2]{\jacoarg{#1}{#2}}

\newcommand{\lag}{\text{lag}}
\newcommand{\lowergamma}[2]{\Gamma\big(#1,#2\big)}
\newcommand{\pts}{\mathscr{X}}
\newcommand{\ptsy}{\mathscr{Y}}
\newcommand{\ball}[2]{B_{#1}(#2)}

\newcommand{\nn}{\boldsymbol{n}} 

\newcommand{\ratio}{\partition_1}

\graphicspath{{./assets/}{}}

\newcommand{\timeT}{\varkappa}
\usepackage{todonotes}

\ifarxiv
	\newcommand{\newparagraph}[1]{\subsection{#1}}
	\newcommand{\newparagraphno}[1]{\subsection*{#1}}
\else
 	\newcommand{\newparagraph}[1]{\paragraph{#1.}}
 	\newcommand{\newparagraphno}[1]{\paragraph{#1.}}
\fi

\newcommand{\partitionzero}{{}}
\newcommand{\gausscst}[1]{\frac{#1}{2}\ln(2\pi)}
\newcommand{\gauss}[2]{\mathscr{N}(#1,#2)}

\newcommand{\sigmamat}{\mathcal{D}}

\usepackage{soulutf8}
\soulregister{\allowbreak}{0}
\usepackage{colortbl}

\sethlcolor{red}

\usepackage{hhline}
\newcommand{\bestcell}{\cellcolor{green!10}}
\newcommand{\bestcellais}{\cellcolor{blue!10}}

 \usepackage{newtxtext,newtxmath}
 \newcommand{\myemph}[1]{\emph{#1}}

 \makeatletter
 \providecommand\@dotsep{5}
 \renewcommand{\listoftodos}[1][\@todonotes@todolistname]{%
 	\@starttoc{tdo}{#1}}
 \makeatother

 \ifshowupdate
 
 \else
 
 \fi

\begin{document}
 \maketitle

\begin{abstract}
Many applications in computational sciences and statistical inference require the computation of expectations with respect to complex high-dimensional distributions with unknown normalization constants, as well as the estimation of these constants.
Here we develop a method to perform these calculations based on generating samples from a simple base distribution, transporting them by the flow generated by a velocity field, and performing averages along these flowlines.
This non-equilibrium importance sampling (NEIS) strategy is straightforward to implement and can be used for calculations with arbitrary target distributions.
On the theory side, we discuss how to tailor the velocity field to the target and establish general conditions under which the proposed estimator is a perfect estimator with zero-variance.
We also draw connections between NEIS and approaches based on mapping a base distribution onto a target via a transport map. On the computational side, we show how to use deep learning to represent the velocity field by a neural network and train it towards the zero variance optimum.
These results are illustrated numerically on benchmark examples (with dimension up to $10$), where after training the velocity field, the variance of the NEIS estimator is reduced by up to 6 {orders} of magnitude than that of a vanilla estimator.
We also compare the performances of NEIS with those of Neal's annealed importance sampling (AIS).
\end{abstract}

\section{Introduction}

Given  a potential function $U_1 :\dom\rightarrow\Real$ on the domain $\dom\subseteq\Rd$, the main goal of this paper is to evaluate
\begin{equation}
\label{eq:partition}
\partition_1 := \int_{\dom} e^{-U_1(x)} \ud x.
\end{equation}
The calculations of such integrals arise in many applications from several scientific fields. For instance, $\partition_1$ is known as the partition function in statistical physics \cite{lifshitz2013statistical}, where it is used to characterize the thermodynamic properties of a system with energy $U_1$, and as the evidence in Bayesian statistics, where it is used for model selection  \cite{feroz_multimodal_2008}.

When the dimension $d$ of the domain $\Omega$ is large, standard numerical quadrature methods are inapplicable to~\eqref{eq:partition} and the method of choice to estimate $\partition_1$ is Monte-Carlo sampling~\cite{liu2001monte,brooks_handbook_2011}. This requires expressing $\partition_1$ as an expectation, which can be done \eg{}, by realizing that
\begin{align}
	\label{eqn::direct_importance}
	\partition_1 = \ee_{0}\big[e^{-U_1}/\rho_0\big],
\end{align}
where $\ee_0$ denotes expectation with respect to the probability density function $\rho_0>0$. If $\rho_0$ is both known (\ie{}, we can evaluate it pointwise in $\Omega$, normalization factor included) and simple to sample from, we can build an estimator for $\partition_1$ by replacing the expectation on the right hand side of~\eqref{eqn::direct_importance} by the empirical average of $e^{-U_1}/\rho_0$ over samples drawn from $\rho_0$. Unfortunately, finding a density $\rho_0$ that has the two properties above is hard: unless $\rho_0$ is well-adapted to $e^{-U_1}$, {the estimator} based on \eqref{eqn::direct_importance} is terrible in general, with a standard deviation that is typically much larger than its mean or even infinite. A similar issue arises if we want to estimate the expectation $\ee_0 f$ of some function $f:\Omega\to\Real$, and the two problems are in fact connected when $f > 0$ since the second reduces to~\eqref{eqn::direct_importance} for $U_1=-\log(f\rho_0)$.

These difficulties have prompted the development of importance sampling strategies~\cite{geyer_importance_2011} whose aim is to produce estimators with a reasonably low variance for $\partition_1$ or $\ee_0f$. These include for example umbrella sampling~\cite{torrie1977nonphysical,thiede_eigenvector_2016}, replica exchange (aka parallel tempering)~\cite{geyer_importance_2011,lu_methodological_2019}, nested sampling~\cite{skilling_nested_2004,skilling_nested_2006}, in which  the estimation of $\partition_1$ is factorized into the calculation of several expectations of the type~\eqref{eqn::direct_importance}, but with better properties, that can then be recombined using thermodynamic integration~\cite{kirkwood1935statistical} or Bennett acceptance ratio method~\cite{bennett1976efficient}.

Complementary to these equilibrium techniques, non-equilibrium sampling strategies have also been introduced for the calculation of~\eqref{eq:partition}. For example, Neal's annealed importance sampling (AIS)~\cite{neal_annealed_2001}  based on the Jarzynski equality~\cite{jarzynski_equilibrium_1997,jarzynski_nonequilibrium_1997,andrieu_sampling_2018} calculates $\partition_1$ using properly weighted averages over sequences of states evolving from samples from $\rho_0$, without requiring that the kernel used to generate these states be in detailed-balance with respect to either $\rho_0$, or $\rho_1 := e^{-U_1}/\partition_1$, or any density interpolating between these two. Instead the weight factors are based on the probability distribution of the sequence of states in path space. Other non-equilibrium sampling strategies in this vein include bridge  and path sampling \cite{gelman_simulating_1998}, and sequential Monte Carlo (SMC) sampling \cite{moral_sequential_2006, arbel_annealed_2021}.

In this paper, we analyze another non-equilibrium importance sampling  (NEIS) method, originally introduced in~\cite{rotskoff_dynamical_2019}. NEIS is based on generating samples from a simple base density $\rho_0$, then propagating them forward and backward in time along the flowlines of a velocity field, and computing averages along these trajectories---the basic idea of the method is to use the flow induced by this velocity field to sweep samples from $\rho_0$ through regions in $\Omega$ that contribute most to the expectation. As shown in~\cite{rotskoff_dynamical_2019} and recalled below, this procedure leads to  consistent estimators for the calculation of $\partition_1$ or $\ee_0 f$ via a generalization of~\eqref{eqn::direct_importance}. One advantage of the method, which is a rare feature among importance sampling strategies, is that it leads to estimators that always have lower variance than the vanilla estimator based on~\eqref{eqn::direct_importance} \cite{rotskoff_dynamical_2019}. The question we investigate in this paper is how low their variance can be made, both in theory and in practice.  Our \textbf{main contributions} are:

\begin{itemize}[leftmargin=.1in]
    \item Under mild assumptions on $U_1$ and $\rho_0$, we show that if the NEIS velocity field is the gradient of a potential that satisfies a Poisson equation, the NEIS estimator for $\partition_1$ has zero variance.
    \item {Under the same assumptions, we show that this optimal flow can be used to construct a perfect transport map from $\rho_0$ to $\rho_1$.}
    This allows us to compare NEIS with importance sampling strategies
    involving transport maps like normalizing flows (NF) that have  recently
    gained
    popularity~\cite{rezende_variational_2015,kobyzev_normalizing_2020,papamakarios_normalizing_2021},
and highlight some potential advantages of the former over the latter.
    \item On the practical side, we derive variational problems for the optimal velocity field in NEIS, and show how to solve these problems by approximating the velocity by a neural network and optimizing its parameters using deep learning training strategies, similar to what is done with neural ODE~\cite{chen_neural_2018}.
    \item We illustrate the feasibility and usefulness of this approach by testing it on numerical examples. First we consider Gaussian mixtures in up to 10 dimensions. In this context, we show that training the velocity used in NEIS allows to reduce the variance of a vanilla estimator using a standard Gaussian distribution as $\rho_0$ by up to 6 orders of magnitude.
{Second we study Neal's 10-dimensional funnel distribution \cite{neal_slice_2003,arbel_annealed_2021}, for which the variance of the vanilla importance sampling method is infinity; training a linear dynamics with 2 parameters in NEIS can lead to an estimator with empirical variance less than $1$.}
    {In these examples we also show that after training, NEIS leads to estimators with lower variance than AIS \cite{neal_annealed_2001}.}
\end{itemize}

\newparagraphno{Related works}

The idea of transporting samples from $\rho_0$ to lower the variance of the
vanilla estimator based on~\eqref{eqn::direct_importance} is also at the core
of importance sampling strategies using normalizing flows (NF)
\cite{tabak_density_2010, tabak_family_2013, rezende_variational_2015,
kobyzev_normalizing_2020,
papamakarios_normalizing_2021,wu_stochastic_2020,wirnsberger_targeted_2020,noe_boltzmann_2019,muller_neural_2019}.
The type of transport used in NF-based method is however different in nature
from the one used in NEIS. With NF, one tries to construct a map that transforms
each sample from $\rho_0$ into a sample from the target $\rho_1
= e^{-U_1}/\partition_1$. In contrast, NEIS uses samples from $\rho_0$ as
initial conditions to generate trajectories, and uses the data along these
entire trajectories to build an estimator.  Intuitively, this means that
samples likely on $\rho_0$ must become likely on $\rho_1$ sometime along these
trajectories rather than at a given time specified beforehand, which is easier
to enforce.

NEIS bears similarities with  Neal's AIS~\cite{neal_annealed_2001}, except that in NEIS the sampling is done once from $\rho_0$ to generate deterministic trajectories to gather data for the estimator, whereas AIS uses random trajectories. There are some methods based on AIS that optimize the transition kernel: for instance, stochastic normalizing flows (SNF) proposed in \cite{wu_stochastic_2020} incorporates NF between annealing steps; and
annealed flow transport (AFT) in \cite{arbel_annealed_2021} combines NF with the sequential Monte Carlo method to provide optimized flow transport. These approaches require learning several maps along the annealed transition, whereas the NEIS herein only needs to learn a single flow dynamics.

A time-discrete version of NEIS, termed NEO, was proposed in \cite{thin_neo_2021}. The current implementation of NEO iterates on a map that needs to be prescribed beforehand, but this map could perhaps be optimized using a strategy similar to the one proposed here.

From a practical standpoint, the idea of optimizing the velocity field in NEIS using a neural network approximation for this field can be viewed as an application of neural ODEs~\cite{chen_neural_2018} that uses the variance of the NEIS estimator as the objective function to minimize. The nature of this objective poses specific challenges in the training procedure, which we investigate here.

\newparagraphno{Notations}  For symmetry, we will denote $\rho_0(x) = e^{-U_0(x)}$ with $U_0 = - \log \rho_0:\Omega \to \Real$ and $\partition_0= \int_\Omega e^{-U_0(x)} \ud x =1$. We denote a $\dimn$-dimensional vector filled with zeros as $\vectorzero_{\dimn}$ and the $\dimn\times \dimn$ identity matrix as $\idmat_{\dimn}$.
$\inner{\cdot}{\cdot}$ is the Euclidean inner product in $\Rd$. We assume that the domain $\Omega$ is either an open and connected subset of $\Real^\dimn$ with smooth boundary or a $\dimn$-dimensional torus (without boundary).
{We denote by  $\gauss{\mu}{\Sigma}$ the multivariate Gaussian density with mean $\mu$ and covariance matrix $\Sigma$.
For two functions $f, g: \domsubset\to \Real$ where $\domsubset$ is a domain of interest, the notation $f\lesssim g$ means that there exists a  constant $C>0$ such that $f(x)\le C g(x)$ for any $x\in \domsubset$.}
{Suppose $\mapT:\dom\to\Rd$ is a map and $\rho$ is a distribution, then the pushforward distribution of $\rho$ by the map $\mapT$ is denoted as $\mapT\#\rho$.}
The notation $\abs{\cdot}$ is the usual $\ell_2$ norm for vectors and $\norm{\cdot}$ is the matrix norm or functional norm.

\section{Flow-based NEIS method}
\label{sec::method}

Here we recall the main ingredients of the non-equilibrium importance sampling (NEIS) method proposed in \cite{rotskoff_dynamical_2019}. Let $\dynb: \Omega \to\Real^d$ be a velocity field which we assume belongs to the vector space

\begin{align}
	\label{eqn::b_space}
	\mani{} := \Big\{
	\dynb\in C^\infty\big(\overline{\dom}, \Rd\big)\ \Big\rvert\
		 \dynb\cdot \nn{} \rvert_{\partial\dom} = 0,\
	 	 \sup_{x\in \dom}\ \abs{\nabla \dynb(x)} < \infty
	 \Big\},
\end{align}
where $\nn{}$ is the normal vector at the boundary $\partial\dom$.
Define the associated flow map $\bm{X}_t: \Omega \to\Omega$ via
\begin{align}
\label{eqn::ode}
\td \state{t}{x} = \dynb\left(\state{t}{x}\right), \qquad \state{0}{x} = x,
\end{align}
and let $\jacoarg{t}{x}$ be the Jacobian of this map:
\begin{align}
\label{eq:jacob}
	\jacoarg{t}{x}:=\left|\det\left(\nabla_x \state{t}{x}\right) \right| \equiv \exp\left(\int_{0}^{t} \nabla \cdot \dynb\left(\state{s}{x}\right) \ud s\right).
\end{align}
Finally, let us denote
\begin{align}
	\label{eqn::testfunc}
	\begin{aligned}
	&\testfuncarg{k}{t}{x} := e^{-U_k\left(\state{t}{x}\right)} \jacoarg{t}{x}
	\end{aligned}
\end{align}
for $k \in \{ 0, 1\}$, $x\in\dom$ and $t\in\Real$. NEIS is based on the following result, proven in \appref{subsec::proof::formulation}:

\begin{proposition}
\label{prop::formulation} If $\dynb\in \mani{}$, then for any $-\infty<\tm<\tp<\infty$, we have
\begin{equation}
    \label{eqn::noneq_sample::finite}
	\partition_1 = \ee_{0} \newafin,
\end{equation}
where
\begin{equation}
    \label{eq::afin}
    \newafin(x) := \int_{\tm}^{\tp} \frac{\testfuncarg{1}{t}{x}}{\int_{t-\tp}^{t-\tm} \testfuncarg{0}{s}{x} \ud s} \ud t.
\end{equation}
In addition, if
\begin{equation}
    \label{eqn::ainf}
    \lim_{\substack{\tm\to\ -\infty\\\tp\to\infty}} \newafin(x) = \newainf(x) := \frac{\int_{\Real} \testfunc{1}_t(x)\ud t }{\int_{\Real} \testfunc{0}_t(x)\ud t}
\end{equation}
exists for almost all $x\sim\rho_0$, then
\begin{equation}
    \label{eqn::noneq_sample::infinite}
	\partition_1 = \ee_{0} \newainf.
\end{equation}
\end{proposition}

When $\dynb = \vectorzero_{\dimn}$, or $\tm\uparrow 0$ and $\tp\downarrow0$,  \eqref{eqn::noneq_sample::finite}  reduces to \eqref{eqn::direct_importance}. The aim, however, is to choose $\dynb$ so that the estimator based on \eqref{eqn::noneq_sample::finite}  has a {lower} variance than the one based \eqref{eqn::direct_importance}: we will show below that this can indeed be done. For now, note that Jensen's inequality implies that an estimator based on \eqref{eqn::noneq_sample::infinite} for \textit{any}~$\dynb$ has lower variance than the one based \eqref{eqn::direct_importance}; {see \cite{rotskoff_dynamical_2019} or \propref{prop::var_dynb} below for details}.

Note also that, if one allows the magnitude of the flow $\dynb$ to be arbitrarily large,  the finite-time NEIS \eqref{eq::afin} will behave like the infinite-time NEIS \eqref{eqn::ainf}; such a relation will be discussed and elaborated in Appendix~\ref{subsec::relation}.

{Finally, note that the estimator \eqref{eqn::noneq_sample::infinite} based on \eqref{eqn::ainf} is invariant with respect to the parameterization of the flowlines generated by the dynamics $\dynb$, as shown by the following result proved in \appref{subsec::proof_invariance}:}
\begin{proposition}[An invariance property]
	\label{prop::invariance_b}
	Suppose $\dynb, \alpha\dynb \in \mani{}$, where $\alpha\in C^{\infty}(\dom,\Real)$ satisfies $\inf_{x\in\dom} \alpha(x)>0$.
	Then the fields $\dynb$ and $\alpha\dynb$ generate the same flowlines, and  $\newainf_{\dynb} = \newainf_{\alpha\dynb}$ where $\newainf_{\dynb}$ and $\newainf_{\alpha\dynb}$ are the function defined in \eqref{eqn::ainf} using $\dynb$ and $\alpha\dynb$, respectively.
\end{proposition}

\section{Optimal NEIS}
\label{sec::optimal_neis}

\begin{figure}[t]
\centering
	\includegraphics[width=0.5\textwidth]{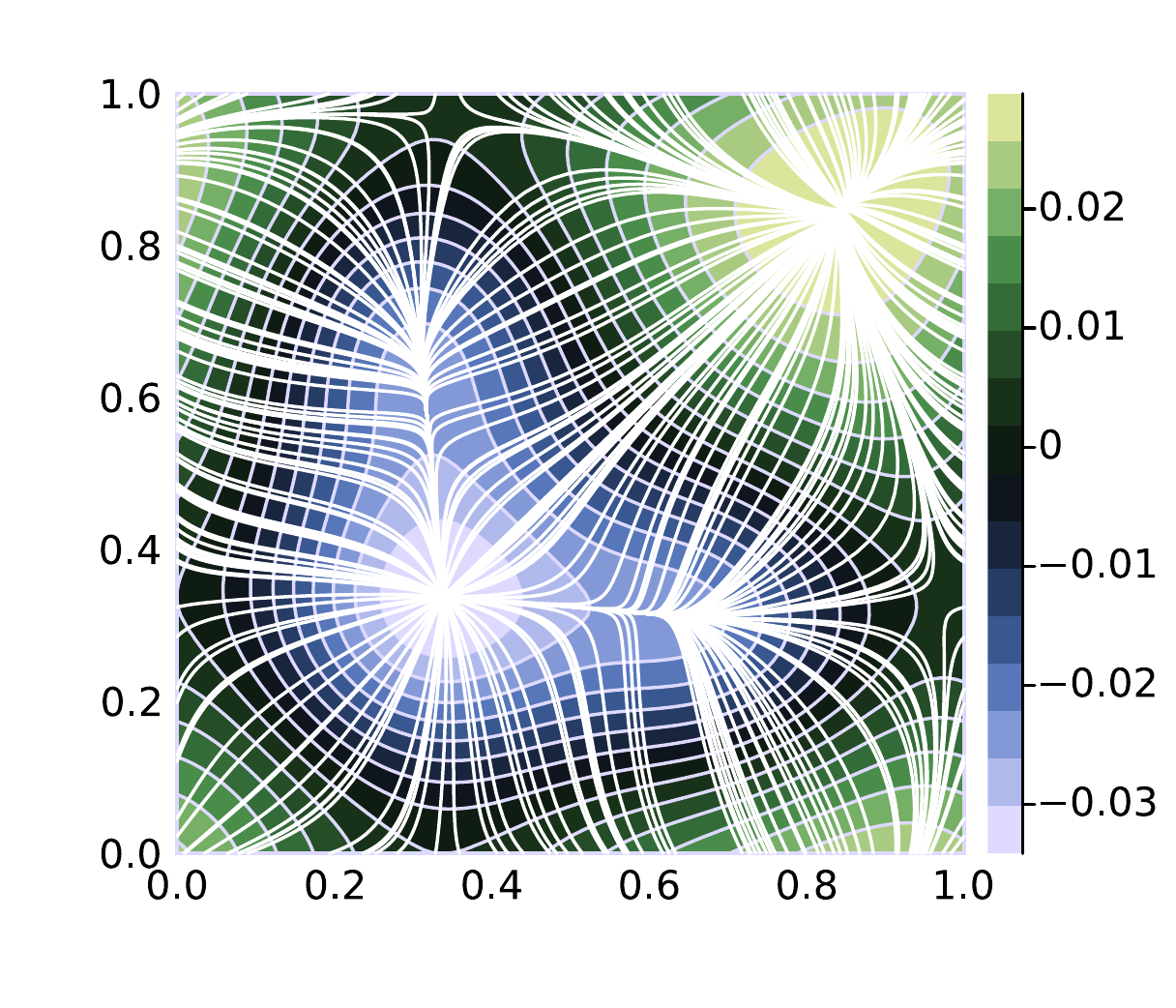}
	\caption{Contour plot of $V$ and flowlines of $\dynb=\nabla V$, where $V$ solves the Poisson's equation \eqref{eqn::poisson} with $\sigmamat=1$, assuming that $\rho_0=1$ and $\rho_1$ is a mixture density with 3 modes; see \eqref{eqn::torus_eg}. With this $\dynb=\nabla V$, we have {$\newainf(x)=\partition_1$} for almost all $x\in[0,1]^2$.
	}
	\label{fig::torus}
 \end{figure}

  The NEIS estimator for \eqref{eqn::noneq_sample::infinite} is unbiased
  no matter what $\dynb$ is. However, its performance relies on the choice of $\dynb$.
Therefore, a natural question is to find the field $\dynb$ that achieves the largest variance reduction. The next result shows that an optimal $\dynb$ exists that leads to a zero-variance estimator:

\begin{proposition}[{Existence of zero-variance dynamics}]
	\label{prop::existence}
	Assume that  $\dom=[0,1]^\dimn$ is a torus and $U_0, U_1\in C^\infty(\dom,\Real)$.
	Let $\sigmamat: \Omega\to (0,\infty)$ be some smooth positive function {with $\inf_{x\in\dom}\sigmamat(x)>0$}, and suppose that $V\in C^\infty(\dom)$ solves the following Poisson's equation on $\dom$
	\begin{align}
		\label{eqn::poisson}
		\begin{aligned}
			\nabla \cdot(\sigmamat \nabla V) = \rho_1 - \rho_0,\qquad \text{with} \quad\int_{\dom} V(x) \ud x= 0.
		\end{aligned}
	\end{align}
	If the solution $V$ is a Morse function, then $\dynb = \nabla V$ is a zero-variance velocity field: that is, if we use it to define~\eqref{eqn::ainf}, we have 
  \begin{align}
		\label{eqn::zero::var::2}
		\int_{\Real} \rho_0(\state{s}{x})\jacoarg{s}{x}\ud s = \int_{\Real} \rho_1(\state{s}{x})\jacoarg{s}{x}\ud s, \qquad \text{for almost all $x\sim \rho_0$},
	\end{align}
	and as a result
	\begin{equation}
	    \label{eq:zero:var}
	    {\newainf(x) = \partition_1} \quad \text{for almost all $x\sim \rho_0$}.
	\end{equation}
\end{proposition}

This proposition is proven in \appref{sec:proof:zero-variance}, where we
also make a connection between \eqref{eqn::poisson} and Beckmann’s transportation problem. We
stress that the optimal $\dynb$ specified in \propref{prop::existence} is not
unique (see \propref{prop::non-unique}): however, we show below in
\propref{prop::local} that, under certain conditions, all local minima of the
variance (viewed as a functional of $\dynb$) are global minima. We also note
that the assumption that the solution is a Morse function is mostly
a technicality, as  discussed in \appref{subsec::proof::existence_approximate}.
Similarly, we consider the torus in \propref{prop::existence} for simplicity
mainly; we expect that the proposition will hold in general when $\dom$ has
compact closure or even when $\dom=\Rd$, see
\appref{sec:proof:zero-variance} for examples including that of
Gaussian mixture distributions.

For illustration, the contour plot of $V$ and the flowlines of  $\dynb=\nabla V$ are shown in \figref{fig::torus} in a simple example in a two-dimensional torus where {$\rho_0 = 1$ and $\rho_1$ is a mixture density with 3 modes}; their explicit expressions are given in \eqref{eqn::torus_eg}; in this example, we solved \eqref{eqn::poisson} numerically with $\sigmamat=1$, see \appref{subsec::eg_poisson} for more details.  Some other examples where the zero-variance dynamics is  explicit are discussed in \appref{app::eg}.

\newparagraphno{Connection to transport maps and normalizing flows}

	The zero-variance dynamics  provides a transport map $\mapT$ from $\rho_0$ to $\rho_1$, as shown in:
	\begin{proposition}[Existence of a perfect generator]
		\label{prop::existence_generator}
		Suppose $\sigmamat = 1$ for simplicity. Under the  same assumption as in \propref{prop::existence}, let $V$ be the Morse function solving \eqref{eqn::poisson} and $\dynb = \nabla V$ the associated zero-variance dynamics. Then there exists a continuously differentiable function $\timeT$ (defined almost everywhere on $\dom$) such that
		\begin{align}
			\label{eqn::optimal_time}
			\int_{-\infty}^{0} \rho_0(\state{s}{x}) \jacoarg{s}{x}\ud s = \int_{-\infty}^{\timeT(x)} \rho_1(\state{s}{x})\jacoarg{s}{x}\ud s.
		\end{align}
		Furthermore, the map $\mapT(x) := \state{\timeT(x)}{x}$ is a transport map from $\rho_0$ to $\rho_1$, i.e., $\mapT\#\rho_0 = \rho_1$.
	\end{proposition}

The proof is given in  \appref{subsec::proof::existence::generator}.
{Note that  we consider again $\dynb=\nabla V$ on the torus for technical simplicity: the statement of the proposition should hold in general for a zero-variance dynamics $\dynb$.}
The solution of~\eqref{eqn::optimal_time} is particularly simple in one-dimension, where we can take $\dynb(x) = 1$, and straightforwardly verify that $\timeT(x) = \mapT(x) - x$  with
\begin{equation}
    \label{eq:1d:map}
    \mapT(x) = F_1^{-1}(F_0(x)) \qquad \text{where} \quad F_i(x)=\int_{-\infty}^x \rho_i(y) \ud y, \ \ i=0,1.
\end{equation}
We also illustrate the statement of \propref{prop::existence_generator}
via numerical examples in \appref{subsec::generator-examples}.

To avoid confusion, we stress that we will \textit{not} use the transport map $\mapT$ of Proposition~\ref{prop::existence_generator} in the examples below.
Indeed, using this map would require identifying $\timeT$, which introduces an unnecessary additional calculation which we can avoid using the NEIS estimator directly. In addition, the NEIS estimator will likely have better properties than those based on transport maps, as we can we can think of NEIS as using a time-parameterized family of transport maps rather than a single one.
In particular, the variance of the NEIS estimator will be small  if samples likely on $\rho_0$ become likely on $\rho_1$ \textit{sometime} along the NEIS trajectories,  rather than at the same fixed time for all samples. The former seems easier to fulfill than the latter. For example, in one-dimension, the NEIS estimator has zero variance for any $\dynb$ bounded away from zero, whereas building a transport map from $\rho_0$ to $\rho_1$ is already nontrivial in that simple case since it requires solving~\eqref{eq:1d:map}.

\section{Variational formulations}
\label{sec::variational}

The Poisson equation~\eqref{eqn::poisson} admits a variational formulation:
\begin{align}
\label{eq:var:poisson}
	\begin{aligned}
		&\min_{V} \int_{\dom} \tfrac12 |\nabla V|^2 \sigmamat + V (\rho_1 - \rho_0).
	\end{aligned}
\end{align}
If $\sigmamat$ is chosen to be a probability density function (for example
$\sigmamat = \rho_0$ or $\sigmamat= \frac12 (\rho_1+\rho_0)$), the two terms in
the objective in~\eqref{eq:var:poisson} are expectations which can be estimated
via sampling (using \eg{}, direct sampling for the expectation with
respect to $\rho_0$ and NEIS for the one with respect to $\rho_1$).
This means that we can in principle use an MCMC estimator of~\eqref{eq:var:poisson} as  empirical loss, and minimize it over all $V$ in some parametric class.
Here however, we will follow a different strategy that allows us  to directly parametrize $\dynb$  instead of $V$ (i.e. relax the requirement that $\dynb=\nabla V$) and simply use the variance of the estimator as objective function.

Specifically, since we  quantify the performance of the estimators based on \eqref{eqn::noneq_sample::finite} and \eqref{eqn::noneq_sample::infinite} by their variance, we can view these quantities as functionals of $\dynb$ that we wish to minimize. Since the estimators are unbiased, these objectives are
\begin{align}
\label{eqn::var}
    \begin{aligned}
    \newvarfin(\dynb) &= \newmsecfin(\dynb) - \partition_1^2,\ & \text{(finite-time)};\\
    \newvarinf(\dynb) &= \newmsecinf(\dynb) - \partition_1^2,\ & \text{(infinite-time)},
\end{aligned}
\end{align}
where we defined the second moments $\newmsecfin(\dynb) := \ee_{0} \big[|\newafin|^2\big]$ and $\newmsecinf(\dynb) := \ee_{0} \big[|\newainf|^2\big]$.
With the finite-time objective,  we know that with $\dynb = \vectorzero_{\dimn}$, \eqref{eqn::noneq_sample::finite} reduces to \eqref{eqn::direct_importance}.
Therefore minimizing $\newmsecfin(\dynb) $ over $\dynb$ by gradient descent starting from $\dynb$ near $\vectorzero_{\dimn}$ will necessarily produce a better estimator: while we cannot guarantee that the variance of this optimized estimator will be zero, the experiments conducted below indicate that it can be {a several} order of magnitude below that of the vanilla estimator.

For the infinite-time objective, we know that for any $\dynb$, \eqref{eqn::ainf} leads to an estimator with a lower variance than the one based on~\eqref{eqn::direct_importance} \cite{rotskoff_dynamical_2019}.  Minimizing $\newvarinf(\dynb) $ over $\dynb$ using gradient descent leads to a local minimum; the next result shows that all such local minima are global:

\begin{proposition}[Global minimum]
	\label{prop::local}
	Under some technical assumptions listed in~\propref{prop::local::full}, if  $\dynb_* \in \mani{}$ is a  local minimum of $\newvarinf(\dynb)$ where  the functional derivative of  $\newvarinf(\dynb)$ with respect to $\dynb$ vanishes, \ie{}, $\delta \newvarinf(\dynb_*)/\delta  \dynb = \vectorzero_{\dimn}$ on $\Omega$, 	then $\dynb_* $ is a global minimum and  $\newvarinf(\dynb_*) = 0$.
\end{proposition}

The expression of the functional derivative $\delta \newvarinf(\dynb_*)/\delta  \dynb $ is given in  \propref{prop::stationary_inf}.
The technical assumptions under which \propref{prop::local} hold are explained in \appref{sec::space_inf} and the proof is given in \appref{sec::proof::local}.

\section{Training towards the optimal \texorpdfstring{$\dynb$}{b}}
\label{sec::numerics}

Here we discuss how to use deep learning techniques to find the optimal $\dynb$; these techniques will be illustrated  on numerical examples in Section~\ref{sec::numerical_experiments}. Some technical details are deferred to \appref{app::numerics_details}.

\newparagraph{Objective}
We use the finite-time objective $\newmsecfinarg{\tm}{\tp}(\dynb)$
in \eqref{eqn::noneq_sample::finite} with
{$\tm \in [-1,0], \tp = \tm+1$.}
{Two natural choices are $\tm = 0$ and $\tm=-1/2$, which will be used below  in the numerical experiments.}
This leads to no loss of generality \textit{a~priori} since in the training scheme we put no restriction on the magnitude that $\dynb$ can reach, and with large $\dynb$ the flow line can travel a large distance even  during $t\in[-1,1]$ (the range of integration in $s,t$ in~\eqref{eq::afin});
see the discussion in Appendix~\ref{subsec::relation} for more details.
In practice, we use a time-discretized version of~\eqref{eq::afin} with {$2N_t$} discretization points, and use the standard Runge-Kutta scheme of order 4 (RK4) to integrate the ODE~\eqref{eqn::ode} over $t\in[-1,1]$ using {uniform time step ($\Delta t = 1/N_t$)}.
{We note that this numerical discretization introduces a bias. However, this bias can be systematically controlled by changing the time step or using higher order integrator.  In our experiments, we observed that the RK4 integrator led to  negligible errors, see Table~\ref{table::eg}.}

\newparagraph{Neural architecture}
In our experiments, we either parameterize $\dynb$ by a neural network directly, or we assume that $\dynb$  is a gradient field,
\begin{align*}
	\begin{aligned}
		\dynb = \nabla V & \qquad \text{(gradient form)},\\
	\end{aligned}
\end{align*}
and parameterize the potential $V$ by a neural network.
We always use an $\ell$-layer neural network with width $m$ for all inner layers; therefore, from now on, we simply refer the neural network structure by a pair $(\ell,\mode)$;
{see \appref{subsec::numerical_details} for more details}.
When we parametrize the potential function $V$ instead of $\dynb$, the only difference is that the output dimension of the neural network becomes $1$ instead of $\dimn$.
The activation function is chosen as the \myemph{softplus} function (a smooth version of ReLU) that gave better empirical results compared to the sigmoid function.
At initialization the neural parameters were randomly generated.
Theoretical results about the gradient of $\newmsecfinarg{\tm}{\tp}(\dynb)$ with respect to parameters are given in \appref{sec::optimal_flow_finite} and corresponding numerical implementations are explained in \appref{app::numerics_details}.

\newparagraph{Direct training method}
We minimize $\newmsecfinarg{\tm}{\tp}(\dynb)$ with respect to the parameters in the neural network using stochastic gradient descent (SGD) in which we evaluate the loss and its gradient empirically using mini-batches of data drawn from $\rho_0$ at every iteration step.

\newparagraph{Assisted training method}
When local minima of $U_1$ are far away from the local minimum of $U_0$ at $x=0$, the direct training method by sampling data from $\rho_0$ and minimizing $\newmsecfinarg{\tm}{\tp}(\dynb) $ fails, because
{the flowlines may not reach the importance region of $\rho_1$ due to poor initializations of $\dynb$. More specifically, if along almost all trajectories, $e^{-U_1(\state{s}{x})} \approx 0$ for $s\in [-1,1]$, then with large probability $\newainf_{\tm,\tp}(x)\approx 0$ where $x\sim \rho_0$; as a result, the empirical variance of the estimator can  be extremely small if the number of samples is small, while the true variance could be extremely large.}
Such a \emph{mode collapse} phenomenon is common in rare event simulations.

To get around this difficulty, recall that ideally we would like to find a dynamics $\dynb$ such that $\newainf$ is approximately a constant function in the infinite-time case. That is, if $\dynb$ is a zero-variance dynamics, then $\newainf$ is  a constant function and $\ee_p\big[(\newainf - (\ee_p\newainf)^2)\big] = 0$ for \textit{any} distribution $p$, and we are not constrained to use the base distribution $\rho_0$ and minimize the functional $\dynb\to\var(\dynb)$ in \eqref{eqn::var}.
Motivated by this idea, we use an assisted training scheme in which, at iteration $i$ of SGD, the loss function is
\begin{equation}
    \label{eq:biasedloss}
    \ee_{p_i}\Big[\big(\newainf_{\tm,\tp} - \ee_{p_i}\newainf_{\tm,\tp}\big)^2\Big].
\end{equation}
Here $\ee_{p_i}$ denotes expectation with respect to the probability density $p_i$ defined as
\begin{align}
	\label{eqn::bias}
	\begin{aligned}
	p_i &=  (1- c_i) \rho_0 + c_i \bm{Z} {\#} \rho_0,\qquad
	c_i = \max\big\{c - i \frac{c}{\upsilon L}, 0\big\},
	\end{aligned}
\end{align}
where $\upsilon\in (0,1)$ controls the number of steps during which the training is assisted, $c\in (0,1)$, $L$ is the total number of training steps, and $\bm{Z} := \bm{Z}_{t=1}$ is the time-1 map of the ODE $\Dot{\bm{Z}}_{t}(x) = -\varsigma\nabla U_1\big(\stateD{t}{x}\big)$ with $\stateD{t=0}{x} = x$ and $\varsigma >0$ is a parameter.
In essence, using \eqref{eq:biasedloss} means that, for the first $\upsilon L$ training steps, there is a probability $c_i$ that the data $x\sim \rho_0$ are replaced by $\bm{Z}(x)$, so that the training method can better explore important regions near local minima of $U_1$. Subsequently, the assistance is turned off so that some subtle adjustment can be made. If some samples from $\rho_1= e^{-U_1}/\partition_1$ were available beforehand, we could equivalently replace $\bm{Z} {\#} \rho_0$ in~\eqref{eqn::bias} by the empirical distribution over these samples.
{We emphasize that the assisted training method is only used to guide the training initially and the NEIS estimator for $\ratio$ is unbiased.}

\section{Numerical experiments}
\label{sec::numerical_experiments}

We consider three benchmark examples to illustrate the effectiveness of NEIS assisted with training.
The first two examples involve Gaussian mixtures, for which we use NEIS with $\tm=0$; the third example is Neal's funnel distribution, for which we use NEIS with $\tm=-1/2$.
In all examples, we compare the performance of NEIS with those of  annealed importance sampling (AIS) \cite{neal_annealed_2001}; the number of transition steps in AIS is denoted as $K$ and we refer to this method as AIS-$K$ below; for more details see~\appref{subsec::ais}.
For the comparison, we choose to record the query costs to $U_1$ and $\nabla U_1$ as the measurement of computational complexity,
which connects to the framework in the \emph{theory of information complexity} (see \eg{}, \cite{novak_deterministic_1988}).
The runtime could depend on coding, machine condition, etc., whereas query complexity more or less only depends on the computational problem ($U_0$, $U_1$ and $\dynb$) itself;
for most examples of interest, $U_0$ is simple whereas $U_1$ and its derivatives will be expensive to compute.

{For simplicity, we always use as base density $\rho_0(x) = (2\pi)^{-\dimn/2} e^{-\frac12 \abs{x}^2}$.
We remark that a better choice of $\rho_0$ (i.e., more adapted to $\rho_1$) can significantly improve the sampling performance; our $\rho_0$ is precisely used to validate the performance of NEIS in situations where $\rho_0$ is \textit{not} well chosen.
It would be interesting to study how to adapt the choice of $\rho_0$ for easier training in NEIS, but this is left for future investigations.}

When presenting results, we rescale the estimates so that the exact value is $\ratio=1$ for all examples.
More implementation details about training are deferred to \appref{subsec::numerical_details}.
All trainings and estimates of $\ratio$ are conducted on a laptop with CPU i7-12700H; we use 15 threads at maximum.
The runtime of training ranges approximately between 40$\sim$80 seconds for Gaussian mixture (2D), 9.5$\sim$12 minutes for Gaussian mixture (10D); 
for the Funnel distribution (10D), the runtime is around 25 minutes for a generic linear ansatz and around 2 minutes for a two-parametric ansatz.
\appref{subsec::train_figure} includes additional figures about training.
When computing the gradient of the variance with respect to parameters, we
use an integration-based method when $\tm=-1/2$ (for the convenience of numerical implementation) and use an ODE-based method when $\tm=0$ (for higher accuracy);
details about these two approximation methods are given in \appref{subsec::integral_method} and \appref{subsec::ode_method} respectively.
The codes are accessible on
\url{https://github.com/yucaoyc/NEIS}.

\newparagraph{An asymmetric $2$-mode Gaussian mixture in 2D}
As a first illustration, we consider an asymmetric $2$-mode Gaussian mixture
	\begin{align}
		\label{eqn::eg_two_mode}
		\begin{aligned}
			e^{-U_1}\ &\propto\ \frac{1}{5} \gauss{\lambda \basis_1}{\sigma_1^2 \idmat_{\dimn}} + \frac{4}{5} \gauss{-\lambda \basis_2}{\sigma_1^2 \idmat_{\dimn}},
		\end{aligned}
\end{align}
where $\basis_1 = \begin{bsmallmatrix}1 \\ 0 \end{bsmallmatrix}$, $\basis_2 = \begin{bsmallmatrix}0 \\ 1\end{bsmallmatrix}$, $\sigma_1 = \sqrt{0.1}$, $\lambda = 5$.
With this choice of parameters, the variance of the vanilla estimator based on \eqref{eqn::direct_importance} is approximately $1.85\times 10^6$.
We use NEIS with $\tm=0$ and set the time step to $\Delta t= 1/50$ for ODE discretization during both training and estimation of $\ratio$.
We train over $L=50$ SGD steps using the loss~\eqref{eq:biasedloss} by imposing bias in the first $60\%$ the training period only (\ie{}, with $\upsilon=0.6$).
The evolution of the variance during the training is shown in \figref{fig::train::1} in Appendix; 
{the best optimized flow has a variance of about $1$, as opposed to $10^6$ for the vanilla estimator.}
Since $p_i$ in \eqref{eqn::bias} is quite different from $\rho_0$ during the assisted learning period, it may happen that the empirical variance significantly exceeds the variance of the vanilla importance sampling; this does not contradict with \propref{prop::var_dynb} below.
As seen in \figref{fig::train::1}, at the end of the assisted period,
the variance is already quite small and in most cases, the variance continues to reduce as $\dynb$ gets further optimized.

After training, 
we estimate $\ratio$ using NEIS with the optimized flow and compare its performance with AIS-10 and AIS-100.
We first record the query cost for training and then set a total number of queries to $U_1, \nabla U_1$ as budgets.
Given the query budget, we estimate $\ratio$ using each method $10$ times, leading to the results given in Table~\ref{table::eg} below. 
When we determine the estimation cost of NEIS, we deduct the query cost of training from the total query budget for fairer comparison.
Note that NEIS uses less total queries to produce more accurate estimate of $\ratio$: in particular, the standard deviation of estimating $\ratio$ by NEIS method is around 1 magnitude smaller than AIS-100.
Moreover, the bias from ODE discretization appears to be negligible.
\figref{fig::cmp} shows an optimized flow and also provides a visual comparison of NEIS with AIS under fixed query budget; more comparisons using various ansatzes or architectures can be found in Figures~\ref{fig::cmp::1::generic} and \ref{fig::cmp::1::grad}.

\begin{figure}[h!]
	\centering
	\subfigure[Asymmetric Gaussian mixture (2D): gradient ansatz, $\ell=2$, $\mode=20$]{\includegraphics[width=0.9\textwidth]{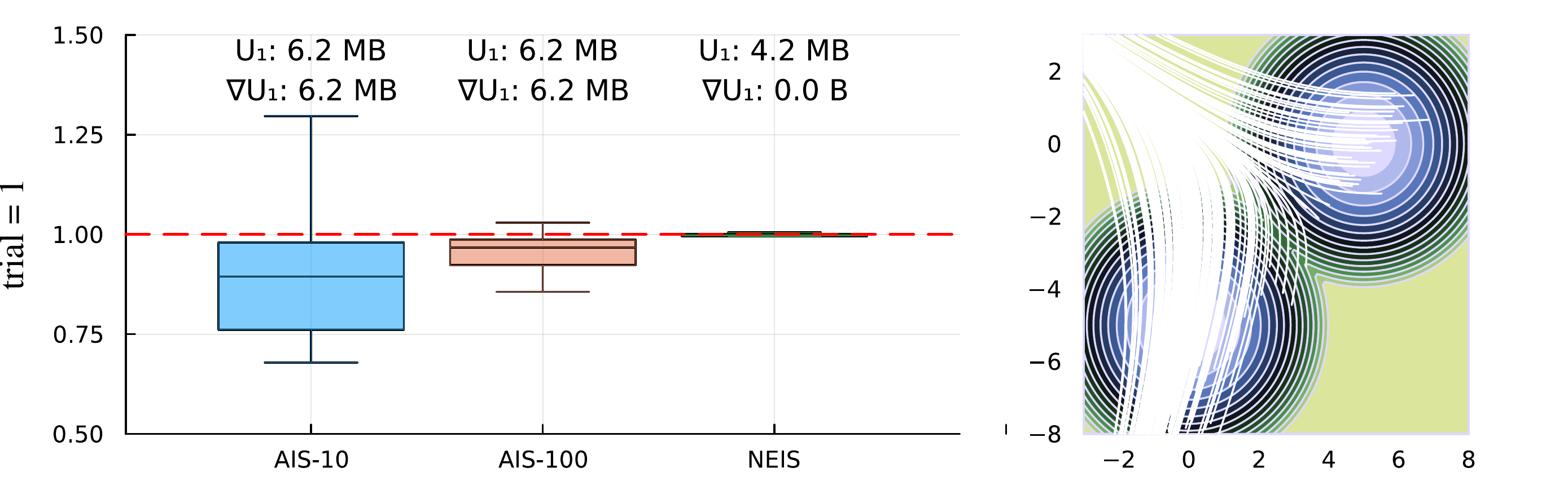}}
	\subfigure[Symmetric Gaussian mixture (10D): gradient ansatz, $\ell=2$, $\mode=30$]{\includegraphics[width=0.9\textwidth]{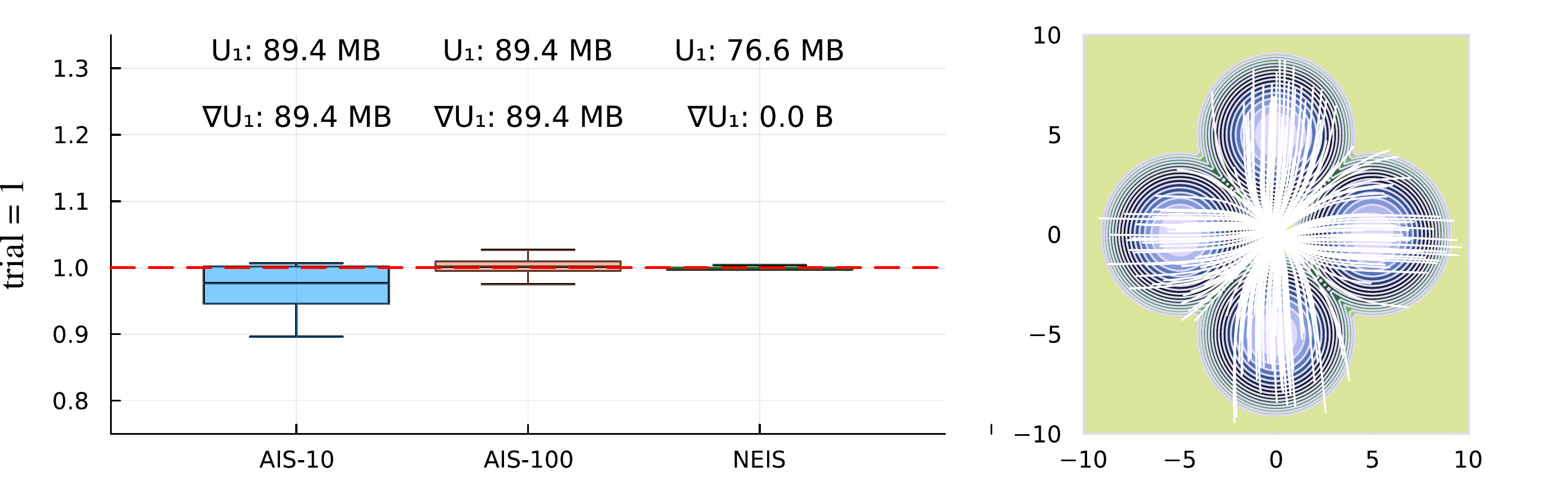}}
	\subfigure[Funnel distribution (10D): generic linear ansatz]{\includegraphics[width=0.9\textwidth]{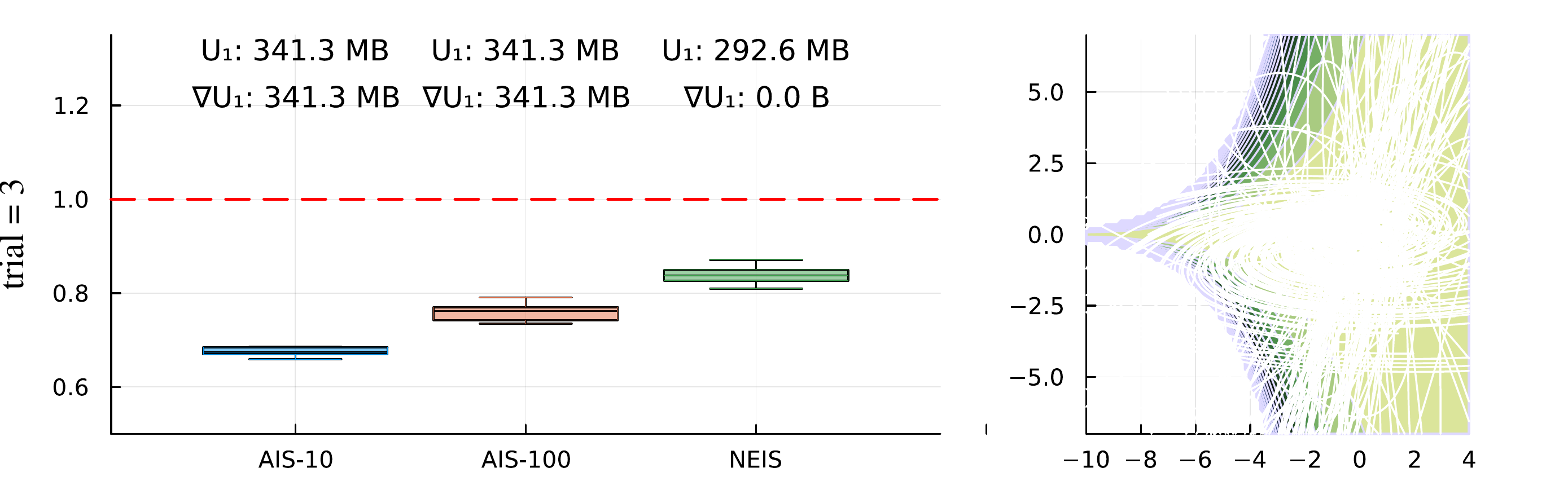}}
	\subfigure[Funnel distribution (10D): two-parametric ansatz]{\includegraphics[width=0.9\textwidth]{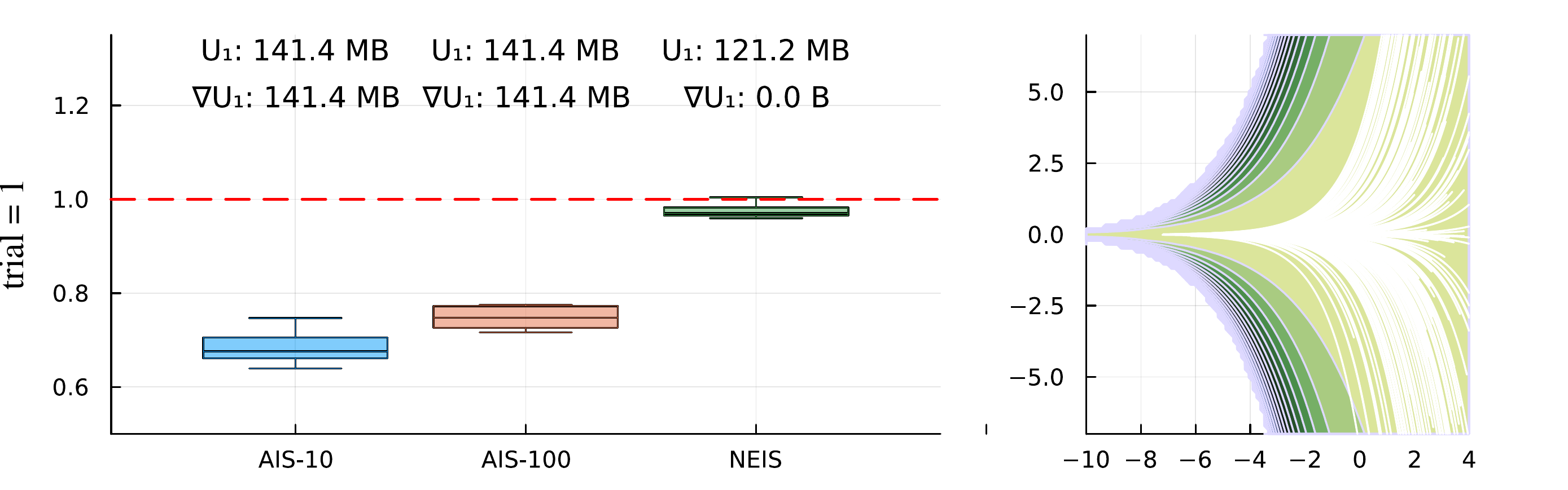}}
	\caption{Selective comparison results for various models. Left panels: estimates of  $\ratio$ by AIS and NEIS with optimized flow under the fixed query budget shown above the panels; we repeat these calculations $10$ times and show boxplots of these $10$ estimates for each method. The trial number refers to the index for randomly chosen initialization. The query numbers refer to the queries used for each estimate of $\ratio$ for each method; $1 \text{MB} = 10^6$. The dashed red lines show the exact value $\ratio = 1$. Right panels: streamlines of optimized flows atop the contours of $U_1$, both projected into the $x_1$-$x_2$ plane for the two 10D examples.  Full comparison and figures can be found in Figures~\ref{fig::cmp::1::generic},\ref{fig::cmp::1::grad},\ref{fig::cmp::2::grad}, and \ref{fig::cmp::3::ln} in Appendix.}
	\label{fig::cmp}
\end{figure}

\newparagraph{A symmetric $4$-mode Gaussian mixture in 10D}
Next we consider a symmetric $4$-mode Gaussian mixture in $d=10$ dimension with 
\begin{align}
		\label{eqn::eg3}
		e^{-U_1(x)}\ &\propto\  
		\sum_{i=1}^4 \gauss{\mu_i}{\Sigma},
\end{align}
where the vector $\mu_i = \begin{bsmallmatrix}
	\lambda \cos(\frac{i\pi}{2}) &
	\lambda \sin(\frac{i\pi}{2}) & 
	0 &
	0 & 
	\cdots & 
	0 \\
\end{bsmallmatrix}$ 
and 
$\Sigma = \text{Diag}\begin{bsmallmatrix} \sigma_1^2 & \sigma_1^2 & \sigma_2^2 & \sigma_2^2 & \cdots & \sigma_2^2 \end{bsmallmatrix}$ is a diagonal matrix.
The parameters are 
$\dimn = 10$, 
$\sigma_1 = \sqrt{0.1}$, $\sigma_2 = \sqrt{0.5}$ and $\lambda = 5$.
With this choice of parameters, the variance of the vanilla estimator based on \eqref{eqn::direct_importance}  is {approximately $2.15\times 10^6$}.
We use NEIS with $\tm=0$ and the time step $\Delta t= 1/60$ is used for ODE discretization during both training and estimation of $\ratio$.
We show the training result in \figref{fig::train::2}. 
The variance reduces to about $10$ after $60$ SGD steps for the gradient ansatz (here we only considered this ansatz as it produces more promising empirical results).

Similar to the last example, we compare NEIS using the optimized $\dynb$ with AIS, under fixed query budgets; see Table \ref{table::eg}. The best result from NEIS has an estimator with the standard deviation less than $1/3$ of the one from AIS-100,
even though NEIS has less query cost. 
This comparison suggests that AIS-100 needs more than $9$ times more resources than NEIS with optimized flow in order to achieve similar precision and the cost spent on training indeed pays off if we require an accurate estimate of $\ratio$ (meaning less fluctuation for Monte Carlo estimates). Moreover, this table also shows that the bias from ODE discretization is negligible.

Figure~\ref{fig::cmp} shows a particular optimized flow: as can be seen, the mass near the origin flows towards four local minima of $U_1$, as we would intuitively expect. More optimized flows and comparisons can be found in \figref{fig::cmp::2::grad} in Appendix.

\newparagraph{A funnel distribution in 10D} 
We consider the following 10D funnel distribution studied in \cite{arbel_annealed_2021,neal_slice_2003}:
for the state $x=[x_1,x_2,\ldots, x_{10}]\in\Real^{10}$, 
\begin{align*}
	x_1 \sim \gauss{0}{9}, \qquad x_i \sim \gauss{0}{e^{x_1}}, \qquad 2\le i\le \dimn.
\end{align*}
For numerical stability, we consider the above funnel distribution restricted to a unit ball centered at the origin with radius $25$.
Instead of heuristically parameterizing the dynamics via neural-networks, we consider a generic linear ansatz and a two-parametric linear ansatz:
\begin{subequations}
\begin{align}
	\dynb(x) &= W_1 x + b_1,\qquad &  W_1\in \Real^{10\times 10}, b_1\in\Real^{10},\label{eqn::ln} \\
	\dynb(x) &= -[
		\beta, \alpha x_2, \alpha x_3, \ldots, \alpha x_{10}],\qquad	& \alpha,\beta\in\Real.\label{eqn::ln_ansatz}
\end{align}	
\end{subequations}
The generic linear ansatz can be regarded as a neural network without inner layers.
With~\eqref{eqn::ln}, we drawn the entries in the matrix $W_1$ randomly  and we set $b_1=\vectorzero_{10}$ initially, and we use the assisted training method; with~\eqref{eqn::ln_ansatz},
we set $\alpha=\beta = 2$ initially, and we use the direct training method. 
In both cases, we choose the finite-NEIS scheme with $\tm=-1/2$. 
We notice that the asymmetric choice $\tm=0$ can also leads into more optimal dynamics, but its performance is not as competitive as the symmetric case $\tm=-1/2$. It is very likely that such a difference is due to the structure of funnel distribution: each coordinate $x_i$ ($1\le i\le 10$) has mean $0$ and therefore, a symmetric version can probably better weight the contribution from both forward and backward flowlines.

The training results are shown in Figures~\ref{fig::train::3::ln} and \ref{fig::train::3::ln_2} in Appendix. In \figref{fig::train::3::ln_2}, we can observe that both error and variance are overall decreasing during the training and the parameters $\alpha, \beta$ tend to increase with a similar speed.
We use the same protocol as in the two previous examples to compare NEIS with AIS.
As can be be observed in Table~\ref{table::eg}, the two-parametric ansatz~\eqref{eqn::ln_ansatz} gives the best estimate; the generic linear ansatz~\eqref{eqn::ln} is not as competitive as the two-parametric ansatz (probably due to over-parameterization), but it still outperforms the AIS-100, under fixed query budget. 
Figure~\ref{fig::cmp} shows these optimized flows; more results can be found in Figure~\ref{fig::cmp::3::ln} in Appendix.
The apparent gap between estimates and the ground truth in Figure~\ref{fig::cmp} (or see Table~\ref{table::eg}) comes from insufficient sample size.

\section{Conclusion and outlook}
\label{sec::discussion}

In this work, we revisited the NEIS strategy proposed in~\cite{rotskoff_dynamical_2019} and analyzed its capabilities, both from  theoretical  and computational standpoints.
Regarding the former, we showed that NEIS leads to a zero-variance estimator for a velocity field $\dynb=\nabla V$ with a potential $V$ that satisfies a certain Poisson equation with the difference between the target and the base density as source.
Moreover, a zero-variance dynamics can be used to construct a transport map from $\rho_0$ to $\rho_1$.
In turn, we highlighted the connection and difference between NEIS and importance sampling strategies based on the normalizing flows (NF).

On the computational side, we showed that the variance of the NEIS estimator can be used as objective function to train the velocity field $\dynb$. This training procedure can be performed in practice by approximating the velocity field by a neural network, and optimizing the parameters in this network using SGD, similar to what is done in the context of neural ODE but with a different objective. Our numerical experiments showed that this strategy is effective
and can lower the variance of a vanilla estimator for $\partition_1$ by several orders of magnitude.

While the numerical examples we used in the present paper are somewhat academic, the results suggest that NEIS has potential in more realistic settings. In order to explore other applications, it would be interesting to investigate how to best parametrize $\dynb$ (e.g., less parameters and non-stiff energy landscape with respect to these parameters) and how to best initiate the training procedure. It would also be interesting to ask whether we can improve the performance of NEIS by optimizing certain parameters in the base density $\rho_0$ in concert with $\dynb$. The answers to these questions are probably  model specific and are left for future work.

 \section*{Acknowledgment} We would like to thank Jonathan Weare and
Fang-Hua Lin for helpful discussions, and the anonymous referees for their useful comments and suggestions. The work of EVE is supported by the National Science Foundation  under awards DMR-1420073, DMS-2012510, and DMS-2134216, by the Simons Collaboration on Wave Turbulence, Grant No. 617006, and by a Vannevar Bush Faculty Fellowship.

 \newpage
 \bibliographystyle{plain}
 \bibliography{reference}

\begin{bibdiv}
\begin{biblist}

\bib{andrieu_sampling_2018}{misc}{
      author={Andrieu, Christophe},
      author={Ridgway, James},
      author={Whiteley, Nick},
       title={Sampling normalizing constants in high dimensions using
  inhomogeneous diffusions},
        date={2018},
        note={arXiv: 1612.07583},
}

\bib{arbel_annealed_2021}{inproceedings}{
      author={Arbel, Michael},
      author={Matthews, Alex},
      author={Doucet, Arnaud},
       title={Annealed flow transport {Monte Carlo}},
        date={2021},
   booktitle={Proceedings of the 38th international conference on machine
  learning},
      editor={Meila, Marina},
      editor={Zhang, Tong},
      series={Proceedings of Machine Learning Research},
      volume={139},
   publisher={PMLR},
       pages={318\ndash 330},
}

\bib{armijo_minimization_1966}{article}{
      author={Armijo, Larry},
       title={Minimization of functions having {Lipschitz} continuous first
  partial derivatives.},
        date={1966},
     journal={Pac. J. Math.},
      volume={16},
      number={1},
       pages={1\ndash 3},
}

\bib{Audin2014}{book}{
      author={Audin, Mich{\`{e}}le},
      author={Damian, Mihai},
       title={{Morse Theory and Floer Homology}},
   publisher={Springer London},
        date={2014},
}

\bib{Beckmann_1952_continuous}{article}{
      author={Beckmann, Martin},
       title={A continuous model of transportation},
        date={1952},
     journal={Econometrica},
      volume={20},
      number={4},
       pages={643\ndash 660},
}

\bib{bennett1976efficient}{article}{
      author={Bennett, Charles~H},
       title={Efficient estimation of free energy differences from {Monte
  Carlo} data},
        date={1976},
     journal={J. Comput. Phys.},
      volume={22},
      number={2},
       pages={245\ndash 268},
}

\bib{braunling_2004}{misc}{
      author={Br{\"a}unling, O.},
       title={Fourier series on the n-dimensional torus},
        date={2004},
        note={\url{https://www.uni-math.gwdg.de/schick/teach/jvnfour1.pdf}},
}

\bib{brooks_handbook_2011}{book}{
      author={Brooks, Steve},
      author={Gelman, Andrew},
      author={Jones, Galin},
      author={Meng, Xiao-Li},
       title={Handbook of {Markov} {Chain} {Monte} {Carlo}},
   publisher={CRC Press LLC},
        date={2011},
}

\bib{chen_neural_2018}{inproceedings}{
      author={Chen, Ricky T.~Q.},
      author={Rubanova, Yulia},
      author={Bettencourt, Jesse},
      author={Duvenaud, David~K},
       title={Neural ordinary differential equations},
        date={2018},
   booktitle={Advances in {Neural} {Information} {Processing} {Systems}},
      volume={31},
   publisher={Curran Associates, Inc.},
}

\bib{e_barron_2019}{misc}{
      author={E, Weinan},
      author={Ma, Chao},
      author={Wu, Lei},
       title={The {Barron} space and the flow-induced function spaces for
  neural network models},
        date={2019},
        note={arXiv: 1906.08039},
}

\bib{evans_pde_2010}{book}{
      author={Evans, Lawrence~C.},
       title={{Partial Differential Equations}},
   publisher={American Mathematical Society},
        date={2010},
}

\bib{feroz_multimodal_2008}{article}{
      author={Feroz, F.},
      author={Hobson, M.~P.},
       title={Multimodal nested sampling: an efficient and robust alternative
  to {Markov} {Chain} {Monte} {Carlo} methods for astronomical data analyses},
        date={2008},
     journal={Mon. Not. R. Astron. Soc.},
      volume={384},
      number={2},
       pages={449\ndash 463},
}

\bib{gelman_simulating_1998}{article}{
      author={Gelman, Andrew},
      author={Meng, Xiao-Li},
       title={Simulating normalizing constants: from importance sampling to
  bridge sampling to path sampling},
        date={1998},
     journal={Statist. Sci.},
      volume={13},
      number={2},
       pages={163\ndash 185},
        note={Number: 2},
}

\bib{geyer_importance_2011}{incollection}{
      author={Geyer, Charles~J.},
       title={Importance sampling, simulated tempering, and umbrella sampling},
        date={2011},
   booktitle={Handbook of {Markov Chain Monte Carlo}},
      editor={Brooks, Steve},
      editor={Gelman, Andrew},
      editor={Jones, Galin},
      editor={Meng, Xiao-Li},
   publisher={CRC press},
}

\bib{jarzynski_equilibrium_1997}{article}{
      author={Jarzynski, C.},
       title={Equilibrium free-energy differences from nonequilibrium
  measurements: {A} master-equation approach},
        date={1997},
     journal={Phys. Rev. E},
      volume={56},
      number={5},
       pages={5018\ndash 5035},
}

\bib{jarzynski_nonequilibrium_1997}{article}{
      author={Jarzynski, C.},
       title={Nonequilibrium equality for free energy differences},
        date={1997},
     journal={Phys. Rev. Lett.},
      volume={78},
      number={14},
       pages={2690\ndash 2693},
}

\bib{kirkwood1935statistical}{article}{
      author={Kirkwood, John~G},
       title={Statistical mechanics of fluid mixtures},
        date={1935},
     journal={J. Chem. Phys.},
      volume={3},
      number={5},
       pages={300\ndash 313},
}

\bib{klenke_2014_probability}{book}{
      author={Klenke, Achim},
       title={{Probability Theory: A Comprehensive Course}},
        date={2014},
}

\bib{kobyzev_normalizing_2020}{article}{
      author={Kobyzev, Ivan},
      author={Prince, Simon J.~D.},
      author={Brubaker, Marcus~A.},
       title={Normalizing flows: {An} introduction and review of current
  methods},
        date={2020},
     journal={IEEE Trans. Pattern Anal. Mach. Intell.},
       pages={1\ndash 1},
}

\bib{lax_change_2001}{article}{
      author={Lax, Peter~D.},
       title={Change of variables in multiple integrals {II}},
        date={2001},
     journal={Am. Math. Mon.},
      volume={108},
      number={2},
       pages={115\ndash 119},
}

\bib{lifshitz2013statistical}{book}{
      author={Lifshitz, Evgenii~Mikhailovich},
      author={Pitaevskii, Lev~Petrovich},
       title={{Statistical Physics: Theory of the Condensed State}},
   publisher={Elsevier},
        date={2013},
      volume={9},
}

\bib{liu2001monte}{book}{
      author={Liu, Jun~S},
       title={{Monte Carlo Strategies in Scientific Computing}},
   publisher={Springer},
        date={2001},
      volume={10},
}

\bib{lu_methodological_2019}{article}{
      author={Lu, Jianfeng},
      author={Vanden-Eijnden, Eric},
       title={Methodological and computational aspects of parallel tempering
  methods in the infinite swapping limit},
        date={2019},
     journal={J. Stat. Phys.},
      volume={174},
      number={3},
       pages={715\ndash 733},
}

\bib{moral_sequential_2006}{article}{
      author={Moral, Pierre~Del},
      author={Doucet, Arnaud},
      author={Jasra, Ajay},
       title={Sequential {Monte} {Carlo} samplers},
        date={2006},
     journal={J. R. Statist. Soc. B},
      volume={68},
      number={3},
       pages={411\ndash 436},
}

\bib{muller_neural_2019}{article}{
      author={Müller, Thomas},
      author={Mcwilliams, Brian},
      author={Rousselle, Fabrice},
      author={Gross, Markus},
      author={Novák, Jan},
       title={Neural importance sampling},
        date={2019},
     journal={ACM Trans. Graph.},
      volume={38},
      number={5},
       pages={145:1\ndash 145:19},
}

\bib{neal_annealed_2001}{article}{
      author={Neal, Radford~M.},
       title={Annealed importance sampling},
        date={2001},
     journal={Stat. Comput.},
      volume={11},
      number={2},
       pages={125\ndash 139},
}

\bib{neal_slice_2003}{article}{
      author={Neal, Radford~M.},
       title={Slice sampling},
        date={2003},
     journal={Ann. Stat.},
      volume={31},
      number={3},
}

\bib{novak_deterministic_1988}{book}{
      author={Novak, Erich},
       title={Deterministic and {Stochastic} {Error} {Bounds} in {Numerical}
  {Analysis}},
      series={Lecture {Notes} in {Mathematics}},
   publisher={Springer-Verlag},
        date={1988},
}

\bib{noe_boltzmann_2019}{article}{
      author={Noé, Frank},
      author={Olsson, Simon},
      author={Köhler, Jonas},
      author={Wu, Hao},
       title={Boltzmann generators: Sampling equilibrium states of many-body
  systems with deep learning},
        date={2019},
     journal={Science},
      volume={365},
      number={6457},
}

\bib{papamakarios_normalizing_2021}{article}{
      author={Papamakarios, George},
      author={Nalisnick, Eric},
      author={Rezende, Danilo~Jimenez},
      author={Mohamed, Shakir},
      author={Lakshminarayanan, Balaji},
       title={Normalizing flows for probabilistic modeling and inference},
        date={2021},
     journal={J. Mach. Learn. Res.},
      volume={22},
      number={57},
       pages={1\ndash 64},
}

\bib{petersen_neural_2020}{misc}{
      author={Petersen, Philipp},
       title={Neural network theory},
        date={2020},
}

\bib{rezende_variational_2015}{inproceedings}{
      author={Rezende, Danilo~Jimenez},
      author={Mohamed, Shakir},
       title={Variational inference with normalizing flows},
        date={2015},
   booktitle={Proceedings of the 32nd {International} {Conference} on
  {International} {Conference} on {Machine} {Learning} - {Volume} 37},
      series={{ICML}'15},
     address={Lille, France},
       pages={1530\ndash 1538},
}

\bib{rotskoff_dynamical_2019}{article}{
      author={Rotskoff, Grant~M.},
      author={Vanden-Eijnden, Eric},
       title={Dynamical computation of the density of states and {Bayes}
  factors using nonequilibrium importance sampling},
        date={2019},
     journal={Phys. Rev. Lett.},
      volume={122},
      number={15},
       pages={150602},
}

\bib{rudin_1976_principles}{book}{
      author={Rudin, Walter},
       title={{Principles of Mathematical Analysis}},
        date={1976},
}

\bib{santambrogio_dacorogna-moser_2014}{article}{
      author={Santambrogio, Filippo},
       title={A dacorogna-moser approach to flow decomposition and minimal flow
  problems},
        date={2014},
     journal={ESAIM: ProcS},
      volume={45},
       pages={265\ndash 274},
}

\bib{santambrogio_2015_optimal}{book}{
      author={Santambrogio, Filippo},
       title={{Optimal Transport for Applied Mathematicians}},
   publisher={Birkh{\"a}user, Cham},
        date={2015},
}

\bib{skilling_nested_2004}{article}{
      author={Skilling, John},
       title={Nested sampling},
        date={2004},
     journal={AIP Conf. Proc.},
      volume={735},
      number={1},
       pages={395\ndash 405},
}

\bib{skilling_nested_2006}{article}{
      author={Skilling, John},
       title={Nested sampling for general {Bayesian} computation},
        date={2006},
     journal={Bayesian Anal.},
      volume={1},
      number={4},
       pages={833\ndash 859},
}

\bib{tabak_family_2013}{article}{
      author={Tabak, E.~G.},
      author={Turner, Cristina~V.},
       title={A family of nonparametric density estimation algorithms},
        date={2013},
     journal={Commun. Pure Appl. Math.},
      volume={66},
      number={2},
       pages={145\ndash 164},
}

\bib{tabak_density_2010}{article}{
      author={Tabak, Esteban~G.},
      author={Vanden-Eijnden, Eric},
       title={Density estimation by dual ascent of the log-likelihood},
        date={2010},
     journal={Commun. Math. Sci.},
      volume={8},
      number={1},
       pages={217\ndash 233},
}

\bib{thiede_eigenvector_2016}{article}{
      author={Thiede, Erik~H.},
      author={Van~Koten, Brian},
      author={Weare, Jonathan},
      author={Dinner, Aaron~R.},
       title={Eigenvector method for umbrella sampling enables error analysis},
        date={2016},
     journal={J. Chem. Phys.},
      volume={145},
      number={8},
       pages={084115},
}

\bib{thin_neo_2021}{misc}{
      author={Thin, Achille},
      author={Janati, Yazid},
      author={Corff, Sylvain~Le},
      author={Ollion, Charles},
      author={Doucet, Arnaud},
      author={Durmus, Alain},
      author={Moulines, Eric},
      author={Robert, Christian},
       title={{NEO}: Non equilibrium sampling on the orbit of a deterministic
  transform},
        date={2021},
        note={arXiv:2103.10943},
}

\bib{torrie1977nonphysical}{article}{
      author={Torrie, Glenn~M},
      author={Valleau, John~P},
       title={Nonphysical sampling distributions in {Monte Carlo} free-energy
  estimation: Umbrella sampling},
        date={1977},
     journal={J. Comput. Phys.},
      volume={23},
      number={2},
       pages={187\ndash 199},
}

\bib{vaswani_painless_2019}{inproceedings}{
      author={Vaswani, Sharan},
      author={Mishkin, Aaron},
      author={Laradji, Issam},
      author={Schmidt, Mark},
      author={Gidel, Gauthier},
      author={Lacoste-Julien, Simon},
       title={Painless stochastic gradient: Interpolation, line-search, and
  convergence rates},
        date={2019},
   booktitle={Advances in {Neural} {Information} {Processing} {Systems}},
      volume={32},
   publisher={Curran Associates, Inc.},
}

\bib{wirnsberger_targeted_2020}{article}{
      author={Wirnsberger, Peter},
      author={Ballard, Andrew~J.},
      author={Papamakarios, George},
      author={Abercrombie, Stuart},
      author={Racanière, Sébastien},
      author={Pritzel, Alexander},
      author={Jimenez~Rezende, Danilo},
      author={Blundell, Charles},
       title={Targeted free energy estimation via learned mappings},
        date={2020},
     journal={J. Chem. Phys.},
      volume={153},
      number={14},
       pages={144112},
}

\bib{wu_stochastic_2020}{inproceedings}{
      author={Wu, Hao},
      author={Köhler, Jonas},
      author={Noe, Frank},
       title={Stochastic normalizing flows},
        date={2020},
   booktitle={Advances in {Neural} {Information} {Processing} {Systems}},
      editor={Larochelle, H.},
      editor={Ranzato, M.},
      editor={Hadsell, R.},
      editor={Balcan, M.~F.},
      editor={Lin, H.},
      volume={33},
       pages={5933\ndash 5944},
}

\end{biblist}
\end{bibdiv}

 \newpage
 \appendix
\onecolumn
\section{The functional space for the infinite-time case}
\label{sec::space_inf}

\subsection{Assumptions}

For simplicity, we make:
\begin{assumption}
	\label{assump::U}
    \begin{enumerate}[label=(\roman*)]
	\item The domain $\Omega$ is either
	\begin{itemize}[leftmargin=4ex]
		\item an open and connected subset of $\Real^\dimn$ with smooth boundary,
		\item or a $\dimn$-dimensional torus (without boundary).
	\end{itemize}
	\item $U_0,\ U_1\in C^\infty(\dom, \Real)$.
	\item $\partition_0 := \int_{\dom} e^{-U_0} = 1$, $\partition_1 < \infty$ and $\varmax \in (0,\infty)$, where
	\begin{align*}
	   \varmax := \ee_{\rho_0}[e^{-2(U_1 - U_0)}] - (\ratio)^2
	\end{align*}
	is the variance for the  vanilla importance sampling method.
\end{enumerate}
\end{assumption}

\subsubsection*{Notations}
The open ball around $x$ with radius $r$ is denoted as $\ball{r}{x}:=\big\{y\in\dom\ \big\rvert\  \abs{y-x}<r\big\}$.
For any subset $\domsubset \subset \dom$, let $\fhit{\domsubset}(x)$ and $\bhit{\domsubset}(x)$ be the first hitting times to the boundary $\partial\domsubset$ in the forward and backward directions, respectively.
More specifically,
\begin{align}
	\label{eqn::hitting_time}
	\begin{aligned}
		\fhit{\domsubset}(x) &:= \inf\big\{t\ge 0: \state{t}{x}\in \partial\domsubset\big\},\\
		\bhit{\domsubset}(x) &:= \sup\big\{t\le 0: \state{t}{x}\in \partial\domsubset\big\}.
	\end{aligned}
\end{align}

For later convenience, let us denote
\begin{align}
	\label{eqn::ainf_binf}
	\binf(x) := \intreal \testfuncarg{0}{t}{x} \ud t,
\end{align}
and thus
\begin{align*}
	\newainf(x)\ \myeq{\eqref{eqn::ainf}}\ \frac{\intreal \testfuncarg{1}{t}{x} \ud t}{\binf(x)}.
\end{align*}

\subsection{Preliminaries}

\begin{lemma}
	If $\dynb \in \mani{}$ in \eqref{eqn::b_space}, then for any $t, s\in \Real$, $x\in \dom$ and $k = 0,1$,
	\begin{align}
		\label{eqn::testfunc_trans}
		\begin{aligned}
			\jacoargbig{t}{\state{s}{x}} = \jacoarg{t+s}{x}/\jacoarg{s}{x}, \qquad
			\testfuncargbig{k}{t}{\state{s}{x}} = \testfuncarg{k}{t+s}{x}/\jacoarg{s}{x}.
		\end{aligned}
	\end{align}
\end{lemma}
  This lemma can be verified easily by definition and thus the proof is
    omitted. A simple consequence of \eqref{eqn::testfunc_trans} is the
    following result that we also state without proof:

\begin{lemma}
	For any $s\in \Real$ and $x\in \dom$,
	\begin{align}
		\label{eqn::time_trans_ainf}
		\binf\big(\state{s}{x}\big) = \frac{\binf(x)}{\jacoarg{s}{x}},\qquad \newainf\big(\state{s}{x}\big) = \newainf(x),
	\end{align}
	provided that these terms are well-defined.
\end{lemma}

\subsection{The functional space}
For the infinite-time case, we consider the following family of vector
fields denoted as $\maniinf{}$.
\begin{definition}
	\label{defn::maniinf}
	$\maniinf{}$ is a set that contains all $\dynb\in \mani{}$ such
        that $\dom\backslash\effdom(\dynb)$ has Lebesgue measure zero, where
        $\effdom(\dynb)\subset\dom$ is the collection of points $x$ at which
    the functions
	\begin{align}
		\label{eqn::func_cts}
		z\rightarrow\int_{0}^{\infty} \testfuncarg{k}{t}{z}\ud t, \qquad  z\rightarrow\int_{-\infty}^{0} \testfuncarg{k}{t}{z}\ud t
	\end{align}
	are continuous on a local neighborhood of $x$ for $k = 0, 1$.
\end{definition}

We use the notation $\effdom(\dynb)$ because in general, such a subset depends on the choice of $\dynb$.
The main reason behind the above definition is that we need functions in \eqref{eqn::func_cts} to behave nicely almost everywhere.
The integrability of $t\rightarrow \testfuncarg{k}{t}{x}$ depends on the long-term behavior of the flow, which is not easy to characterize in general; thus we simply include the integrability into the assumption.
However, we can indeed expect that the above conditions should  hold for most interesting examples, for instance:

\begin{example}
	If $U_0(x) = \frac{\abs{x}^2}{2}+\gausscst{\dimn}$ and $U_1(x)
        = \frac{\abs{x}^2}{2\sigma^2}$ (so that both $\rho_0$ and $\rho_1$ are
    Gaussian densities),
    one optimal flow is $\dynb(x) =  x$ (cf. \appref{app::eg} below). For this choice,  we can direct compute that when $z\neq \vectorzero_{\dimn}$, $\state{t}{z} = e^t z$ and hence
	\begin{align*}
		\int_{0}^{\infty} \testfuncarg{1}{t}{z}\ud t &= \frac{1}{2} \Big(\frac{\abs{z}^2}{2\sigma^2}\Big)^{-\dimn/2} \int_{\frac{\abs{z}^2}{2\sigma^2}}^{\infty} s^{\frac{\dimn}{2}-1} e^{-s}\ud s,\\
		\int_{-\infty}^{0} \testfuncarg{1}{t}{z}\ud t &= \frac{1}{2} \Big(\frac{\abs{z}^2}{2\sigma^2}\Big)^{-\dimn/2} \int_{0}^{\frac{\abs{z}^2}{2\sigma^2}} s^{\frac{\dimn}{2}-1} e^{-s}\ud s,
	\end{align*}
	and similar expressions hold for $\testfunc{0}$ by letting $\sigma=1$ in above equations.
	Then clearly, $\effdom(\dynb) = \dom\backslash\{\vectorzero_{\dimn}\}$ and such a dynamics belongs to $\maniinf$.
However, the constant function $\dynb = \vectorzero_{\dimn}\notin \maniinf{}$: we can easily verify that e.g., for this dynamics, $\int_{0}^{\infty}\testfuncarg{1}{t}{z}\ud t = \infty$ for any $z$ and hence  $\effdom(\vectorzero_{\dimn}) = \emptyset$; see also \lemref{lem::G} below.
\end{example}

Here are a few immediate properties of the set $\effdom(\dynb)$:
\begin{lemma}
	\label{lem::property_Omega_b}
	We have:
	\begin{enumerate}[label=(\roman*)]
        \item \label{lem::property:open}
            $\effdom(\dynb)$ is an open subset of $\dom$.

        \item \label{lem::property:close}
            The set $\effdom(\dynb)$ is closed under the evolution of the dynamics $\dynb$, i.e., if $\xst\in \effdom(\dynb)$, then $\state{t}{\xst}\in
		\effdom(\dynb)$ for any $t\in \Real$.

    \item \label{lem::property:close2}
        If $\xst\in \overline{\effdom(\dynb)}$, then
        $\state{t}{\xst}\in \overline{\effdom(\dynb)}$ for any $t\in\Real$,
    where $\overline{\effdom(\dynb)}$ is the closure of $\effdom(\dynb)$ in
$\Rd$.

    \item \label{lem::property:cts} $x\rightarrow\intreal \testfuncarg{k}{t}{x}\ud t$ is continuous on $\effdom(\dynb)$.

    \item \label{lem::property:ctsxt}
        $(x,t)\rightarrow \int_{t}^{\infty}\testfuncarg{k}{s}{x}\ud s$ is continuous on $\effdom(\dynb)\times \Real$.
	\end{enumerate}
\end{lemma}

\begin{proof}
    Part \ref{lem::property:open} follows easily from the definition of
    $\effdom(\dynb)$ and part \ref{lem::property:cts} also trivially holds.
    For part \ref{lem::property:close}, notice that if $y=\state{t}{x}$, then
	\begin{align*}
		\int_{0}^{\infty} \testfuncarg{k}{s}{y}\ud s &= \int_{0}^{\infty} \testfuncargbig{k}{s}{\state{t}{x}}\ud s
		\myeq{\eqref{eqn::testfunc_trans}} \Big(\int_{0}^{\infty} \testfuncarg{k}{t+s}{x}\ud s\Big)/\jacoarg{t}{x} \\
		&= \Big(\int_{0}^{\infty} \testfuncarg{k}{s}{x}\ud s - \int_{0}^{t}\testfuncarg{k}{s}{x}\ud s\Big)/\jacoarg{t}{x}\\
		&= \Big(\int_{0}^{\infty} \testfuncargbig{k}{s}{\state{-t}{y}}\ud s - \int_{0}^{t}\testfuncargbig{k}{s}{\state{-t}{y}}\ud s\Big)/\jacoargbig{t}{\state{-t}{y}}.
	\end{align*}
	Since $(x,t)\rightarrow \state{t}{x}$, $(x,t)\rightarrow \jacoarg{t}{x}$ are continuous by the smoothness assumption of $\dynb$, it is clear that $\jacoargbig{t}{\state{-t}{y}}$ is continuous on $\dom$. The continuity of $y\rightarrow \int_{0}^{t}\testfuncargbig{k}{s}{\state{-t}{y}}\ud s$ also trivially holds on $\dom$. Next, the continuity of $y\rightarrow \int_{0}^{\infty} \testfuncargbig{k}{s}{\state{-t}{y}}\ud s$ in a local neighborhood of $\state{t}{\xst}$ comes from the assumption that $\xst\in \effdom(\dynb)$. Thus, $y\to \int_{0}^{\infty} \testfuncarg{k}{s}{y}\ud s$ is continuous in a local neighborhood of $\state{t}{\xst}$, if $\xst\in \effdom(\dynb)$. The other case for $y\to \int_{-\infty}^{0} \testfuncarg{k}{s}{y}\ud s$ can be similarly verified.

    Part \ref{lem::property:close2} is an immediate consequence of 
    part \ref{lem::property:close}. Let us prove it by contradiction.
    Assume that the conclusion in part \ref{lem::property:close2} does not
    hold, then there exists $\xst\in \overline{\effdom(\dynb)}$ such that
    $y^\star :=
    \state{t}{\xst}\notin \overline{\effdom(\dynb)}$ for some $t\in\Real$.
    By part \ref{lem::property:close}, we know this $\xst\not\in
    \effdom(\dynb)$ and thus $\xst\in \partial \effdom(\dynb)$ is in the
    boundary.
    As $\overline{\effdom(\dynb)}$ is closed and $y^\star\notin
    \overline{\effdom(\dynb)}$, there must exist an $\eps>0$ such that
    for any $y\in\Rd$ with $\abs{y-y^\star}<\eps$, we have
    $y\notin \overline{\effdom(\dynb)}$.
    By the smoothness assumption on $\dynb$, we know there exists a $\delta>0$
    such that for any $x\in \ball{\delta}{\xst}$, we have
$\abs{\state{t}{x}-y^\star}<\eps$ outside of $\overline{\effdom(\dynb)}$.
    Since $\xst$ is in the boundary of $\effdom(\dynb)$,
    we must be able to find an $x\in \effdom(\dynb)$ such that $x\in
    \ball{\delta}{\xst}$ and by part \ref{lem::property:close},
    $\state{t}{x}\in\effdom(\dynb)$. Thus, we reach a contradiction.

    For part \ref{lem::property:ctsxt},
	\begin{align*}
		\int_{t}^{\infty} \testfuncarg{k}{s}{x}\ud s = \int_{0}^{\infty} \testfuncarg{k}{t+s}{x}\ud s \myeq{\eqref{eqn::testfunc_trans}} \jacoarg{t}{x}\int_{0}^{\infty}  \testfuncargbig{k}{s}{\state{t}{x}}\ud s.
	\end{align*}
	Because $(x,t)\rightarrow \state{t}{x}$ is continuous (by the smoothness of
    $\dynb\in \mani{}$) and $z\rightarrow\int_{0}^{\infty}
    \testfuncarg{k}{s}{z}\ud s$ is continuous in a local neighborhood of
    $\state{t}{x}\in \effdom(\dynb)$ by part \ref{lem::property:close},
    it is then clear that $\int_{0}^{\infty}  \testfuncargbig{k}{s}{\state{t}{x}}\ud s$ is continuous with respect to $(x,t)$, and hence the conclusion follows easily.
\end{proof}

A notable family of points that are excluded from $\effdom(\dynb)$ are points at which $\int_{-\infty}^{\infty} \testfuncarg{k}{t}{x}\ud t = \infty$. For instance, these include stationary points of $\dynb$ and periodic orbits, which we state as:

\begin{lemma}
	\label{lem::G}
    Given a vector field $\dynb\in \mani{}$, we have $x\notin \effdom(\dynb)$,
	\begin{enumerate}[label=(\roman*)]
		\item if $x$ is a stationary point of $\dynb$ (\ie{}, $\dynb(x) = \vectorzero_{\dimn}$), or
		\item if $x$ is on a periodic orbit.
	\end{enumerate}
\end{lemma}

\begin{proof}
	If $x$ is a stationary point of $\dynb$, the trajectory is $\state{t}{x} = x$ for all $t\in \Real$. Then it is obvious that $\intreal\testfuncarg{k}{t}{x}\ud t = e^{-U_k(x)}\intreal e^{(\div\dynb)(x) t}\ud t = \infty$, which establishes part (i).
	For the  part (ii), let us assume that the period is $\tau$. If $c := \int_{0}^{\tau} \nabla \cdot \dynb\big(\state{s}{x}\big)\ud s\ge 0$, then
	\begin{align*}
		\int_{0}^{\infty} \testfuncarg{k}{t}{x}\ud t &= \int_{0}^{\infty} e^{-U_k\big(\state{t}{x}\big) + \int_{0}^{t}\div\dynb\big(\state{s}{x}\big)\ud s}\ud t \\
		&\ge \int_{0}^{\infty} e^{-\max_{0\le s\le \tau} U_k\big(\state{s}{x}\big) + c \floor{t/\tau}+\int_{\tau \floor{t/\tau}}^{t-\tau \floor{t/\tau}} \nabla \cdot \dynb \big(\state{s}{x}\big)\ud s}\ud t \\
		&\ge C \int_{0}^{\infty} e^{c(t/\tau -1)}\ud t = \infty,
	\end{align*}
	where the constant $C = e^{-\max_{0\le s\le \tau} U_k\big(\state{s}{x}\big)} e^{\min_{0\le s\le \tau} \int_{0}^{s} \nabla \cdot \dynb\big(\state{r}{x}\big)\ud r}$.
	Therefore, $\intreal \testfuncarg{k}{t}{x} = \infty$.
	If $c := \int_{0}^{\tau} \nabla \cdot \dynb\big(\state{s}{x}\big)\ud s < 0$, we can similarly show that $\int_{-\infty}^{0} \testfuncarg{0}{t}{x}\ud t = \infty$, and the same conclusion holds.
\end{proof}

\subsection{Perturbation of the dynamics}

Next we investigate the following question:
given $\dynb\in \maniinf$ and $\delta\dynb\in C_c^{\infty}(\dom,\Rd)$, can we guarantee that  $\dynb+\eps\delta\dynb\in \maniinf{}$ for small enough $\eps$?
Such a conclusion trivially holds for the finite-time case;
however, more underlying structures are needed for the infinite-time case, due to the fact that the long-term behavior of the flow $\dynb$
can sensitively depend on the local perturbation $\delta\dynb$.
This question is probably unavoidable in order to understand the mathematical structure of $\maniinf{}$. 

\smallskip
{\noindent {\bf Notation:}} We shall use the notation $\state{t}{\cdot}$ to represent the flow map under $\dynb$, and use $\statesup{t}{\cdot}{\eps}$ to represent the flow map under $\dynb^{\eps} := \dynb + \eps \delta \dynb$ in this section below. Moreover, we shall use $\testfuncepsarg{k}{\eps}{\cdot}{\cdot}$ to denote the function defined in \eqref{eqn::testfunc} for the dynamics $\dynb^{\eps}$ and the Jacobian of the flow $\jacoargsup{(\cdot)}{\cdot}{\eps}$ is similarly defined.

\begin{definition}[$\dynb$-stability]
	\label{defn::int_cond_dynb}
Given $\dynb\in \maniinf$:
\begin{enumerate}[label=(\alph*)]
\item (For an open bounded set).
A nonempty open bounded set $\domsubset \subset \effdom(\dynb)$
  is said to be \emph{$\dynb$-stable} if:
\begin{enumerate}[label=(\roman*)]
\item
There exists a point $\xst\in \domsubset$ and $\zeta \in (0,1)$ such that
\begin{align}
  \label{eqn::int_cond_dynb::1}
  \abs{\dynb(\xst)}>0, \qquad
  \abs{\dynb(x) - \dynb(\xst)} <\zeta \abs{\dynb(\xst)}, \qquad \forall x\in \domsubset;
\end{align}
\item For any $x\in \domsubset$, the trajectory $\big\{\state{t}{x}\big\}_{t\in \Real}$ intersects with the boundary $\partial \domsubset$ at exactly two points.
\end{enumerate}

\item (For a point). A point $x\in \effdom(\dynb)$ is said to be $\dynb$-stable if there exists a neighborhood $\ball{\eps}{x}$ such that the region $\ball{\eps}{x}$ is $\dynb$-stable. Otherwise, the point $x$ is said to be $\dynb$-unstable.
\end{enumerate}
\end{definition}

The assumption in part (i) ensures that the trajectory $t\rightarrow\state{t}{x}$ will leave this region $\domsubset$ within a finite amount of time;
see \lemref{lem::exit_time} below.
Part (ii)
is used to ensure that the trajectory is not (infinitely) recurrent to $\domsubset$; once the trajectory leaves $\domsubset$, it will not return to $\domsubset$ again.

\begin{lemma}
	\label{lem::exit_time}
	Suppose $\dynb\in \mani{}$ and
	$\domsubset\subset\dom$ is open and bounded.
	If there exists a point $\xst\in \domsubset$ and $\zeta \in (0,1)$ such that
	\begin{align*}
		\abs{\dynb(\xst)}>0, \qquad
		\abs{\dynb(x) - \dynb(\xst)} <\zeta \abs{\dynb(\xst)}, \qquad \forall x\in \domsubset;
	\end{align*}
	then for any $x\in \domsubset$, we have
	\begin{align*}
		\fhit{\domsubset}(x) - \bhit{\domsubset}(x) \le \frac{\diam(\domsubset)}{(1-\xi) \abs{\dynb(\xst)}} < \infty.
	\end{align*}
\end{lemma}

\begin{proof}
	Consider the quantity $h_t(x) := \innerbig{\dynb(\xst)}{\state{t}{x} - \xst}$ for $x\in \domsubset$.
	Then when $t\in \big(\bhit{\domsubset}(x), \fhit{\domsubset}(x)\big)$,
	\begin{align*}
		\td{} h_t(x)  = \innerbig{\dynb(\xst)}{\dynb(\state{t}{x}) - \dynb(\xst) + \dynb(\xst) } \ge (1-\zeta)\abs{\dynb(\xst)}^2 > 0.
	\end{align*}
	Therefore,
	\begin{align*}
		(1-\xi) \abs{\dynb(\xst)}^2 \big(\fhit{\domsubset}(x) - \bhit{\domsubset}(x)\big) &\le \int_{\bhit{\domsubset}(x)}^{\fhit{\domsubset}(x)} \td h_t(x)\ud t \\
		&= h_{\fhit{\domsubset}(x)}(x) - h_{\bhit{\domsubset}(x)}(x)\le \abs{\dynb(\xst)}\ \diam(\domsubset).
	\end{align*}
	Then the conclusion can be immediately obtained.
\end{proof}

We now state the main result of this section:

\begin{proposition}
\label{prop::perturbation}
Suppose $\dynb \in \maniinf$ and $\xst \in \effdom(\dynb)$ is $\dynb$-stable.
Then there exists an open bounded neighborhood of $\xst$, denoted as $\domsubset\subset\effdom(\dynb)$, such that for an arbitrary $\delta\dynb\in C_c^{\infty}(\domsubset, \Rd)$,
there exists an $\eps_0>0$ and  $\dynb + \eps \delta \dynb\in \maniinf$ for any $\eps\in (0, \eps_0)$.
\end{proposition}

\begin{proof}
The main idea is that if the path $\big\{\state{t}{x}\big\}_{t\in\Real}$ passes through $\domsubset$, then a small perturbation within a bounded time period does not affect the long-term behavior; if the path does not pass through $\domsubset$, then $\dynb^{\eps} = \dynb$ along this path and therefore, $\int_{0}^{\infty}\testfuncarg{k,\eps}{t}{x}\ud t=\int_{0}^{\infty} \testfuncarg{k}{t}{x}\ud t$; hence, the continuity is also preserved locally around $x$.
\smallskip

{\noindent \bf {Step (\rom{1}):}} Setup and the choice of $\domsubset$.

Since $\xst\in \effdom(\dynb)$, we know $\abs{\dynb(\xst)}>0$ by \lemref{lem::G}.
Moreover, we can find a small $\dynb$-stable ball $\ball{\theta}{\xst}$ by \defref{defn::int_cond_dynb} with a parameter $\zeta\in (0,1)$ in \eqref{eqn::int_cond_dynb::1}. Then consider the following cross-section within $\ball{\theta}{\xst}$
\begin{align*}
	S :=  \big\{ y:\ \abs{y-\xst} < \theta/2, \qquad \inner{y-\xst}{\dynb(\xst)} =0\big\},
\end{align*}
and define the streamtube $\tube$ passing through $S$ as
\begin{align*}
	\tube := \big\{\state{t}{y}:\ y\in S,\ t\in \Real\big\}.
\end{align*}
It is not hard to see that $\tube$ is an open subset of $\effdom(\dynb)$ by \lemref{lem::property_Omega_b}.
Then let us choose $\domsubset$ as an arbitrary open ball around $\xst$ such that
\begin{align*}
\domsubset\subset \tube\cap \ball{\theta/2}{\xst}.
\end{align*}

Next let us consider an arbitrary $\delta\dynb\in C_c^{\infty}({D},\Rd)$ and from now on, we shall fix $\delta\dynb$.
It is easy to verify that $\dynb+\eps\delta\dynb\in \mani{}$ for any $\eps\in\Real$. The non-trivial part is to check that
$\Omega\backslash\effdom(\dynb^\eps)$ has measure zero
for sufficiently small $\eps$ and hence $\dynb^{\eps}\in \maniinf{}$.

\smallskip
{\noindent \bf {Step (\rom{2}): }}Choice of $\eps_0$.

Let us pick
\begin{align}
	\label{eqn::epsilon_0}
	\eps_0 = \min\Big\{ \frac{\abs{\dynb(\xst)}}{1 + \abs{\delta \dynb(\xst)}},\
	\frac{(\zeta^{\star} - \zeta)\abs{\dynb(\xst)}}{2 \norm{\delta \dynb}_{L^{\infty}(\domsubset)} + \zeta^{\star} \abs{\delta \dynb(\xst)}}\Big\} > 0,
\end{align}
for any $\zeta^{\star}\in (\zeta, 1)$.
The main motivation is that we need \eqref{eqn::int_cond_dynb::1} to hold for $\dynb^{\eps}$ as well. Indeed, if $\eps \le \eps_0$,
\begin{align*}
	\abs{\dynb^{\eps}(\xst)} = \abs{\dynb(\xst) + \eps \delta \dynb(\xst)} \ge \abs{\dynb(\xst)} - \eps \abs{\delta\dynb(\xst)}\ \myge{\eqref{eqn::epsilon_0}}\ \frac{\abs{\dynb(\xst)}}{1 + \abs{\delta\dynb(\xst)}}>  0,
\end{align*}
and for any $x\in \ball{\theta}{\xst}$,
\begin{align*}
	\abs{\dynb^{\eps}(x) - \dynb^{\eps}(\xst)} &\le \abs{\dynb(x) - \dynb(\xst)} + \eps \abs{\delta \dynb(x) - \delta \dynb(\xst)}\\
	&\myle{\eqref{eqn::int_cond_dynb::1}} \zeta \abs{\dynb(\xst)} + \eps (2\norm{\delta\dynb}_{L^{\infty}(\domsubset)})\\
	&\myless{\eqref{eqn::epsilon_0}} \zeta^{\star} \abs{\dynb(\xst)} - \zeta^{\star} \eps \abs{\delta\dynb(\xst)} \\
	&\le \zeta^{\star} \abs{\dynb^{\eps}(\xst)}.
\end{align*}
As discussed above, this property ensures that any trajectory $\big\{\statesup{t}{x}{\eps}\big\}_{t\in\Real}$ with $x\in \ball{\theta}{\xst}$ will pass through the boundary $\partial \ball{\theta}{\xst}$, and at the same time, since $\ball{\theta}{\xst}$ is $\dynb$-stable, the condition (ii) in Definition \ref{defn::int_cond_dynb} ensures that such a trajectory only intersects with the boundary $\partial \ball{\theta}{\xst}$ at exactly $2$ points.

Let us denote
\begin{align*}
	\tau := \sup_{\eps\in [0, \eps_0)}\frac{\diam(\domsubset)}{(1-\xi^{\star}) \abs{\dynb^{\eps}(\xst)}} \le \frac{\diam(\domsubset) \big(1+\abs{\delta\dynb(\xst)}\big)}{(1-\xi^{\star})\abs{\dynb(\xst)}} < \infty.
\end{align*}
By \lemref{lem::exit_time}, we know
\begin{align}
	\label{eqn::exit}
	\statesup{t}{x}{\eps}\notin \ball{\theta}{\xst}, \qquad \forall x\in \ball{\theta}{\xst},\ \forall \eps\in [0,\eps_0),\ \forall t \in (-\infty,-\tau]\cup [\tau,\infty).
\end{align}

\smallskip
{\noindent \bf Step (\rom{3}):} {Prove that  $\Omega\backslash\effdom(\dynb^{\eps})$ has measure zero for any $\eps\in (0,\eps_0)$.}

We will prove that for any $\eps\in (0,\eps_0)$,
\begin{align*}
	\tube \subset \effdom(\dynb^{\eps}), \qquad \effdom(\dynb)\backslash \overline{\tube}\subset \effdom(\dynb^{\eps}).
\end{align*}
It could be observed that the boundary $\partial\tube$ has measure zero: the boundary contains flowlines from a hyper-surface with dimension $\dimn-2$, that is, $\partial\tube$ is a set of flowlines passing through
$
	\big\{y:\ \abs{y-\xst}=\theta/2,\ \inner{y-\xst}{\dynb(\xst)}=0\big\}.
$
Hence, provided that the above equation holds, we immediately know that
\begin{align*}
\Omega\backslash\effdom(\dynb^{\eps}) &\equiv  \effdom(\dynb^\eps)^c \subset \big(\tube \cup (\effdom(\dynb)\backslash \overline{\tube})\big)^c
= \tube^c \cap (\effdom(\dynb)\cap \overline{\tube}^c)^c
= \tube^c \cap (\effdom(\dynb)^c \cup \overline{\tube})\\
& = (\tube^c \cap \effdom(\dynb)^c\big)\cup (\tube^c \cap \overline{\tube}) \subset \effdom(\dynb)^c \cup \partial\tube
= \big(\Omega\backslash\effdom(\dynb)\big)\cup \partial\tube
\end{align*}
has measure zero, where the superscript $c$ means set complement.
As a remark, from now on, we shall fix $\eps\in (0,\eps_0)$.

\medskip
{\noindent \emph{Proof of $\tube \subset \effdom(\dynb^{\eps})$}}.
Let us pick an arbitrary $x\in \tube$ and we shall prove that  $z\rightarrow \int_{0}^{\infty} \testfuncarg{k,\eps}{t}{z}\ud t$ is continuous locally near $x$. Similarly, we can verify that $z\rightarrow \int_{-\infty}^{0} \testfuncarg{k,\eps}{t}{z}\ud t$ is locally continuous near $x$.
Therefore, $x\in \effdom(\dynb^{\eps})$ and thus $\tube\subset \effdom(\dynb^{\eps})$.

Next we return to verify that $z\rightarrow \int_{0}^{\infty} \testfuncarg{k,\eps}{t}{z}\ud t$ is continuous locally near $x$.
We claim that there exists a $y\in S \cup \domsubset\in \ball{\theta/2}{\xst}$ and $s\in\Real$ such that $\statesup{s}{y}{\eps}=x$.
To prove this, we need to discuss two cases:

\begin{itemize}[leftmargin=4ex]
	\item 	Suppose $\big\{\statesup{t}{x}{\eps}\big\}_{t\in\Real}$ never enters $\domsubset$. Because $\dynb^{\eps} = \dynb$ on $\dom\backslash\domsubset$, we know $\statesup{t}{x}{\eps} = \state{t}{x}$ for any $t\in\Real$. By the definition of the streamtube $\tube$,  there exists a $y\in S$ and $s\in\Real$ such that $y = \state{-s}{x} = \statesup{-s}{x}{\eps}$, which immediately gives $\statesup{s}{y}{\eps} = x$.

	\item Suppose $\big\{\statesup{t}{x}{\eps}\big\}_{t\in\Real}$ enters $\domsubset$ at some time. Then the above conclusion follows easily.
\end{itemize}

Because $\dynb^\eps$ is smooth, for small enough $\delta$, we can ensure that $\ball{\delta}{x}\subset \tube$ and also $\statesup{-s}{z}{\eps}\in \ball{\theta/2}{\xst}$ for any $z\in \ball{\delta}{x}$. 
By \eqref{eqn::exit},
\begin{align}
	\label{eqn::exit_time_state}
	\statesup{t-s}{z}{\eps} = \statesupbig{t}{\statesup{-s}{z}{\eps}}{\eps}\notin \ball{\theta}{x^{\star}}, \qquad \forall z\in \ball{\delta}{x},\ \forall t\ge \tau. 
\end{align}

We divide the proof of continuity of $z\rightarrow \int_{0}^{\infty} \testfuncarg{k,\eps}{t}{z}\ud t$ into two cases:
\begin{enumerate}[label=(\alph*)]
	\item If $\tau \le s$, then we already know for any point $z\in \ball{\delta}{x}$, we have $\statesup{t}{z}{\eps} \notin \ball{\theta}{\xst}$ for $t\ge 0$.
	Recall that $\dynb = \dynb^{\eps}$ outside of $\ball{\theta}{\xst}$. Hence $\statesup{t}{z}{\eps} = \state{t}{z}$ for any $t\ge 0$, and 
	\begin{align*}
		\int_{0}^{\infty} \testfuncarg{k,\eps}{t}{z}\ud t = \int_{0}^{\infty} \testfuncarg{k}{t}{z}\ud t,
	\end{align*}
	which is continuous on a neighbor of $x$ by $x\in \effdom(\dynb)$.
	
	\item If $\tau > s$, then by \eqref{eqn::exit_time_state}, 
	we know $ \statesup{t}{\statesup{\tau-s}{z}{\eps}}{\eps}\notin \ball{\theta}{\xst}$ for any $z\in \ball{\delta}{x}$ and $t\ge 0$ and in particular, $\statesup{\tau-s}{z}{\eps}\notin \ball{\theta}{\xst}$ for any $z\in \ball{\delta}{x}$.
	Let us rewrite 
	\begin{align*}
		\int_{0}^{\infty} \testfuncarg{k,\eps}{t}{z}\ud t &= \int_{0}^{\tau-s} \testfuncarg{k,\eps}{t}{z}\ud t + \int_{\tau-s}^{\infty} \testfuncarg{k,\eps}{t}{z}\ud t \\
		&= \int_{0}^{\tau-s} \testfuncarg{k,\eps}{t}{z}\ud t + \int_{0}^{\infty} \testfuncargbig{k,\eps}{t+(\tau-s)}{z}\ud t \\
		&\myeq{\eqref{eqn::testfunc_trans}} \int_{0}^{\tau-s} \testfuncarg{k,\eps}{t}{z}\ud t + \jacoargsup{\tau-s}{z}{\eps}\int_{0}^{\infty} \testfuncargbig{k,\eps}{t}{\statesup{\tau-s}{z}{\eps}}\ud t.
	\end{align*}
	Since $\dynb^{\eps}$ is smooth, apparently $z\to \int_{0}^{\tau-s} \testfuncarg{k,\eps}{t}{z}\ud t$ and $z\to J^{\eps}_{\tau-s}(z)$ are continuous on $\dom$. The continuity of $z\rightarrow \int_{0}^{\infty} \testfuncargbig{k,\eps}{t}{\statesup{\tau-s}{z}{\eps}}\ud t$ locally near $x$ can be exactly proved in the same way as Part (a) above for the new point $\statesup{\tau-s}{x}{\eps}$. 
	
	One small technical result to verify is that $\statesup{\tau-s}{x}{\eps}\in \effdom(\dynb)$ in order to apply the same argument from Part (a). Note that $\statesup{\tau-s}{x}{\eps}= \statesup{\tau}{y}{\eps} = \statesupbig{\tau - \wt{\tau}}{\statesup{\wt{\tau}}{y}{\eps}}{\eps}$ 
	where 
	\begin{align*}
		\wt{\tau} := \inf\big\{t\ge 0\ \big\rvert\ \statesup{t}{y}{\eps} \in \partial B_{\theta}(\xst)\big\}.
	\end{align*} 
Since $\statesup{\wt{\tau}}{y}{\eps}\in \partial \ball{\theta}{\xst}$, we know $\statesup{\wt{\tau}}{y}{\eps}\in \effdom(\dynb)$ (\eg{}, by choosing a small enough $\theta$). Outside of $\ball{\theta}{\xst}$, 
	we know $\dynb^{\eps}=\dynb$ and thus 
	$\statesupbig{\tau - \wt{\tau}}{\statesup{\wt{\tau}}{y}{\eps}}{\eps} 
	= \statebig{\tau - \wt{\tau}}{\statesup{\wt{\tau}}{y}{\eps}}\in \effdom(\dynb)$ by \lemref{lem::property_Omega_b}.
	
\end{enumerate}

\medskip

{\noindent \emph{Proof of $\effdom(\dynb)\backslash \overline{\tube}\subset \effdom(\dynb^{\eps})$}}.
Let us consider an arbitrary point $x\in \effdom(\dynb)\backslash \overline{\tube}$.	
Note that $\dynb^{\eps} = \dynb$ outside of $\domsubset\subset\tube$. 
For a local neighborhood $\ball{\delta}{x}$ outside of  $\overline{\tube}$, we also know for any $y\in \ball{\delta}{x}$, $\statesup{t}{y}{\eps} = \state{t}{y} \notin \overline{\tube}$ for $t\in \Real$, by both the definition of $\tube$ and the construction $\dynb^{\eps}=\dynb$ outside of $\tube$. 
By the same argument as in Part (a) above, it could be readily shown that $x\in \effdom(\dynb^{\eps})$ and thus $\effdom(\dynb)\backslash \overline{\tube}\subset \effdom(\dynb^{\eps})$.
\end{proof}

\section{Supplementary material for \secref{sec::method}}
\label{sec::proof_method}

\subsection{Proof of \propref{prop::formulation}}
\label{subsec::proof::formulation}

	We first prove the finite-time case.
	As $\dynb\in \mani{}$,
	we know that $\int_{\tm}^{\tp} \testfuncarg{0}{-s}{x} \ud s \equiv
    \int_{\tm}^{\tp} e^{-U_0\big(\state{-s}{x}\big)} \jacoarg{-s}{x}\ud
    s < \infty$ by the continuity assumption of $\dynb$ and
    Assumption~\ref{assump::U}.
	Then
    \begin{align*}
		\begin{aligned}
			\ratio &= 
			\int_{\dom} e^{-U_1(x)} \frac{\int_{\tm}^{\tp} \testfuncarg{0}{-t}{x}\ud t}{\int_{\tm}^{\tp} \testfuncarg{0}{-s}{x}\ud s} \ud x \\
			&= \int_{\tm}^{\tp} \int_{\dom} e^{-U_1(x)} \frac{\testfuncarg{0}{-t}{x}}{\int_{\tm}^{\tp} \testfuncarg{0}{-s}{x}\ud s} \ud x \ud t.
		\end{aligned}
	\end{align*}
	By the change of variables $x=\state{t}{\wt{x}}$, we have 
	\begin{align}
		\label{eqn::proof_formulation}
		\begin{aligned}
			\ratio =& \int_{\tm}^{\tp} \int_{\dom}  e^{-U_1\big(\state{t}{\wt{x}}\big)} \frac{\testfuncarg{0}{-t}{\state{t}{\wt{x}}}}{\int_{\tm}^{\tp} \testfuncarg{0}{-s}{\state{t}{\wt{x}}}\ud s} \jacoarg{t}{\wt{x}} \ud \wt{x} \ud t \\
			\myeq{\eqref{eqn::testfunc_trans}}& \int_{\tm}^{\tp} \int_{\dom} e^{-U_1\big(\state{t}{\wt{x}}\big)} \frac{\testfuncarg{0}{0}{\wt{x}}}{\int_{\tm}^{\tp} \testfuncarg{0}{t-s}{\wt{x}}\ud s} \jacoarg{t}{\wt{x}} \ud \wt{x} \ud t\\
			=& \int_{\tm}^{\tp} \int_{\dom} e^{-U_1\big(\state{t}{\wt{x}}\big)} \jacoarg{t}{\wt{x}} \frac{\rho_0(\wt{x})}{\int_{\tm}^{\tp} \testfuncarg{0}{t-s}{\wt{x}}\ud s}  \ud \wt{x} \ud t \\
			=& \ee_{0} \Big[ \int_{\tm}^{\tp} \frac{\testfuncarg{1}{t}{\cdot}}{\int_{t-\tp}^{t-\tm} \testfuncarg{0}{s}{\cdot}\ud s} \ud t\Big] \equiv \ee_0 \big[\newafinarg{\tm}{\tp}\big].
		\end{aligned}
	\end{align}
	Note that as the integrand is non-negative, switching the order of time integration and space integration is justified by Fubini–Tonelli theorem.

	The proof of \propref{prop::formulation} for the infinite-time case is essentially the same as the finite-time case, except the followings: 
\begin{itemize}[leftmargin=4ex]
\item We need to replace the domain $\dom$ in \eqref{eqn::proof_formulation} by $\effdom(\dynb)$ defined in \defref{defn::maniinf}.

\item We need the continuity of
$x\rightarrow\intreal\testfuncarg{0}{t}{x}\ud t$ in order to use
Theorem 2 in \cite{lax_change_2001} to get the first line in \eqref{eqn::proof_formulation}. 
A generalization with discontinuity should be possible by e.g., considering piecewise continuous $\dynb$.
However, we shall not explore this further in this work.
As a remark, the map $x\to\state{t}{x}$ being bijective (due to the nature of ODE flows) on $\overline{\effdom(\dynb)}$ was proved
in \lemref{lem::property_Omega_b} \ref{lem::property:close2}, when
applying \cite[Theorem 2]{lax_change_2001}.

\end{itemize}

\subsection{Proof of \propref{prop::invariance_b}}
\label{subsec::proof_invariance}
Consider
\begin{align*}
	&\td \state{t}{x} = \dynb\big(\state{t}{x}\big),\qquad \state{0}{x} = x;\\
	&\td \stateD{t}{x} = \alpha\big(\stateD{t}{x}\big) \dynb\big(\stateD{t}{x}\big),\qquad  \stateD{0}{x} = x.
\end{align*}
To prove the first conclusion, we need to verify that the trajectory $\big\{\state{t}{x}\big\}_{t\in\Real} = \big\{\stateD{t}{x}\big\}_{t\in \Real}$.
From now on, let us fix $x$ and introduce a scalar-valued function $\theta$ by 
$\theta_t := \int_{0}^{t} \alpha\big(\stateD{s}{x}\big)\ud s$
	(\ie{}, $\td \theta_t = \alpha\big(\stateD{t}{x}\big)$.
	By taking the time derivative, it is not hard to verify that $\stateD{t}{x} = \statebig{\theta_t}{x}$ as both satisfy the same ODE.
	Of course, $\theta$ also depends on $x$
	but we shall not explicitly specify this dependence for simplicity of notations.
	Thus, the trajectory $t\rightarrow \stateD{t}{x}$ is the same as $t\rightarrow \state{t}{x}$ under time rescaling specified by $\theta$.

	Next we shall prove the following lemma, which immediately leads into the second result in \propref{prop::invariance_b}.
	\begin{lemma} 
		Suppose $g:\dom\rightarrow\Real$ is a non-negative continuous function. For any $x\in \dom$, 
	\begin{align*}
		& \intreal g\big(\stateD{t}{x}\big) \exp\Big(\int_{0}^{t} \nabla \cdot(\alpha \dynb)\big(\stateD{s}{x}\big)\ud s\Big)\ud t \\
		&\qquad = \frac{1}{\alpha(x)} \intreal g\big(\state{t}{x}\big) \exp\Big(\int_{0}^{t} \big(\nabla \cdot\dynb)\big(\state{s}{x}\big)\ud s\Big)\ud t.
	\end{align*}
	\end{lemma}

	\begin{proof}
	Because $\alpha$ is strictly positive,
	$\theta_\cdot:\Real\to\Real$ is bijective, and by the inverse function theorem
	\begin{align*}
		\td\theta^{-1}_t = \frac{1}{\alpha\big(\stateD{\theta^{-1}_t}{x}\big)} = \frac{1}{\alpha\big(\state{t}{x}\big)}.
	\end{align*}
	Then by the change of variables $\wt{t} = \theta_t$ and $\wt{s} = \theta_s$,
	\begin{align*}
		&\ \intreal g\big(\stateD{t}{x}\big) \exp\Big(\int_{0}^{t} \nabla \cdot(\alpha \dynb)\big(\stateD{s}{x}\big)\ud s\Big)\ud t \\
		=&\ \intreal g\big(\state{\theta_t}{x}\big) \exp\Big(\int_{0}^{t} \nabla \cdot(\alpha \dynb)\big(\state{\theta_s}{x}\big)\ud s\Big)\ud t\\
		=&\ \intreal g\big(\state{\wt{t}}{x}\big) \frac{1}{\alpha\big(\state{\wt{t}}{x}\big)} \exp\Big(\int_{0}^{\wt{t}} \div(\alpha\dynb)\big(\state{\wt{s}}{x}\big)\frac{1}{\alpha\big(\state{\wt{s}}{x}\big)}\ud \wt{s} \Big)\ud \wt{t}.
	\end{align*}
	It is then sufficient to show that $\psi_1(t) = \psi_2(t)$ for any $t\in \Real$, where
	\begin{align*}
		\psi_1(t) &:= \frac{1}{\alpha\big(\state{t}{x}\big)} \exp\Big(\int_{0}^{t} \div(\alpha\dynb)\big(\state{s}{x}\big)\frac{1}{\alpha\big(\state{s}{x}\big)}\ud s \Big), \\
		\psi_2(t) &:= \frac{1}{\alpha(x)} \exp\Big(\int_{0}^{t} (\nabla \cdot\dynb)\big(\state{s}{x}\big)\ud s\Big).
	\end{align*}
	It is easy to observe that $\psi_1(0) = \psi_2(0)$.
	Let us consider the time derivative of $\psi_1$
	\begin{align*}
		\td {\psi}_1(t) &= \psi_1(t)\Big( -\frac{1}{\alpha(\state{t}{x})} \innerbig{\nabla\alpha(\state{t}{x})}{\dynb(\state{t}{x})} + \nabla \cdot(\alpha \dynb)\big(\state{t}{x}\big) \frac{1}{\alpha\big(\state{t}{x}\big)} \Big) \\
		&= \psi_1(t) \big(\nabla \cdot \dynb)\big(\state{t}{x}\big).
	\end{align*}
	It is clear that $\psi_2(t)$ satisfies the same ODE and thus $\psi_1 = \psi_2$.
	\end{proof}

\subsection{Remarks on the discrete-time analogy of \eqref{eqn::noneq_sample::finite}}
\label{sec::remark-discrete}

Let us briefly explain how \eqref{eqn::noneq_sample::finite} connects to the method in \cite{thin_neo_2021} by time discretization.
Suppose $N_{-} = \tm/h$ and $N_{+} = \tp/h$, where $h\ll 1$ is the time step size;
for simplicity of notation,
let us assume that $N_{-}$ and $N_{+}$ are simply integers.
By discretizing the time integration for the \ftneis{} scheme in \eqref{eqn::noneq_sample::finite},
\begin{align*}
	\ee_{x\sim\rho_0}\Big[\int_{{\tm}}^{{\tp}} \frac{\testfuncarg{1}{t}{x}}{\int_{t-T_{+}}^{t-T_{-}} \testfuncarg{0}{s}{x} \ud s} \ud t \Big] 
	&\approx \ee_{x\sim\rho_0} \bigg[ \sum_{k=N_{-}}^{N_{+}}  \frac{e^{-U_1\big(\state{kh}{x}\big)}\jacoarg{kh}{x}}{\sum_{j = k - N_{+}}^{k-N_{-}} e^{-U_0\big(\state{jh}{x}\big)} \jacoarg{jh}{x}} \bigg]\\
	&= \ee_{x\sim \rho_0} \bigg[\sum_{k=N_{-}}^{N_{+}}  \frac{e^{-(U_1-U_0)\big(\mapT^{-k}(x)\big)} \rho_0\big(\mapT^{-k}(x)\big)\ \jacomap{\mapT^{-k}}{x}}{\sum_{j = k-N_{+}}^{k-N_{-}} \rho_0\big(\mapT^{-j}(x)\big) \jacomap{\mapT^{-j}}{x}}  \bigg],
\end{align*}
where we denote $\mapT(\cdot) = \state{-h}{\cdot}$, and $\jacomap{\mapT^{-j}}{x} := \Abs{\det\Big(\nabla \mapT^{-j}(x)\Big)}$ is the Jacobian for the map $\mapT^{-j}$.
Then by choosing $N_{-} = -K$ and $N_{+} = 0$, we have
\begin{align*}
	\ee_{x\sim\rho_0} \bigg[ \sum_{k=0}^{K}  e^{-(U_1-U_0)\big(\mapT^k(x)\big)} w_k(x)\bigg], 
	\qquad 
	w_k(x) = \frac{\rho_0\big(\mapT^k(x)\big) \jacomap{\mapT^k}{x}}{\sum_{j = -k}^{-k+K} \rho_0\big(\mapT^{-j}(x)\big) \jacomap{\mapT^{-j}}{x}},
\end{align*}
which are Eqs. (8) and (10) in arXiv v1 of \cite{thin_neo_2021}.

\subsection{Remark on the relation between finite-time and infinite-time NEIS}
\label{subsec::relation}

In what follows, we briefly elaborate on the relation between the finite-time and infinite-time NEIS schemes.
Suppose we fix $-\infty < \tm < 0 < \tp < \infty$ and consider a fixed valid flow $\dynb$ for the infinite-time NEIS, i.e.,
\begin{align*}
  \newainf^\dynb(x)  \equiv \frac{\int_\Real\exp\big(-U_1(\state{t}{x}) +\int_{0}^{t} \nabla\cdot\dynb(\state{r}{x})\ud r\big)\ud t}{\int_\Real \exp\big(-U_0(\state{t}{x}) +\int_{0}^{t} \nabla\cdot\dynb(\state{r}{x})\ud r\big)\ud t}
\end{align*}
is well-defined for almost all $x\in \dom$, where
the state $\state{t}{x}$ solves the ODE $\frac{\ud}{\ud t}\state{t}{x} = \dynb(\state{t}{x})$ for any $x\in\dom$. The superscript in $\newainf^\dynb$ is used to emphasize that it is the estimator for this particular flow $\dynb$ and similar notations are used below.

Next we consider a family of rescaled flow $\dynb^\alpha$ parameterized by $\alpha > 0$, defined as
\begin{align*}
  \dynb^\alpha(x) = \alpha \dynb(x), \qquad \forall x\in \dom.
\end{align*}
Let us study its estimator for the finite-time NEIS:
\begin{align}
  \label{eqn::afin_rescale_alpha}
  \newafin^{\dynb^\alpha}(x)\ & \myeq{\eqref{eq::afin}}\ \int_{\tm}^{\tp} \frac{\exp\big(-U_1(\stateD{t}{x}) + \int_{0}^{t} \nabla\cdot\dynb^\alpha(\stateD{r}{x})\ud r\big)}{\int_{t-\tp}^{t-\tm} \exp\big(-U_0(\stateD{s}{x})+\int_{0}^{s} \nabla \cdot \dynb^\alpha(\stateD{r}{x})\ud r\big) \ud s} \ud t,
\end{align}
where $\stateD{t}{x}$ solves the ODE $\frac{\ud}{\ud t}\stateD{t}{x} = \dynb^\alpha(\stateD{t}{x})$ for any $x\in\dom$.
From Appendix~\ref{subsec::proof_invariance}, we already know that $\stateD{t}{x} = \state{\alpha t}{x}$ for any $x\in\dom$ and $t\in\Real$.
By change of time variables $t=\wt{t}/\alpha$, $s=\wt{s}/\alpha$ and $r=\wt{r}/\alpha$, we have
\begin{align*}
  \newafin^{\dynb^\alpha}(x) &=
  \int_{\alpha \tm}^{\alpha \tp}
  \frac{\exp\big(-U_1(\state{\wt{t}}{x}) + \int_{0}^{\wt{t}} \nabla\cdot\dynb(\state{\wt{r}}{x})\ud \wt{r}\big)}{\int_{\wt{t}-\alpha \tp}^{\wt{t}-\alpha \tm} \exp\big(-U_0(\state{\wt{s}}{x}) +\int_{0}^{\wt{s}} \nabla\cdot\dynb(\state{\wt{r}}{x})\ud \wt{r}\big)\ud \wt{s}} \ud \wt{t}\\
 &= \int_{\Real} \frac{\indi_{[\alpha\tm,\alpha\tp]}(\wt{t})\ \exp\big(-U_1(\state{\wt{t}}{x}) + \int_{0}^{\wt{t}} \nabla\cdot\dynb(\state{\wt{r}}{x})\ud \wt{r}\big)}{\int_{\wt{t}-\alpha \tp}^{\wt{t}-\alpha \tm} \exp\big(-U_0(\state{\wt{s}}{x}) +\int_{0}^{\wt{s}} \nabla\cdot\dynb(\state{\wt{r}}{x})\ud \wt{r}\big)\ud \wt{s}} \ud \wt{t}.
\end{align*}
For any $\wt{t}\in\Real$, as $\alpha\to\infty$, the integrand in the last equation converges pointwise to
\begin{align*}
  \frac{\exp\big(-U_1(\state{\wt{t}}{x}) + \int_{0}^{\wt{t}} \nabla\cdot\dynb(\state{\wt{r}}{x})\ud \wt{r}\big)}{\int_{\Real} \exp\big(
    -U_0(\state{\wt{s}}{x}) + \int_{0}^{\wt{s}}
      \nabla\cdot\dynb(\state{\wt{r}}{x})\ud \wt{r}
  \big)\ud \wt{s}}.
\end{align*}
If we heuristically swap the order of taking the limit $\alpha\to\infty$ and the integral with respect to $\wt{t}$ (which should generally hold for most examples),
we end up with an identity:
\begin{align*}
  \lim_{\alpha\to\infty} \newafin^{\dynb^\alpha}(x) = \int_{\Real}
\frac{\exp\big(-U_1(\state{\wt{t}}{x}) + \int_{0}^{\wt{t}} \nabla\cdot\dynb(\state{\wt{r}}{x})\ud \wt{r}\big)}{\int_{\Real} \exp\big(
    -U_0(\state{\wt{s}}{x}) + \int_{0}^{\wt{s}}
      \nabla\cdot\dynb(\state{\wt{r}}{x})\ud \wt{r}
  \big)\ud \wt{s}}\ud \wt{t} \equiv \newainf^\dynb(x).
  \end{align*}

The above relation heuristically justifies that due to the space-time rescaling, in the limit of large magnitude of the flow (i.e., $\alpha\to\infty$ above), it does not matter how $\tm, \tp$ are chosen as long as $\tm < 0 < \tp$. In particular, if the flow $\dynb$ is a zero-variance dynamics, i.e., $\newainf^\dynb(x) = \ratio$ for $x\in\dom$ almost surely,
then in the finite-time NEIS, the flow $\dynb^\alpha = \alpha \dynb$ (with $\alpha\gg 1$) should be approximately a zero-variance dynamics for the finite-time scheme.
The finite-time NEIS scheme may not have explicit analytical results about zero-variance dynamics in the same way as the infinite-time NEIS scheme; however, due to the above discussed relation,
the finite-time version still possesses the ability to handle and learn an approximately zero-variance dynamics, which is good enough in practice, e.g., during training the optimal flow in \secref{sec::numerics}.

\section{The first-order perturbation of the variance for the finite-time scheme}
\label{sec::optimal_flow_finite}

{Here we study how the variance (or equivalently, the second moment) of the estimator depends on~$\dynb$, since the performance of the \ftneis{} scheme
\eqref{eqn::noneq_sample::finite} largely depends on this choice.}
{More specifically,} in the following proposition, we study how the second moment changes under a small perturbation of $\dynb$.
The expression \eqref{eqn::pert_M_2} below will be useful for training optimal dynamics in \secref{sec::numerics} (see also Appendix \ref{app::numerics_details} for details).

\begin{proposition}
	\label{prop::pert_M_2}
	Suppose $\dynb \in \mani{}$ and
	for any perturbation $\delta \dynb\in C_c^{\infty}(\dom,\Rd)$,
	denote $\dynb^{\eps} := \dynb + \eps \delta \dynb$.
	Then,
	\begin{align}
		\label{eqn::pert_M_2}
		\begin{aligned}
			& \frac{\ud}{\ud\epsilon}\newmsecfin(\dynb+\eps\delta \dynb)\Big\rvert_{\eps=0} \\
			=\ & 2\ee_{x\sim\rho_0} \bigg[\newafin(x)\ \Big(\int_{\tm}^{\tp}  \frac{\deritestfuncarg{1}{t}{x} \int_{t-\tp}^{t-\tm} \testfuncarg{0}{s}{x} \ud s - \testfuncarg{1}{t}{x} \int_{t-\tp}^{t-\tm} \deritestfuncarg{0}{s}{x} \ud s}{\Big(\int_{t-\tp}^{t-\tm} \testfuncarg{0}{s}{x} \ud s\Big)^2} \ud t\Big)\bigg],
		\end{aligned}
	\end{align}
	where for $k\in \{0,1\}$, we define
	\begin{align}
		\label{eqn::deri_test_fun}
		\deritestfuncarg{k}{t}{x} 
		:=  \testfuncarg{k}{t}{x} \left(\begin{aligned}
			\innerbig{-\nabla U_k\big(\state{t}{x}\big)}{\deristate{t}{x}} 
			 &+ \int_{0}^{t} \innerbig{\nabla (\nabla \cdot \dynb)\big(\state{s}{x}\big)}{\deristate{s}{x}} \ud s  \\
			&+ \int_{0}^{t} (\nabla \cdot \delta \dynb)\big(\state{s}{x}\big)\ud s 
		\end{aligned}\right),
	\end{align}
	and $\deristate{t}{x}$ is the solution of the following ODE:
	\begin{align}
		\label{eqn::deriX}
		\td \deristate{t}{x} = \nabla \dynb\big(\state{t}{x}\big) \deristate{t}{x} + \delta \dynb\big(\state{t}{x}\big),\qquad \deristate{0}{x} = 0.
	\end{align}
\end{proposition}

The expression of the functional derivative $\derimsec$ (not presented in this work) for the finite-time case can be derived in the same way as \lemref{lem::functional_derivative_inf_v1} below for the infinite-time case. However, it appears that the expression of $\derimsec$ is too complicated to provide useful analytical results.

\begin{proof}
	
Let us perturb $\dynb$ by $\epsilon\delta \dynb$ where $\eps\ll 1$.
Let us consider 
\begin{align*}
	&\td \state{t}{x} = \dynb\big(\state{t}{x}\big),\\
	&\td \statesup{t}{x}{\eps} = \dynb^{\eps}\big(\statesup{t}{x}{\eps}\big) \equiv  (\dynb+\eps\delta \dynb)\big(\statesup{t}{x}{\eps}\big),
\end{align*} 
with a fixed initial condition $\state{0}{x}= \statesup{0}{x}{\eps} = x$. For a small $\eps$, we can expect that $\state{t}{x} \approx \statesup{t}{x}{\eps}$ and also we know $\statesup{t}{x}{0} \equiv \state{t}{x}$.
Define $\testfuncepsarg{k}{\eps}{t}{x}$ for the dynamics $\dynb^{\eps}$ in the same way as in \eqref{eqn::testfunc}, namely,
\begin{align*}
\testfuncepsarg{k}{\eps}{t}{x} := \exp\Big(-U_k\big(\statesup{t}{x}{\eps}\big) + \int_{0}^{t} (\nabla \cdot \dynb^{\eps})\big(\statesup{s}{x}{\eps}\big) \ud s\Big).
\end{align*}
By these notations,
\begin{align*}
\newmsecfin(\dynb^{\eps})
= \ee_{x\sim\rho_0} \bigg[\Big(\int_{\tm}^{\tp} \frac{\testfuncepsarg{1}{\eps}{t}{x}}{\int_{t-\tp}^{t-\tm} \testfuncepsarg{0}{\eps}{s}{x} \ud s} \ud t\Big)^2\bigg].
\end{align*}
Then we take the derivative of $\newmsecfin(\dynb^{\eps})$ with respect to $\eps$:
\begin{align*}
&\ \frac{\ud}{\ud\epsilon}\newmsecfin(\dynb^{\eps}) \\
=&\ 2\ee_{x\sim\rho_0} \left[\begin{aligned}
& \Big(\int_{\tm}^{\tp} \frac{\testfuncepsarg{1}{\eps}{t}{x}}{\int_{t-\tp}^{t-\tm} \testfuncepsarg{0}{\eps}{s}{x} \ud s} \ud t \Big) \\
& \qquad \times\Big(\int_{\tm}^{\tp}  \frac{\frac{\ud}{\ud\epsilon}\testfuncepsarg{1}{\eps}{t}{x}\int_{t-\tp}^{t-\tm} \testfuncepsarg{0}{\eps}{s}{x} \ud s - \testfuncepsarg{1}{\eps}{t}{x} \int_{t-\tp}^{t-\tm} \frac{\ud}{\ud\eps} \testfuncepsarg{0}{\eps}{s}{x} \ud s}{\Big(\int_{t-\tp}^{t-\tm} \testfuncepsarg{0}{\eps}{s}{x} \ud s\Big)^2} \ud t\Big)
\end{aligned}
\right].
\end{align*}
Next, we need to compute $\frac{\ud}{\ud\eps}\testfuncepsarg{k}{\eps}{t}{x}$.
Let us first consider the perturbation to the trajectory.
Let $\deristatesup{t}{x}{\eps} := \frac{\ud}{\ud \epsilon} \big(\statesup{t}{x}{\eps}\big)$ and then
\begin{align*}
\begin{aligned}
\td \deristatesup{t}{x}{\eps}&= \frac{\ud}{\ud\eps} \Big(\big(\dynb + \eps \delta \dynb)\big(\statesup{t}{x}{\eps}\big)\Big) \\
&= \nabla \dynb\big(\statesup{t}{x}{\eps}\big) \deristatesup{t}{x}{\eps} + \delta \dynb\big(\statesup{t}{x}{\eps}\big) + \eps \Big(\nabla \delta \dynb\big(\statesup{t}{x}{\eps}\big) \Big)\deristatesup{t}{x}{\eps}, \\
\deristatesup{0}{x}{\eps} &= 0.
\end{aligned}
\end{align*}
When $\eps = 0$, $\deristatesup{t}{x}{0}$ is the solution to
\begin{align*}
\td \deristatesup{t}{x}{0} = \nabla \dynb\big(\state{t}{x}\big) \deristatesup{t}{x}{0} + \delta \dynb\big(\state{t}{x}\big),\qquad \deristatesup{0}{x}{0} = 0.
\end{align*}

Now we are ready to explicitly write down $\frac{\ud}{\ud\eps}\testfuncepsarg{k}{\eps}{t}{x}$. It is straightforward to derive that
\begin{align*}
\begin{aligned}
&\hspace{3ex} \frac{\ud}{\ud\eps}\testfuncepsarg{k}{\eps}{t}{x} \\
&= \testfuncepsarg{k}{\eps}{t}{x} \frac{\ud}{\ud\eps}\Big(-U_k\big(\statesup{t}{x}{\eps}\big) + \int_{0}^{t} \nabla \cdot (\dynb + \eps \delta \dynb)\big(\statesup{s}{x}{\eps}\big) \ud s\Big)\\
&\ = \testfuncepsarg{k}{\eps}{t}{x} \left(
	\begin{aligned}
	& \innerbig{-\nabla U_k\big(\statesup{t}{x}{\eps}\big) }{ \deristatesup{t}{x}{\eps}} \\
	&\qquad  + \int_{0}^{t} \innerBig{\nabla (\nabla \cdot \dynb)\big(\statesup{s}{x}{\eps}\big)}{ \deristatesup{s}{x}{\eps}}
+ (\nabla \cdot \delta \dynb)\big(\statesup{s}{x}{\eps}\big) \ud s \\
&\qquad + \eps \int_{0}^{t}\innerBig{\nabla (\nabla \cdot \delta \dynb)\big(\statesup{s}{x}{\eps}\big)}{\deristatesup{s}{x}{\eps}} \ud s
\end{aligned}\right).
\end{aligned}
\end{align*}
When we let $\eps=0$, we have
\begin{align*}
\deritestfuncarg{k}{t}{x} &:= \frac{\ud}{\ud\eps}\testfuncepsarg{k}{\eps}{t}{x}\Big\rvert_{\eps=0} \\
&=\  \testfuncarg{k}{t}{x} 
\left(\begin{aligned}
&\innerbig{-\nabla U_k\big(\state{t}{x}\big)}{ \deristatesup{t}{x}{0}} \\
&\qquad + \int_{0}^{t} \innerbig{\nabla (\nabla \cdot \dynb)\big(\state{s}{x}\big)}{ \deristatesup{s}{x}{0}}
+ (\nabla \cdot \delta \dynb)\big(\state{s}{x}\big)\ud s 
\end{aligned}\right).
\end{align*}
Finally, we arrive at the conclusion by combining previous results and dropping the superscript in $\deristatesup{t}{x}{0}$ for simplicity.
\end{proof}

\section{The first-order perturbation of the variance for the infinite-time scheme}
\label{sec::perturb::infinite}

The goal  of this section is to derive the functional derivative of $\newmsecinf(\dynb)$ with respect to $\dynb$, denoted as $\derimsecinf$, defined as follows: for any $\delta\dynb\in C_c^{\infty}(\dom,\Rd)$ such that $\dynb + \eps \delta\dynb\in \maniinf$
for small enough $\eps$, we have 
\begin{align}
	\label{eqn::func_deri_defn}
	\frac{\ud}{\ud\epsilon}\newmsecinf(\dynb + \eps\delta \dynb)\Big\rvert_{\eps=0} = \int_{\dom} \innerBig{\derimsecinf }{ \delta \dynb}.
\end{align}
Since $\newmsecinf$ and $\newvarinf$ only differ by a constant (which is independent of $\dynb$), it is apparent that $\derivarinf \equiv \derimsecinf$.

\begin{proposition}
	\label{prop::stationary_inf}
	The functional derivative $\derimsecinf:\dom\to\Rd$ has the following form
	\begin{align}
		\label{eqn::func_deri_inf}
		\begin{aligned}
			&\derivarinf(x) \equiv \derimsecinf(x) \\
			&=
			\frac{2\nabla \newainf(x)}{\partitionzero\binf(x)}\left(
			\int_{0}^{\infty}
			\testfuncarg{0}{t}{x}\ud t \int_{-\infty}^{0} \testfuncarg{1}{t}{x}\ud t 
			- \int_{-\infty}^{0} \testfuncarg{0}{t}{x}\ud t  \int_{0}^{\infty} \testfuncarg{1}{t}{x}\ud t \right).
		\end{aligned}
	\end{align}
\end{proposition}

\begin{remark}
	The proof of the last formula is slightly formal: for instance, conditions on $\dynb$ to ensure the existence of $\nabla \newainf$ are not discussed.
\end{remark}

Recall the expression of $\binf{}$ from \eqref{eqn::ainf_binf}.
The proof of~ Proposition~\ref{prop::stationary_inf} is given in
    Appendix~\ref{app:proof:10}: it relies on a few explicit formulas, that we
state first.

\subsection{Some explicit formulas}

We need some explicit formulas for $\deristate{t}{x}$ \eqref{eqn::deriX} and $\deritestfuncarg{k}{t}{x}$ \eqref{eqn::deri_test_fun}.
We notice that $\deritestfunc{k}$ depends on $\deristate{\cdot}{\cdot}$ and $\delta \dynb$ linearly,
and $\deristate{t}{x}$ also depends on $\delta \dynb$ linearly.
Therefore, the first step is to rewrite the expression of $\deristate{t}{x}$ more explicitly 
in terms of $\delta\dynb$.

\begin{lemma}
	\label{lem::Z}
	Suppose the dynamics $\dynb\in \mani{}$ and $\delta\dynb\in C_c^{\infty}(\dom,\Rd)$. Then we have
	\begin{align}
	\label{eqn::Zt_sol}
	\deristate{t}{x} = \int_{0}^{t} \corrz{t}{s}(x)\ \delta \dynb\big(\state{s}{x}\big) \ud s, \qquad \forall x\in \dom.
	\end{align}
	The kernel $\corrz{t}{s}(x)\in \Real^{\dimn\times\dimn}$ has the following form
	\begin{align}
	\label{eqn::corrz}
	\corrz{t}{s}(x) = \left\{
	\begin{aligned}
	\exp_{\chrontime} \Big(\int_{s}^{t} \nabla \dynb\big(\state{r}{x}\big)\ud r\Big),\qquad & \text{ if } t\ge s \ge 0,\\
	\bigg(\exp_{\chrontime} \Big(\int_{t}^{s} \nabla \dynb\big(\state{r}{x}\big)\ud r\Big)\bigg)^{-1} ,\qquad & \text{ if } t \le s \le 0,
	\end{aligned}\right.
	\end{align}
	where $\exp_{\chrontime}$ is the chronological time-ordered operator exponential.
\end{lemma}

\begin{proof}
	By plugging the ansatz \eqref{eqn::Zt_sol} into \eqref{eqn::deriX}, we immediately know that $\corrz{s}{s}(x) = \idmat_\dimn$ for all $s\in\Real$ and 
	\begin{align}
		\label{eqn::corrz_evol}
		\partial_t \corrz{t}{s}(x) = \nabla \dynb\big(\state{t}{x}\big) \corrz{t}{s}(x).
	\end{align}
	This linear ODE has an explicit solution as in \eqref{eqn::corrz}, by introducing the time-ordered operator exponential.
\end{proof}

Next, we shall rewrite $\deritestfuncarg{k}{t}{x}$.
\begin{lemma}
	We can rewrite $\deritestfuncarg{k}{t}{x}$ as follows
	\begin{align}
	\label{eqn::deritestfunc}
	\deritestfuncarg{k}{t}{x} = \testfuncarg{k}{t}{x} \Big(\int_{0}^{t}
	\innerBig{\corrg{t}{s}{k}(x) }{ \delta \dynb\big(\state{s}{x}\big)} + (\nabla \cdot \delta\dynb)\big(\state{s}{x}\big) \ud s\Big),
	\end{align}
	where $\corrg{t}{s}{k}(x)$ is defined as 
	\begin{align}
	\label{eqn::corrg}
	\corrg{t}{s}{k}(x) := - \corrz{t}{s}(x)^{T} \nabla U_k\big(\state{t}{x}\big)  + \int_{s}^{t}\corrz{r}{s}(x)^{T} \nabla(\nabla \cdot \dynb)\big(\state{r}{x}\big)\ud r.
	\end{align}
\end{lemma}

\begin{proof}
	Recall the expression of $\deritestfuncarg{k}{t}{x}$ from  \eqref{eqn::deri_test_fun}. After plugging \eqref{eqn::Zt_sol}, we have
	\begin{align*}
	\deritestfuncarg{k}{t}{x} &\myeq{\eqref{eqn::deri_test_fun}}\   \testfuncarg{k}{t}{x} \left(\begin{aligned} & \innerbig{-\nabla U_k\big(\state{t}{x}\big)}{\deristate{t}{x}} + \int_{0}^{t} \innerbig{\nabla (\nabla \cdot \dynb)\big(\state{s}{x}\big)}{\deristate{s}{x}} \ud s \\
		&\qquad + \int_{0}^{t}(\nabla \cdot \delta \dynb)\big(\state{s}{x}\big) \ud s\end{aligned}\right)\\
	&\myeq{\eqref{eqn::Zt_sol}} \testfuncarg{k}{t}{x}\left(\begin{aligned}
	&-\int_{0}^{t} \innerBig{\corrz{t}{s}(x)^{T} \nabla U_k\big(\state{t}{x}\big)}{ \delta \dynb\big(\state{s}{x}\big)}\ud s + \int_{0}^{t} (\nabla \cdot \delta \dynb)\big(\state{s}{x}\big)\ud s \\
	& \qquad + \int_{0}^{t} \int_{0}^{s} \innerBig{\corrz{s}{r}(x)^{T} \nabla (\nabla \cdot \dynb)\big(\state{s}{x}\big)}{ \delta \dynb\big(\state{r}{x}\big)} \ud r\ud s
	\end{aligned}\right)\\
	&= \testfuncarg{k}{t}{x}\left(\begin{aligned}
	&-\int_{0}^{t} \innerBig{\corrz{t}{s}(x)^{T} \nabla U_k\big(\state{t}{x}\big)}{ \delta \dynb \big(\state{s}{x}\big)}\ud s + \int_{0}^{t} (\nabla \cdot \delta \dynb)\big(\state{s}{x}\big)\ud s \\
	& \qquad + \int_{0}^{t} \innerBig{\int_{s}^{t} \corrz{r}{s}(x)^{T} \nabla (\nabla \cdot \dynb)\big(\state{r}{x}\big) \ud r}{ \delta \dynb\big(\state{s}{x}\big)} \ud s
	\end{aligned}\right)\\
	&= \testfuncarg{k}{t}{x} \Big(\int_{0}^{t} \innerbig{\corrg{t}{s}{k}(x) }{ \delta \dynb \big(\state{s}{x}\big)} + (\nabla \cdot \delta \dynb)\big(\state{s}{x}\big) \ud s\Big).
	\end{align*}
\end{proof}

Then we present a few properties, which will be useful when computing the functional derivative $\derimsecinf$.
The following lemma shows how $\corrz{t}{-s}(\cdot)$ and $\corrg{t}{-s}{k}(\cdot)$ change under the dynamical evolution $\state{s}{\cdot}$.

\begin{lemma}
	When $t\le -s  \le 0$ or $t \ge -s \ge 0$,
	\begin{align}
	\label{eqn::time_trans_corrz}
	\corrz{t}{-s}\big(\state{s}{x}\big) = \corrz{t+s}{0}(x),\qquad
	\corrg{t}{-s}{k}\big(\state{s}{x}\big) = \corrg{t+s}{0}{k}(x).
	\end{align}
\end{lemma}

\begin{proof}
	We first consider the term $\corrz{t}{-s}\big(\state{s}{x}\big)$.
	When $t\ge -s \ge 0$, 
	\begin{align*}
	    \corrz{t}{-s}\big(\state{s}{x}\big) =& 
	    \exp_{\chrontime}\Big(\int_{-s}^{t} \nabla\dynb\big(\statebig{r}{\state{s}{x}}\big)\ud r\Big) \\
	    =& \exp_{\chrontime}\Big(\int_{0}^{t+s} \nabla\dynb(\state{r}{x})\ud r\Big) = \corrz{t+s}{0}(x).
	\end{align*}
	The case for $t\le -s \le 0$ can be similarly verified.
	
	Recall from \eqref{eqn::corrg} that
	\begin{align*}
	&\ \corrg{t}{-s}{k}\big(\state{s}{x}\big) \\
	\myeq{\eqref{eqn::corrg}}& - \corrz{t}{-s}\big(\state{s}{x}\big)^{T} \nabla U_k\Big(\statebig{t}{\state{s}{x}}\Big)  + \int_{-s}^{t}\corrz{r}{-s}\big(\state{s}{x}\big)^{T} \nabla(\nabla \cdot \dynb) \Big(\statebig{r}{\state{s}{x}}\Big)\ud r\\
	=& -\corrz{t+s}{0}(x)^{T} \nabla U_k\big(\state{t+s}{x}\big) + \int_{0}^{t+s} \corrz{r}{0}(x)^{T} \nabla (\nabla \cdot \dynb)\big(\state{r}{x}\big) \ud r \\
	\myeq{\eqref{eqn::corrg}}&\ \corrg{t+s}{0}{k}(x),
	\end{align*}
	where to get the third line, we use the above formula \eqref{eqn::time_trans_corrz} about $\corrz{t}{-s}\big(\state{s}{x}\big)$.
\end{proof}

The following lemma connects $\corrz{t}{0}$ and $\corrg{t}{0}{k}$ with gradients.

\begin{lemma}
	\label{lem::deri_X_t}
	For any $x\in \dom$ and $t\in \Real$, we have
	\begin{align}
	\label{eqn::deri_X_t}
	\nabla_x \state{t}{x} := \begin{bmatrix}
	\frac{\partial \stateidx{t}{x}{i}}{\partial x_j}\end{bmatrix}_{i,j} = \corrz{t}{0}(x),\\
	\label{eqn::deri_F_t}
	\testfuncarg{k}{t}{x} \corrg{t}{0}{k}(x) = \nabla_x \testfuncarg{k}{t}{x},\qquad \text{for } k\in \{0,1\}.
	\end{align}
{As a consequence, $\det\big(\corrz{t}{0}(x)\big) = \jacoarg{t}{x}$.}
\end{lemma}

\begin{proof}
	We fix an index $1\le j\le \dimn$ and consider the dynamics $\td \statesupT{t}{x}{\eps} = \dynb\big(\statesupT{t}{x}{\eps}\big)$,
	$\statesupT{0}{x}{\eps} = x + \eps \basis_j$
	where $\basis_j$ is a vector with the $j^{\text{th}}$ element to be one,
	and all other elements to be zero. Clearly, $\statesupT{t}{x}{0} = \state{t}{x}$.
	
	Let $\deristatesupT{t}{x}{\eps} := \frac{\ud }{\ud\eps}\statesupT{t}{x}{\eps}$. Then
	\begin{align*}
	\td \deristatesupT{t}{x}{\eps} = \frac{\ud}{\ud \eps} \dynb\big(\statesupT{t}{x}{\eps}\big) = \nabla \dynb\big(\statesupT{t}{x}{\eps}\big) \deristatesupT{t}{x}{\eps},\qquad \deristatesupT{0}{x}{\eps} = \basis_j.
	\end{align*}
	Moreover, when $\eps = 0$, we have
	\begin{align*}
	\td \deristatesupT{t}{x}{0} = \nabla \dynb\big(\state{t}{x}\big) \deristatesupT{t}{x}{0}, \qquad \deristatesupT{0}{x}{0} = \basis_j,
	\end{align*}
	whose solution is simply the
	$j^{\text{th}} \text{ column of } \corrz{t}{0}(x)$.
	Besides, the $j^{\text{th}}$ column of $\nabla_x \state{t}{x}$ is given by
	\begin{align*}
	\lim_{\eps\rightarrow 0} \frac{\state{t}{x+\eps \basis_j} - \state{t}{x}}{\eps} 
	= \lim_{\eps\rightarrow 0} \frac{\statesupT{t}{x}{\eps} - \statesupT{t}{x}{0}}{\eps} = \frac{\ud}{\ud \eps} \statesupT{t}{x}{\eps}\Big\rvert_{\eps=0}  = \deristatesupT{t}{x}{0}.
	\end{align*}
	By combining above results, we easily know that $\nabla_x \state{t}{x} = \corrz{t}{0}(x)$.
	
	Next, for $k\in \{0,1\}$ and any $x\in \dom$,
	\begin{align*}
	\nabla \testfuncarg{k}{t}{x} 
	=&\ \testfuncarg{k}{t}{x} \Big(-\big(\nabla_x \state{t}{x}\big)^T\nabla U_k\big(\state{t}{x}\big) + \int_{0}^{t} \big(\nabla_x \state{s}{x}\big)^{T} \nabla (\nabla \cdot \dynb)\big(\state{s}{x}\big)\ud s \Big) \\
	=&\  \testfuncarg{k}{t}{x} \Big(-\corrz{t}{0}(x)^T\nabla U_k\big(\state{t}{x}\big) + \int_{0}^{t} \corrz{s}{0}(x)^{T} \nabla (\nabla \cdot \dynb)\big(\state{s}{x}\big)\ud s \Big) \\
	\myeq{\eqref{eqn::corrg}}&\ \testfuncarg{k}{t}{x} \corrg{t}{0}{k}(x).
	\end{align*}
\end{proof}

\subsection{Proof of \propref{prop::stationary_inf}}
\label{app:proof:10}

We list without proof the following result for the infinite-time case, which can be derived in the same way as \propref{prop::pert_M_2}.
\begin{lemma}
	\label{lem::perturb_M_2_infinite}
	Let $\dynb^{\eps} := \dynb + \eps \delta \dynb$. Suppose that for small enough
	$\eps$, we have $\dynb^{\eps}\in \maniinf$. Then
	\begin{align*}
		\frac{\ud}{\ud\epsilon} \newmsecinf (\dynb+\eps\delta \dynb) \Big\rvert_{\eps=0}
		= 2\ee_{x\sim\rho_0} \bigg[\frac{\newainf(x)}{\binf(x)}
		\Big(\intreal \deritestfuncarg{1}{t}{x}\ud t -
		\newainf(x) \intreal \deritestfuncarg{0}{t}{x} \ud t\Big)
		\bigg],
	\end{align*}
	where $\deritestfuncarg{k}{t}{x}$ is defined in \eqref{eqn::deri_test_fun} for $k = 0, 1$.
\end{lemma}

\begin{lemma}
	\label{lem::functional_derivative_inf_v1}
	The functional derivative of the second moment for the infinite-time case is 
	\begin{align}
		\label{eqn::functional_derivative_inf}
		\begin{aligned}
			\derimsecinf(x) &=
			2 \bigg(\intreal \testfuncarg{0}{s}{x} \opSinfargbig{-s}{\state{s}{x}}
			- \nabla_{x}\Big(\testfuncarg{0}{s}{x} \opGinfargbig{-s}{\state{s}{x}} \Big)\ud s \bigg),
		\end{aligned}
	\end{align}
	where
	\begin{align*}
		& \opSinfarg{s}{x} := \left\{\begin{aligned}
			\frac{\newainf(x)}{\binf(x)} \int_{s}^{\infty} \testfuncarg{1}{t}{x} \corrg{t}{s}{1}(x) \ud t 
			- \frac{\big(\newainf(x)\big)^2}{\binf(x)} \int_{s}^{\infty} \testfuncarg{0}{t}{x} \corrg{t}{s}{0}(x)\ud t,\qquad & \text{ if } s > 0;\\
			-\frac{\newainf(x)}{\binf(x)} \int_{-\infty}^{s} \testfuncarg{1}{t}{x} \corrg{t}{s}{1}(x)\ud t + \frac{\big(\newainf(x)\big)^2}{\binf(x)} \int_{-\infty}^{s} \testfuncarg{0}{t}{x} \corrg{t}{s}{0}(x)\ud t,\qquad & \text{ if } s < 0.
		\end{aligned}\right.\\
		& \opGinfarg{s}{x} := \left\{\begin{aligned}
			\frac{\newainf(x)}{\binf(x)} \int_{s}^{\infty} \testfuncarg{1}{t}{x} \ud t - \frac{\big(\newainf(x)\big)^2}{\binf(x)} \int_{s}^{\infty} \testfuncarg{0}{t}{x} \ud t,\qquad & \text{ if } s > 0;\\
			-\frac{\newainf(x)}{\binf(x)} \int_{-\infty}^{s} \testfuncarg{1}{t}{x} \ud t + \frac{\big(\newainf(x)\big)^2}{\binf(x)} \int_{-\infty}^{s} \testfuncarg{0}{t}{x} \ud t,\qquad & \text{ if } s < 0.
		\end{aligned}\right.
	\end{align*}
\end{lemma}
When $s = 0$, $\opSinfarg{0}{\cdot}, \opGinfarg{0}{\cdot}$ are not specified above, because they will not affect the functional derivative $\derimsecinf$ by changing values at a single point.

\begin{proof}
	In \lemref{lem::perturb_M_2_infinite}, we need to simplify the term $\intreal \deritestfuncarg{k}{t}{x}\ud t$. By plugging the formula of $\deritestfunc{k}$ from \eqref{eqn::deritestfunc}, we have
	\begin{align*}
		&\ \intreal \deritestfuncarg{k}{t}{x}\ud t \\
		\myeq{\eqref{eqn::deritestfunc}}&\ \intreal \testfuncarg{k}{t}{x} \Big(\int_{0}^{t}
		\innerBig{\corrg{t}{s}{k}(x) }{ \delta \dynb\big(\state{s}{x}\big)} + (\nabla \cdot \delta\dynb)\big(\state{s}{x}\big) \ud s\Big) \ud t \\
		=&\ \int_{0}^{\infty} \int_{s}^{\infty} \testfuncarg{k}{t}{x} \innerBig{\corrg{t}{s}{k}(x) }{ \delta \dynb\big(\state{s}{x}\big)} + \testfuncarg{k}{t}{x} (\nabla \cdot \delta\dynb)\big(\state{s}{x}\big) \ud t \ud s \\
		&\qquad - \int_{-\infty}^{0} \int_{-\infty}^{s} \testfuncarg{k}{t}{x} \innerBig{\corrg{t}{s}{k}(x)}{ \delta \dynb\big(\state{s}{x}\big)} + \testfuncarg{k}{t}{x} (\nabla \cdot \delta\dynb)\big(\state{s}{x}\big) \ud t \ud s.
	\end{align*}
	By plugging the last equation into \lemref{lem::perturb_M_2_infinite} and with straightforward simplification, we can verify that 
	\begin{align*}
		& \frac{\ud}{\ud\epsilon} \newmsecinf (\dynb+\eps\delta \dynb) \Big\rvert_{\eps=0} \\
		=&\ 2 \ee_{x\sim\rho_0} \Big[\intreal \innerBig{\opSinfarg{s}{x}}{\delta\dynb\big(\state{s}{x}\big)}\ud s + \intreal \opGinfarg{s}{x}(\nabla \cdot \delta\dynb)\big(\state{s}{x}\big)\ud s \Big] \\
		=&\ 2 \intreal \int_{\dom} \rho_0(x) \Big( \innerBig{\opSinfarg{s}{x}}{\delta\dynb\big(\state{s}{x}\big)} +  \opGinfarg{s}{x} (\nabla \cdot \delta\dynb)\big(\state{s}{x}\big) \Big)\ud x\ud s\\
		\myeq{$\wt{x}=\state{s}{x}$}& \hspace{1.5em}  2 \intreal \int_{\dom} \bigg( \rho_0\big(\state{-s}{\wt{x}}\big) \innerBig{\opSinfargbig{s}{\state{-s}{\wt{x}}} }{\delta\dynb(\wt{x})} \jacoarg{-s}{\wt{x}} \\
		&\hspace{10em} + \rho_0\big(\state{-s}{\wt{x}}\big) \opGinfargbig{s}{\state{-s}{\wt{x}}} (\div\delta\dynb)(\wt{x}) \jacoarg{-s}{\wt{x}}\bigg) \ud \wt{x}\ud s \\
		\myeq{\eqref{eqn::testfunc}}&\ 2
		\intreal \int_{\dom} \testfuncarg{0}{-s}{x} \innerBig{\opSinfargbig{s}{\state{-s}{x}}}{\delta\dynb({x})} - \innerBig{\nabla \Big(\testfuncarg{0}{-s}{x}\opGinfargbig{s}{\state{-s}{x}}\Big)}{\delta\dynb(x)}\ud x \ud s.
	\end{align*}
	The integration by parts in the last line is valid because $\delta\dynb$ vanishes at the boundary $\partial\dom$. By comparing the last equation with \eqref{eqn::func_deri_defn}, we can immediately obtain \eqref{eqn::functional_derivative_inf} after straightforward simplifications.
\end{proof}

Then we need to simplify $\opSinfargbig{-s}{\state{s}{x}}$ and $\opGinfargbig{-s}{\state{s}{x}}$.
\begin{lemma}
	\label{lem::opsg}
	$\opSinfargbig{-s}{\state{s}{x}}$ and $\opGinfargbig{-s}{\state{s}{x}}$ have the following form 
	\begin{align*}
		\opSinfargbig{-s}{\state{s}{x}} &= \left\{
		\begin{aligned}
			-\frac{\newainf(x)}{\binf(x)} \int_{-\infty}^{0} \nabla \testfuncarg{1}{t}{x} \ud t
			+ \frac{\big(\newainf(x)\big)^2}{\binf(x)} \int_{-\infty}^{0} \nabla \testfuncarg{0}{t}{x} \ud t,\qquad & \text{ if } s > 0,\\
			\frac{\newainf(x)}{\binf(x)}\int_{0}^{\infty} \nabla \testfuncarg{1}{t}{x} \ud t
			- \frac{\big(\newainf(x)\big)^2}{\binf(x)} \int_{0}^{\infty} \nabla \testfuncarg{0}{t}{x} \ud t,\qquad & \text{ if } s < 0.
		\end{aligned}\right.\\
		\opGinfargbig{-s}{\state{s}{x}} &= \left\{
		\begin{aligned}
			-\frac{\newainf(x)}{\binf(x)} \int_{-\infty}^{0} \testfuncarg{1}{t}{x} \ud t
			+ \frac{\big(\newainf(x)\big)^2}{\binf(x)} \int_{-\infty}^{0} \testfuncarg{0}{t}{x} \ud t,\qquad & \text{ if } s > 0,\\
			\frac{\newainf(x)}{\binf(x)}\int_{0}^{\infty} \testfuncarg{1}{t}{x} \ud t
			- \frac{\big(\newainf(x)\big)^2}{\binf(x)} \int_{0}^{\infty} \testfuncarg{0}{t}{x} \ud t,\qquad & \text{ if } s < 0.
		\end{aligned}\right.
	\end{align*}
\end{lemma}
\begin{proof}
	When $s > 0$,
	\begin{align*}
		&\ \opSinfargbig{-s}{\state{s}{x}} \\
		=&\ -\frac{\newainf\big(\state{s}{x}\big)}{\binf\big(\state{s}{x}\big)} \int_{-\infty}^{-s} \testfuncargbig{1}{t}{\state{s}{x}} \corrg{t}{-s}{1}\big(\state{s}{x}\big) \ud t \\
		& \qquad + \frac{\Big(\newainf\big(\state{s}{x}\big)\Big)^2}{\binf\big(\state{s}{x}\big)} \int_{-\infty}^{-s} \testfuncargbig{0}{t}{\state{s}{x}} \corrg{t}{-s}{0}\big(\state{s}{x}\big) \ud t \\
		\myeq{\eqref{eqn::testfunc_trans},\eqref{eqn::time_trans_ainf},\eqref{eqn::time_trans_corrz}}&\hspace{1em} -\frac{\newainf(x)}{\binf(x)} \int_{-\infty}^{-s} \testfuncarg{1}{t+s}{x} \corrg{t+s}{0}{1}(x) \ud t
		+ \frac{\big(\newainf(x)\big)^2}{\binf(x)} \int_{-\infty}^{-s} \testfuncarg{0}{t+s}{x} \corrg{t+s}{0}{0}(x) \ud t \\
		=&\ -\frac{\newainf(x)}{\binf(x)} \int_{-\infty}^{0} \testfuncarg{1}{t}{x} \corrg{t}{0}{1}(x) \ud t
		+ \frac{\big(\newainf(x)\big)^2}{\binf(x)} \int_{-\infty}^{0} \testfuncarg{0}{t}{x} \corrg{t}{0}{0}(x) \ud t\\
		\myeq{\eqref{eqn::deri_F_t}}&\ -\frac{\newainf(x)}{\binf(x)} \int_{-\infty}^{0} \nabla\testfuncarg{1}{t}{x} \ud t
		+ \frac{\big(\newainf(x)\big)^2}{\binf(x)} \int_{-\infty}^{0} \nabla \testfuncarg{0}{t}{x} \ud t,
	\end{align*}
	which is clearly independent of $s$ from this expression.
	Similarly, when $s < 0$,
	\begin{align*}
		&\ \opSinfargbig{-s}{\state{s}{x}} \\
		=&\ \frac{\newainf\big(\state{s}{x}\big)}{\binf\big(\state{s}{x}\big)} \int_{-s}^{\infty} \testfuncargbig{1}{t}{\state{s}{x}} \corrg{t}{-s}{1}\big(\state{s}{x}\big) \ud t \\
		&\qquad - \frac{\Big(\newainf\big(\state{s}{x}\big)\Big)^2}{\binf\big(\state{s}{x}\big)} \int_{-s}^{\infty} \testfuncargbig{0}{t}{\state{s}{x}} \corrg{t}{-s}{0}\big(\state{s}{x}\big) \ud t\\
		=&\ \frac{\newainf(x)}{\binf(x)}\int_{-s}^{\infty} \testfuncarg{1}{t+s}{x} \corrg{t+s}{0}{1}(x) \ud t
		- \frac{\big(\newainf(x)\big)^2}{\binf(x)} \int_{-s}^{\infty} \testfuncarg{0}{t+s}{x} \corrg{t+s}{0}{0}(x) \ud t \\
		=&\ \frac{\newainf(x)}{\binf(x)}\int_{0}^{\infty} \testfuncarg{1}{t}{x} \corrg{t}{0}{1}(x) \ud t
		- \frac{\big(\newainf(x)\big)^2}{\binf(x)} \int_{0}^{\infty} \testfuncarg{0}{t}{x} \corrg{t}{0}{0}(x) \ud t\\
		=&\ \frac{\newainf(x)}{\binf(x)}\int_{0}^{\infty} \nabla \testfuncarg{1}{t}{x} \ud t
		- \frac{\big(\newainf(x)\big)^2}{\binf(x)} \int_{0}^{\infty} \nabla\testfuncarg{0}{t}{x} \ud t,
	\end{align*}
	which is again independent of $s$. We can similarly simplify $\opGinfargbig{-s}{\state{s}{x}}$.
\end{proof}

By plugging the formula in \lemref{lem::opsg} into \eqref{eqn::functional_derivative_inf}, 
\begin{align*}
	& \intreal \testfuncarg{0}{s}{x} \opGinfargbig{-s}{\state{s}{x}} \ud s \\
	=& \int_{0}^{\infty} \testfuncarg{0}{s}{x} \ud s\ 
	\Big( -\frac{\newainf(x)}{\binf(x)} \int_{-\infty}^{0} \testfuncarg{1}{t}{x} \ud t
	+ \frac{\big(\newainf(x)\big)^2}{\binf(x)} \int_{-\infty}^{0} \testfuncarg{0}{t}{x} \ud t
	\Big) \\
	&\qquad  + \int_{-\infty}^{0} \testfuncarg{0}{s}{x} \ud s\ 
	\Big(\frac{\newainf(x)}{\binf(x)}\int_{0}^{\infty} \testfuncarg{1}{t}{x} \ud t
	- \frac{\big(\newainf(x)\big)^2}{\binf(x)} \int_{0}^{\infty} \testfuncarg{0}{t}{x} \ud t\Big)\\
	=& \frac{\newainf(x)}{\binf(x)} \Big(\int_{-\infty}^{0} \testfuncarg{0}{t}{x} \ud t \int_{0}^{\infty} \testfuncarg{1}{t}{x} \ud t - \int_{0}^{\infty} \testfuncarg{0}{t}{x} \ud t \int_{-\infty}^{0} \testfuncarg{1}{t}{x} \ud t \Big).
\end{align*}
Besides,
\begin{align*}
	&  \intreal \testfuncarg{0}{s}{x} \opSinfargbig{-s}{\state{s}{x}} \ud s \\
	=& -\frac{\newainf(x)}{\binf(x)} \int_{0}^{\infty} \testfuncarg{0}{s}{x} \ud s  \int_{-\infty}^{0} \nabla \testfuncarg{1}{t}{x} \ud t
	+ \frac{\big(\newainf(x)\big)^2}{\binf(x)} \int_{0}^{\infty} \testfuncarg{0}{s}{x} \ud s  \int_{-\infty}^{0} \nabla \testfuncarg{0}{t}{x} \ud t\\
	& \qquad + \frac{\newainf(x)}{\binf(x)} \int_{-\infty}^{0} \testfuncarg{0}{s}{x} \ud s  \int_{0}^{\infty} \nabla \testfuncarg{1}{t}{x} \ud t
	- \frac{\big(\newainf(x)\big)^2}{\binf(x)} \int_{-\infty}^{0} \testfuncarg{0}{s}{x} \ud s  \int_{0}^{\infty} \nabla \testfuncarg{0}{t}{x} \ud t.
\end{align*}

By combining previous results,
\begin{align*}
	&\ \ \ \ 
	\frac{1}{2}\times \derimsecinf(x)\\
	&\myeq{\eqref{eqn::functional_derivative_inf}}
	\intreal \testfuncarg{0}{s}{x} \opSinfargbig{-s}{\state{s}{x}} \ud s - \nabla \Big(\intreal \testfuncarg{0}{s}{x} \opGinfargbig{-s}{\state{s}{x}} \ud s\Big)\\
	&= -\frac{\newainf(x)}{\binf(x)} \int_{0}^{\infty} \testfuncarg{0}{t}{x} \ud t  \int_{-\infty}^{0} \nabla \testfuncarg{1}{t}{x} \ud t
	+ \frac{\big(\newainf(x)\big)^2}{\binf(x)} \int_{0}^{\infty} \testfuncarg{0}{t}{x} \ud t  \int_{-\infty}^{0} \nabla \testfuncarg{0}{t}{x} \ud t\\
	&\ \ + \frac{\newainf(x)}{\binf(x)} \int_{-\infty}^{0} \testfuncarg{0}{t}{x} \ud t  \int_{0}^{\infty} \nabla \testfuncarg{1}{t}{x} \ud t
	- \frac{\big(\newainf(x)\big)^2}{\binf(x)} \int_{-\infty}^{0} \testfuncarg{0}{t}{x} \ud t  \int_{0}^{\infty} \nabla\testfuncarg{0}{t}{x} \ud t\\
	& \ \ - \nabla \Big(\frac{\newainf(x)}{\binf(x)}\Big) \Big(\int_{-\infty}^{0} \testfuncarg{0}{t}{x} \ud t  \int_{0}^{\infty} \testfuncarg{1}{t}{x} \ud t - \int_{0}^{\infty} \testfuncarg{0}{t}{x} \ud t \int_{-\infty}^{0} \testfuncarg{1}{t}{x} \ud t \Big)\\
	& \ \ - \frac{\newainf(x)}{\binf(x)} \nabla \Big(\int_{-\infty}^{0} \testfuncarg{0}{t}{x} \ud t \int_{0}^{\infty} \testfuncarg{1}{t}{x} \ud t - \int_{0}^{\infty} \testfuncarg{0}{t}{x} \ud t \int_{-\infty}^{0} \testfuncarg{1}{t}{x} \ud t \Big)\\
	&= \frac{\big(\newainf(x)\big)^2}{\binf(x)} \Big(\int_{0}^{\infty} \testfuncarg{0}{t}{x} \ud t \int_{-\infty}^{0} \nabla \testfuncarg{0}{t}{x} \ud t -  \int_{-\infty}^{0} \testfuncarg{0}{t}{x} \ud t  \int_{0}^{\infty} \nabla\testfuncarg{0}{t}{x} \ud t \Big)  \\
	&\ \ - \nabla \Big(\frac{\newainf(x)}{\binf(x)}\Big) \Big(\int_{-\infty}^{0} \testfuncarg{0}{t}{x} \ud t  \int_{0}^{\infty} \testfuncarg{1}{t}{x} \ud t - \int_{0}^{\infty} \testfuncarg{0}{t}{x} \ud t \int_{-\infty}^{0} \testfuncarg{1}{t}{x} \ud t \Big)\\
	& \ \ + \frac{\newainf(x)}{\binf(x)} \Big(\int_{0}^{\infty} \nabla \testfuncarg{0}{t}{x}\ud t \int_{-\infty}^{0} \testfuncarg{1}{t}{x}\ud t - \int_{-\infty}^{0} \nabla \testfuncarg{0}{t}{x} \ud t \int_{0}^{\infty} \testfuncarg{1}{t}{x}\ud t \Big).
\end{align*}

It can be directly verify that
\begin{align*}
	\nabla \big(\frac{\newainf}{\binf}\big)(x) =
	\frac{\newainf(x)}{\binf(x)}
	\Big( \frac{\intreal\nabla\testfuncarg{1}{t}{x}\ud t}{\intreal\testfuncarg{1}{t}{x}\ud t} - 2 \frac{\intreal\nabla\testfuncarg{0}{t}{x}\ud t}{\intreal\testfuncarg{0}{t}{x}\ud t}\Big).
\end{align*}
By plugging this into the expression of the functional derivative and dividing both sides by $\frac{\newainf}{\binf}$,
\begin{align*}
	&\hspace{1em} \frac{\partitionzero \binf(x)}{2\newainf(x)} \times \derimsecinf(x)\\
	&= \frac{\intreal\testfuncarg{1}{t}{x}\ud t}{\binf(x)} \Big(\int_{0}^{\infty} \testfuncarg{0}{t}{x}\ud t  \int_{-\infty}^{0} \nabla \testfuncarg{0}{t}{x} \ud t -  \int_{-\infty}^{0} \testfuncarg{0}{t}{x} \ud t  \int_{0}^{\infty} \nabla\testfuncarg{0}{t}{x} \ud t \Big)  \\
	&\ \ -\frac{\intreal\nabla\testfuncarg{1}{t}{x}\ud t}{\intreal\testfuncarg{1}{t}{x}\ud t} \Big(\int_{-\infty}^{0} \testfuncarg{0}{t}{x} \ud t  \int_{0}^{\infty} \testfuncarg{1}{t}{x} \ud t - \int_{0}^{\infty} \testfuncarg{0}{t}{x} \ud t  \int_{-\infty}^{0} \testfunc{1}_t \ud t \Big)\\
	&\ \ +2 \frac{\intreal\nabla\testfuncarg{0}{t}{x}\ud t}{\intreal\testfuncarg{0}{t}{x}\ud t} \Big(\int_{-\infty}^{0} \testfuncarg{0}{t}{x} \ud t \int_{0}^{\infty} \testfuncarg{1}{t}{x} \ud t - \int_{0}^{\infty} \testfuncarg{0}{t}{x} \ud t  \int_{-\infty}^{0} \testfuncarg{1}{t}{x} \ud t \Big)\\
	& \ \ + \Big(\int_{0}^{\infty} \nabla \testfuncarg{0}{t}{x}\ud t \int_{-\infty}^{0} \testfuncarg{1}{t}{x}\ud t - \int_{-\infty}^{0} \nabla \testfuncarg{0}{t}{x} \ud t \int_{0}^{\infty} \testfuncarg{1}{t}{x}\ud t \Big).
\end{align*}
We keep the terms involving $\int\nabla\testfuncarg{1}{t}{x}\ud t$ untouched and we only try to simplify terms involving $\int\nabla\testfuncarg{0}{t}{x}\ud t$.
The coefficient for $\int_{0}^{\infty} \nabla \testfuncarg{0}{t}{x}\ud t$ is
\begin{align*}
	&  \frac{1}{\binf(x)}\left(
	\begin{aligned}
		& -\intreal \testfuncarg{1}{t}{x}\ud t \int_{-\infty}^{0} \testfuncarg{0}{t}{x}\ud t
		+ 2 \int_{-\infty}^{0} \testfuncarg{0}{t}{x} \ud t  \int_{0}^{\infty} \testfuncarg{1}{t}{x} \ud t \\
		& - 2\int_{0}^{\infty} \testfuncarg{0}{t}{x} \ud t \int_{-\infty}^{0} \testfuncarg{1}{t}{x} \ud t
		+ \intreal\testfuncarg{0}{t}{x}\ud t \int_{-\infty}^{0} \testfuncarg{1}{t}{x}\ud t
	\end{aligned}
	\right) \\
	=& \frac{1}{\binf(x)} \Big(\int_{-\infty}^0 \testfuncarg{0}{t}{x}\ud t \int_{0}^{\infty} \testfuncarg{1}{t}{x}\ud t - \int_{0}^{\infty} \testfuncarg{0}{t}{x} \ud t \int_{-\infty}^{0} \testfuncarg{1}{t}{x}\ud t \Big).
\end{align*}
Similarly, the coefficient for $\int_{-\infty}^{0} \nabla \testfuncarg{0}{t}{x}\ud t$ is
\begin{align*}
	& \frac{1}{\binf(x)}\left(
	\begin{aligned}
		&    \intreal \testfuncarg{1}{t}{x}\ud t \int_{0}^{\infty} \testfuncarg{0}{t}{x}\ud t
		+ 2 \int_{-\infty}^{0} \testfuncarg{0}{t}{x} \ud t  \int_{0}^{\infty} \testfuncarg{1}{t}{x} \ud t \\
		& - 2\int_{0}^{\infty} \testfuncarg{0}{t}{x} \ud t \int_{-\infty}^{0} \testfuncarg{1}{t}{x} \ud t
		- \intreal \testfuncarg{0}{t}{x}\ud t  \int_{0}^{\infty} \testfuncarg{1}{t}{x}\ud t
	\end{aligned}
	\right)\\
	=& \frac{1}{\binf(x)}  \Big(\int_{-\infty}^{0} \testfuncarg{0}{t}{x}\ud t  \int_{0}^{\infty} \testfuncarg{1}{t}{x}\ud t - \int_{0}^{\infty} \testfuncarg{0}{t}{x}\ud t \int_{-\infty}^{0} \testfuncarg{1}{t}{x}\ud t\Big)
\end{align*}
Hence,
\begin{align*}
	&\ \ \ \  \frac{\partitionzero \binf(x)}{2\newainf(x)} \times \derimsecinf(x)\\
	&= \Big(\int_{-\infty}^{0} \testfuncarg{0}{t}{x}\ud t  \int_{0}^{\infty} \testfuncarg{1}{t}{x}\ud t - \int_{0}^{\infty} \testfuncarg{0}{t}{x}\ud t  \int_{-\infty}^{0} \testfuncarg{1}{t}{x}\ud t\Big) \times \\
	&\qquad \qquad\qquad   \Big(\frac{\intreal \nabla \testfuncarg{0}{t}{x}\ud t}{\intreal \testfuncarg{0}{t}{x}\ud t} - \frac{\intreal \nabla \testfuncarg{1}{t}{x}\ud t}{\intreal\testfuncarg{1}{t}{x}\ud t}\Big)\\
	&= \Big(\int_{-\infty}^{0} \testfuncarg{0}{t}{x}\ud t \int_{0}^{\infty} \testfuncarg{1}{t}{x}\ud t - \int_{0}^{\infty} \testfuncarg{0}{t}{x}\ud t \int_{-\infty}^{0} \testfuncarg{1}{t}{x}\ud t\Big)
	\big(-\nabla \ln (\newainf)(x)\big).
\end{align*}
After straightforward simplification, we obtain \eqref{eqn::func_deri_inf}. \hfill $\square$

\section{Proof of Proposition \ref{prop::existence} and discussion
about its assumptions and implications}
\label{sec:proof:zero-variance}

In this section, we prove Proposition~\ref{prop::existence}, and discuss its assumptions (in particular the Morse function condition) as well as some of it implications (including {the} settings that go beyond the ones in the proposition). In particular, we solve the Poisson equation~\eqref{eqn::poisson} when $\rho_0$ is the standard Gaussian density in~$\Rd$ and $\rho_1$ is the density of a Gaussian mixture distribution.

\subsection{Proof of \propref{prop::existence} when $\sigmamat = 1$}
\label{app::proof::existence}

{We proceed in three steps:}

{\bf Step 1: }{We shall first establish the limiting behavior of the dynamics.}\medskip

More specifically, the trajectory $t\to\state{t}{x}$ will converge to a local maximum of $V$ in the forward direction and converge to a local minimum in the backward direction, except at a set of points with measure zero.

  \begin{lemma}
  	\label{lemma::critical_pts}
  	For any $x\in \dom$, we have  $\lim_{\abs{t}\rightarrow\infty}\abs{\dynb(\state{t}{x})} = 0$.
  \end{lemma}
\begin{proof}
	We only need to show one direction $t\rightarrow\infty$ and the other case follows similarly. If this does not hold, then there exists $\eps>0$ and a monotone increasing sequence $\{t_k\}_{k=1}^{\infty}$ such that $\absbig{\dynb\big(\state{t_k}{x}\big)}\ge \eps$ and $\lim_{k\rightarrow\infty} t_k = \infty$.
	Consider 
	\begin{align*}
		\td \absbig{\dynb\big(\state{t}{x}\big)}^2 &= \td \absbig{\nabla V\big(\state{t}{x}\big)}^2 \\
		&= 2\innerBig{\nabla V\big(\state{t}{x}\big)}{\nabla^2 V\big(\state{t}{x}\big)\nabla V\big(\state{t}{x}\big)} 
		\ge -2C \absbig{\dynb\big(\state{t}{x}\big)}^2,
	\end{align*}
	where $C := \sup_{x\in \dom} \norm{\nabla^2 V(x)}<\infty$. 
	By Gr{\"o}nwall's inequality,
	\begin{align*}
		\absbig{\dynb\big(\state{t}{x}\big)} \ge \absbig{\dynb\big(\state{s}{x}\big)} e^{-C(t-s)},
	\end{align*} for all $t\ge s\ge 0$. 
	Without loss of generality, we can ensure that $t_k - t_{k-1} \ge 1$. 
	Hence,
	\begin{align*}
		V\big(\state{t}{x}\big) - V\big(\state{0}{x}\big) &= \int_{0}^{t} \absbig{\dynb\big(\state{s}{x}\big)}^2\ud s 
		\ge \sum_{k=1}^{\infty} \indi_{[0,t]}(t_{k+1}) \int_{t_k}^{t_{k}+1} \eps e^{-C (t-t_k)}\ud t \\
		&= \sup\{k: t_{k+1} \le t\}  \frac{\eps\big(1-e^{-C}\big)}{C},
	\end{align*}
	which will diverge to infinity as $t\rightarrow\infty$. This contradicts with the boundedness of $V$ and thus the assumption does not hold. As a remark, the notation $\indi_{A}()$ is an indicator function for the set $A$.
\end{proof}

\begin{lemma}
	\label{lemma::limiting_behavior}
	Suppose $x$ is not on the stable or unstable manifold of a saddle point. 
	Then the trajectory $\{\state{t}{x}\}_{t\in\Real}$ must converge to a local maximum of $V$ in the forward direction and a local minimum of $V$ in the backward direction.
\end{lemma}

\begin{proof}
	Since the torus $\dom = [0,1]^\dimn$ is bounded, the trajectory $\{\state{t}{x}\}_{t\ge 0}$ must be bounded and there exists an increasing sequence $\{t_k\}_{k=1}^{\infty}$ such that $\big\{\state{t_k}{x}\big\}_{k=1}^{\infty}$ is convergent by Bolzano-Weierstrass theorem and let us denote the limit as $\xst$. By \lemref{lemma::critical_pts}, we know $\dynb(\xst) = \vectorzero_{\dimn}$ and thus $\xst$ is a critical point. By the assumption, $\xst$ is not a saddle point nor a local minimum, that is, $\xst$ must be a local maximum of $V$. 
	By the assumption that $V$ is a Morse function, the critical point $\xst$ has a non-degenerate Hessian. 
	After the trajectory enters its basin of attraction (containing an open ball around $\xst$), the trajectory $\{\state{t}{x}\}_{t\in\Real}$ will eventually converge to $\xst$. The backward direction can be proved in a similar way.
\end{proof}

\begin{lemma}
	\label{lem::conv_testfunc}
	Under the same assumption as in \lemref{lemma::limiting_behavior},
	we know $\int_{0}^{\infty} \testfuncarg{k}{t}{x}\ud t < \infty$ and $\int_{-\infty}^{0} \testfuncarg{k}{t}{x}\ud t < \infty$. 
	In particular, $\lim_{\abs{t}\rightarrow\infty} \testfuncarg{k}{t}{x} = 0$ and $\lim_{\abs{t}\rightarrow\infty} \jacoarg{t}{x} = 0$.
\end{lemma}

\begin{proof}
Without loss of generality, we only consider the forward branch. 
Since $V$ is assumed to be a Morse function, 
the Hessian $\nabla^2 V(\xst) < 0$ is non-degenerate and $0 > \tr\big(\nabla^2 V(\xst)\big) = \Laplace V(\xst)$. Therefore, 
\begin{align}
	\label{eqn::limit_divg}
	\lim_{t\rightarrow\infty} \div\dynb\big(\state{t}{x}\big) = \div\dynb(\xst) = \Laplace V(\xst) < 0.
\end{align} 
Since $\rho_k = e^{-U_k}/\partition_k$ are bounded {on the torus}, we know 
\begin{align*}
	\int_{0}^{\infty} \testfuncarg{k}{t}{x}\ud t = \int_{0}^{\infty} e^{-U_k\big(\state{t}{x}\big)} \jacoarg{t}{x}\ud t \le C \int_{0}^{\infty} \jacoarg{t}{x}\ \ud t = C\int_{0}^{\infty} e^{\int_{0}^{t} \div\dynb\big(\state{s}{x}\big)\ud s}\ud t,
\end{align*}
where $C := \sup_{x\in\dom} \max\{e^{-U_0(x)}, e^{-U_1(x)}\} < \infty$ herein.
From \eqref{eqn::limit_divg}, there exists $\beta > 0$ and $\tau>0$ such that $\div\dynb(\state{s}{x})\le -\beta$ for all $s\ge \tau$, then if $t\ge \tau$,
\begin{align*}
	\int_{0}^{t}\div\dynb\big(\state{s}{x}\big)\ud s \le \int_{0}^{\tau}\div\dynb\big(\state{s}{x}\big)\ud s - \beta(t-\tau),
\end{align*}
and therefore, 
\begin{align*}
	\int_{0}^{\infty} \testfuncarg{k}{t}{x}\ud t\le C\Big(\int_{0}^{\tau} e^{\int_{0}^{t} \div\dynb\big(\state{s}{x}\big)}\ud t + \int_{\tau}^{\infty} e^{\int_{0}^{\tau}\div\dynb\big(\state{s}{x}\big)\ud s} e^{-\beta(t-\tau)}\ud t \Big)< \infty.
\end{align*}
In particular, when $t\ge \tau$, 
\begin{align*}
	\jacoarg{t}{x} \le e^{\int_{0}^{\tau}\div\dynb\big(\state{s}{x}\big)\ud s} e^{-\beta(t-\tau)},
\end{align*}
which converges to zero exponentially fast as $t\rightarrow\infty$. 
The same conclusion holds for $\testfuncarg{k}{t}{x}\equiv e^{-U_k\big(\state{t}{x}\big)}\jacoarg{t}{x}$ when $t\rightarrow\infty$.
\end{proof}

{\bf Step 2:}{ We verify that $\dynb=\nabla V$ is a zero-variance dynamics.}
\medskip

We need to show that $\intreal (\rho_1-\rho_0)\big(\state{t}{x}\big) \jacoarg{t}{x} \ud t = 0$ almost everywhere on $\dom$.

{Under the same assumption as \lemref{lemma::limiting_behavior}}, let us consider  
\begin{align*}
	\intreal (\rho_1-\rho_0)\big(\state{t}{x}\big) \jacoarg{t}{x} \ud t &= \intreal \Laplace V\big(\state{t}{x}\big) e^{\int_{0}^{t} \Laplace V\big(\state{s}{x}\big)\ud s}\ud t \\
	&= \intreal \frac{\ud}{\ud t}\jacoarg{t}{x}\ud t \\
	&= \lim_{t\rightarrow\infty}\jacoarg{t}{x} - \lim_{t\rightarrow-\infty}\jacoarg{t}{x} = 0.
\end{align*}
The last line comes from \lemref{lem::conv_testfunc}. 
{The validity of the above equation almost everywhere on $\dom$ will be explained in Step 3.}

\medskip
{\bf Step 3:} {We  prove that $\dynb = \nabla V \in \maniinf{}$ in the sense of \defref{defn::maniinf}, \ie{}, such a gradient ascent dynamics is a valid one for the \itneis{} scheme.}

If we can find two open subsets $\domsubset_1, \domsubset_2$ such that $x\to\int_{0}^{\infty}\testfuncarg{k}{t}{x}\ud t$ is continuous on $\domsubset_1$, $x\to\int_{-\infty}^{0}\testfuncarg{k}{t}{x}\ud t$ is continuous on $\domsubset_2$, and both $\dom\backslash\domsubset_1$ and $\dom\backslash\domsubset_2$ have Lebesgue measure zero, then clearly $\mho(\dynb)\supset \domsubset_1\cap \domsubset_2$ and $\dynb\in\maniinf{}$.

Due to the symmetric role of forward and backward branches of trajectories, it is then sufficient to prove the following lemma.
\begin{lemma}
	\label{lem::continuity_F_on_D}
 There exists an open subset $\domsubset\subset \dom$ such that $x\rightarrow\int_{0}^{\infty} \testfuncarg{k}{t}{x}$ is continuous and $\dom\backslash\domsubset$ has measure zero.
\end{lemma}

\begin{proof}
	Let us denote the local maxima of $V$ as $\pts_1, \pts_2, \cdots, \pts_r$. 
	The index $r <\infty$ because $V$ is a Morse function and $\dom$ is compact.
	Since $\nabla^2 V(\pts_i)<0$ for $1\le i \le r$, there exists a local neighborhood $\ball{\delta_i}{\pts_i}$ such that $\lim_{t\rightarrow\infty} \state{t}{z} = \pts_i$ if $z\in \ball{\delta_i}{\pts_i}$. Hence, it is not hard to characterize the basin of attraction of $\pts_i$
	\begin{align*}
		O_i = \Big\{\state{t}{x}:\ x\in \ball{\delta_i}{\pts_i},\ t\le 0\Big\}
	\end{align*}
	which is open. Then define an open subset $\domsubset := \cup_{i=1}^{r} O_i$. By \lemref{lemma::limiting_behavior}, we know $\dom\backslash\domsubset$ has Lebesgue measure zero, 
	since there is only a finite number of critical points and the {stable/unstable} manifold of saddle points has measure zero.
	Next, we still need to verify $z\rightarrow\int_{0}^{\infty} \testfuncarg{k}{t}{z}\ud t$ is continuous at an arbitrary point $x\in\domsubset$. Since $\domsubset=\cup_{i=1}^{r} O_i$ and $O_i$ are open and disjoint, it is sufficient to verify this conclusion for $x\in O_i$ for an arbitrary index $i$.
	
	By the smoothness of $V$, there exists a local neighborhood $O_{\delta} := \big\{z\in O_i:\ V(\pts_i) - \delta < V(z)\le V(\pts_i)
	\big\}$ such that 
	$\frac{(\div\dynb)(z)}{(\div\dynb)(\pts_i)} \in (\frac{1}{2}, \frac{3}{2})$ for every $z\in O_{\delta}$.
	{Let $M = \sup\big\{\max\{e^{-U_1(x)},e^{-U_0(x)}\}: x\in \dom\big\}$. We know $M < \infty$ because $\dom$ is compact and $U_0, U_1$ are smooth.}
	Define $\tau := \inf\{t\ge 1:\ \state{t}{x}\in O_{\delta}\}$. 
	Due to the smoothness of $\dynb$, for any integer $j\ge 2$, we can choose a small neighborhood $\ball{\delta_j}{x}$ such that $\state{t}{z}\in O_{\delta}$ for all $t\ge j\tau$ and for all $z\in \ball{\delta_j}{x}$ (note that $O_{\delta}$ is automatically a trapping region of $\dynb$ by construction). Then {when $z\in \ball{\delta_j}{x}$,}
	\begin{align*}
		\int_{j\tau}^{\infty} \testfuncarg{k}{t}{z}\ud t
		& \le \int_{j\tau}^{\infty} M \jacoarg{t}{z}\ud t\\
		& \le \int_{j\tau}^{\infty} M \jacoarg{j\tau}{z} e^{(\div\dynb)(\pts_i)\frac{1}{2}(t - j\tau)}\ud t\\
		&\le 2M \jacoarg{j\tau}{z} \frac{1}{-(\div\dynb)(\pts_i)}\\
		&=  \frac{2M}{-(\div\dynb)(\pts_i)} \Big(\jacoarg{j\tau}{z} - \jacoarg{j\tau}{x} +  \jacoarg{j\tau}{x}\Big).
	\end{align*}
	For an arbitrary $\eps>0$, by \lemref{lem::conv_testfunc}, we can pick $j$ large enough such that $\jacoarg{j\tau}{x} < \frac{-(\div\dynb)(\pts_i)}{2M} \frac{\eps}{6}$. Next we can accordingly pick $\delta_j$ small enough such that $\jacoarg{j\tau}{z} - \jacoarg{j\tau}{x} < \frac{-(\div\dynb)(\pts_i)}{2M} \frac{\eps}{6}$ for all $z\in \ball{\delta_j}{x}$. In this way, we can ensure that 
	\begin{align}
		\label{eqn::tail_int}
		\int_{j\tau}^{\infty} \testfuncarg{k}{t}{z}\le \frac{\eps}{6} + \frac{\eps}{6} = \frac{\eps}{3}, \qquad \forall z\in \ball{\delta_j}{x}.
	\end{align}
	Furthermore, {due to the smoothness of $\dynb$,} we can choose $\delta_j$ (possibly even smaller) so that 
	\begin{align}
		\label{eqn::diff_int}
		\Abs{\int_{0}^{j\tau} \testfuncarg{k}{t}{z}\ud t - \int_{0}^{j\tau} \testfuncarg{k}{t}{x}\ud t} \le \frac{\eps}{3}, \qquad \forall z\in \ball{\delta_j}{x}.
	\end{align}
	The continuity of $z\to \int_{0}^{j\tau} \testfuncarg{k}{t}{z}\ud t$ can be easily established due to the differentiability of $z\to\testfuncarg{k}{t}{z}$. 
	By combining previous results, for each $\eps>0$, we can find a $\delta_j$ such that for any $z\in \ball{\delta_j}{x}$, 
	\begin{align*}
		&\ \Abs{\int_{0}^{\infty} \testfuncarg{k}{t}{z}\ud t - \int_{0}^{\infty} \testfuncarg{k}{t}{x}\ud t}\\
		\le &\ \Abs{\int_{0}^{j\tau} \testfuncarg{k}{t}{z}\ud t - \int_{0}^{j\tau} \testfuncarg{k}{t}{x}\ud t} + \int_{j\tau}^{\infty} \testfuncarg{k}{t}{z}\ud t + \int_{j\tau}^{\infty} \testfuncarg{k}{t}{x}\ud t\ \\  \myle{\eqref{eqn::tail_int},\eqref{eqn::diff_int}}&\ \  \frac{\eps}{3} + \frac{\eps}{3} + \frac{\eps}{3} = \eps.
	\end{align*}
	This proves the continuity of $z\to\int_{0}^{\infty} \testfuncarg{k}{t}{z}\ud t$ at the point $x$.
\end{proof}

\subsection{A remark about the general case}
\label{subsec::remark::existence}

By \propref{prop::invariance_b}, to prove that $\nabla V$ is a zero-variance dynamics, it is equivalent to prove that $\dynb = \sigmamat \nabla V$ is a zero-variance dynamics where $V$ solves \eqref{eqn::poisson}.

Notice that $\jacoarg{t}{x} = e^{\int_{0}^{t} \nabla \cdot (\sigmamat V)\big(\state{s}{x}\big)\ud s}$ and  
\begin{align*}
\intreal \big(\rho_1(\state{t}{x}) - \rho_0(\state{t}{x})\big) \jacoarg{t}{x}\ud t &= \intreal \nabla \cdot (\sigmamat \nabla V) (\state{t}{x}) \jacoarg{t}{x}\ud t\\ 
&= \intreal \td \jacoarg{t}{x}\ud t  = \lim_{t\rightarrow\infty} \jacoarg{t}{x} - \lim_{t\rightarrow-\infty} \jacoarg{t}{x}.
\end{align*}
As long as $\jacoarg{t}{x}$ vanishes when $\abs{t}\to\infty$, such a dynamics $\dynb = \sigmamat \nabla V$ is indeed a zero-variance dynamics. 

{For a point $x\in\dom$, suppose the gradient ascent trajectory under $\nabla V$ will converge to a (non-degenerate) local maximum of $V$, denoted as $\xst$;
by \propref{prop::invariance_b}, the trajectory initiated from $x$ under $\dynb=\sigmamat\nabla V$ is the same and $\state{t}{x} \to \xst$ as $t\to\infty$ under the flow $\dynb$.
Since $\sigmamat$ is strictly positive, 
it also does not change the concavity of local extreme points: when $\xst$ is a local maximum of $V$ (with $\nabla V(\xst) =\vectorzero_{\dimn}$ and $\nabla^2 V(\xst) < 0$), then 
\begin{align*}
	\nabla \dynb(\xst) = \nabla V(\xst) \nabla\sigmamat(\xst)^{T} + \sigmamat(\xst) \nabla^2 V(\xst) = \sigmamat(\xst)\nabla^2 V(\xst) < 0,
\end{align*}
which implies that as $t\to\infty$,
\begin{align*}
	\nabla\cdot\dynb\big(\state{t}{x}\big)\to \nabla\cdot\dynb(\xst) = \tr\big(\nabla\dynb(\xst)\big) < 0.
\end{align*}
By the same argument as in the case $\sigmamat = 1$ (\ie{}, \lemref{lem::conv_testfunc}), we can establish the validity that $\jacoarg{t}{x}\to 0$ as $\abs{t}\to \infty$.
}

\subsection{A remark about the existence of Morse function}
\label{subsec::proof::existence_approximate}

In Poisson's equation \eqref{eqn::poisson}, a Morse function $V$ does not always exist for an arbitrary smooth density function $\rho_1$, e.g., when $\rho_1 = \rho_0$, $\sigmamat=1$, we know $V=0$ is the solution of \eqref{eqn::poisson} but $V=0$ is not a Morse function. However, since Morse functions are dense in $C^\infty(\dom,\Real)$ \cite{Audin2014}, we can always find a Morse function such that the dynamics $\dynb =\nabla V$ behaves almost like a zero-variance dynamics, which is summarized in the next proposition.

\begin{proposition}
\label{prop::existence_approximate}
Suppose $\dom=[0,1]^\dimn$ is a torus and $U_0, U_1\in C^\infty(\dom,\Real)$.
Without loss of generality, assume $\partition_0 = \partition_1 = 1$. 
For any $\eps\in (0,1)$, there exists a Morse function $V$ such that the dynamics $\dynb=\nabla V$ provides an estimator $1-\eps \le \newainf(x) \le 1 + \eps$ for almost all $x\in\dom$ in the infinite-time NEIS method. Consequently, the variance $\newvarinf(\dynb) \le \eps^2$.
\end{proposition}

\begin{proof}
    Denote $\theta := \inf\{ \rho_0(x): x\in \dom\} \equiv e^{-\sup\{U_0(x):\ x\in \dom\}} > 0$ since $U_0$ is smooth and $\dom$ is compact.
    Since both $\rho_k = e^{-U_k}$ are smooth for $k=0,1$, one could approximate $\rho_1 - \rho_0$ by trigonometric polynomials $T_N(x) = \sum_{\abs{\mu}_{\infty}\le N, \mu\neq \vectorzero_\dimn} a_{\mu} e^{i 2\pi \inner{\mu}{x}}$ such that  \begin{align*}
    \norm{(\rho_1 - \rho_0) - T_N}_{C^0(\dom)} < \frac{\eps\theta}{2},
    \end{align*}
    where {$a_\mu = \int_{\dom} e^{-i2\pi\inner{\mu}{x}}(\rho_1-\rho_0)(x)\ud x\in \mathbb{C}$} are Fourier coefficients, $\mu\in \Int^\dimn$ and $N\in \Natural$; see \eg{}, \cite[Theorem 16]{braunling_2004}.
    Let $\Psi_N(x) = \sum_{\abs{\mu}_{\infty} \le N, \mu\neq \vectorzero_\dimn} \frac{a_\mu}{-4\pi^2\abs{\mu}^2} e^{i 2\pi \inner{\mu}{x}}\in C^\infty(\dom,\Real)$. 
    It is clear that $\Laplace \Psi_N = T_N$.
    As Morse functions are dense, we can find a Morse function $V$ such that $\norm{V - \Psi_N}_{C^2(\dom)} < \frac{\eps\theta}{2}$ \cite[Proposition 1.2.4]{Audin2014}, and in particular, $\norm{\Laplace V - \Laplace \Psi_N}_{C^0(\dom)} < \frac{\eps\theta}{2}$.
    Therefore, 
    \begin{align}
    \label{eqn::diff_rho_1_tilde}
    \norm{\Laplace V - (\rho_1 - \rho_0)}_{C^0(\dom)} \le \norm{\Laplace V - \Laplace \Psi_N}_{C^0(\dom)} + \norm{\Laplace \Psi_N - (\rho_1 - \rho_0)}_{C^0(\dom)} \le \eps \theta.
    \end{align}
    
    By \propref{prop::existence}, we know that $\dynb = \nabla V$ is a zero-variance dynamics for $\wt{\rho}_1 := \rho_0 + \Laplace V$. 
    The \propref{prop::existence} is proved under the assumption that densities are positive smooth functions for convenience and it is straightforward to verify that it also holds if $\rho_1$ is an arbitrary smooth function in \propref{prop::existence}. In particular, using the same argument in Appendix \ref{app::proof::existence} Step 2, we have for almost all $x \sim \rho_0$, 
    \begin{align}
    \label{eqn::zero_var_rho_tilde}
    \frac{\int_{\Real} \wt{\rho_1}(\state{t}{x}) \jacoarg{t}{x}\ud t}{\int_{\Real} \rho_0(\state{t}{x}) \jacoarg{t}{x}\ud t} = 1 + \frac{\int_\Real \Laplace V(\state{t}{x})\jacoarg{t}{x}\ud t}{{\int_\Real \rho_0(\state{t}{x})\jacoarg{t}{x}\ud t}} = 1 + \frac{\jacoarg{t}{x}\rvert_{t=-\infty}^{t=\infty}}{\int_\Real \rho_0(\state{t}{x})\jacoarg{t}{x}\ud t} = 1.
    \end{align}
    Hence, 
    \begin{align*}
        \abs{\newainf(x) - 1} & \myeq{\eqref{eqn::ainf}} \abs{\frac{\int_{\Real} \rho_1(\state{t}{x}) \jacoarg{t}{x}\ud t}{\int_{\Real} \rho_0(\state{t}{x}) \jacoarg{t}{x}\ud t} - 1} \\
        &= \abs{\frac{\int_{\Real} \wt{\rho}_1(\state{t}{x}) \jacoarg{t}{x}\ud t}{\int_{\Real} \rho_0(\state{t}{x}) \jacoarg{t}{x}\ud t} + \frac{\int_{\Real} (\rho_1-\wt{\rho}_1)(\state{t}{x}) \jacoarg{t}{x}\ud t}{\int_{\Real} \rho_0(\state{t}{x}) \jacoarg{t}{x}\ud t} - 1}\\
        & \myeq{\eqref{eqn::zero_var_rho_tilde}} \abs{\frac{\int_{\Real} (\rho_1-\wt{\rho}_1)(\state{t}{x}) \jacoarg{t}{x}\ud t}{\int_{\Real} \rho_0(\state{t}{x}) \jacoarg{t}{x}\ud t}}\\
        & \le \frac{\int_{\Real} \abs{(\rho_1-\wt{\rho}_1)(\state{t}{x})} \jacoarg{t}{x}\ud t}{\int_{\Real} \rho_0(\state{t}{x}) \jacoarg{t}{x}\ud t} \\
        & \myle{\eqref{eqn::diff_rho_1_tilde}} \frac{\int_\Real \eps\theta \jacoarg{t}{x}\ud t}{\int_\Real \theta \jacoarg{t}{x}\ud t} = \eps.
    \end{align*}
    Since we assumed $\partition_1 = 1$, 
    \begin{align*}
        \newvarinf(\dynb) &= \ee_0[\abs{\newainf}^2] - \big(\partition_1\big)^2
        = {\ee_0\abs{\newainf-\partition_1}^2 \le \eps^2.}
    \end{align*}
\end{proof}

\subsection{Solution of Poisson's equation \eqref{eqn::poisson} for Gaussian mixtures}

\begin{lemma}
	One solution of the Poisson's equation $\Laplace V = C e^{-\frac{\abs{x-\mu}^2}{2\sigma^2}}$ with $\dimn\ge 2$ on $\dom=\Rd$ is 
	$V(x) = f(\abs{x - \mu})$, where the function $f:\Real^{+}\rightarrow\Real$ has the derivative 
	\begin{align*}
		f'(r) = C 2^{\dimn/2-1} \sigma^d r^{1-d} \int_{0}^{\frac{r^2}{2\sigma^2}} t^{\dimn/2-1} e^{-t}\ud t \equiv C 2^{\dimn/2-1} \sigma^d r^{1-d}\lowergamma{\dimn/2}{\frac{r^2}{2\sigma^2}},
	\end{align*}
where $\lowergamma{a}{x}:=\int_{0}^x t^{a-1}e^{-t}\ud t$ is the lower incomplete gamma function.
\end{lemma}

\begin{proof}
	Without loss of generality, let $\mu=\vectorzero_{\dimn}$. Then a natural radial solution is given as $V(x) = f(\abs{x})$ for some scalar-valued function $f$. The above Poisson's equation becomes
	\begin{align*}
		f''(\abs{x}) + f'(\abs{x})\frac{\dimn-1}{\abs{x}} = C e^{-\frac{\abs{x}^2}{2\sigma^2}}.
	\end{align*}
	By some straightforward computation,
	\begin{align*}
		f'(r) &= C r^{-(d-1)}\int_{0}^{r} s^{\dimn-1} e^{-s^2/(2\sigma^2)}\ud s 
		= C 2^{\dimn/2-1} \sigma^d r^{1-d} \underbrace{\int_{0}^{\frac{r^2}{2\sigma^2}} t^{\dimn/2-1} e^{-t}\ud t}_{\equiv \lowergamma{\dimn/2}{\frac{\abs{r}^2}{2\sigma^2}}}.
	\end{align*}
\end{proof}

\begin{proposition}
	\label{prop::dynb_poisson_gaussian}
	Suppose $V$ solves the following Poisson's equation on $\dom=\Rd$
	\begin{align*}
		\Laplace V = \rho_1 - \rho_0
	\end{align*}
	where 
	\begin{align*}
		\rho_0(x) &= \frac{1}{\sqrt{2\pi}^d} \exp\Big(-\frac{\abs{x}^2}{2}\Big),\\
		\rho_1(x) &=  \sum_{i=1}^{n} \omega_i \frac{1}{\sqrt{2\pi\sigma_i^2}^d} \exp\Big(-\frac{\abs{x-\mu_i}^2}{2\sigma_i^2}\Big), \qquad n\in \Natural,\ \mu_i\in\Rd, \sigma_i\in\Real^{+},\ \forall 1\le i \le n, 
	\end{align*}
	and $\mu_i\neq \mu_j$ if $i\neq j$.
	Then one solution for the gradient flow dynamics $\dynb = \nabla V$ is given as 
	\begin{align}
		\label{eqn::dynb_poisson_gaussian}
		\begin{aligned}
			\dynb(x) = 2^{-1}\pi^{-\dimn/2} \bigg(\sum_{i=1}^{n} \omega_i \abs{x-\mu_i}^{-\dimn} \lowergamma{\dimn/2}{\frac{\abs{x-\mu_i}^2}{2\sigma_i^2}} (x - \mu_i) 
			- \abs{x}^{-\dimn} \lowergamma{\dimn/2}{\frac{\abs{x}^2}{2}} x\bigg).
		\end{aligned}
	\end{align}
\end{proposition}
\begin{proof}
	We just need to apply the last lemma and the formula $\nabla V(x) = \frac{f'(\abs{x})}{\abs{x}} x$ if $V(x) = f(\abs{x})$.
\end{proof}

\begin{proposition}
	\label{prop::dynb_poisson_gaussian_limit}
	Suppose $\dynb$ is given in \eqref{eqn::dynb_poisson_gaussian}. Then 
	\begin{align*}
		&\lim_{\big(\max_{i=1}^{n} \sigma_i\big)\rightarrow 0}\dynb(x) \\
		=& \left\{
		\begin{aligned}
			2^{-1}\pi^{-\dimn/2}\bigg(\sum_{i\neq j} \omega_i \abs{x-\mu_i}^{-\dimn} \Gamma(\dimn/2) (x - \mu_i) 
			- \abs{x}^{-\dimn} \lowergamma{\dimn/2}{\frac{\abs{x}^2}{2}} x\bigg), \qquad \text{ if } x = \mu_j,\\
			2^{-1}\pi^{-\dimn/2} \bigg(\sum_{i}^{} \omega_i \abs{x-\mu_i}^{-\dimn} \Gamma(\dimn/2) (x - \mu_i) 
			- \abs{x}^{-\dimn} \lowergamma{\dimn/2}{\frac{\abs{x}^2}{2}} x\bigg), \qquad \text{otherwise}.
		\end{aligned}\right.
	\end{align*}
The limiting dynamics $x\to \lim_{\big(\max_{i=1}^{n} \sigma_i\big)\rightarrow 0}\dynb(x)$ is continuous on the region $\Rd\backslash\{\mu_i\}_{i=1}^{n}$.
\end{proposition}

\begin{proof}
If $x = \mu_j$, then 
\begin{align*}
	\dynb(x) = \frac{1}{2\pi^{\dimn/2}}\bigg(\sum_{i\neq j} \omega_i \abs{x-\mu_i}^{-\dimn} \lowergamma{\dimn/2}{\frac{\abs{x-\mu_i}^2}{2\sigma_i^2}} (x - \mu_i) 
	- \abs{x}^{-\dimn} \lowergamma{\dimn/2}{\frac{\abs{x}^2}{2}} x\bigg).
\end{align*}
When $\sigma_i\rightarrow 0$ for all $i$,we know $\lowergamma{\dimn/2}{\frac{\abs{\mu_j-\mu_i}^2}{2\sigma_i^2}} \rightarrow \Gamma(\dimn/2)$ when $i\neq j$ and hence we have the above result.
Similarly, we can obtain the expression when $x\neq \mu_j$ for any $j$.
\end{proof}

\subsection{Example: Poisson's equation yields a zero-variance dynamics}
\label{subsec::eg_poisson}
{
The example for \figref{fig::torus} is 
\begin{align}
\label{eqn::torus_eg}
	\begin{aligned}
		\rho_0(x) = e^{-U_0(x)} &= 1,\\
		\rho_1(x)\propto e^{-U_1(x)} &= \frac{\phi\big(x-\begin{bsmallmatrix} 0.3\\ 0.3\end{bsmallmatrix}\big) 
			+ \phi\big(x-\begin{bsmallmatrix}0.7\\ 0.3\end{bsmallmatrix}\big) 
			+ \phi(x-\begin{bsmallmatrix}0.3\\ 0.7\end{bsmallmatrix}\big)}{3},\\
		\phi(x) &= e^{2\cos(2\pi x_1) + 2 \cos(2\pi x_2)}.
	\end{aligned}
\end{align}
}

The periodic boundary condition in \propref{prop::existence} helps to ease the technicalities in proving that the gradient dynamics $\dynb=\nabla V$ from solving the Poisson's equation \eqref{eqn::poisson} is a zero-variance dynamics, by removing the effect from the boundary $\partial\dom$. The same conclusion, however, should hold if $V$ solves the Poisson's equation with Neumann boundary condition: 
\begin{align}
	\label{eqn::poisson_vn}
	\Laplace V = \rho_1 -\rho_0,\qquad \nabla V\cdot \nn{} = 0 \text{ on } \partial\dom,
\end{align}
{where $\nn{}$ is the normal vector of the boundary $\partial\dom$.}
We consider the same model \eqref{eqn::torus_eg} and $\dom = (0,1)^2$. 
The potential $V$ and flowlines of $\dynb=\nabla V$ are visualized in \figref{fig::vn_2d::V} and 
we can numerically verify that {$\newainf(x)=\ratio$} for almost all $x\in \dom$.

\begin{figure}[h!]
	\centering
     \includegraphics[width=0.5\textwidth]{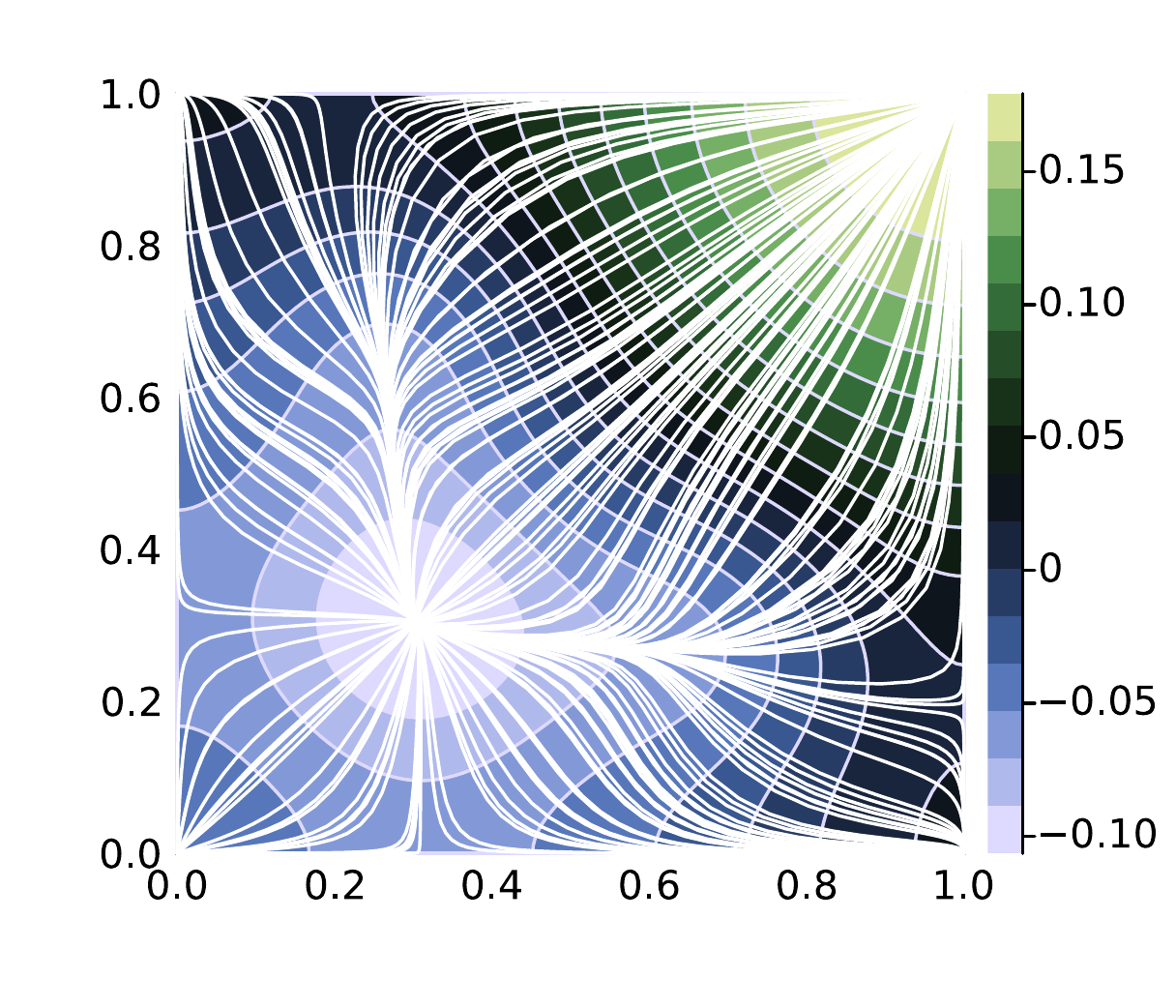}
	\caption{Contour plot of $V$ and flowlines of $\dynb = \nabla V$ for the model \eqref{eqn::torus_eg} on the domain $\dom=(0,1)^2$ with Neumann boundary condition.}
	\label{fig::vn_2d::V}
\end{figure}

\subsection{Non-uniqueness of zero-variance dynamics}

Recall from \propref{prop::invariance_b} that there are certain degrees of freedom to choose the dynamics:
for a given $\dynb_1$, if we choose $\dynb_2 = \alpha \dynb_1$
where $\alpha\in C^{\infty}(\dom,\Real)$ is strictly positive,
then this function $\alpha$ can be absorbed into the time rescaling
and it does not affect the variance of the sampling scheme.
However, even if we remove this parameterization redundancy,
zero-variance dynamics may still not be unique, \eg{}, due to the {geometric rotational symmetry}.

\begin{proposition}[Non-uniqueness]
	\label{prop::non-unique}
	For given $\rho_0$ and $\rho_1$, there might exist more than one zero-variance dynamics (let us say $\dynb_1, \dynb_2$) but there is no scalar-valued function $\alpha$ such that $\dynb_2 = \alpha \dynb_1$.
\end{proposition}

\begin{proof}
We construct an example to prove the non-uniqueness: let $\dimn=2$, $U_0(x) = \abs{x}^2/2 {+\ln(2\pi)}$ and $U_1$ be given by 
\begin{align*}
	\exp\big(-U_1(x)\big) &= \exp\big(-U_0(x)\big) + {\frac{1}{2\pi}}x_1 x_2 \exp\big(-x_1^4 - x_2^4\big).
\end{align*}
We can easily verify that $\abs{x} e^{-x^4} < \frac{3}{4} e^{-x^2/2}$ for any $x\in \Real$. Then we know 
$\frac{7}{16} e^{-U_0(x)} < e^{-U_1(x)} < \frac{25}{16} e^{-U_0(x)}$.
Therefore, $U_1$ is clearly well-defined and $\ratio = 1$.
For the dynamics $\dynb(x) = \begin{bsmallmatrix} v_1 \\ v_2 \end{bsmallmatrix}$ with $v_1^2 + v_2^2 > 0$, we have $\jacoarg{t}{x} = 1$ for any $x\in\dom, t\in \Real$, and 
\begin{align*}
	& \intreal e^{-U_1\big(\state{t}{x}\big)} \jacoarg{t}{x} \ud t\\ 
	=& \intreal e^{-U_0\big(\state{t}{x}\big)} \jacoarg{t}{x} \ud t \\
	&\qquad +  {\frac{1}{2\pi}}\intreal (x_1 + v_1 t) (x_2 + v_2 t) \exp\big(-(x_1+v_1 t)^4 - (x_2 + v_2 t)^4\big) \ud t.
\end{align*}
When either $v_1 = 0$ or $v_2 = 0$, we can easily verify that
\begin{align*}
	\frac{\intreal e^{-U_1\big(\state{t}{x}\big)} \jacoarg{t}{x} \ud t}{\intreal e^{-U_0\big(\state{t}{x}\big)} \jacoarg{t}{x}\ud t} = 1,\qquad \forall x\in \Real^2.
\end{align*}
Therefore, the variance is zero for two dynamics with orthogonal directions $\dynb = \begin{bsmallmatrix} 1 \\ 0 \end{bsmallmatrix}$ and $\dynb = \begin{bsmallmatrix} 0 \\ 1 \end{bsmallmatrix}$.
It is clear that there is no scalar-valued function $\alpha$ such that $\dynb_2 = \alpha \dynb_1$, and thus the non-uniqueness is established.
\end{proof}

\subsection{Connection to the Beckmann’s problem.}
\label{subsec::beckmann}

The Poisson's equation \eqref{eqn::poisson} with $\sigmamat=1$ is the Euler-Lagrange equation associated with 
\begin{align}
	\label{eq:beckman}
	\min \int_{\dom} \abs{\dynb(x)}^p \ud x \ \ \ \text{subject to}\ \ \ \div\dynb = \rho_1 - \rho_0,
\end{align}
when $p = 2$.
The variational problem in~\eqref{eq:beckman} is known as Beckman's problem of continuous transportation~\cite{Beckmann_1952_continuous}; when $p=1$, it is also related to optimal transport in $W_1$ Wasserstein distance \cite{santambrogio_dacorogna-moser_2014,santambrogio_2015_optimal}.

\section{Explicitly solvable zero-variance dynamics}
\label{app::eg}

In this section, we provide some examples with explicitly solvable zero-variance dynamics. Throughout this section, we consider $\dom = \Rd$.

\begin{table}[!ht]
	\caption{Examples with explicitly solvable zero-variance dynamics for the infinite-time case with the domain $\dom=\Rd$. 
	By \propref{prop::invariance_b}, given a zero-variance dynamics $\dynb$, any dynamics of the form $\alpha \dynb$ for some scalar-valued positive function $\alpha$ is also a zero-variance dynamics. In this table, we have removed such a degree of freedom.}
	\begin{tabular}{p{0.1\textwidth}|p{0.32\textwidth}|p{0.34\textwidth}|p{0.12\textwidth}}
		\toprule
		Dimension & $U_0$ and $U_1$ & $\dynb$  & Details\\
		\hline
		$\dimn=1$ & arbitrary & $\dynb(x) = 1$ & \appref{app::eg::1d}\\
		\hline
		general $\dimn$ & $U_0(x) = \abs{x}^2/2 {+\gausscst{\dimn}}$\par $U_1(x) = (x-\varpi)^{T} \Sigma^{-1} (x-\varpi)/2$ & $\dynb(x) = \Lambda x + v$ with\par $\Lambda = \ln(\Sigma^{-1/2})$,\par $v = -\big(\idmat_\dimn - \Sigma^{1/2}\big)^{-1}\ln(\Sigma^{-1/2}) \varpi.$ &  \appref{subsec::linear_gaussian}\\
		\hline
		general $\dimn$ &  $\rho_0$ and $\rho_1$ have the same marginal distribution on the orthogonal subspace of $\{cv:\ c\in \Real\}$ & $\dynb(x) = v$ & \appref{subsec::parallel_velocity}\\
		\bottomrule
	\end{tabular}

	\label{table::solvable_eg}
\end{table}

\subsection{Some general properties}

Given a $\dynb\in \mani{}$ and a distribution $\rho_0$, 
we study the family of $U_1$ such that $\dynb$ is a zero-variance dynamics.
Given an ODE flow map $\state{\tau}{\cdot}$ based on the dynamics $\dynb$, let us introduce 
\begin{align}
		\label{eqn::Utau}
		U^{\tau}(x) := U_0\big(\state{-\tau}{x}\big) -\log\big(\jacoarg{-\tau}{x}\big).
	\end{align}
Then $\rho \propto e^{-U^{\tau}}$ is the push-forward distribution of the flow map $\state{\tau}{\cdot}$, \ie{}, $\rho = \big(\state{\tau}{\cdot}\big)\#\rho_0$.
The family of distributions that can be written as a linear combination of such pushforward distributions can be characterized by 
\begin{align*}
		\accF :=
		\Big\{U:\
		e^{-U} \in {\text{Span}}
		\big\{e^{-U^{\tau}}\big\}_{\tau\in\Real} \Big\}.
	\end{align*}

We have:
\begin{proposition}
	\label{prop::zero_acc_propty}
	For every $U_1\in \accF$, the variance $\newvarinf(\dynb)$ (if well-defined) is exactly zero, i.e., if the distribution $\rho_1\propto e^{-U_1}$ is a linear combination of push-forward distributions by the ODE flows maps, then $\ratio$ can be estimated with zero-variance by the \itneis{} scheme.
	\end{proposition}

This proposition is proven in \appref{app:proof:E1} below. 
In  words it says that, if we can learn a perfect neural ODE such that $\rho_1 = \state{\tau}{\cdot}\#\rho_0$ for some $\tau$, then such a dynamics is also the optimal one (i.e., zero-variance dynamics) for the \itneis{} scheme.
Conversely, if $U \notin  \accF$,
then is it still possible that $\newvarinf(\dynb) =0$? The answer is positive:

\begin{proposition}
		\label{prop::expressivility}
		The exists a dynamics $\dynb\in \mani{}$ and a $\rho_1 \propto e^{-U_1}$ such that $\ratio$ can be computed with zero-variance but  $\rho_1$ does not need to be a linear combination of $\big\{\state{\tau}{\cdot}\#\rho_0\big\}_{\tau\in\Real}$ (namely, $U_1\notin \accF$).
\end{proposition}

This proposition is proven in \appref{app:proof:E2} below. 

\subsection{Flows for the 1D case}
\label{app::eg::1d}

Let us consider $\dynb = 1$. Then we can compute $\ratio$ via the \itneis{} scheme with zero-variance for arbitrary potentials $U_0$ and $U_1$. 
This could be verified via direct computation: $\jacoarg{t}{x} = 1$ and $\state{t}{x} = x + t$ for any $t, x\in \Real$, and thus for an arbitrary $x$, 
\begin{align*}
\newainf(x) = \frac{\intreal e^{-U_1\big(\state{t}{x}\big)}\jacoarg{t}{x}\ud t}{\intreal e^{-U_0\big(\state{t}{x}\big)}\jacoarg{t}{x}\ud t}
= \frac{\intreal e^{-U_1(x+t)}\ud t}{\intreal e^{-U_0(x+t)}\ud t} = \ratio.
\end{align*}
Therefore, the variance is exactly zero.

Another perspective to understand this comes from \propref{prop::zero_acc_propty}. For $\dynb=1$, we know $e^{-U^{\tau}(x)} = e^{-U_0(x-\tau)}$ in \eqref{eqn::Utau}. By \propref{prop::zero_acc_propty}, if the potential $U_1$ can be expressed as follows
\begin{align*}
e^{-U_1} &= \intreal f(\tau) e^{-U^{\tau}}\ d\tau = (e^{-U_0} * f),
\end{align*}
then $\ratio$ can be computed with zero-variance, where $f$ is a tempered distribution and $*$ means the convolution. Then it is sufficient to show the existence of such a $f$ for a generic $U_1$.

For a given potential $U_1$, we can solve the above equation for $f$ using Fourier transform; more specifically, 
\begin{align*}
f = \fouriertrans^{-1}\Big(\fouriertrans(e^{-U_1})/\fouriertrans(e^{-U_0})\Big),
\end{align*}
where $\fouriertrans^{(-1)}$ are (inverse) Fourier transform.

\subsection{Linear flows for Gaussian distributions}
\label{subsec::linear_gaussian}

\begin{proposition}
	\label{prop::gaussian}
	Suppose $U_0(x) = \abs{x}^2/2{+\gausscst{\dimn}}$ and $U_1(x) = (x-\varpi)^{T} \Sigma^{-1} (x-\varpi)/2$, where the covariance matrix $\Sigma$  is non-degenerate.
	A zero-variance linear dynamics is $\dynb(x) = \Lambda x + v$ with 
	\begin{align}
	\label{eqn::linear_choice}
	\Lambda = \ln(\Sigma^{-1/2}),\qquad v = -\Big(\idmat_\dimn - \Sigma^{1/2}\Big)^{-1}\ln(\Sigma^{-1/2}) \varpi.
	\end{align}
\end{proposition}

Before presenting detailed proofs, let us make a few remarks about zero-variance dynamics in \eqref{eqn::linear_choice}:
\begin{itemize}[leftmargin=4ex]
	\item 
	
	If we further let $\Sigma = (1 - \eps) \idmat_{\dimn}$ where $\eps\ll 1$ is an asymptotic parameter, 
	then the above choice \eqref{eqn::linear_choice} can be approximated as follows: 
	\begin{align*}
	\td \state{t}{x} = \Lambda \state{t}{x} + v \approx \frac{\eps}{2} \state{t}{x} - \big(1 + \frac{\eps}{4}\big)\varpi + \mathcal{O}(\eps^2). 
	\end{align*}
	The leading order dynamics $\td \state{t}{x} \approx -\varpi$ is consistent with the parallel velocity case below in \appref{subsec::parallel_velocity}.

	\item For the above Gaussian case, the dynamics \eqref{eqn::linear_choice} can be regarded as the gradient flow dynamics of the following quadratic potential 
	\begin{align*}
	V(x) = \frac{1}{2}\Big(x -\big(\idmat_\dimn - \Sigma^{1/2}\big)^{-1}\varpi\Big)^{T} \ln(\Sigma^{-1/2}) \Big(x -\big(\idmat_\dimn - \Sigma^{1/2}\big)^{-1}\varpi\Big).
	\end{align*}
	Indeed, it is not hard to guess that the optimal dynamics might have a linear form in order to transport Gaussian distributions.
	It is natural to guess that $\dynb = -\nabla U_1$ or $\dynb = -\nabla (U_1 - U_0)$ might be zero-variance dynamics, but it could be verified that neither of them are zero-variance dynamics.
\end{itemize}

\begin{proof}[Proof of \propref{prop::gaussian}]
	Suppose we consider a family of linear dynamics
	\begin{align*}
	\td \state{t}{x} = \Lambda \state{t}{x} + v,\qquad \state{0}{x} = x.
	\end{align*}
	Then $\state{t}{x} = e^{\Lambda t} x + \Lambda^{-1} \big(e^{\Lambda t}-\idmat_\dimn\big) v$ and  $(\nabla \cdot \dynb)(x) = \tr(\Lambda)$.
	Therefore, $U^{-\tau}$ defined in \eqref{eqn::Utau} has the following form for any $\tau\in \Real$, 
	\begin{align*}
	U^{-\tau}(x) = U_0\Big(e^{\Lambda\tau} x + \Lambda^{-1} (e^{\Lambda\tau}-\idmat_\dimn) v\Big) - \tr(\Lambda)\tau.
	\end{align*}
	By \propref{prop::zero_acc_propty},
	the variance is zero if 
	\begin{align*}
		U_1(x) = U^{-1}(x) + C,
	\end{align*}
where $C$ is some constant.
	This condition can be simplified as
	\begin{align*}
		(x-\varpi)^{T} \Sigma^{-1} (x-\varpi)/2 = U_0\Big(e^{\Lambda} x + \Lambda^{-1} (e^{\Lambda}-\idmat_\dimn) v\Big) - \tr(\Lambda) + C.
	\end{align*}
	Thus, we just need to ensure 
	\begin{align*}
	-\Lambda^{-1} (\idmat_d - e^{-\Lambda}) v = \varpi, \qquad e^{\Lambda^{T}} e^{\Lambda} = \Sigma^{-1},
	\end{align*}
	by matching {the order of $x$} and, more specifically, we can choose $\Lambda$ and $v$ as in \eqref{eqn::linear_choice}.
\end{proof}

\subsection{Flows with parallel velocity}
\label{subsec::parallel_velocity}

As we have mentioned, in the 1D case, the choice $\dynb = 1$ gives a zero-variance estimator for the \itneis{} scheme.
A straightforward generalization is to consider the following parallel velocity case 
\begin{align*}
\dynb(x) = \alpha(x)v,
\end{align*}
where $v\in \Rd$ and $\alpha$ is a positive scalar-valued function. 
Due to \propref{prop::invariance_b}, it suffices to consider $\dynb(x)= v$. 
As we can always rotate the coordinate without affecting partition functions, without loss of generality, let us assume
$v = \basis_1$
for simplicity, where $\basis_1$ is a vector with the first element to be $1$ and zeros otherwise. 
For an arbitrary initial proposal $x$, only the first coordinate $x_1$ is changing under the dynamics $\dynb = \basis_1$. The estimator essentially works like the 1D case above:
\begin{align*}
\newainf(x) = \frac{ \intreal e^{-U_1\big(\state{t}{x}\big)} \jacoarg{t}{x}\ud t}{\intreal e^{-U_0\big(\state{t}{x}\big)} \jacoarg{t}{x}\ud t } = \frac{\intreal e^{-U_1(q, x_2, x_3, \cdots, x_{\dimn})} \ud q}{\intreal e^{-U_0(q, x_2, x_3, \cdots, x_{\dimn})} \ud q} = \ratio{} \times \frac{\wt{\rho}_{1}(x_2, x_3, \cdots, x_{\dimn})}{\wt{\rho}_0(x_2, x_3, \cdots, x_{\dimn})},
\end{align*}
where $\wt{\rho}_k$ is the marginal distribution of $\rho_k$ for the subspace $\Real^{\dimn-1}$ by tracing out the first coordinate.
Therefore, the dynamics $\dynb=\basis_1$ is a zero-variance dynamics iff $\wt{\rho}_0 = \wt{\rho}_1$.
\begin{proposition}
	Suppose $\dynb(x)=\alpha(x) v$ where $\alpha\in C^{\infty}(\Rd,\Real)$ with $\inf_{x\in\dom} \alpha(x)>0$ and $v\in \Rd$.
	Such a dynamics $\dynb$ gives a zero-variance estimator iff
	$\rho_0$ and $\rho_1$ have the same marginal distribution on the orthogonal space of $\big\{cv:\ c\in \Real\big\}$.
\end{proposition}

\subsection{Proof of \propref{prop::zero_acc_propty}}
\label{app:proof:E1}

\begin{lemma}
Fix the potential $U_0$ and a dynamics $\dynb\in \mani{}$.
If $\frac{\intreal e^{-U_k\big(\state{t}{x}\big)}\jacoarg{t}{x}\ud t}{\intreal e^{-U_0\big(\state{t}{x}\big)}\jacoarg{t}{x}\ud t} = C_k$ is a constant function for $k \in \{2, 3\}$,   
then any mixture of $U_2$ and $U_3$, given below, also ensures that $\frac{\intreal e^{-U\big(\state{t}{x}\big)}\jacoarg{t}{x}\ud t}{\intreal e^{-U_0\big(\state{t}{x}\big)}\jacoarg{t}{x}\ud t}$ is a constant function, as long as $U$ is a valid potential function:
\begin{align*}
	U = -\log(\omega_2 e^{-U_2} + \omega_3 e^{-U_3}),\qquad \omega_2, \omega_3\in\Real.
\end{align*}
\end{lemma}

\begin{proof}
	We can easily observe that
	\begin{align*}
			\intreal e^{-U\big(\state{t}{x}\big)} \jacoarg{t}{x} \ud t &=
			\intreal \omega_2 e^{-U_2(\state{t}{x})} \jacoarg{t}{x} + \omega_3 e^{-U_3(\state{t}{x})} \jacoarg{t}{x} \ud t\\
			&=  (C_2 \omega_2 + C_3 \omega_3) \intreal e^{-U_0(\state{t}{x})} \jacoarg{t}{x} \ud t.
		\end{align*}
	\end{proof}

\begin{lemma}
	\label{lem::pushforward}
		For the distribution $\rho_1 \propto e^{-U^{\tau} + C}$, the partition function $\ratio$ can be computed with zero-variance.
	\end{lemma}
\begin{proof}
		We only need to verify the case $C=0$. Then
	\begin{align*}
			\intreal e^{-U^{\tau}\big(\state{t}{x}\big)} \jacoarg{t}{x} \ud t &=
			\intreal e^{-U_0\big(\state{t-\tau}{x}\big)} \jacoargbig{-\tau}{\state{t}{x}} \jacoarg{t}{x} \ud t \\
			&\myeq{\eqref{eqn::testfunc_trans}}\ \intreal e^{-U_0\big(\state{t-\tau}{x}\big)} \frac{\jacoarg{t-\tau}{x}}{\jacoarg{t}{x}} \jacoarg{t}{x} \ud t \\
			&= \intreal e^{-U_0\big(\state{t-\tau}{x}\big)} \jacoarg{t-\tau}{x} \ud t \\
			&=  \intreal e^{-U_0\big(\state{t}{x}\big)} \jacoarg{t}{x} \ud t.
		\end{align*}
	This means $\dynb$ is a zero-variance dynamics for the distribution $\rho_1$.
	\end{proof}

\begin{proof}[Proof of \propref{prop::zero_acc_propty}]
		Combine the above two lemmas.
	\end{proof}

\subsection{Proof of \propref{prop::expressivility}} 
\label{app:proof:E2}
Consider a 2D example with $U_0(x) = \frac{\abs{x}^2}{2} {+\ln(2\pi)}$ and $\dynb = \begin{bsmallmatrix} 1 \\ 0\end{bsmallmatrix}$.
Then $\jacoarg{t}{x} = 1$,
and $U^{\tau}(x) = \frac{(x_1-\tau)^2 + x_2^2}{2} {+\ln(2\pi)}$ where $x = \begin{bsmallmatrix}x_1 \\ x_2\end{bsmallmatrix}$.
Let us consider
\begin{align*}
		e^{-U_1(x)} := e^{-U^{1}(x)} + \eps x_1 e^{-x_1^4-x_2^4},
	\end{align*}
where 
$\eps := {\frac{1}{2\pi}}\big(2\sup_{y\in\Real} \abs{y} e^{-y^4 + (y-1)^2/2}\sup_{y\in\Real}  e^{-y^4+y^2/2}\big)^{-1} > 0$.
It could be straightforwardly verified that $U_1$ is a well-defined potential and
\begin{align*}
		\int_{-\infty}^{\infty} e^{-U_1\big(\state{t}{x}\big)} \jacoarg{t}{x} \ud t
		=& \int_{-\infty}^{\infty} e^{-U^{1}\big(\state{t}{x}\big)} \jacoarg{t}{x}\ud t
		+ \eps \int_{-\infty}^{\infty} (x_1+t) e^{-(x_1+t)^4-x_2^4} \ud t\\
		=& \int_{-\infty}^{\infty} e^{-U^{1}\big(\state{t}{x}\big)}\jacoarg{t}{x} \ud t \\ =& \int_{-\infty}^{\infty} e^{-U_0\big(\state{t}{x}\big)} \jacoarg{t}{x} \ud t.
	\end{align*}
{The last equality holds by \lemref{lem::pushforward}.}
Therefore, $\dynb$ is a zero-variance dynamics for $\rho_0$ and $\rho_1$. 
However, $U_1\notin \accF$ because $\intreal f(\tau) e^{-U^{\tau}(x)}\ud \tau$ must be a separable function, whereas $U_1$ is not.

\section{Proof of \propref{prop::existence_generator} and more discussions}

\subsection{Proof of \propref{prop::existence_generator}}
\label{subsec::proof::existence::generator}

{We proceed in three steps:}

{\noindent {\bf Step 1:}} Let us first assume the existence of $\timeT\in C^1(\domsubset,\Real)$ {satisfying \eqref{eqn::optimal_time}}, where $\domsubset$ is an open subset of $\dom$ and $\dom\backslash\domsubset$ has measure zero. We shall verify that $\mapT\#\rho_0 = \rho_1$ almost everywhere.

From \eqref{eqn::optimal_time}, let us replace $x$ by $\state{\eps}{x}$, 
\begin{align*}
	0\ \myeq{\eqref{eqn::testfunc_trans}}&\ \frac{1}{\jacoarg{\eps}{x}}\Big( \int_{-\infty}^0 \rho_0(\state{s+\eps}{x}) \jacoarg{s+\eps}{x}\ud s - \int_{-\infty}^{\timeT(\state{\eps}{x})} \rho_1(\state{s+\eps}{x})\jacoarg{s+\eps}{x}\ud s\Big).
\end{align*}
By straightforward simplification, 
\begin{align*}
	0 = \int_{-\infty}^{0+\eps} \rho_0(\state{s}{x}) \jacoarg{s}{x}\ud s - \int_{-\infty}^{\timeT(\state{\eps}{x})+\eps} \rho_1(\state{s}{x})\jacoarg{s}{x}\ud s.
\end{align*}
By taking the derivative with respect to $\eps$ at $\eps=0$, 
\begin{align*}
	\rho_0(x) &= \rho_1(\state{\timeT(x)}{x})\jacoarg{\timeT(x)}{x} \big(1 + \inner{\nabla\timeT(x)}{\dynb(x)}\big)\\
	&= \rho_1(\mapT(x)) \jacoarg{\timeT(x)}{x} \big(1 + \inner{\nabla\timeT(x)}{\dynb(x)}\big).
\end{align*}
To show that $\mapT\#\rho_0 = \rho_1$, we need to verify that $\rho_0(x) = \rho_1(\mapT(x))\jacomap{\mapT}{x}$ where $\jacoarg{\mapT}{x} = \absbig{\det\big(\nabla_x \state{\timeT(x)}{x}\big)}$. Therefore, it remains to prove that 
\begin{align}
	\label{eqn::jaco_generator}
	\det\big(\nabla_x \state{\timeT(x)}{x}\big) = \jacoarg{\timeT(x)}{x} \big(1 + \innerbig{\nabla\timeT(x)}{\dynb(x)}\big).
\end{align}

By direct computation, 
\begin{align*}
	\nabla_x \state{\timeT(x)}{x} \myeq{\eqref{eqn::deri_X_t}} \corrz{\timeT(x)}{0}(x) + \dynb\big(\state{\timeT(x)}{x}\big) \big(\nabla\timeT(x)\big)^{T}.
\end{align*}
By the matrix determinant lemma, 
\begin{align*}
	\det\big(\nabla_x \state{\timeT(x)}{x}\big) &= \Big(1+\innerBig{\nabla\timeT(x)}{\big(\corrz{\timeT(x)}{0}(x)\big)^{-1} \dynb\big(\state{\timeT(x)}{x}\big)}\Big)\det\big(\corrz{\timeT(x)}{0}(x)\big)\\
	&= \Big(1+\innerBig{\nabla\timeT(x)}{\dynb(x)}\Big) \jacoarg{\timeT(x)}{x},
\end{align*}
where we used \eqref{eqn::corrz} and \lemref{lem::deri_X_t} to get the second line. The last equation verifies \eqref{eqn::jaco_generator} and therefore, $\mapT\#\rho_0 = \rho_1$ almost everywhere.

\smallskip
{\noindent{\bf Step 2:}} We shall explain $\domsubset$ and establish the existence of  $\timeT\in C^1(\domsubset,\Real)$.

Suppose we denote the local minima of $V$ as $\pts_1, \pts_2, \cdots, \pts_{a}$, and local maxima of $V$ as $\ptsy_1, \ptsy_2, \cdots, \ptsy_{b}$, where $a, b\in \Natural$. Then in the proof of \propref{prop::existence}, we have already mentioned that 
\begin{align*}
	\domsubset := \Big\{x\in \dom\ \Big\rvert \lim_{t\to-\infty} \state{t}{x} = \pts_i,\ \lim_{t\to\infty} \state{t}{x} = \ptsy_j,\text{ for some } i, j\Big\}
\end{align*}
is an open subset of $\effdom(\dynb)$ and $\dom\backslash\domsubset$ has measure zero, due to the assumption that $V$ is a Morse function.

Consider the following function $L(x,t):\domsubset\times\Real\to \Real$, defined as 
\begin{align*}
	L(x,t) := \int_{-\infty}^{0} \rho_0(\state{s}{x}) \jacoarg{s}{x}\ud s - \int_{-\infty}^{t} \rho_1(\state{s}{x})\jacoarg{s}{x}\ud s.
\end{align*}
We can observe that 
\begin{itemize}[leftmargin=4ex]
	\item $\lim_{t\to-\infty}L(x,t) = \int_{-\infty}^{0} \rho_0(\state{s}{x}) \jacoarg{s}{x}\ud s > 0$.
	\item 
	\begin{align*}
		\lim_{t\to\infty} L(x,t) &= \int_{-\infty}^{0} \rho_0(\state{s}{x}) \jacoarg{s}{x}\ud s - \int_{-\infty}^{\infty} \rho_1(\state{s}{x})\jacoarg{s}{x}\ud s \\
		&= -\int_{0}^{\infty} \rho_0(\state{s}{x}) \jacoarg{s}{x}\ud s < 0,
		\end{align*} 
	where the second equality comes from the fact that $\dynb = \nabla V$ is a zero-variance dynamics; see e.g., \eqref{eqn::zero::var::2}.
	\item With fixed $x$, the function $t\to L(x,t)$ is continuously differentiable and is strictly monotonically decreasing.
\end{itemize} 
These imply that for each $x\in\domsubset$, there exists a unique $\timeT(x)$ such that $L\big(x, \timeT(x)\big) = 0$ by the intermediate value theorem. Therefore, $\timeT$ is well-defined via \eqref{eqn::optimal_time}.

Next, we need to prove that such a function $\timeT\in C^1(\domsubset,\Real)$,
which can be immediately obtained by the implicit function theorem \cite{rudin_1976_principles}, provided that we can prove $L\in C^1(\domsubset\times\Real,\Real)$.
Due to the smoothness assumption on $\rho_0$ and $\rho_1$, it is clear that $\partial_t L(x,t) = -\rho_1(\state{t}{x})\jacoarg{t}{x} < 0$ exists and is continuous. 
Therefore, the task becomes to prove that $\nabla_x L(x,t)$ exists and is continuous. 
Since it is clear that $\int_{0}^{t}\rho_1(\state{s}{x})\jacoarg{s}{x}\ud s$ is continuously differentiable with respect to $x$, 
it is then sufficient to prove that 
\begin{align*}
	G_k(x):= \int_{-\infty}^0 \rho_k(\state{s}{x})\jacoarg{s}{x}\ud s
\end{align*} 
is continuously differentiable for $k \in\{ 0,1\}$. 

\medskip
{\noindent {\bf Step 3:} Prove that $G_k\in C^1(\domsubset,\Real)$.}

\medskip
In Appendix \ref{app::proof::existence}, we have proved that $G_k$ is continuous; see \lemref{lem::continuity_F_on_D} in particular. Next, we first verify that $G_k$ is differentiable and then verify that $\nabla G_k$ is also continuous.

\medskip
{\noindent\emph{Part 1: $G_k$ is differentiable.}}

We want to verify that $G_k$ is differentiable and in particular
\begin{align}
	\label{eqn::Gx}
	\nabla_x G_k(x) = \int_{-\infty}^{0} \nabla_x \big(\rho_k(\state{s}{x})\jacoarg{s}{x}\big)\ud s.
\end{align}
Let us consider an arbitrary $x\in \domsubset$. Without loss of generality, suppose its limit in the backward branch is $\pts_1 = \lim_{t\to-\infty}\state{t}{x}$. 
Let us focus on 
a local neighborhood $\ball{\delta}{x}$ such that $\lim_{t\to-\infty} \state{t}{z} = \pts_1$ for all $z\in \ball{\delta}{x}$. 
Such a small $\delta$ exists because $\pts_1$ is a strict local minimum and $\dynb =\nabla V$ is smooth from the assumption that $V$ is a Morse function. In order to verify that $G_k$ is differentiable (i.e., \eqref{eqn::Gx}), 
by the Leibniz rule (see \eg{}, \cite[Theorem 6.28]{klenke_2014_probability}), 
it is sufficient to prove that there exists an integrable function $Q:(-\infty,0]\to\Real$ such that 
\begin{align}
	\label{eqn::Q}
	\abs{\nabla_z \big(\rho_k(\state{s}{z})\jacoarg{s}{z}\big)} \le Q(s), \qquad \forall s\in (-\infty,0],\ \forall z\in \ball{\delta}{x}.
\end{align} 
By direct computation, 
\begin{align}
	\label{eqn::deri:H}
	\begin{aligned}
	&\nabla_z \big(\rho_k(\state{s}{z})\jacoarg{s}{z}\big) \\
	=&
	 \rho_k(\state{s}{z})\jacoarg{s}{z}\Big( -(\nabla\state{s}{z})^{T} \nabla U_k(\state{s}{z}) + \int_{0}^{s} \big(\nabla \state{r}{z}\big)^{T} \nabla(\div\dynb)(\state{r}{z})\ud r\Big)\\
	 \myeq{\eqref{eqn::deri_X_t}}& \rho_k(\state{s}{z})\jacoarg{s}{z}\Big( -(\corrz{s}{0}(z))^{T} \nabla U_k(\state{s}{z}) + \int_{0}^{s} \big(\corrz{r}{0}(z)\big)^{T} \nabla(\div\dynb)(\state{r}{z})\ud r\Big).
	 \end{aligned}
\end{align}
Since the domain $\dom$ is assumed to be a torus,
we know $\rho_k$, $\nabla U_k$ and $\nabla(\div\dynb)$ are uniformly bounded on the domain $\dom$. Therefore, 
\begin{align}
	\label{eqn::deri_rho_J_bound}
	\begin{aligned}
	\abs{\nabla_z \big(\rho_k(\state{s}{z})\jacoarg{s}{z}\big)} \lesssim 
	 \jacoarg{s}{z} \Big(\norm{\corrz{s}{0}(z)} + \int_{s}^{0} \norm{\corrz{r}{0}(z)}\ud r\Big).
	\end{aligned}
\end{align}
Recall from \eqref{eq:jacob} and \eqref{eqn::corrz} that 
\begin{align*}
	\partial_s \jacoarg{s}{z} &= (\div\dynb)(\state{s}{z}) \jacoarg{s}{z}, \qquad \jacoarg{0}{z} = 1;\\
	\partial_s \corrz{s}{0}(z) &= \nabla\dynb(\state{s}{z})\corrz{s}{0}(z), \qquad \corrz{0}{0}(z) = \idmat_{\dimn}.
\end{align*}
Recall that $\state{s}{z}\to \pts_1$ as $s\to-\infty$; the value $\div\dynb(\pts_1)$ and the hessian matrix $\nabla\dynb(\pts_1)$ are both strictly positive, which imply that $\jacoarg{s}{x}$ and $\norm{\corrz{s}{0}(z)}$ are both decaying as $s\to-\infty$.

Let us denote $\upsilon>0$ as the smallest eigenvalue of $\nabla\dynb(\pts_1)$. 
For a given $\eps$ with $0 < \eps < \min\big\{\div\dynb(\pts_1), \upsilon\big\}$, let us define 
\begin{align*}
	\mathcal{E} = \big\{z\in\dom: \div\dynb(z) > \nabla\cdot\dynb(\pts_1) - \eps,\ \nabla\dynb(z) > (\upsilon - \eps) \idmat_{\dimn}\big\},
\end{align*}
which is an open neighborhood of $\pts_1$.
We can find a subset of $\mathcal{E}$, denoted as $\mathcal{E}_{\text{trap}}$, such that $\mathcal{E}_{\text{trap}}$ is a trapping region of the dynamics $\dynb$ for the backward branch. Thus we can find a negative time $\tau$ (which possibly depends on $\eps$) such that $\state{s}{x} \in \mathcal{E}_{\text{trap}}$ for all $s\le \tau$.
If we choose a $\delta$ small enough, then we can even ensure that 
\begin{align}
	\label{eqn::state_trap}
	\state{s}{z}\in \mathcal{E}_{\text{trap}}, \qquad \forall s\le \tau, \forall z\in \ball{\delta}{x},
\end{align}
due to the smoothness of $\dynb$ and the construction that $\mathcal{E}_1$ is a trapping region. Therefore, when $s\le \tau$, $\jacoarg{s}{z}$ decays to zero exponentially fast as $s\to-\infty$ with a rate at least $\div\dynb(\pts_1) - \eps$; similarly, the matrix norm $\norm{\corrz{s}{0}(z)}$ also decays to zero exponentially fast as $s\to-\infty$ with a rate at least $\upsilon-\eps$.
Therefore, on the region $(-\infty,0]\times \ball{\delta}{x}$, we can readily obtain
\begin{align}
	\label{eqn::J_C_decay}
	\jacoarg{s}{z} \lesssim e^{\big(\div\dynb(\pts_1) - \eps\big) s}, \qquad \norm{\corrz{s}{0}(z)} \lesssim e^{\big(\upsilon-\epsilon\big) s}.
\end{align}
By plugging the above estimates into \eqref{eqn::deri_rho_J_bound}, we know that 
\begin{align}
	\label{eqn::H_bound}
	\abs{\nabla_z \big(\rho_k(\state{s}{z})\jacoarg{s}{z}\big)} \lesssim \jacoarg{s}{z} \lesssim e^{\big(\div\dynb(\pts_1) - \eps\big) s}.
\end{align}
The function $e^{\big(\div\dynb(\pts_1) - \eps\big) s}$ is integrable on $(-\infty,0]$ and this serves as the function $Q$ needed in \eqref{eqn::Q} with some multiplicative constant.

\medskip
{\noindent \emph{Part 2: $\nabla G_k$ is continuous.}}

Let us denote $H(z,s) := \nabla_z \big(\rho_k(\state{s}{z})\jacoarg{s}{z}\big)$.
From \eqref{eqn::deri:H}, we can observe that $\nabla_z H(z,s)$ is continuous with respect to $z$ with  
\begin{align*}
	\nabla_z H(z,s) =& \frac{H(z,s) \big(H(z,s)\big)^{T}}{\rho_k(\state{s}{z})\jacoarg{s}{z}} + \\ & \rho_k(\state{s}{z})\jacoarg{s}{z} \nabla_z\Big(-(\corrz{s}{0}(z))^{T} \nabla U_k(\state{s}{z})\Big) + \\ 
	& \rho_k(\state{s}{z})\jacoarg{s}{z} \int_{0}^{s}\nabla_z\Big( \big(\corrz{r}{0}(z)\big)^{T} \nabla(\div\dynb)(\state{r}{z})\Big)\ud r.
\end{align*}

\begin{lemma}
	\label{lem::Hderi_bound}
	For a given smooth vector field $W\in C^{\infty}(\dom,\Real^{\dimn})$, there exists a constant $C$ such that for any $z\in \ball{\delta}{x}$ and $s\in (-\infty,0]$,
	 we have 
\begin{align*}
	\norm{\nabla_z \big(\corrz{s}{0}(z)^T W(\state{s}{z})} \le C e^{(\upsilon-\eps) s}.
\end{align*}
\end{lemma} 

By this lemma and the estimates in \eqref{eqn::H_bound}, 
	\begin{align*}
		\norm{\nabla_z H(z,s)} &\lesssim \norm\Big{\frac{H(z,s) \big(H(z,s)\big)^{T}}{\rho_k(\state{s}{z})\jacoarg{s}{z}}} + \jacoarg{s}{z} e^{(\upsilon-\eps) s}  + \jacoarg{s}{z} \int_{s}^{0} e^{(\upsilon-\eps)r}\ud r \\
		&\lesssim \jacoarg{s}{z} \lesssim e^{\big(\div\dynb(\pts_1) - \eps\big)s},
	\end{align*}
which readily leads into the continuity of $\nabla G_k$ based on \eqref{eqn::Gx}.

\medskip
{\noindent \emph{Part 3: Proof of \lemref{lem::Hderi_bound}}}

 By direct computation
 \begin{align*}
 	&\ \nabla_{z_j} \Big(\corrz{s}{0}(z)^T W(\state{s}{z}\Big)_{i} \\
 	=&\ \sum_{l} \nabla_{z_j} \Big(\big(\corrz{s}{0}(z)\big)_{l, i} W_{l}(\state{s}{z}) \Big) \\
 	\myeq{\eqref{eqn::deri_X_t}}&\ \sum_{l} \Big(\nabla_{z_j} \big(\corrz{s}{0}(z)\big)_{l, i}\Big) W_{l}(\state{s}{z}) + \sum_{l,m} \big(\corrz{s}{0}(z)\big)_{l,i} \big(\nabla W(\state{s}{z})\big)_{l,m} \big(\corrz{s}{0}(z))_{m,j}.
 \end{align*}
Since $W$ is assumed to be smooth on the torus $\dom$, we know $W$ and $\nabla W$ are uniformly bounded. Besides, previously, we know that  $\norm{\corrz{s}{0}(z)} \lesssim e^{\big(\upsilon-\epsilon\big) s}$. Therefore, we only need to prove that for any $1\le \ell\le \dimn$, 
\begin{align}
	\label{eqn::norm_corrz_ell}
	\norm{\partial_{z_\ell} \corrz{s}{0}(z)}\lesssim e^{(\upsilon-\eps) s}.
\end{align}

From \eqref{eqn::corrz_evol}, we have 
\begin{align}
	\label{eqn::corrz_deri_time}
	\begin{aligned}
	& \partial_{s} \big(\partial_{z_\ell} \corrz{s}{0}(z)\big) = \nabla\dynb(\state{s}{z}) \big(\partial_{z_\ell} \corrz{s}{0}(z)\big) + \bm{S}(s,z), \qquad 
	\big(\partial_{z_\ell} \corrz{0}{0}(z)\big) = \vectorzero_{\dimn\times\dimn},\\
	& \big(\bm{S}(s,z) \big)_{i,j} = \sum_{n,m}(\partial_{z_i,z_n,z_m} V) (\state{s}{z}) \big(\corrz{s}{0}(z)\big)_{m,\ell} \big(\corrz{s}{0}(z)\big)_{n,j}.
	\end{aligned}
\end{align}
Since $\norm{\corrz{s}{0}(z)}\lesssim e^{(\upsilon-\eps) s}$ from \eqref{eqn::J_C_decay}, the source term $\norm{\bm{S}(s,z)}$ also decays exponentially fast with rate $2(\upsilon-\eps)$ as $s\to -\infty$, namely, 
\begin{align}
	\label{eqn::S_norm}
	\norm{\bm{S}(s,z)}\lesssim e^{2(\nu-\eps)s}, \qquad \forall s\in (-\infty,0], z\in \ball{\delta}{x}.
\end{align}
By rewriting \eqref{eqn::corrz_deri_time} in the integral form and by \eqref{eqn::corrz}, we have
\begin{align*}
	\big(\partial_{z_\ell} \corrz{s}{0}(z)\big) 
	&= - \int_{s}^{0}\corrz{s}{0}(z) \big(\corrz{r}{0}(z)\big)^{-1} \bm{S}(r,z)\ \dd r
	= - \int_{s}^{0} \corrz{s}{r}(z) \bm{S}(r,z)\ \dd r.
\end{align*}
To prove \lemref{lem::Hderi_bound}, we only need to consider the case $s\ll 0$. 
Suppose we consider $s\le \tau$ only; recall the role of $\tau$ in \eqref{eqn::state_trap}. 
Then we could obtain that when $s\le r\le \tau$
\begin{align}
	\label{eqn::corrz_norm}
	\norm{\corrz{s}{r}(z)} = \norm{\exp_{\antichrontime}\Big(-\int_{s}^{r} \nabla\dynb(\state{u}{z})\ud u\Big)}\ \myle{\eqref{eqn::state_trap}}\ e^{-(\nu-\eps)(r-s)},
\end{align}
where $\exp_{\antichrontime{}}$ is the anti-chronological time-ordered operator exponential.
We can separate the above integral form using $\tau$ and obtain the following estimates: when $s\le \tau$, 
\begin{align*}
	\norm{\partial_{z_\ell} \corrz{s}{0}(z)} &\le \norm{\int_{s}^{\tau} \corrz{s}{r}(z) \bm{S}(r,z)\ \dd r} + \norm{\int_{\tau}^{0} \corrz{s}{r}(z) \bm{S}(r,z)\ \dd r}\\
	&\le \int_{s}^{\tau} \norm{\corrz{s}{r}(z) \bm{S}(r,z)}\ \dd r + \int_{\tau}^{0} \norm{\corrz{s}{0}(z) \big(\corrz{r}{0}(z)\big)^{-1} \bm{S}(r,z)}\ \dd r\\
	&\le \int_{s}^{\tau} \underbrace{\norm{\corrz{s}{r}(z)}}_{\text{use } \eqref{eqn::corrz_norm}} 
	\underbrace{\norm{\bm{S}(r,z)}}_{\text{use } \eqref{eqn::S_norm}}\ \dd r 
	+ 
	\underbrace{\norm{\corrz{s}{0}(z)}}_{\text{use } \eqref{eqn::J_C_decay}} 
	\int_{\tau}^{0} \underbrace{\norm{\big(\corrz{r}{0}(z)\big)^{-1}}}_{=\mathcal{O}(1)} \cdot \underbrace{\norm{\bm{S}(r,z)}}_{\text{use } \eqref{eqn::S_norm}}\ \dd r\\
	&\lesssim  \int_{s}^{\tau} e^{-(\upsilon-\eps)(r-s)} e^{2(\upsilon-\eps) r}\dd r 
	+ e^{(\upsilon-\eps) s} \underbrace{\int_{\tau}^{0} e^{2(\upsilon-\eps) r}\dd r}_{=\mathcal{O}(1)}\\
	&\lesssim e^{(\upsilon-\eps) s} + e^{(\upsilon-\eps) s} \lesssim e^{(\upsilon-\eps) s}.
\end{align*}
This verifies \eqref{eqn::norm_corrz_ell} for any $z\in \ball{\delta}{x}$ and $s\le\tau$. When $s\in [-\tau,0]$, we can simply choose the prefactor large enough so that \eqref{eqn::norm_corrz_ell} holds, as $\tau$ is a finite value.
Hence, \lemref{lem::Hderi_bound} is verified.

\subsection{Examples}
\label{subsec::generator-examples}

We elaborate on \propref{prop::existence_generator} by concrete examples. For these examples, we estimate $\timeT$ from \eqref{eqn::optimal_time} either analytically or numerically, and then we validate \propref{prop::existence_generator} by comparing $\rho_1$ and the empirical distribution of $\state{\timeT(x)}{x}$ where $x\sim\rho_0$.

\subsubsection{Gaussian examples in 1D}
\label{subsubsec::gaussian_1d}

For the case $U_0(x) = \frac{\abs{x}^2}{2} +\gausscst{1}$ and $U_1(x) = \frac{\abs{x-\omega}^2}{2\sigma^2}$ with $\sigma<1$, from Table \ref{table::solvable_eg}, we already know that a zero-variance dynamics is $\dynb(x) = x-\frac{\omega}{1-\sigma}$. 
Then $\state{t}{x} = e^t x - (e^t-1)\frac{\omega}{1-\sigma}$ for any $t,x\in\Real$ and $\jacoarg{t}{x} = e^t$ for any $t\in\Real$.
By direct computation 
\begin{align*}
	\frac{\int_{-\infty}^\theta \rho_1(\state{s}{x})\jacoarg{s}{x}\dd s}{		\int_{-\infty}^{0} \rho_0(\state{s}{x})\jacoarg{s}{x}\dd s} = -\frac{\erf(\frac{\omega}{\sqrt{2}(1-\sigma)}) 
		+ \erf\big(\frac{-e^\theta (\omega + x(\sigma-1) + \omega\sigma)}{\sqrt{2}\sigma(\sigma-1)}\big)
	}{
		\erf(\frac{x}{\sqrt{2}}) 
		+ \erf\big(\frac{\omega}{\sqrt{2}(\sigma-1)}\big)
	}.
\end{align*}
By solving $\timeT$ in \eqref{eqn::optimal_time} using the last equation, we have 
\begin{align*}
	\timeT(x) = \log(\sigma),\qquad \forall x \neq \frac{\omega}{1-\sigma},
\end{align*}
which is independent of the state $x$.

\begin{figure}[h!]
\centering
	\includegraphics[width=0.5\textwidth]{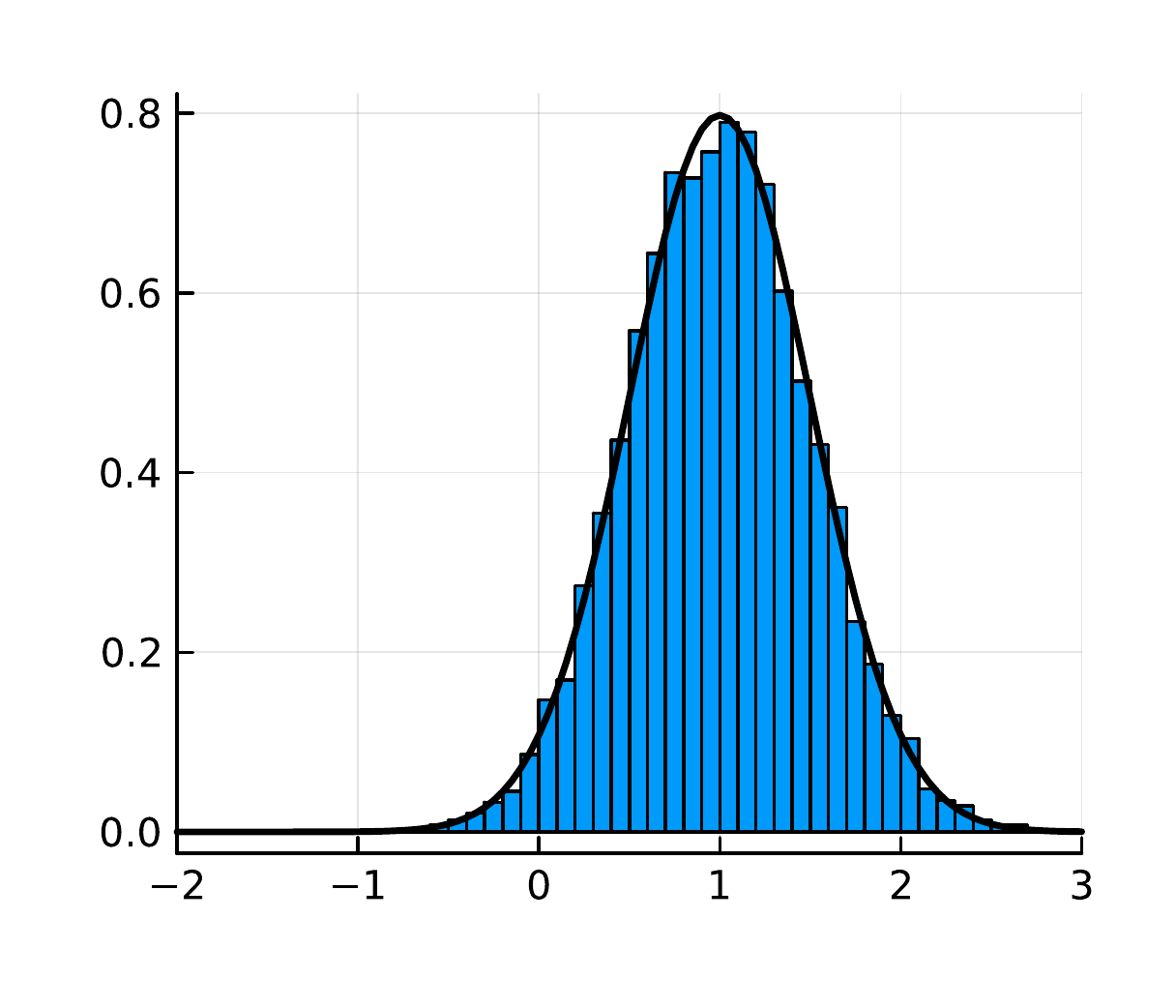}
	\caption{This figure shows the histogram of sample points of $\state{\timeT(x)}{x}$ (blue) and the distribution $\rho_1$ (black) for the 1D Gaussian example in Appendix~\ref{subsubsec::gaussian_1d} with $\omega=1$ and $\sigma=0.5$.}
\end{figure}

\subsubsection{Three-mode mixture on a 2D torus}
We consider the model \eqref{eqn::torus_eg} on the torus $\dom=[0,1]^2$ and recall that the zero-variance dynamics has been shown in \figref{fig::torus}.
Then the contour plot of $\timeT$ is visualized in \figref{fig::torus_optimal_time}. Moreover, in the same figure, the empirical distributions of $\state{\timeT(x)}{x}$ and contour plots of $\rho_1$ are provided, which numerically verifies \propref{prop::existence_generator}.

\subsubsection{An example on $(0,1)^2$ with Neumann boundary condition}

We consider the model
\begin{align}
	      \begin{aligned}
	      V(x) = \gamma \cos(2\pi x_1) & \cos(2\pi x_2), \qquad
	      \rho_0(x) = 1,\qquad    \rho_1(x) = \rho_0(x) + \Laplace V(x),\\
	      & x = \begin{bsmallmatrix}
	              x_1 \\ x_2 \\
	      \end{bsmallmatrix}, 
	     \qquad \gamma = \frac{0.45}{4\pi^2},
	      \end{aligned}
	      \label{eqn::eg_neumann_cos}
	\end{align}
on the domain $\dom=(0,1)^2$
and the potential $V$ automatically solves the Poisson's equation with Neumann boundary condition (see \eqref{eqn::poisson_vn}) by the above construction. Apart from verifying that $\dynb = \nabla V$ is a zero-variance dynamics numerically, we can also observe that \propref{prop::existence_generator} holds in this case; see \figref{fig::neumann_cos}.

\begin{figure}[h!]
\centering
	\subfigure[$\timeT$]{\includegraphics[width=0.48\textwidth]{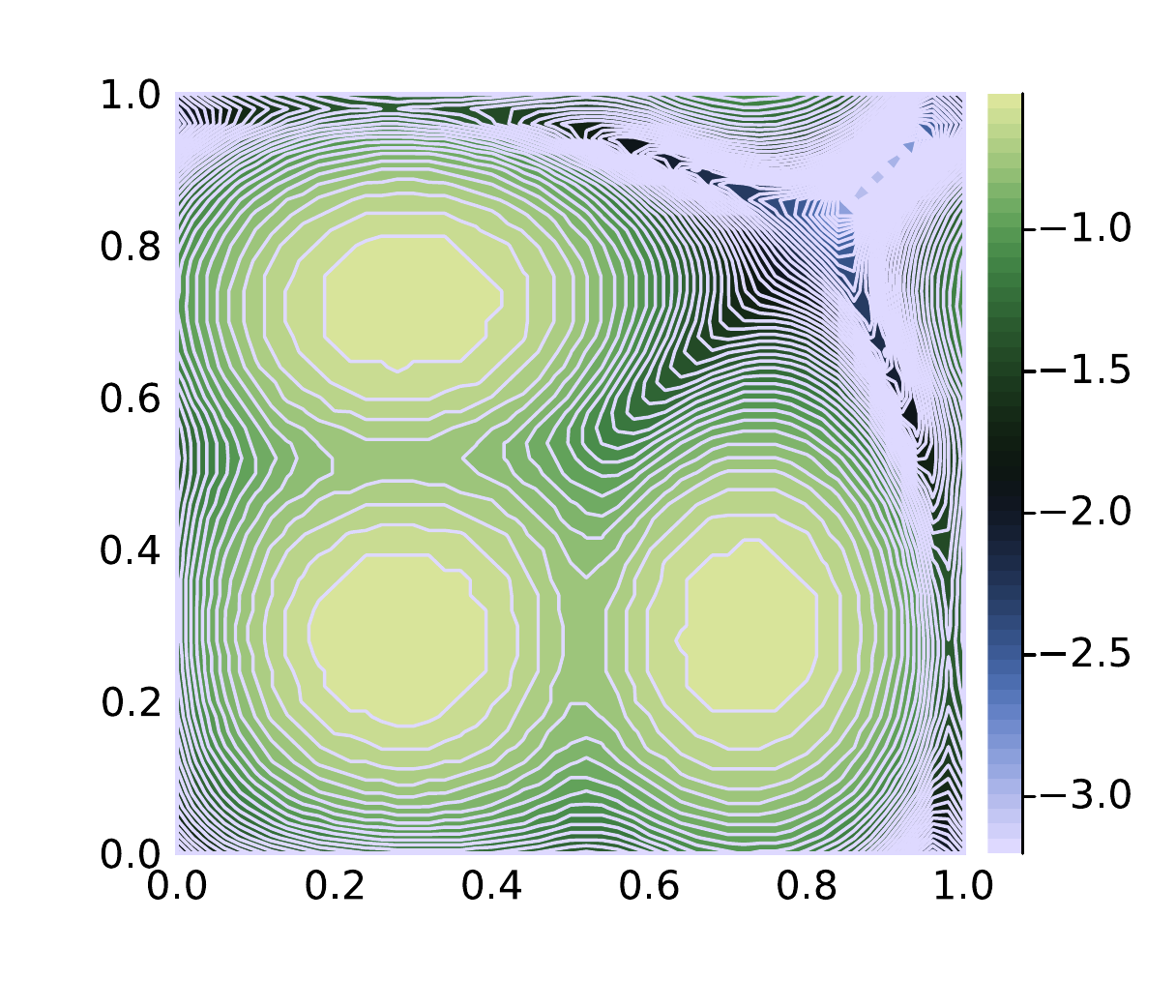}}\\
	\subfigure[histogram of samples of $\state{\timeT(x)}{x}$ where $x\sim \rho_0$]{\includegraphics[width=0.48\textwidth]{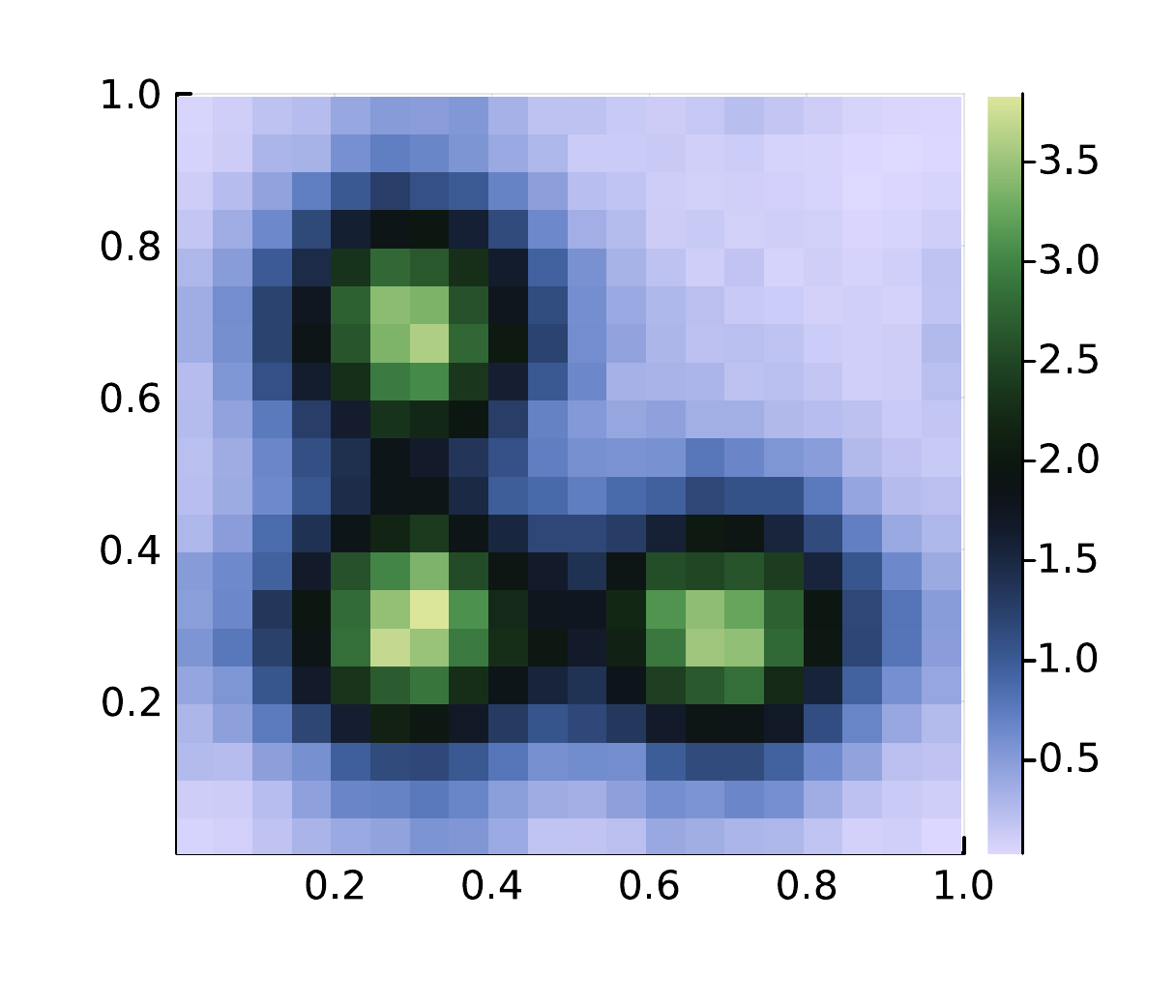}}
	\subfigure[$\rho_1$]{\includegraphics[width=0.48\textwidth]{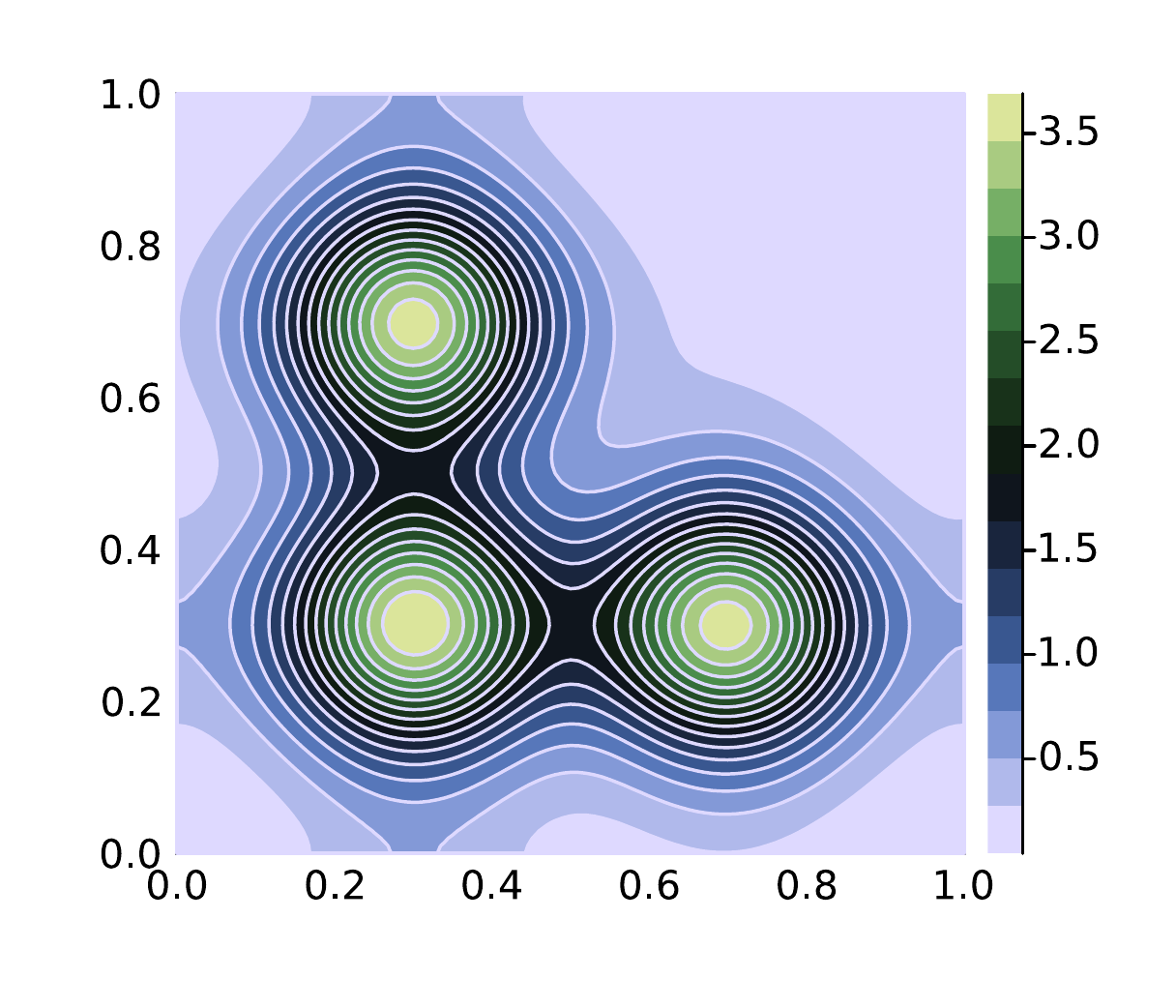}}
	\caption{The figure shows the time map $\timeT$ defined via \eqref{eqn::optimal_time} and it also numerically verifies that $\rho_1$ is the distribution of $\state{\timeT(x)}{x}$ with $x\sim\rho_0$, for the model \eqref{eqn::torus_eg} on the torus $\dom=[0,1]^2$ (with periodic boundary condition).}
	\label{fig::torus_optimal_time}
\end{figure}

\begin{figure}[h!]
\centering
	\subfigure[$V$ and trajectories]{
		\includegraphics[width=0.48\textwidth]{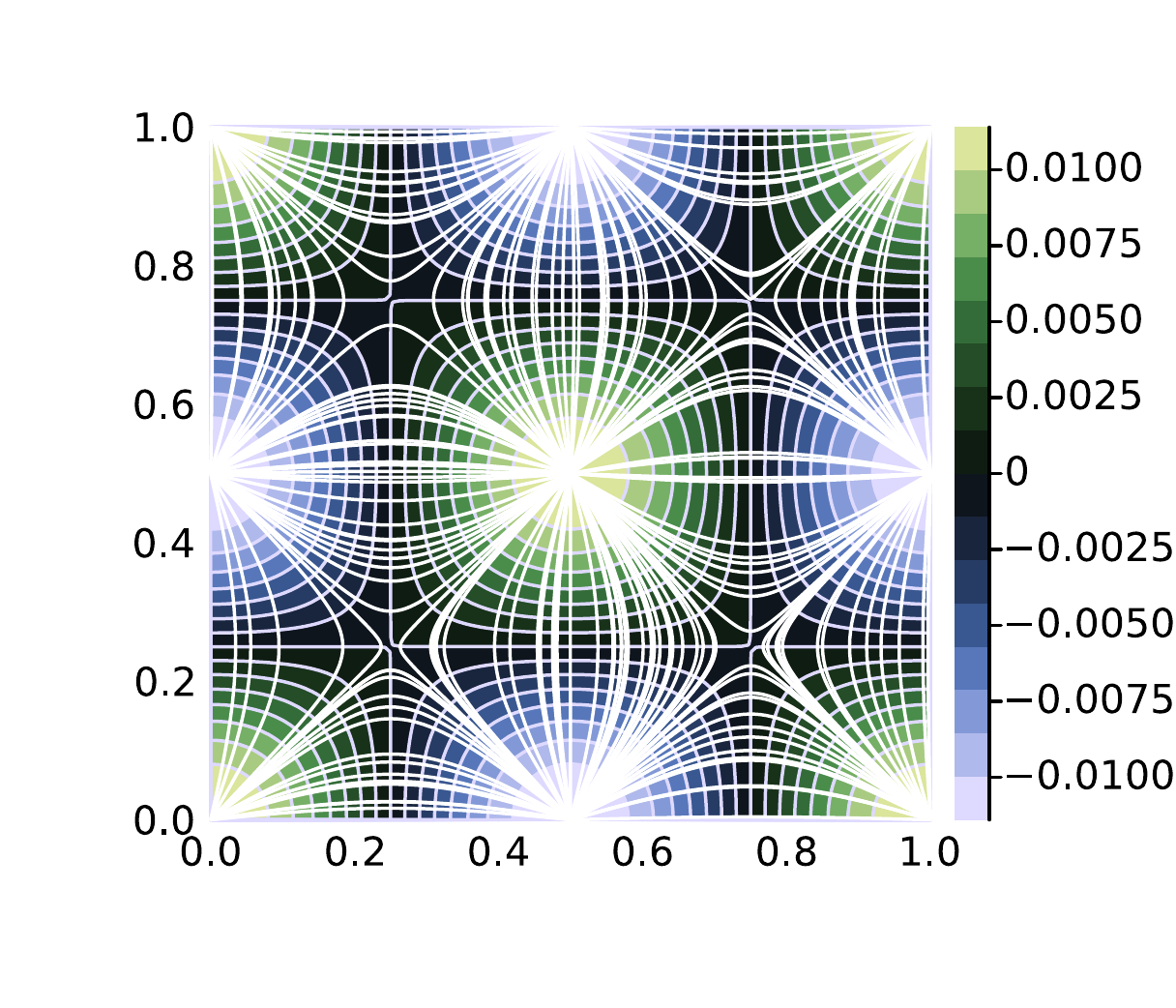}
	}
	\subfigure[$\rho_1$]{
		\includegraphics[width=0.48\textwidth]{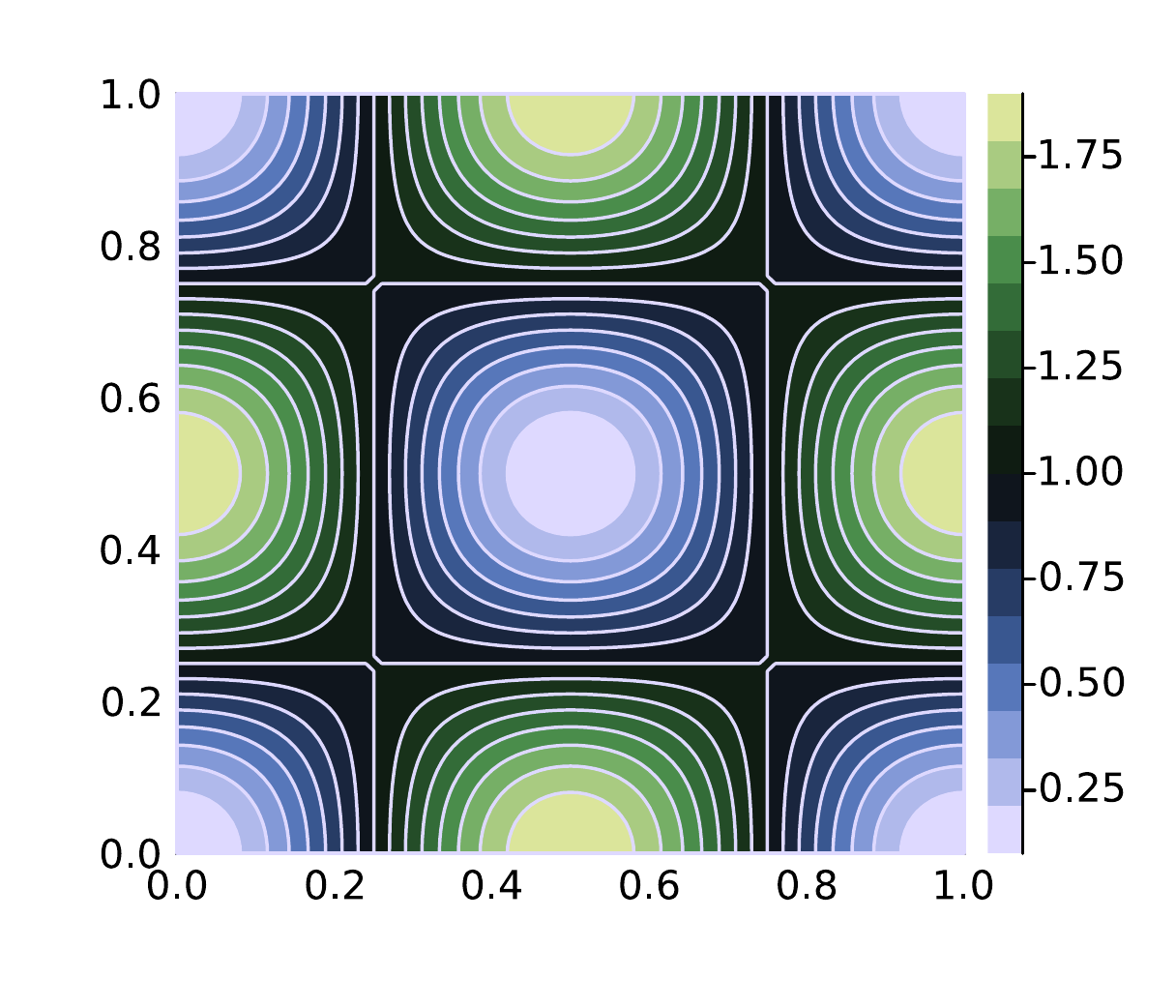}
	}
	\\
	\subfigure[$\timeT$ in symlog scale]{
		\includegraphics[width=0.48\textwidth]{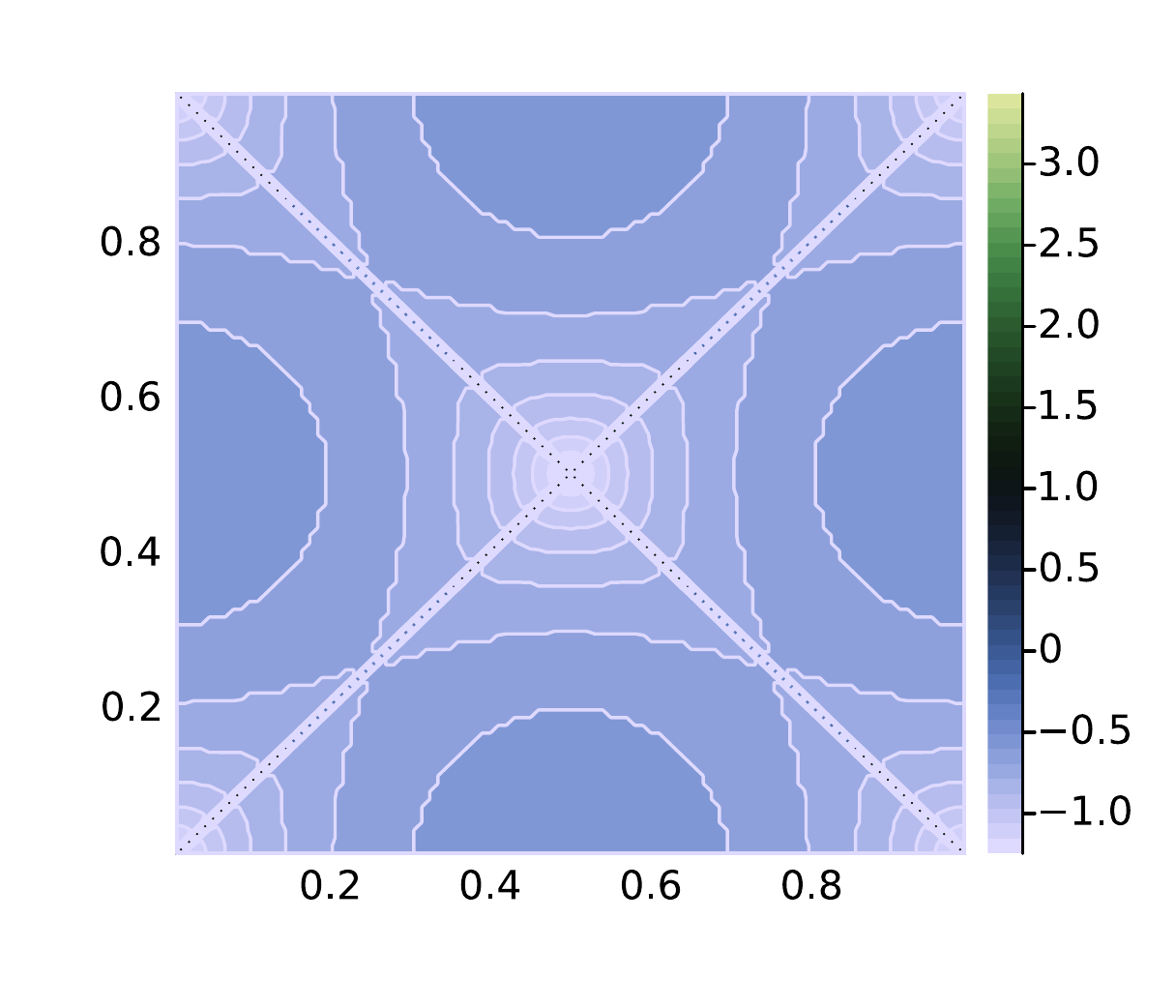}
	}
	\subfigure[histogram of samples of $\state{\timeT(x)}{x}$ where $x\sim \rho_0$]{
		\includegraphics[width=0.48\textwidth]{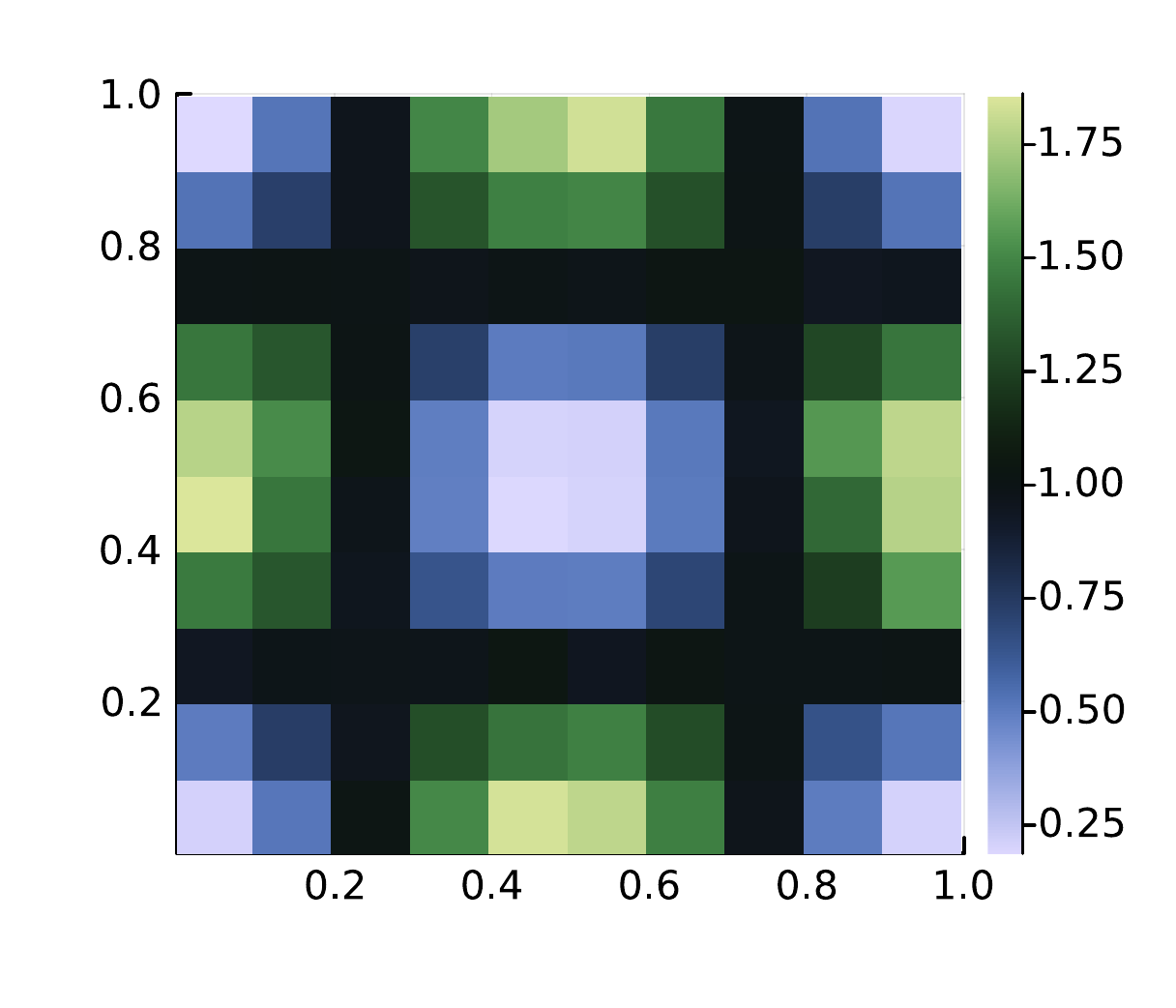}
	}
	\caption{This figure visualizes the model \eqref{eqn::eg_neumann_cos} on the domain $\dom=(0,1)^2$ (with Neumann boundary condition). 
		In the panel (a), we show the potential $V$ and sample trajectories under the dynamics $\dynb = \nabla V$; in the panel (b), we show the distribution $\rho_1$. In the panel (c), we present the time map $\varphi \circ \timeT$ where $\timeT$ is defined via \eqref{eqn::optimal_time} and the rescaling function $\varphi(z) := \text{sign}(z) \log(1+\abs{z})$ is the symlog function. In the panel (d), we show a histogram of samples of $\state{\timeT(x)}{x}$ where $x\sim \rho_0$. The panel (d) resembles the panel (b), which numerically verifies that $\rho_1$ is the distribution of $\state{\timeT(x)}{x}$ with $x\sim \rho_0$.}
	\label{fig::neumann_cos}
\end{figure}

\section{Proof of \propref{prop::local}}
\label{sec::proof::local}

Below is the full version of \propref{prop::local}.

\begin{proposition}[Local minimum]
	\label{prop::local::full}
	Assume that 
	\begin{enumerate}[label=(\roman*)]
		
		\item (Local minimum). $\dynb \in \maniinf$ is a (non-trivial) local minimum of $\newvarinf$, namely, $\newvarinf(\dynb) < \varmax$ and if there is a perturbation $\delta\dynb\in C_c^{\infty}(\Omega,\Rd)$ such that $\dynb + \eps \delta\dynb\in \maniinf$ for sufficiently small $\eps$, then $\newvarinf(\dynb+ \eps\delta\dynb) \ge  \newvarinf(\dynb)$.
		
		\item (Continuity assumption). The functional derivative 
		$\derivarinf$ is continuous and 
		$\derivarinf = \vectorzero_{\dimn}$ on $\Omega$. 
		In particular, $\nabla\newainf$ exists and is continuous on $\dom$.
		
		\item (Technical assumptions). $\effdom(\dynb) = \dom$ (see \defref{defn::maniinf}) and the set of $\dynb$-unstable points (see  \defref{defn::int_cond_dynb}) has Lebesgue measure zero. Moreover, the set $\big\{x\in \dom:\ \nabla (U_1 - U_0)(x) = \vectorzero_{\dimn}\big\}$ has Lebesgue measure zero.
	\end{enumerate}
	Then $\dynb$ is a zero-variance dynamics for the \itneis{} scheme, i.e., $\newvarinf(\dynb) = 0$.
\end{proposition}

Recall the formula of $\derivarinf$ from \eqref{eqn::func_deri_inf}:
	\begin{align*}
	\begin{aligned}
		&\derivarinf(x) =
		\frac{2\nabla \newainf(x)}{\partitionzero\binf(x)}\left(
		\int_{0}^{\infty}
		\testfuncarg{0}{t}{x}\ud t \int_{-\infty}^{0} \testfuncarg{1}{t}{x}\ud t 
		- \int_{-\infty}^{0} \testfuncarg{0}{t}{x}\ud t  \int_{0}^{\infty} \testfuncarg{1}{t}{x}\ud t \right).
	\end{aligned}
\end{align*}
Here is a sketch of the main idea behind the proof.
If $\nabla \newainf = \vectorzero_{\dimn}$ on the domain $\dom$, then $\newainf$ is a constant function. Hence, $\dynb$ provides a zero-variance estimator for the \itneis{} scheme and it must be a global minimum of the functional $\newvarinf$ as well.
If the other term $\int_{0}^{\infty} \testfuncarg{0}{t}{x}\ud t  \int_{-\infty}^{0} \testfuncarg{1}{t}{x}\ud t - \int_{-\infty}^{0} \testfuncarg{0}{t}{x}\ud t  \int_{0}^{\infty} \testfuncarg{1}{t}{x}\ud t=0$ locally on an open subset, then it could be shown that this is equivalent to $\inner{\nabla (U_1 - U_0)(x)}{\dynb(x)} = 0$ (see \lemref{lem::cond_equiv} and \lemref{prop::cond_stationary}).
Such a dynamics $\dynb$ should not be optimal, because $\dynb$ is perpendicular to the gradient of the potential difference and such a $\dynb$ does not explore the local structure (cf. \propref{prop::var_dynb}). 
This intuition leads into the following idea: if $\dynb$ is a local minimum of $\newvarinf$, then we should only have $\nabla\newainf = \vectorzero_\dimn$ almost everywhere on $\dom$; otherwise, we should be able to perturb $\dynb$ so that the dynamics $\dynb$ can better explore the landscapes of $U_0$ and $U_1$ and the variance can be further reduced; see \propref{prop::not_local_min}.

\begin{remark}
	As we work on the domain $\dom$ only, we shall consider the topological space for the domain $\dom$ instead of $\Rd$ from now on.
\end{remark}

\subsection{A characterization of the global maximum}

\begin{proposition}\label{prop::var_dynb}
 If the dynamics $\dynb \in \maniinf$, then 
\begin{align}
	\label{eqn::var_reduction}
\newvarinf(\dynb) \le \varmax.
\end{align}
The equality can be achieved iff 
\begin{align}
	\label{eqn::orthogonal_flow}
	\innerbig{\dynb(x)}{\nabla (U_1 - U_0)(x)} = 0, \qquad \forall x\in \dom.
\end{align}

\end{proposition}

\begin{proof}
	We only need to prove that $\newmsecinf(\dynb) \le \ee_{\rho_0}[e^{-2(U_1 - U_0)}]$
and the equality is achieved iff \eqref{eqn::orthogonal_flow} holds.

	By Jensen's inequality,
	\begin{align*}
		\Big(\frac{\intreal \testfuncarg{1}{t}{x} \ud t}{\intreal \testfuncarg{0}{t}{x}\ud t}\Big)^2 &= \Big(\frac{\intreal e^{-U_1\big(\state{t}{x}\big) + U_0\big(\state{t}{x}\big)}\testfuncarg{0}{t}{x}\ud t}{\intreal \testfuncarg{0}{t}{x}\ud t}\Big)^2 \\
		&\le \frac{\intreal e^{-2(U_1 - U_0)\big(\state{t}{x}\big)} \testfuncarg{0}{t}{x} \ud t}{\intreal \testfuncarg{0}{t}{x}\ud t}.
	\end{align*}
	By taking the expectation $\ee_{x\sim\rho_0}\big[\cdot\big]$ for both sides and by \eqref{eqn::noneq_sample::infinite} for new potentials $\wt{U}_1 = 2U_1-U_0$ and $\wt{U}_0 = U_0$, we immediately have the inequality:
	\begin{align*}
		\newmsecinf(\dynb) &= \ee_{x\sim \rho_0} \bigg[\Big(\frac{\intreal \testfuncarg{1}{t}{x} \ud t}{\intreal \testfuncarg{0}{t}{x}\ud t}\Big)^2\bigg] \\
		&\le \ee_{x\sim \rho_0} \bigg[\frac{\intreal e^{-2(U_1 - U_0)\big(\state{t}{x}\big)} \testfuncarg{0}{t}{x} \ud t}{\intreal \testfuncarg{0}{t}{x}\ud t}\bigg] \\
		&\myeq{\eqref{eqn::noneq_sample::infinite}}\ \frac{\int_{\dom} e^{-2 U_1 + U_0}}{\int_{\dom} e^{-U_0}} = \ee_{x\sim\rho_0} \big[e^{-2(U_1 - U_0)}\big].
	\end{align*}
	
	Next in order to achieve the equality, we need the equality to hold in the above Jensen's inequality:
	\begin{align*}
		t\to e^{-\big(U_1 - U_0\big)\big(\state{t}{x}\big)} \text{ is a constant function}.
	\end{align*}
	By taking derivative with respect to $t$, we immediately know that
	\begin{align*}
		\innerbig{\nabla (U_1 - U_0)(x)}{\dynb(x)} = 0, \qquad \rho_0\text{-almost surely}.
	\end{align*}
	By the continuity assumption on $U_1, U_0$ and $\dynb$, we obtain \eqref{eqn::orthogonal_flow}.
	
	Conversely, when \eqref{eqn::orthogonal_flow} holds, 
	\begin{align*}
		\frac{\ud}{\ud t} (U_1 - U_0)\big(\state{t}{x}\big) = \innerbig{\nabla (U_1 - U_0)(\state{t}{x})}{\dynb\big(\state{t}{x}\big)} = 0.
	\end{align*}
	Hence, $
	U_1\big(\state{t}{x}\big) - U_0\big(\state{t}{x}\big) = U_1(x) - U_0(x)$ for any $t\in \Real.$
	Then
	\begin{align*}
		\frac{\intreal \testfuncarg{1}{t}{x}\ud t}{\intreal \testfuncarg{0}{t}{x}\ud t} = e^{-U_1(x) + U_0(x)}\frac{\intreal e^{-U_0(\state{t}{x})} \jacoarg{t}{x}\ud t}{\intreal e^{-U_0(\state{t}{x})} \jacoarg{t}{x}\ud t} = e^{-U_1(x) + U_0(x)}.
	\end{align*}
	Hence, the equality in \eqref{eqn::var_reduction} holds under the condition \eqref{eqn::orthogonal_flow}.
\end{proof}

\subsection{Some observations about the functional derivative}
We need some simplified understanding of the condition $\int_{0}^{\infty} \testfuncarg{0}{t}{x}\ud t \int_{-\infty}^{0} \testfuncarg{1}{t}{x}\ud t = \int_{-\infty}^{0} \testfuncarg{0}{t}{x} \ud t \int_{0}^{\infty} \testfuncarg{1}{t}{x}\ud t$ arising from  $\derivarinf=\vectorzero_{\dimn}$, 
which are presented in the following two lemmas.

\begin{lemma}
	\label{lem::cond_equiv}
	$\int_{0}^{\infty} \testfuncarg{0}{t}{x}\ud t \int_{-\infty}^{0} \testfuncarg{1}{t}{x}\ud t = \int_{-\infty}^{0} \testfuncarg{0}{t}{x} \ud t \int_{0}^{\infty} \testfuncarg{1}{t}{x}\ud t$ is equivalent to 
	\begin{align}
		\label{eqn::cond_2_v1} 
		\frac{\int_{0}^{\infty} \testfuncarg{1}{t}{x}\ud t}{\int_{0}^{\infty}\testfuncarg{0}{t}{x}\ud t} = \newainf(x).
	\end{align}
\end{lemma}

\begin{proof}
	\begin{align*}
		& \int_{0}^{\infty} \testfuncarg{0}{t}{x}\ud t \int_{-\infty}^{0} \testfuncarg{1}{t}{x}\ud t = \int_{-\infty}^{0} \testfuncarg{0}{t}{x} \ud t \int_{0}^{\infty} \testfuncarg{1}{t}{x}\ud t\\
		\Longleftrightarrow\  &
		\frac{\int_{-\infty}^{0} \testfuncarg{0}{t}{x}\ud t}{\int_{0}^{\infty} \testfuncarg{0}{t}{x}\ud t}
		= \frac{\int_{-\infty}^{0} \testfuncarg{1}{t}{x}\ud t}{\int_{0}^{\infty} \testfuncarg{1}{t}{x}\ud t}\qquad \text{(add } 1 \text{ to both sides)} \\
		\Longleftrightarrow\ &
		\frac{\int_{-\infty}^{\infty} \testfuncarg{0}{t}{x}\ud t}{\int_{0}^{\infty} \testfuncarg{0}{t}{x}\ud t}
		= \frac{\int_{-\infty}^{\infty} \testfuncarg{1}{t}{x}\ud t}{\int_{0}^{\infty} \testfuncarg{1}{t}{x}\ud t} \\
		\Longleftrightarrow\ & \frac{\int_{0}^{\infty} \testfuncarg{1}{t}{x}\ud t}{\int_{0}^{\infty}\testfuncarg{0}{t}{x}\ud t} = \newainf(x).
	\end{align*}
\end{proof}

The following result provides a simplified characterization of the equality \eqref{eqn::cond_2_v1}.
\begin{lemma}{\hspace{2em}}\par
	\label{prop::cond_stationary}
	\begin{enumerate}[label=(\roman*)]
		\item Suppose the condition in \eqref{eqn::cond_2_v1} holds
		for any $x$ in an open set $\domsubset$.
		Then
		\begin{align}
			\label{eqn::stationary_cond}
			\inner{\nabla (U_1 - U_0)(x)}{ \dynb(x)} = 0, \qquad \forall x\in \domsubset.
		\end{align}
		\label{lem::cond_stationary::1}
		
		\item Conversely, if \eqref{eqn::stationary_cond} holds for $\domsubset = \dom$, then $\dynb$ is a global maximum of $\newmsecinf$ (as well as $\newvarinf$).
		\label{lem::cond_stationary::2}
	\end{enumerate}
\end{lemma}

\begin{proof}
	Part (ii) immediately follows from \propref{prop::var_dynb}. Next we prove part (i). From previous results, the condition is that 
	\begin{align*}
		\frac{\int_{0}^{\infty} \testfuncarg{1}{t}{x}\ud t}{\int_{0}^{\infty} \testfuncarg{0}{t}{x}\ud t} = \newainf(x), \qquad \forall x\in \domsubset.
	\end{align*}
	After replacing $x$ by $\state{s}{x}\in \domsubset$ in the above equation
	and by \eqref{eqn::testfunc_trans}, we have
	\begin{align*}
		\newainf(x) \myeq{\eqref{eqn::time_trans_ainf}} \newainf\big(\state{s}{x}\big) =\frac{\int_{0}^{\infty} \testfuncargbig{1}{t}{\state{s}{x}}\ud t}{\int_{0}^{\infty} \testfuncargbig{0}{t}{\state{s}{x}}\ud t} = \frac{\int_{s}^{\infty} \testfuncarg{1}{t}{x}\ud t}{\int_{s}^{\infty} \testfuncarg{0}{t}{x}\ud t},\ \
		\forall\ x\in \domsubset,\ s\in \big(\bhit{\domsubset}(x), \fhit{\domsubset}(x)\big),
	\end{align*}
	where $\fhit{\domsubset}(x)$ and $\bhit{\domsubset}(x)$ defined in \eqref{eqn::hitting_time} are hitting times for the forward and backward branches to the boundary of $\domsubset$. 
	Note that the right hand side of the last equation depends on $s$, whereas the left hand side does not. Let us take the derivative with respect to $s$ and with straightforward simplifications, we obtain 
	\begin{align*}
		\testfuncarg{1}{s}{x}/\testfuncarg{0}{s}{x} = \newainf(x),\qquad \forall\ x\in \domsubset,\ s\in \big(\bhit{\domsubset}(x), \fhit{\domsubset}(x)\big).
	\end{align*}
	By \eqref{eqn::testfunc}, we know 
	$\exp\Big(U_0\big(\state{s}{x}\big) - U_1\big(\state{s}{x}\big)\Big) = \newainf(x).$
	Again, 
	the left hand side depends on $s$ whereas the right hand side does not.
	So we take the derivative with respect to $s$ again and obtain
	$\innerbig{\nabla(U_1-U_0)\big(\state{s}{x}\big)}{\dynb\big(\state{s}{x}\big)} = 0$ for any $x\in \domsubset$ and $s\in \big(\bhit{\domsubset}(x), \fhit{\domsubset}(x)\big)$. Then \eqref{eqn::stationary_cond} follows immediately by choosing $s=0$.
\end{proof}

\subsection{A weaker version}

\lemref{prop::cond_stationary} leads into the following intuition: if there is a certain open subset $\domsubset$ in which $\inner{\nabla(U_1 - U_0)}{\dynb} = 0$, then such a dynamics $\dynb$ should not be a local minimum of $\newvarinf$, because such a $\dynb$ cannot explored the landscape structure of $U_1$ on $\domsubset$. 
This intuition is more rigorously formulated in the following proposition.

\begin{proposition}
\label{prop::not_local_min}
Consider a dynamics $\dynb\in \maniinf$.
Suppose $\domsubset \subset \effdom(\dynb)$ is nonempty and open. Let $K := \Omega\backslash \domsubset$.
Assume that
\begin{enumerate}[label=(\roman*)]
\item For any $x\in \domsubset$, we have
$\inner{\dynb(x)}{\nabla (U_1 - U_0)(x)} = 0.$

\item For any $t\in \Real$ and $x\in \domsubset$, we have $\state{t}{x} \notin K^{\circ}$, \ie{}, trajectories from $\domsubset$ are confined inside $\overline{\domsubset}$.

\item There exists a $\dynb$-stable point $\xst\in \domsubset$ such that $\nabla (U_1 - U_0)(\xst) \neq
 \vectorzero_{\dimn}$.
\end{enumerate}
Then such a dynamics $\dynb$ must not be a local minimum of the second moment $\newmsecinf$ (as well as the variance).
\end{proposition}

\begin{proof}
We proceed in two steps.

{\bf Step (\rom{1}):} The first goal is
to find a smooth function $\delta \dynb\in C_c^\infty(D,\Rd)$ and $\eps_0 > 0$ such that
\begin{subequations}
\begin{align}
\rho_0(E)>0& \text{ where } E:= \big\{x\in \domsubset\ \big\rvert\ \innerbig{\delta \dynb(x) }{ \nabla (U_1 - U_0)(x)} \neq 0\big\}\subset \text{supp}(\delta\dynb); \label{eqn::delta_b_non_perp}\\
\dynb + &\eps\delta \dynb \in \maniinf,\qquad \forall \eps \in (0, \eps_0);
\label{eqn::delta_b_finite}\\
\text{dist} &\big(E, \partial D\big) > 0\label{eqn::dist_positive}.
\end{align}
\end{subequations}
By the assumption (iii) of this proposition and \propref{prop::perturbation}, we know there is an open ball $\ball{\lambda}{\xst}\subset \domsubset$ such that for any $\delta\dynb\in C_c^{\infty}\big(\ball{\lambda}{\xst}, \Rd\big)$, $\dynb + \eps\delta\dynb\in \maniinf$ for small enough $\epsilon$ and thus \eqref{eqn::delta_b_finite} is satisfied. 
It is clear that we can easily choose $\lambda$ small enough so that \eqref{eqn::dist_positive} holds.

Next, the task is to find a smooth function $\delta\dynb$ supported on $\ball{\lambda}{\xst}$ such that \eqref{eqn::delta_b_non_perp} holds. 
Since $\nabla(U_1 - U_0)(\xst)\neq \vectorzero_{\dimn}$ and $U_0, U_1$ are assumed to be smooth, we can choose $\lambda$ small enough such that
	\begin{align*}
		\abs{\nabla (U_1 - U_0)(x) - \nabla(U_1-U_0)(\xst)} \le  {\frac{1}{4}} \abs{\nabla(U_1-U_0)(\xst)},\qquad \forall x\in \ball{\lambda}{\xst}.
	\end{align*}
It is well-known that
\begin{align*}
\varphi(x) :=  \left\{\begin{aligned}
	e^{-\frac{1}{1-\abs{x}^2}}, & \qquad \text{ if }\abs{x} < 1;\\
	0, & \qquad \text{ if } \abs{x} \ge 1,
\end{aligned}\right.
\end{align*}
is a smooth function compactly supported on $\ball{1}{0}$.
Then let us consider
\begin{align*}
\delta \dynb_1(x) := \varphi\big(\frac{x-\xst}{\lambda_1}\big) \nabla (U_1 - U_0)(x), \qquad \text{where } \lambda_1 \in(0, \lambda).
\end{align*}
It is clear that $\delta \dynb_1$ is compactly supported on $\ball{\lambda_1}{\xst}\subset \ball{\lambda}{\xst}$ and for any $x\in \ball{\lambda_1}{\xst}$, we have $\inner{\delta\dynb_1(x)}{\nabla(U_1 - U_0)(x)}>0$ so that \eqref{eqn::delta_b_non_perp} clearly holds.
Next, we still need to further smooth out $\delta\dynb_1$ (see \eg{}, Appendix C of \cite{evans_pde_2010}) by introducing 
\begin{align*}
\delta \dynb_2(x) :=
\int_{\Real^{\dimn}} \varphi_{\varepsilon}(x-y) \delta \dynb_1(y)\ud y,
\end{align*}
where $\varphi_{\varepsilon}(x) := \frac{1}{\varepsilon^{\dimn}}\varphi(x/\varepsilon)$.
{Note that we can easily extend $\delta\dynb_1$ to $\Rd$ by letting $\delta\dynb_1=\vectorzero_\dimn$ outside of $\ball{\lambda_1}{\xst}$ so that $\delta\dynb_1$ can be well-defined on $\Rd$.}
By choosing $0 < \varepsilon < \lambda - \lambda_1$, we can ensure that the smooth function $\delta\dynb_2$ is compactly supported on $\ball{\lambda}{\xst}$.
It is also not hard to show that \eqref{eqn::delta_b_non_perp} still holds for $\delta \dynb_2$: for any $x\in \ball{\lambda_1}{\xst}$, 
\begin{align*}
	\innerbig{\delta\dynb_2(x)}{\nabla(U_1 - U_0)(x)} =& \int_{\Rd} \varphi_{\varepsilon}(x-y) \innerbig{\delta\dynb_1(y)}{\nabla (U_1 - U_0)(x)}\ud y \\
	=&   \int_{\Rd} \varphi_{\varepsilon}(x-y) \varphi\big(\frac{y-\xst}{\lambda_1}\big)\innerbig{\nabla (U_1-U_0)(y)}{\nabla (U_1 - U_0)(x)}\ud y\\
	\ge &  {\frac{7}{16}}\abs{\nabla (U_1 - U_0)(\xst)}^2\int_{\Rd} \varphi_{\varepsilon}(x-y) \varphi\big(\frac{y-\xst}{\lambda_1}\big)\ud y > 0.
\end{align*}
In summary, $\delta\dynb_2$ constructed above satisfies all requirements.

\medskip

{\noindent {\bf Step (\rom{2}):}} We prove that $\dynb$ is not a local minimum by showing that $\newmsecinf(\dynb + \eps \delta \dynb) < \newmsecinf(\dynb)$ for any $\eps \in (0, \eps_0)$, where $\delta\dynb$ satisfies all conditions in the Step (\rom{1}).

By the construction of $\delta\dynb$, we know $\dynb^{\eps} := \dynb + \eps\delta \dynb$ does not change the velocity field at $\partial \domsubset$, and therefore,  $\statesup{t}{x}{\eps} \notin K^\circ$ for any $x\in \domsubset$ still holds for the dynamics $\dynb^{\eps}$
 i.e., trajectories from $\domsubset$ do not enter $K^\circ$.
 As an immediately consequence, trajectories $t\to\statesup{t}{x}{\eps}$ with $x\in K^\circ$ will not enter $\domsubset$ (as ODE trajectories are reversible).
 The slightly technical part is to consider trajectories $t\rightarrow \statesup{t}{x}{\eps}$ with $x\in \partial \domsubset \equiv \partial K$, where the trajectory $t\to\statesup{t}{x}{\eps}$ evolves under the dynamics $\dynb^{\eps}$. 
Let us consider two disjoint sets:
 \begin{align*}
 \wt{\domsubset}&:= \big\{x\in D\cup \partial \domsubset \equiv \Omega\backslash K^\circ\ \big\rvert\ \statesup{t}{x}{\eps} \notin K^\circ,\ \forall t\in \Real\big\}\supset \domsubset,\qquad \wt{K}:= \Omega\backslash \wt{\domsubset} \subset K.
 \end{align*}
 By such a construction, we can observe that for any trajectory $t\rightarrow \statesup{t}{x}{\eps}$ with $x\in \wt{K}$ (or $x\in\wt{\domsubset}$), it must remain inside $\wt{K}$ (or $\wt{\domsubset}$).
 In other words, the flows within $\wt{\domsubset}$ and $\wt{K}$ are completely separated from each other. 
Moreover, because $\delta\dynb$ is only supported on $E\subset\domsubset$ which is completely inside $\domsubset$ by \eqref{eqn::dist_positive}, we know the above definitions of these two sets  $\wt{\domsubset}$ and $\wt{K}$ are independent of $\eps\in [0,\eps_0)$.

Let us use $\newainf$ to denote the function defined in \eqref{eqn::ainf_binf} corresponding to the dynamics $\dynb$ and use  $\newainf_\eps$ to denote the one corresponding to the perturbed dynamics $\dynb^{\eps}$.
Recall the assumption (i) that $\inner{\dynb(x)}{\nabla(U_1 - U_0)(x)}=0$ for all $x\in \domsubset$. By the fact that $\partial D = \partial K$, and by the continuity of $\dynb$, $\nabla U_0$, and $\nabla U_1$, we know 
\begin{align*}
	\innerbig{\dynb(x)}{\nabla (U_1 - U_0)(x)} = 0,\qquad  \forall x\in \partial \domsubset \cup \domsubset \equiv \Omega\backslash K^\circ.
\end{align*}
For any trajectory $t\rightarrow \state{t}{x}$ with $x\in \wt{\domsubset}$, we can easily show that $e^{-(U_1-U_0)\big(\state{t}{x}\big)} = e^{-(U_1-U_0)(x)}$ for all $t\in\Real$, and thus we have $\newainf(x) = e^{-(U_1 - U_0)(x)}$ for any $x\in \wt{\domsubset}$ (the same calculation, in fact, has been shown in the proof of \propref{prop::var_dynb}).
Hence,
\begin{align*}
\newmsecinf(\dynb) &= \ee_{\rho_0} \Big[\indi_{\wt{K}}(\cdot) \big(\newainf(\cdot)\big)^2 \Big] + \ee_{\rho_0} \Big[\indi_{\wt{\domsubset}}(\cdot) \big(\newainf(\cdot)\big)^2\Big] \\
&= \ee_{\rho_0} \Big[\indi_{\wt{K}}(\cdot) \big(\newainf(\cdot)\big)^2 \Big] +
\ee_{\rho_0} \Bigl[\indi_{\wt{\domsubset}}(\cdot) e^{-2(U_1 - U_0)(\cdot)}\Big],
\end{align*}
where $\indi_{A}(\cdot)$ is an indicator function for a set $A$.

Next we consider the trajectory $t\rightarrow \statesup{t}{x}{\eps}$ 
with $x\in \wt{K}$. 
Since such a trajectory never enters $\wt{\domsubset}\supset\domsubset$ and
 $\dynb^\eps = \dynb$ on $\Omega\backslash D$, we know $\statesup{t}{x}{\eps} = \state{t}{x}$ for any $t\in \Real$ and $x\in \wt{K}$, and thus
$\newainf = \newainf_{\eps}$ on $\wt{K}$ for any $\eps\in [0,\eps_0)$.
Hence, $\ee_{\rho_0} \Big[\indi_{\wt{K}}(\cdot) \big(\newainf(\cdot)\big)^2 \Big] = \ee_{\rho_0} \Big[\indi_{\wt{K}}(\cdot) \big(\newainf_\eps(\cdot)\big)^2\Big]$.
By the same argument as in \propref{prop::var_dynb} (by treating $\wt{\domsubset}$ as the domain),
\begin{align*}
\newmsecinf(\dynb^{\eps}) - \newmsecinf(\dynb) =
 \ee_{\rho_0} \Big[\indi_{\wt{\domsubset}}(\cdot) \bigl(\newainf_{\eps}(\cdot)\bigr)^2\Big] - \ee_{\rho_0} \Bigl[\indi_{\wt{\domsubset}}(\cdot) e^{-2(U_1 -U_0)(\cdot)}\Big] \le 0,
\end{align*}
where the equality is achieved only if $\inner{\dynb^{\eps}}{\nabla (U_1 - U_0)} = 0$ on $\wt{\domsubset}$.
Note that on $\wt{\domsubset}$, 
\begin{align*}
	\inner{\dynb^{\eps}}{\nabla (U_1 - U_0)} = \inner{\dynb + \eps \delta\dynb}{\nabla (U_1 - U_0)} = \eps \inner{\delta \dynb}{\nabla (U_1 - U_0)}.
\end{align*}
However, due to the fact that $\inner{\delta \dynb}{\nabla (U_1 - U_0)}\neq 0$ for some open subset of $\domsubset$ with strictly positive $\rho_0$-measure (as constructed in the Step (\rom{1})), the equality $\newmsecinf(\dynb^{\eps}) = \newmsecinf(\dynb)$ cannot be achieved and thus
\begin{align*}
	\newmsecinf(\dynb^{\eps}) < \newmsecinf(\dynb).
\end{align*}
Since we can find a local perturbation $\delta\dynb$ such that $\newmsecinf(\dynb^{\eps}) < \newmsecinf(\dynb)$ for any $\eps\in (0, \eps_0)$,
then $\dynb$ must not be a local minimum of $\newmsecinf$.
\end{proof}

\subsection{Proof of \propref{prop::local::full}}

By \propref{prop::stationary_inf},
we know either 
\begin{align*}
& \nabla \newainf(x) = \vectorzero_{\dimn}, 
\text{ or }
\int_{0}^{\infty} \testfuncarg{0}{t}{x}\ud t \int_{-\infty}^{0} \testfuncarg{1}{t}{x}\ud t - \int_{-\infty}^{0} \testfuncarg{0}{t}{x}\ud t  \int_{0}^{\infty} \testfuncarg{1}{t}{x}\ud t = 0.
\end{align*}
Define 
\begin{align*}
	K := \big\{x\in \dom:\ \nabla \newainf(x) = \vectorzero_{\dimn}\big\},
\end{align*}
which is a closed subset of $\dom$ by the continuity assumption on $\nabla\newainf$.

Hence, $\domsubset:=\Omega\backslash K$ is open and by \lemref{prop::cond_stationary} \ref{lem::cond_stationary::1}, we know $\inner{\dynb}{\nabla (U_1 - U_0)} = 0$ on $\domsubset$.
Here are a few cases to discuss.

\emph{Case (\rom{1})}: $K^\circ = \emptyset$. \par\smallskip
If $K^{\circ} = \emptyset$, then we claim that $\dynb$ must be a global maximum and this contradicts with the assumption that $\newvarinf(\dynb)<\varmax$.
It is not hard to observe that $K^\circ = \emptyset$ implies that $\overline{\domsubset}= \Omega$.
By continuity, we know $\inner{\dynb}{\nabla (U_1 - U_0)} = 0$ on $\overline{\domsubset} = \Omega$.
Then by \lemref{prop::cond_stationary} part \ref{lem::cond_stationary::2}, $\dynb$ must be a global maximum.

\emph{Case (\rom{2})}: $\domsubset = \emptyset$.\par\smallskip

If $\domsubset = \emptyset$ (\ie{}, $K = \Omega$), then it is clear that $\newainf$ is a constant function, and thus the variance $\newvarinf(\dynb)  = 0$. This means $\dynb$ is a zero-variance dynamics.

\emph{Case (\rom{3})}: $K^{\circ} \neq \emptyset$ and $\domsubset \neq \emptyset$ \par\smallskip 

In order to use \propref{prop::not_local_min}, we need to deal with the case that some trajectories $t\rightarrow \state{t}{x}$ for $x\in \domsubset$ might enter $K^\circ$. Let us introduce
\begin{align}
\label{eqn::setS}
S = \Big\{ x\in D:\ \big\{\state{t}{x}\big\}_{t\in\Real}\cap K^\circ \neq \emptyset \Big\},
\end{align}
which essentially contains all points in $\domsubset$ whose trajectories enter $K^\circ$ at some time.
We can easily show that $S$ is open: because $\dynb$ is assumed to be smooth, the trajectories are continuous under a small perturbation {for initial states}.
Then the new disjoint sub-regions to consider are $\wt{K}:=\overline{K^\circ \cup S}$ and $\wt{\domsubset} := \Omega\backslash \wt{K}$.
We collect some facts for clarity:
\begin{itemize}[leftmargin=4ex]
\item $\wt{K}^\circ = K^\circ \cup S \neq \emptyset$;

\item $\wt{\domsubset}\subset \domsubset$, which immediately implies that $\inner{\dynb}{\nabla (U_1 - U_0)} = 0$ on $\wt{\domsubset}$.

\item If $x\in \wt{\domsubset}$, then $\state{t}{x}$ must not enter $\wt{K}^\circ$. Indeed, if not, then we have either $t\rightarrow \state{t}{x}$ entering $K^\circ$ (which contradicts with $x\notin S$), or $\state{t}{x}$ entering $S$ (which still means $\state{t}{x}$ will enter $K^\circ$ due to the reversibility of deterministic trajectories).
\end{itemize}

{We need to discuss two cases:}
\begin{enumerate}[label=(\alph*)]
	\item 
Firstly, let us consider $\wt{\domsubset} = \emptyset$, \ie{}, $\wt{K}=\dom$ and thus $\wt{K}^\circ = \dom$ by the assumption that $\dom$ is open in the topology of the space $\Rd$. 
The connectivity granted by the definition of $S$ in \eqref{eqn::setS} implies that we can divide $\wt{K}^\circ$ into a countable number of sub-regions on which the function $\newainf$ is a constant. More specifically, for $x\in \Omega$, define
\begin{align*}
R_x := \big\{y\in \Omega:\ \exists \gamma_{\cdot}\in C([0,1],\Rd), \gamma_0 = x, \gamma_1 = y,\ \newainf(\gamma_t) = \newainf(x),\ \forall t\in [0,1]\big\}.
\end{align*}
Obviously, $x\in R_x$. 
By the invariance of $\newainf$ under the dynamical flow (see \eqref{eqn::time_trans_ainf}) and the definition of $S$ \eqref{eqn::setS}, we have $\cup_{x\in K^\circ} R_x \supset \wt{K}^\circ =\dom$, which implies that $\cup_{x\in K^\circ} R_x = \dom$.
\begin{itemize}[leftmargin=4ex]
	\item Suppose $x, y\in K^\circ$ and $R_x\cap R_y\neq \emptyset$, then there exists a $z\in R_x\cap R_y$ such that $z$ connects to both $x$ and $y$ via a continuous path and thus $\newainf(x) = \newainf(z)=\newainf(y)$. It is then clear that $R_x = R_y$ via treating $z$ as a bridge.
	Therefore, $\cup_{x\in K^\circ} R_x$ can be simplified as $\cup_{x\in E} R_x$ where $\{R_x\}_{x\in E}$ are disjoint and $E\subset K^\circ$.
	 
	\item Next, we can show that $E$ is countable. Since $K^\circ$ is open and $\nabla\newainf(x)=\vectorzero_\dimn$ on $K$, we can easily verify that for each $x\in K^\circ$, there exists a local neighbor $\ball{\delta}{x}\subset R_x$ and due to the fact that $\mathbb{Q}^\dimn$ is dense and countable, $E$ is at most countable.
	
\end{itemize} 
To summarize, we have 
\begin{align*}
\bigcup_{x\in E} R_{x} =\dom,
\end{align*}
where $\{R_x\}_{x\in E}$ are disjoint and $E$ is countable.

As $\newainf$ is a constant function on $R_{x}$ by the definition of $R_{(\cdot)}$, $\newainf$ is a step function on $\dom$. By the continuity of $\newainf$ from the assumption and $\dom$ is a connected open domain,
we readily know $\newainf$ must be a constant function on $\dom$ instead, and such a $\dynb$ must be a zero-variance dynamics.

\item Next, let us consider the case $\wt{\domsubset}\neq \emptyset$. 
By the assumption that $\nabla (U_1 - U_0) = \vectorzero_{\dimn}$ only on a set with Lebesgue measure zero, we know there must exist $y\in\wt{\domsubset}$ such that $\nabla (U_1 - U_0) (y)\neq \vectorzero_{\dimn}$.
By the assumption that $\dynb$-unstable points has Lebesgue measure zero, there must exist a $\dynb$-stable point $\xst$ (around $y$) such that $\nabla(U_1 - U_1)(\xst) \neq \vectorzero_{\dimn}$.
Then \propref{prop::not_local_min} tells us that $\dynb$ must not be a local minimum, which contradicts with the assumption.
\end{enumerate}

\section{Supplementary material for numerical experiments}
\label{app::numerics_details}

\subsection{Details about AIS.} 
\label{subsec::ais}

For AIS method, we use the equally-spaced temperature distribution, i.e., $\pi_k \propto \rho_0^{1-\beta_k} \rho_1^{\beta_k}$, where $\beta_k = k/K$ for $0\le k\le K$; for each transition step, we use Metropolis-adjusted Langevin algorithm with time step $0.1$ to generate the chain.

More specifically, suppose $M_j(\cdot,\cdot)$ is a transition kernel which leaves $\pi_j$ invariant, then 
\begin{align}
	\label{eqn::ais_formula}
	\ratio = \ee\bigg[ e^{-\sum_{j=1}^{K} (\beta_j-\beta_{j-1}) \big(U_1(x_{j-1}) - U_0(x_{j-1})\big)}\bigg],
\end{align}
where $x_{j} \sim M_j(x_{j-1}, \cdot)$ and $x_0\sim \rho_0$ \cite{neal_annealed_2001,arbel_annealed_2021}.
To implement such a transition kernel, we use Metropolis-adjusted Langevin algorithm:
suppose $\tau$ (which is chosen as $0.1$ as a prescribed parameter) is the time step, then let $\wt{x}_{j} := x_{j-1} + \tau \nabla \log\pi_j(x_{j-1}) + \sqrt{2\tau} \xi_j$, where $\xi_j$ are \emph{i.i.d.} $\dimn$-dimensional standard normal random variables;
the state $\wt{x}_j$ is accepted with a rate 
\begin{align*}
	\min\Big\{1, \frac{\pi_j(\wt{x}_{j}) q(x_{j-1} \mid \wt{x}_{j})}{\pi_{j}(x_{j-1}) q(\wt{x}_j\mid x_{j-1})}\Big\},
\end{align*}
where $q(x'\mid x) = \exp\big(-\frac{1}{4\tau} \abs{x' - x - \tau \nabla\log\pi_j(x)}^2\big)$ coming from the transition probability for the Langevin step.

For each $j$, inside the Metropolis-Hasting correction term, 
we need $2$ queries to $\nabla U_1$ (i.e., $\nabla U_1(x_{j-1})$ and $\nabla U_1(\wt{x}_j)$ hidden inside the computation of $\nabla\log\pi_j$ for the acceptance rate),
and we need two queries
to $U_1$ when computing $U_1(x_{j-1})$ (inside \eqref{eqn::ais_formula} and $\pi_j(x_{j-1})$) and $U_1(\wt{x}_j)$ (inside $\pi_j(\wt{x}_j)$).

\subsection{More implementation details}
\label{subsec::numerical_details}

\newparagraphno{Neural network architecture} We use the following $\ell$-layer neural network \cite{e_barron_2019,petersen_neural_2020} to parameterize the dynamics $\dynb:\dom\to\Rd$ during training:
\begin{align*}
		x\to W_{\ell}\big(f_{\ell-1} \circ \cdots f_2 \circ f_1(x)\big) + b_{\ell},
\end{align*}
where $\ell$ is the layer depth (one output layer and $\ell-1$ hidden layers), $f_j(\cdot) = \AF(W_{j} (\cdot) + b_{j})$ for $j = 1, 2, \cdots, \ell -1$,
$\AF$ is the activation function,
$W_{j}\in \Real^{n_j}\times \Real^{n_{j+1}}$,
$b_j \in \Real^{n_{j+1}}$ for $j = 1, 2, \cdots, \ell$.
When we choose $n_j = \mode$ for all $2\le j\le \ell$, we refer such a neural network as $(\ell,\mode)$-architecture which is mentioned in \secref{sec::numerics}. 

More specifically, 
let us take $\ell = 2$ as an example: an $(\ell,\mode)=(2,\mode)$ architecture for a generic ansatz for $\dynb:\dom \to\Rd$ refers to the following parameterization
\begin{align*}
\dynb(x) = W_2 \AF(W_1 x + b_1) + b_2, \qquad \text{(generic ansatz)},
\end{align*}
where weights $W_1\in\Real^\dimn\times\Real^\mode$, $W_2\in\Real^\mode\times\Real^\dimn$ and bias vectors $b_1\in\Real^\mode$, $b_2\in\Real^\dimn$. 

For the gradient-form ansatz, as we essentially need to parameterize a potential $V:\dom\to\Real$, the $(\ell,\mode)$-architecture for $V$ refers to the following choice when $\ell = 2$,
\begin{align*}
    V(x) = W_2 \AF(W_1 x + b_1)
\end{align*}
where $W_2 \in \Real^\mode\times\Real$, $W_1\in\Real^\dimn\times\Real^\mode$, $b_1\in\Real^\mode$. The bias vector in the output layer is chosen as zero (i.e., $b_2 = 0$) because $b_2$ is a redundant parameter after taking the gradient.
The dynamics $\dynb$ in the gradient form refers to 
$\dynb = \nabla V$.

\newparagraphno{Initialization and trial repetition}

When we parameterize the flow via neural networks, weights and biases in the neural network are randomly generated. 
Therefore, we consider two (or three) independent trials (associated with different random initializations of $\dynb$) for the same neural network architecture characterized by a pair $(\ell,m)$.

\newparagraphno{Optimization algorithm}
We use the SGD algorithm to optimize parameters for $\dynb$.
The {Armijo line search algorithm} (see \eg{}, \cite{armijo_minimization_1966,vaswani_painless_2019}) can used to find the learning rate;
in practice, we notice that solving $\dot{\vartheta} = - \frac{\nabla_\vartheta  \newmsecfinarg{\tm}{\tp}(\dynb_{\vartheta})}{\abs{\nabla_\vartheta  \newmsecfinarg{\tm}{\tp}(\dynb_{\vartheta})}}$ with a relatively large learning rate also works well for numerical examples considered in \secref{sec::numerical_experiments}, and it is used for training in \secref{sec::numerical_experiments}; $\vartheta$ are trainable parameters and $\dynb_{\vartheta}$ is the flow parameterized by $\vartheta$.
The way to approximate the loss function $\newmsecfinarg{\tm}{\tp}(\dynb) \equiv \ee_{0} \big[|\newafin|^2\big]$ (see \secref{sec::variational}) and in particular its gradient with respect to parameters in $\dynb$ will be explained in Appendix \ref{subsec::integral_method} and \ref{subsec::ode_method}.

\subsection{An integration-based forward propagation method to compute the estimator and its gradient}
\label{subsec::integral_method}

To estimate $\newafinarg{\tm}{\tp}(x)$ in \eqref{eq::afin} with $\tp=\tm+1$ and $\tm\in [-1,0]$, the most straightforward approach is to compute $\testfuncarg{k}{t}{x}$ at time grid points and then employ some integration scheme like the trapezoidal quadrature method. 

More specifically, we first discretize the time interval $[-1,1]$ by $2N_t + 1$ equally-spaced points $t_m := \frac{1}{N_t} m$ where $-N_t \le m\le N_t$. Then we use classical ODE integration schemes (we use RK4) to propagate the dynamics $\state{t}{x}$ both forward and backward in time to estimate the following quantities:
\begin{align*}
	\hat{U}_{0,m} = U_0\big(\state{t_m}{x}\big), \quad \hat{U}_{1,m} = U_1\big(\state{t_m}{x}\big), \quad \hat{D}_m = \div\dynb\big(\state{t_m}{x}\big).
\end{align*}
Then 
\begin{align*}
	\testfuncarg{k}{t_m}{x} \approx e^{-\hat{U}_{k,m} + \frac{1}{N_t}\text{sign}(m) \mathcal{Q}_{\text{Tr}}(\hat{D}_0, \hat{D}_1, \cdots, \hat{D}_m)},
\end{align*}
where 
\begin{align*}
	\mathcal{Q}_{\text{Tr}}(\hat{D}_0, \hat{D}_1, \cdots, \hat{D}_m) = \sum_{j=0}^{m} \hat{D}_j - \frac{1}{2}\big(\hat{D}_0 + \hat{D}_m\big)
\end{align*}
is the trapezoidal quadrature scheme. Then similarly, we can use the trapezoidal quadrature scheme to estimate $\int_{t_j - \tp}^{t_j - \tm} \testfuncarg{0}{s}{x}\ud s$ given the values $\testfuncarg{0}{t_m}{x}$ for $-N_t\le m\le N_t$ and the time $t_j$ with $t_j \in [\tm,\tp]$. 
After we have approximated values for $\frac{\testfuncarg{1}{t_j}{x}}{\int_{t_j - \tp}^{t_j - \tm} \testfuncarg{0}{s}{x}\ud s}$, the trapezoidal quadrature scheme is utilized again to approximate $\newafinarg{\tm}{\tp}(x)$. 
The computational cost is mostly dominated by the ODE integration, \eg{}, propagating $\state{t}{x}$ or evaluating $\div\dynb(\state{s}{x})$ in general. 
Therefore, this straightforward integration-based method has linear computational cost with respect to $N_t$ and is thus expected to be optimal.
Using the same principle, we can estimate the gradient of the second moment with respect to parameters in the dynamics (see \eqref{eqn::pert_M_2}) during the training.

\subsection{{An ODE-based} forward propagation method to compute the estimator and its gradient for $\tm = 0$, $\tp=1$}
\label{subsec::ode_method}
To simplify notations, let us introduce 
\begin{align}
	\label{eqn::ab_finite}
	\bfin_t(x) := \int_{t-\tp}^{t-\tm} \testfuncarg{0}{s}{x} \ud s,\ \text{and thus, }\ 
    \newafin(x) 
    \myeq{\eqref{eq::afin}} \int_{\tm}^{\tp} \frac{\testfuncarg{1}{t}{x}}{\bfin_t(x)} \ud t,
\end{align}
which implicitly depends on $\dynb$.
Let us denote 
$\alpha_t(x):= \int_{\tm}^{t} \frac{\testfuncarg{1}{s}{x}}{\bfin_s(x)}\ud s$. 

\subsubsection{ODE dynamics to compute the estimator}
Then the estimator $\newafinarg{0}{1}$ is simply $\alpha_{1}(x)$. To compute $\newafinarg{0}{1}$, we simply need to run the following ODE:
\begin{align}
	\label{eqn::T0_estimator}
	\left\{
	\begin{aligned}
	\td \alpha_t(x) &= e^{-U_1\big(\state{t}{x}\big)}\jacoarg{t}{x}/\bfin_t(x), \quad &\alpha_0(x) = 0,\\
	\td \bfin_t(x) &= e^{-U_0\big(\state{t}{x}\big)} \jacoarg{t}{x} - e^{-U_0\big(\statesup{t}{x}{\lag}\big)}\jacoargsup{t}{x}{\lag}, \quad & \bfin_0(x) = \bfin_1^R(x), \\
	\td\state{t}{x} &= \dynb\big(\state{t}{x}\big),\quad &\state{0}{x} = x,\\
	\td\statesup{t}{x}{\lag} &= \dynb\big(\statesup{t}{x}{\lag}\big),\quad & \statesup{0}{x}{\lag} = \statesup{1}{x}{R},\\
	\td \jacoarg{t}{x} &= \div\dynb\big(\state{t}{x}\big) \jacoarg{t}{x},\quad & \jacoarg{0}{x} = 1,\\
	\td \jacoargsup{t}{x}{\lag} &= \div\dynb\big(\statesup{t}{x}{\lag}\big) \jacoargsup{t}{x}{\lag}, \quad & \jacoargsup{0}{x}{\lag} = \jacoargsup{1}{x}{R},
	\end{aligned}\right.
\end{align}
where $\statesup{t}{x}{\lag} := \state{t-1}{x}$ and $\jacoargsup{t}{x}{\lag} := \jacoarg{t-1}{x}$. Moreover, in order to obtain the initial condition, we shall run the dynamics backward in time, i.e., simulating the following ODE on $[0,1]$, 
\begin{align}
	\label{eqn::aug_T0_estimator}
\left\{\begin{aligned}
	\td \statesup{t}{x}{R} &= -\dynb\big(\statesup{t}{x}{R}\big), \quad & \statesup{0}{x}{R} = x,\\
	\td \jacoargsup{t}{x}{R} &= -\div\dynb\big(\statesup{t}{x}{R}\big)\jacoargsup{t}{x}{R}, \quad &\jacoargsup{0}{x}{R} = 1,\\
	\td \bfin_t^R(x) &= e^{-U_0\big(\statesup{t}{x}{R}\big)} \jacoargsup{t}{x}{R}, \quad &\bfin_0^R(x) = 0,
\end{aligned}\right.
\end{align}
where the superscript $R$ means the reversed process. 
As a remark, the auxiliary backward ODE \eqref{eqn::aug_T0_estimator} has dimension $\dimn + 2$ and the ODE \eqref{eqn::T0_estimator} has dimension $2\dimn + 4$. 

\subsubsection{ODE dynamics to compute the gradient}

Denote $\vartheta$ as a vector containing parameters in $\dynb$ and 
let $N_p$ be the number of parameters. In what follows, we also vectorize all quantities involving $\vartheta$, \eg{}, for each parameter $\vartheta_j$, there is a corresponding $\deristatesup{t}{x}{(j)}$ in \eqref{eqn::deriX} and we simply use the notation $\deristate{t}{x}$ to be a matrix whose $j^{\text{th}}$ column is $\deristatesup{t}{x}{j}$ to save notations; the same convention applies to other quantities.

Next, we shall similarly re-write the expression $\nabla_\vartheta \newafinarg{0}{1}$ (implicitly given in $\nabla_{\vartheta}\newmsecfinarg{0}{1}(\dynb)$ \eqref{eqn::pert_M_2}) in terms of outputs from an ODE.
Using the same idea, let us introduce
\begin{align*}
	\left\{\begin{aligned}
	    \stated{t}{x} &:= \int_{0}^{t} \frac{\deritestfuncarg{1}{r}{x} \int_{r-1}^{r} \testfuncarg{0}{s}{x} \ud s - \testfuncarg{1}{r}{x} \int_{r-1}^{r} \deritestfuncarg{0}{s}{x} \ud s}{\Big(\int_{r-1}^{r} \testfuncarg{0}{s}{x} \ud s\Big)^2} \ud r\\
		\stateg{t}{x} &:= \int_{t-1}^{t} \deritestfuncarg{0}{s}{x}\ud s,\\
		\stateL{t}{x} &:= \int_{0}^{t} \nabla_{\vartheta}(\nabla_x\cdot\dynb)\big(\state{s}{x}\big)\ud s,\\
		\stateH{t}{x} &:= \int_{0}^{t} \Big(\nabla(\div\dynb)(\state{s}{x})\Big)^{T} \deristate{s}{x} \ud s, 
	\end{aligned}\right.
\end{align*} 
and then we can rewrite quantities involved inside $\nabla_{\vartheta}\newmsecfinarg{0}{1}(\dynb)$ as follows:
\begin{align*}
	&\text{the evolution of } \state{t}{x}, \statesup{t}{x}{\lag}, \jacoarg{t}{x},\ \jacoargsup{t}{x}{\lag},\ \bfin_t(x),\ \alpha_t(x)  \text{ in } \eqref{eqn::T0_estimator}, \hspace{-2em} &\\
	\td \stated{t}{x} &= \left(\begin{aligned}
		& \frac{e^{-U_1\big(\state{t}{x}\big)}\jacoarg{t}{x}\big(-\nabla U_1\big(\state{t}{x}\big)^T \deristate{t}{x} + \stateH{t}{x}+ \stateL{t}{x}\big)}{\bfin_t(x)} \\
		&- \frac{e^{-U_1\big(\state{t}{x}\big)}\jacoarg{t}{x} \stateg{t}{x}}{\big(\bfin_t(x)\big)^2}
		\end{aligned}\right), \hspace{-2.5em} & \stated{0}{x} = \vectorzero_{N_p},\\
	\td \stateg{t}{x} &= \left(\begin{aligned}
		e^{-U_0\big(\state{t}{x}\big)}\jacoarg{t}{x}\big(&-\nabla U_0\big(\state{t}{x}\big)^{T} \deristate{t}{x} \\
		& + \stateH{t}{x} + \stateL{t}{x} \big)\\
		- e^{-U_0\big(\statesup{t}{x}{\lag}\big)}\jacoargsup{t}{x}{\lag}
		\big(&-\nabla U_0\big(\statesup{t}{x}{\lag}\big)^{T} \deristatesup{t}{x}{\lag} \\
		&+ \stateHsup{t}{x}{\lag} + \stateLsup{t}{x}{\lag}\big)
	\end{aligned}\right), & \stateg{0}{x} = \stategsup{1}{x}{R},\\
	\td \stateH{t}{x} &= \Big(\nabla(\div\dynb)\big(\state{t}{x}\big)\Big)^{T} \deristate{t}{x},  & \stateH{0}{x} = \vectorzero_{N_p},\\
	\td \stateHsup{t}{x}{\lag} &= \Big(\nabla(\div\dynb)\big(\statesup{t}{x}{\lag}\big)\Big)^{T} \deristatesup{t}{x}{\lag},& \stateHsup{0}{x}{\lag} = \stateHsup{1}{x}{R},\\
	\td \stateL{t}{x} &= \nabla_{\vartheta}(\nabla_x\cdot\dynb)\big(\state{t}{x}\big), & \stateL{0}{x} = \vectorzero_{N_p},\\
	\td \stateLsup{t}{x}{\lag} &= \nabla_{\vartheta}(\nabla_x\cdot\dynb)\big(\statesup{t}{x}{\lag}\big), & \stateLsup{0}{x}{\lag} = \stateLsup{1}{x}{R},\\
	\td \deristate{t}{x} &= \nabla\dynb\big(\state{t}{x}\big) \deristate{t}{x} + \nabla_{\vartheta}\dynb\big(\state{t}{x}\big), &\deristate{0}{x} = \vectorzero_{\dimn\times N_p},\\
	\td \deristatesup{t}{x}{\lag} &= \nabla\dynb\big(\statesup{t}{x}{\lag}\big) \deristatesup{t}{x}{\lag} + \nabla_{\vartheta}\dynb\big(\statesup{t}{x}{\lag}\big), &\deristatesup{0}{x}{\lag} = \deristatesup{1}{x}{R},
\end{align*}
where the initial conditions come from solving an auxiliary backward equation on $[0,1]$, given as follows:
\begin{align*}
	&\text{the evolution of } \statesup{t}{x}{R}, \jacoargsup{t}{x}{R}, \bfin^R_t(x) \text{ in } \eqref{eqn::aug_T0_estimator},  & \\ 
	\td \stategsup{t}{x}{R} &= e^{-U_0(\statesup{t}{x}{R})}\jacoargsup{t}{x}{R}\left(\begin{aligned} -\nabla U_0\big(\statesup{t}{x}{R}\big)^{T} \deristatesup{t}{x}{R}\\
	+ \stateHsup{t}{x}{R} + \stateLsup{t}{x}{R}\end{aligned} \right), & \stategsup{0}{x}{R} = \vectorzero_{N_p},\\
	\td \stateHsup{t}{x}{R} &= -\Big(\nabla(\div\dynb)\big(\statesup{t}{x}{R}\big)\Big)^T \deristatesup{t}{x}{R}, & \stateHsup{0}{x}{R} = \vectorzero_{N_p},\\ 
	\td \stateLsup{t}{x}{R} &= -\nabla_\vartheta(\nabla_x \cdot \dynb)\big(\statesup{t}{x}{R}\big), & \stateLsup{0}{x}{R} = \vectorzero_{N_p},\\
	\td \deristatesup{t}{x}{R} &= -\nabla\dynb\big(\statesup{t}{x}{R}\big) \deristatesup{t}{x}{R} - \nabla_{\vartheta}\dynb\big(\statesup{t}{x}{R}\big), &\deristatesup{0}{x}{R} = \vectorzero_{\dimn\times N_p}.
\end{align*}
As a remark, we combine the auxiliary states $\deristate{t}{x}$ for all parameters together, so that all three terms $\deristate{t}{x}$, $\deristatesup{t}{x}{\lag}$, $\deristatesup{t}{x}{R}$ have dimension $\dimn\times N_p$. Below is a table that summarizes the dimension of all components. 

\begin{table}[h!]
	\centering
	\caption{This table summarizes all variables and their corresponding dimensions.}
\begin{tabular}{c|c}
	\toprule
	notations & dimensions \\
	\toprule
	$\jacoarg{t}{x}$, $\jacoargsup{t}{x}{\lag}$, $\bfin_t(x)$, $\alpha_t(x)$, $\jacoargsup{t}{x}{R}$, $\bfin_t^R(x)$ & $1$\\
	$\state{t}{x}$, $\statesup{t}{x}{\lag}$, $\statesup{t}{x}{R}$ & $\dimn$\\
	$\stated{t}{x}$, $\stateg{t}{x}$, $\stateH{t}{x}$, $\stateHsup{t}{x}{\lag}$, $\stateL{t}{x}$, $\stateLsup{t}{x}{\lag}$, $\stategsup{t}{x}{R}$, $\stateLsup{t}{x}{R}$, $\stateHsup{t}{x}{R}$ & $N_p$\\ 
	$\deristate{t}{x}$, $\deristatesup{t}{x}{\lag}$, $\deristatesup{t}{x}{R}$ & $\dimn\times N_p$\\
	\bottomrule
\end{tabular}
\end{table}
In the end, the outputs we need are 
\begin{align*}
	\newafinarg{0}{1}(x) = \alpha_1(x),\quad \nabla_{\vartheta} \newafinarg{0}{1}(x) =\stated{1}{x}, \quad \text{ so that } \nabla_{\vartheta}\newmsecfinarg{0}{1}(\dynb) = 2 \ee_{\rho_0} \big[\alpha_1(x) \stated{1}{x}\big].
\end{align*}

\subsection{A discussion on the backward propagation for differentiation}

In the formulas \eqref{eqn::pert_M_2} and \eqref{eqn::deri_test_fun}, a crucial step is to evaluate 
\begin{align*}
	\Phi(t,x) := \innerbig{\nabla \varphi(\state{t}{x})}{\deristate{t}{x}}, \quad \varphi = U_k, \text{ or } \div\dynb.
\end{align*}
for time $t\in [-T,T]$ where $T = \tp - \tm$. For any fixed $t$, 
this value can be efficiently measured by the adjoint equation with backward propagation proposed in \cite{chen_neural_2018}. For instance, let us consider $t>0$ and 
\begin{align*}
	\Phi(t,x)\ &\myeq{\eqref{eqn::Zt_sol}}\ \innerBig{\nabla \varphi\big(\state{t}{x}\big)}{\int_{0}^{t} \corrz{t}{s}(x)\ \delta \dynb\big(\state{s}{x}\big) \ud s}\\
	&=\ \int_{0}^{t}\innerBig{ \underbrace{\corrz{t}{s}(x)^{T} \nabla \varphi(\state{t}{x})}_{=:\mathscr{A}(s,x)}}{\delta\dynb\big(\state{s}{x}\big)} \ud s
\end{align*}
Moreover, it can be easily verified by the definition \eqref{eqn::corrz} that for $s\in [0,t]$, 
\begin{align*}
	\frac{\ud}{\ud s} \mathscr{A}(s,x) &= -\nabla\dynb^{T}\big(\state{s}{x}\big) \mathscr{A}(s,x), \quad \mathscr{A}(t,x) = \nabla\varphi\big(\state{t}{x}\big).
\end{align*}
The above two equations are Eqs. (4) \& (5) in \cite{chen_neural_2018}. This can eliminate the need to simulate $\deristate{t}{x}$, whose memory cost scales like $\mathcal{O}(\dimn N_p)$. 
However, in the \ftneis{} scheme, we need to estimate $\Phi(t,x)$ not only for a fixed $t$ but also for $t\in [-1,1]$. Therefore, the computational cost scales like $\mathcal{O}(N_t^2)$ where $N_t$ is the number of time-discretization (or say the depth of the flow map in the normalizing flow context).
As a comparison, the forward propagation method uses more memory but could be computationally cheaper as it only needs to visit the whole trajectory for $\mathcal{O}(1)$ times (see the table below). 

\begin{table}[h!]
\centering
\caption{A comparison between the forward and backward propagation in computing the derivative of $\newmsecfinarg{0}{1}(\dynb)$ with respect to parameters; see \eqref{eqn::pert_M_2} and \eqref{eqn::deri_test_fun}. The notation $N_p$ is the number of parameters for the training and $N_t$ is the number of grid points in time-discretization.}
\begin{tabular}{c|ccc}
	\toprule
	Method & Key difference & Memory & Computational cost \\
	\toprule
	Backward propagation & simulate $\mathscr{A}(s,x)$ instead of $\deristate{s}{x}$ & $\mathcal{O}(N_p)$ & $\mathcal{O}(N_t^2)$ \\
	Forward propagation & simulate $\deristate{s}{x}$ & $\mathcal{O}(\dimn N_p)$ & $\mathcal{O}(N_t)$ \\
	\bottomrule
\end{tabular}

\end{table}

\subsection{A discussion on query complexity}
The query complexities to estimate $\ratio$ by both AIS and NEIS are summarized in the Table~\ref{table::query} below:
\begin{table}[h!]
	\centering
	\caption{A summary of query complexities to $U_1$ when estimating $\ratio$ for various methods. $K$ is the transition step in AIS; $N$ is the number of time steps ($\Delta t=1/N$) for either integration-based or ODE-based discretization; 
	$s$ is the order of ODE integration schemes. The integration-based method does not depend on $s$ as the query to $U_1$ is achieved during trapezoidal integral (see Appendix~\ref{subsec::integral_method}).
	See the AIS-$K$ algorithm and relevant analysis in Appendix~\ref{subsec::ais}}

\label{table::query}

	\begin{tabular}{p{0.1\textwidth}p{0.1\textwidth}p{0.3\textwidth}p{0.35\textwidth}}
	\toprule
	 & AIS-$K$ & ODE-based discretization \eqref{eqn::T0_estimator} & integration-based discretization \par (see Appendix~\ref{subsec::integral_method})\\
	 \toprule
	 $U_1$ & $2K$ & $sN$ & $2N$ \\
	 $\nabla U_1$ & $2K$ & 0 & 0 \\
	 \bottomrule
	\end{tabular}
\end{table}

\newpage
\section{More training and comparisons results}
\label{subsec::train_figure}

\begin{table}[h!]
	\caption{Comparison of NEIS with AIS: we include training and estimation query costs for AIS and NEIS, as well as statistics for 10 independent estimates of $\ratio$; 
		for each method, we first set the query cost to obtain one approximated value of $\ratio$ and then accordingly choose the sample size; we repeatedly estimate $\ratio$ $10$ times and report the mean and std of these $10$ estimates (in the form of mean $\pm$ std).
		For NEIS, we consider multiple random initializations for training and therefore, there are multiple estimates about $\ratio$ in the last row in each panel (either 2 or 3 estimates).
		Estimates in AIS, however, simply refer to estimating results from independent experiments.
		The exact value $\ratio = 1$ and $1 \text{ MB} = 10^6$;
		the best result using AIS is colored in blue and the best result using NEIS is colored in green.
		The layer number $\ell=2$ for both generic and gradient ansatz.
	}
	
	\label{table::eg}
	\begin{tabular}{p{0.08\textwidth}p{0.11\textwidth}p{0.11\textwidth}p{0.11\textwidth}p{0.11\textwidth}p{0.11\textwidth}p{0.11\textwidth}}
		\multicolumn{7}{l}{\ul{\emph{Asymmetric 2-mode Gaussian mixture (2D):}}} \smallskip \\
		& AIS-10 & AIS-100 & NEIS\par (Generic, $\mode=20$) & NEIS\par (Generic, $\mode=30$) & NEIS\par (Grad, $\mode=20$) & NEIS \par (Grad, $\mode=30$) \\
		\hline
		\multirow{2}{*}{\shortstack[l]{training\\ cost}} & \multirow{2}{*}{N/A} & \multirow{2}{*}{N/A} &  \multicolumn{4}{c}{$U_1$: 2 MB}\\
		& & & \multicolumn{4}{c}{$\nabla U_1$: 2.1 MB} \\
		\hline
		\multirow{2}{*}{\shortstack[l]{cost per\\ estimate}} & \multicolumn{2}{c}{$U_1$: 6.2 MB} & \multicolumn{4}{c}{$U_1$: 4.2 MB}\\
		& \multicolumn{2}{c}{$\nabla U_1$: 6.2 MB} & \multicolumn{4}{c}{$\nabla U_1$: 0}\\ 
		\hline
		\multirow{2}{*}{\shortstack[l]{10\\ estimates}} & $0.932 \pm 0.466$ & \bestcellais $1.033 \pm 0.071$ & $1.011 \pm 0.013$ & $0.997 \pm 0.016$ & \bestcell $0.997 \pm 0.007$  & $0.999 \pm 0.012$ \\ \hhline{~*6-}
		& $1.068 \pm 0.499$ & $1.031 \pm 0.166$ & $0.997 \pm 0.009$ & $1.001 \pm 0.011$ & $1.000 \pm 0.008$ & \bestcell $1.003 \pm 0.007$ \\
	\end{tabular}
	
	\bigskip
	
	\begin{tabular}{p{0.15\textwidth}p{0.15\textwidth}p{0.15\textwidth}p{0.15\textwidth}p{0.15\textwidth}}
		\multicolumn{5}{l}{\ul{\emph{Symmetric 4-mode Gaussian mixture (10D):}}} \smallskip \\
		& AIS-10 & AIS-100  & NEIS\par (Grad, $\mode=30$) & NEIS \par (Grad, $\mode=40$) \\
		\hline
		\multirow{2}{*}{\shortstack[l]{training cost}} & \multirow{2}{*}{N/A} & \multirow{2}{*}{N/A} &  \multicolumn{2}{c}{$U_1$: 11.5 MB}\\
		& & & \multicolumn{2}{c}{$\nabla U_1$: 12.8 MB} \\
		\hline
		\multirow{2}{*}{\shortstack[l]{\\ cost per\\ estimate}} & \multicolumn{2}{c}{$U_1$:  89.4 MB} & \multicolumn{2}{c}{$U_1$: 76.6$\sim$76.8 MB}\\
		& \multicolumn{2}{c}{$\nabla U_1$:  89.4 MB} & \multicolumn{2}{c}{$\nabla U_1$: 0}\\ 
		\hline
		\multirow{2}{*}{\shortstack[l]{\medskip 10 estimates}} & $0.987 \pm 0.074$ & $1.005 \pm 0.019$ & \bestcell $0.998 \pm 0.004$ & $0.998 \pm 0.011$ \\ \hhline{~*4-}
		& $0.997 \pm 0.052$ & \bestcellais $1.000 \pm 0.014$ & $1.001 \pm 0.005$ & $0.996 \pm 0.008$ \\
	\end{tabular}
	
	\bigskip
	
	\begin{tabular}{p{0.15\textwidth}p{0.15\textwidth}p{0.15\textwidth}p{0.15\textwidth}p{0.22\textwidth}}
		\multicolumn{5}{l}{\ul{\emph{Funnel distribution (10D):}}}\\
		& AIS-10 & AIS-100  & NEIS\par (linear \eqref{eqn::ln}) & NEIS \par (two-parametric  \eqref{eqn::ln_ansatz}) \\
		\hline
		\multirow{2}{*}{\shortstack[l]{training cost}} & \multirow{2}{*}{N/A} & \multirow{2}{*}{N/A} &  $U_1$: 40.4 MB & $U_1$: 20.2 MB \\
		& & & $\nabla U_1$: 48.8 MB &  $\nabla U_1$: 20.2 MB \\
		\hline
		\multirow{2}{*}{\shortstack[l]{\\ cost per\\ estimate}} & \multicolumn{2}{c}{$U_1$: 341.3 MB} & $U_1$:  293 MB & $U_1$: 121.2 MB \\
		& \multicolumn{2}{c}{$\nabla U_1$: 341.3 MB} & $\nabla U_1$: 0 & $\nabla U_1$: 0\\  \hhline{*5-}
		\multirow{3}{*}{\shortstack[l]{\medskip 10 estimates}} & \multirow{3}{*}{$0.681 \pm 0.022$} & \bestcellais & $0.816 \pm 0.089$ &  \bestcell\\ \hhline{~~~-~}
		& & \bestcellais & $0.847 \pm 0.054$ & \bestcell\\ \hhline{~~~-~}
		& & \multirow{-3}{*}{\bestcellais $0.767 \pm 0.035$} & \bestcell $0.890 \pm 0.171$ & \multirow{-3}{*}{\bestcell $0.984 \pm 0.036$}\\
	\end{tabular}
\end{table}

\afterpage{
\begin{figure}[h!]
	\centering
	\includegraphics[width=0.9\textwidth]{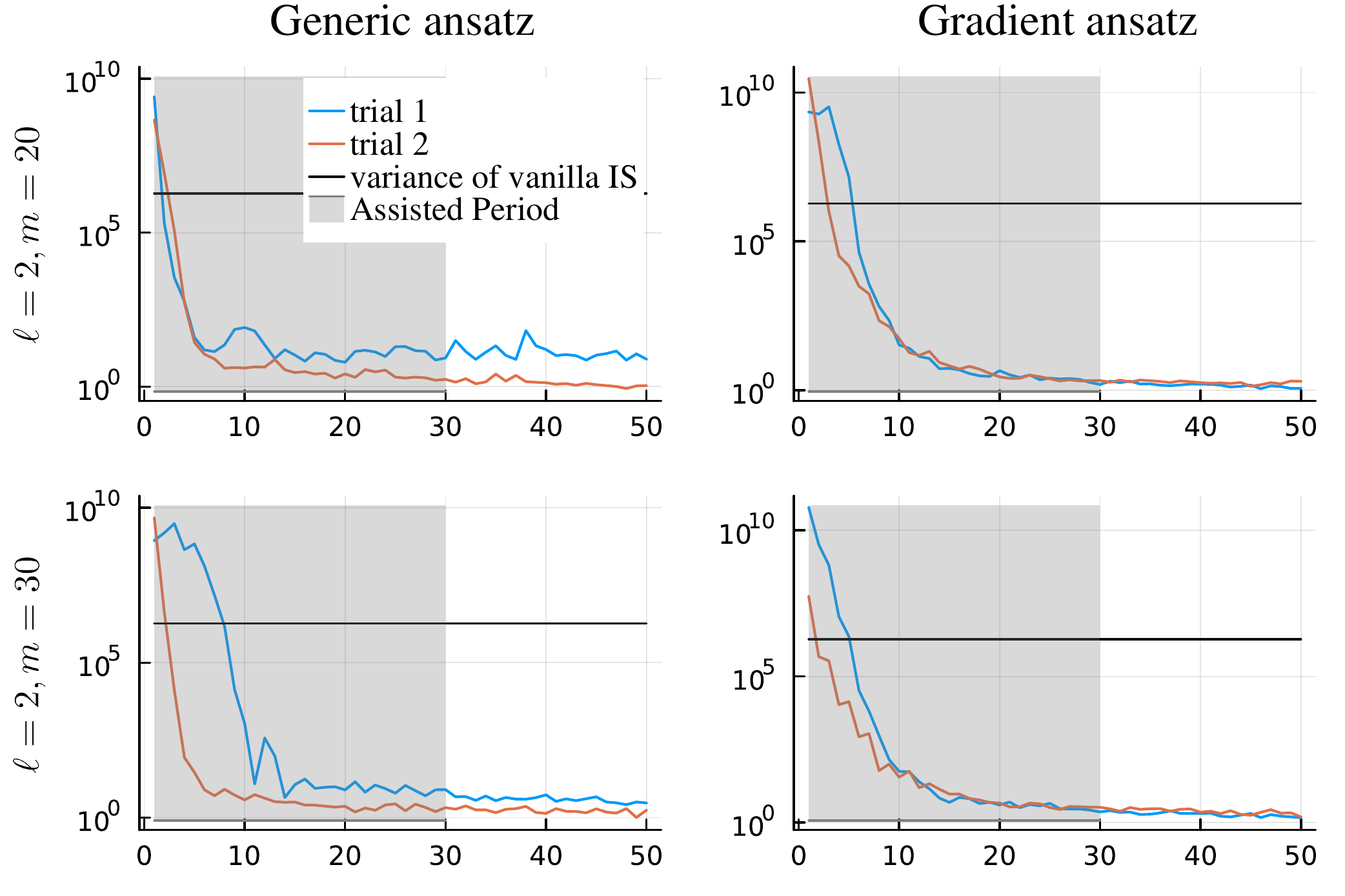}
	\caption{Asymmetric $2$-mode Gaussian mixture in 2D: variance against SGD steps for the asymmetric $2$-mode Gaussian mixture \eqref{eqn::eg_two_mode} in two trial runs.
		\eqref{eq::afin} was discretized with time step $\frac{1}{50}$.
		A mini-batch of sample size $200$ was used throughout the training. The biasing parameters used in~\eqref{eqn::bias} were $\upsilon=0.6$, $c=0.1$ and $\varsigma=1$ (see \eqref{eqn::bias} and the follow-up paragraph).}
	\label{fig::train::1}
\end{figure}

\begin{figure}[h!]
	\centering
	\includegraphics[width=0.45\textwidth]{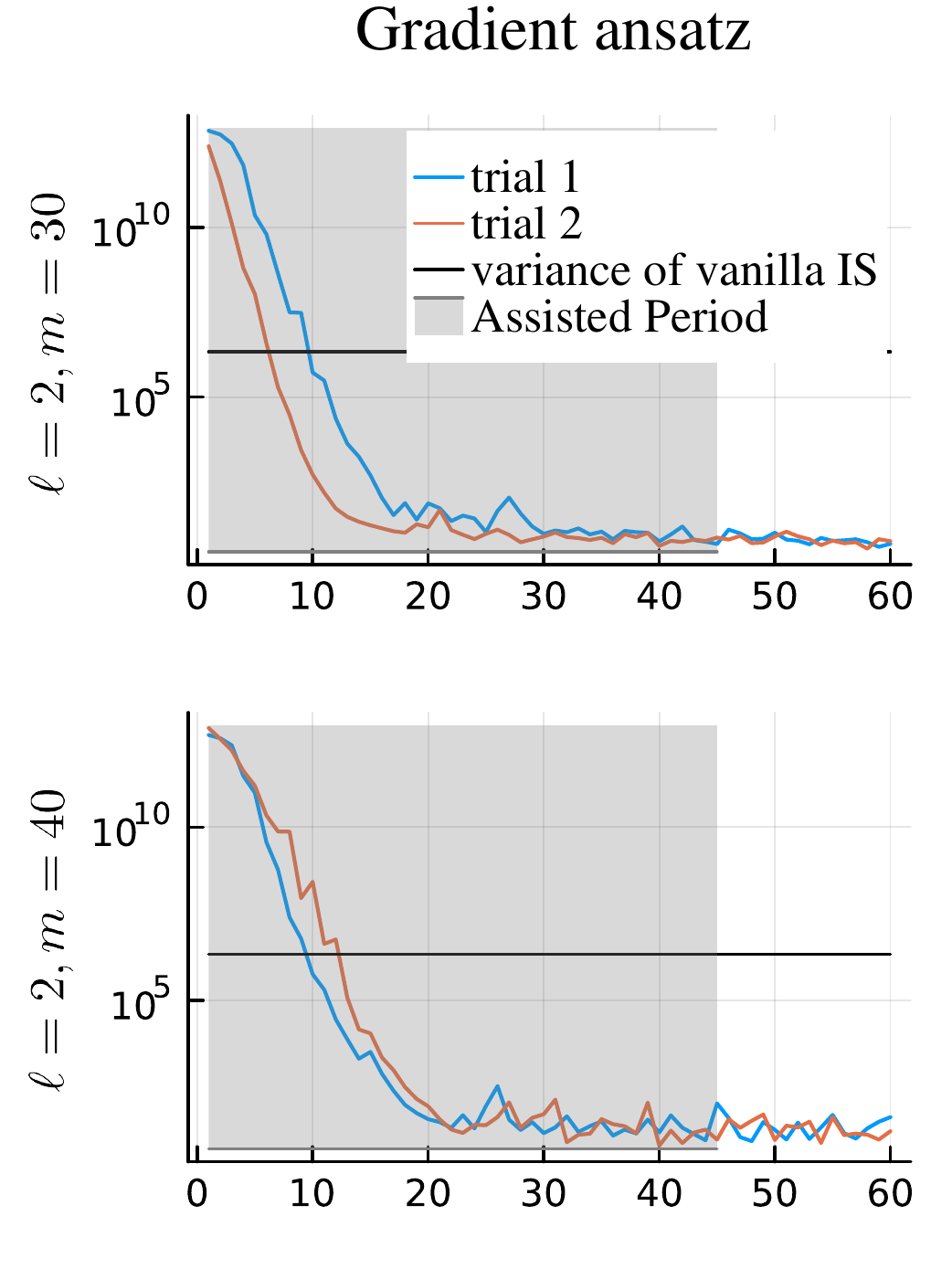}
	\caption{Symmetric $4$-mode Gaussian mixture in 10D:  
		variance as a function of SGD training step in 2 trial runs, with gradient form ansatz for $\dynb$. 
		\eqref{eq::afin} was discretized with time step $\frac{1}{60}$.
		The sample size of the mini-batch was $800$ during the training. The biasing parameters were $\upsilon = 0.75$, $c=0.3$ and $\varsigma=1$ (see \eqref{eqn::bias} and the follow-up paragraph).
	}
	\label{fig::train::2}
\end{figure}
\clearpage
}

\afterpage{
\begin{figure}[h!]
	\centering
	\includegraphics[width=0.55\textwidth]{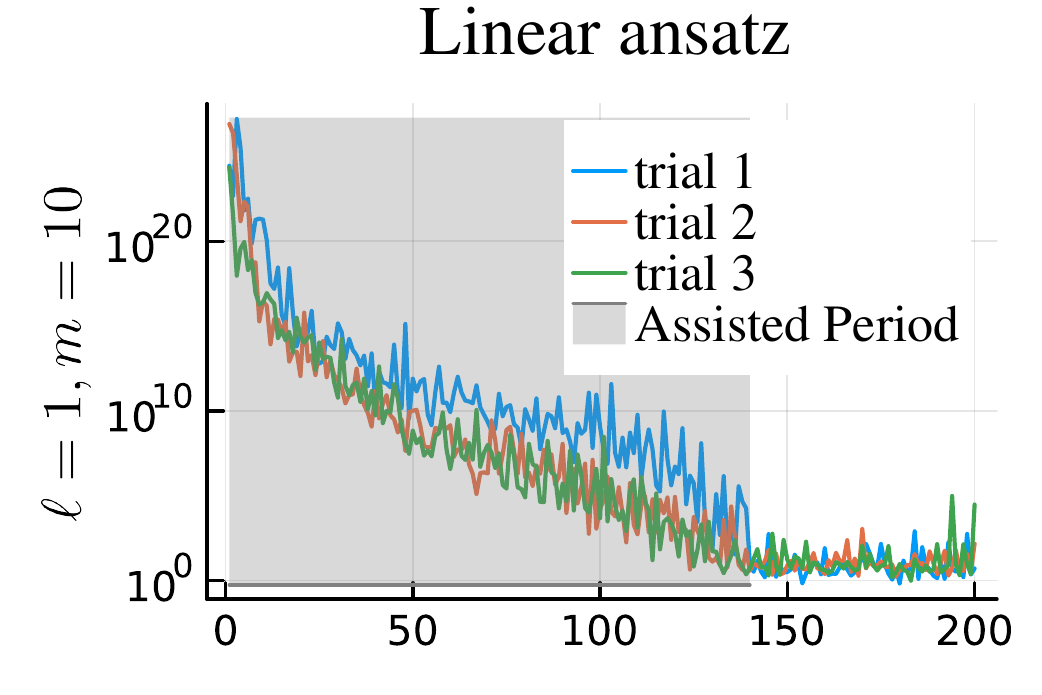}
	\caption{Funnel distribution in 10D: variance as a function of SGD training step in 3 trial runs.
	We consider the finite-time NEIS scheme with $\tm=-\frac{1}{2}$ and use the linear ansatz \eqref{eqn::ln}.
	During the training, the integral inside $\newafinarg{-\frac{1}{2}}{\frac{1}{2}}$ was discretized with time step $\frac{1}{100}$, and
		the mini-batch sample size was $10^3$. 
	 The biasing parameters were $\upsilon = 0.7$, $c=0.3$ and $\varsigma=1$ (see \eqref{eqn::bias} and the follow-up paragraph).
	}
	\label{fig::train::3::ln}
\end{figure}

\begin{figure}[ht!]
	\centering
	\includegraphics[width=\textwidth]{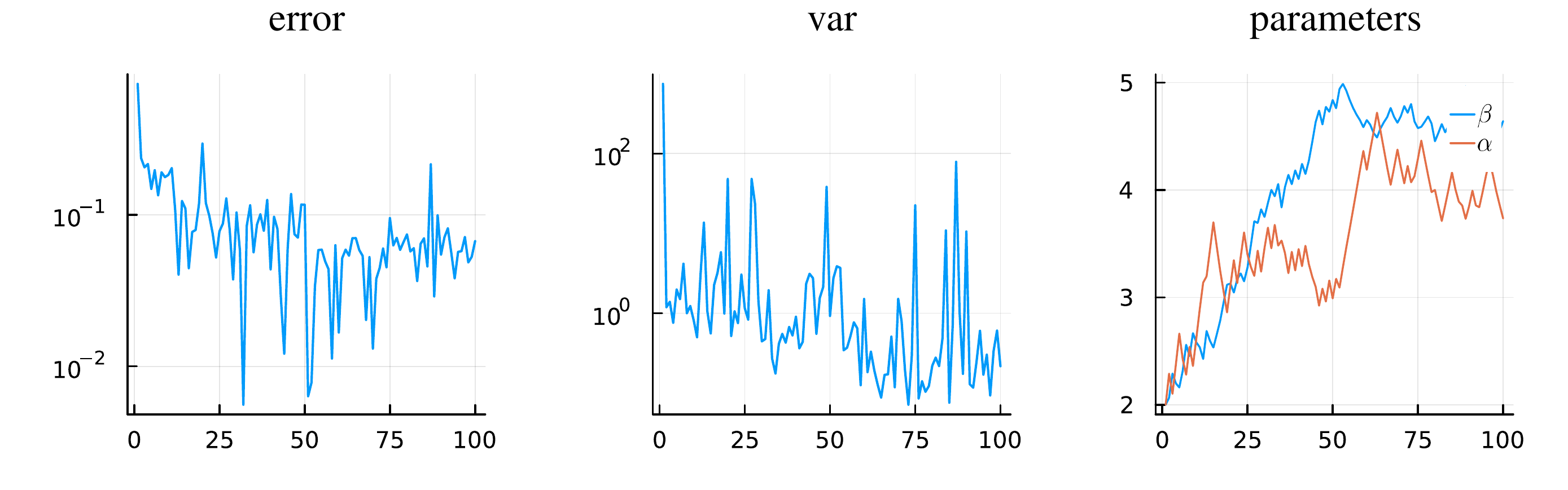}
	\caption{Funnel distribution in 10D: we consider the finite-time NEIS scheme with $\tm=-\frac{1}{2}$ and use the two-parametric ansatz \eqref{eqn::ln_ansatz}.
		During the training, the integral inside $\newafinarg{-\frac{1}{2}}{\frac{1}{2}}$ was discretized with time step $\frac{1}{100}$, and
		the mini-batch sample size was $10^3$.
		The direct training method was used.
	}
	\label{fig::train::3::ln_2}
\end{figure}
\clearpage
}

\afterpage{
\begin{figure}[h!]
	\centering
	\subfigure[Generic ansatz, $\ell=2$, $\mode=20$]{\includegraphics[width=0.9\textwidth]{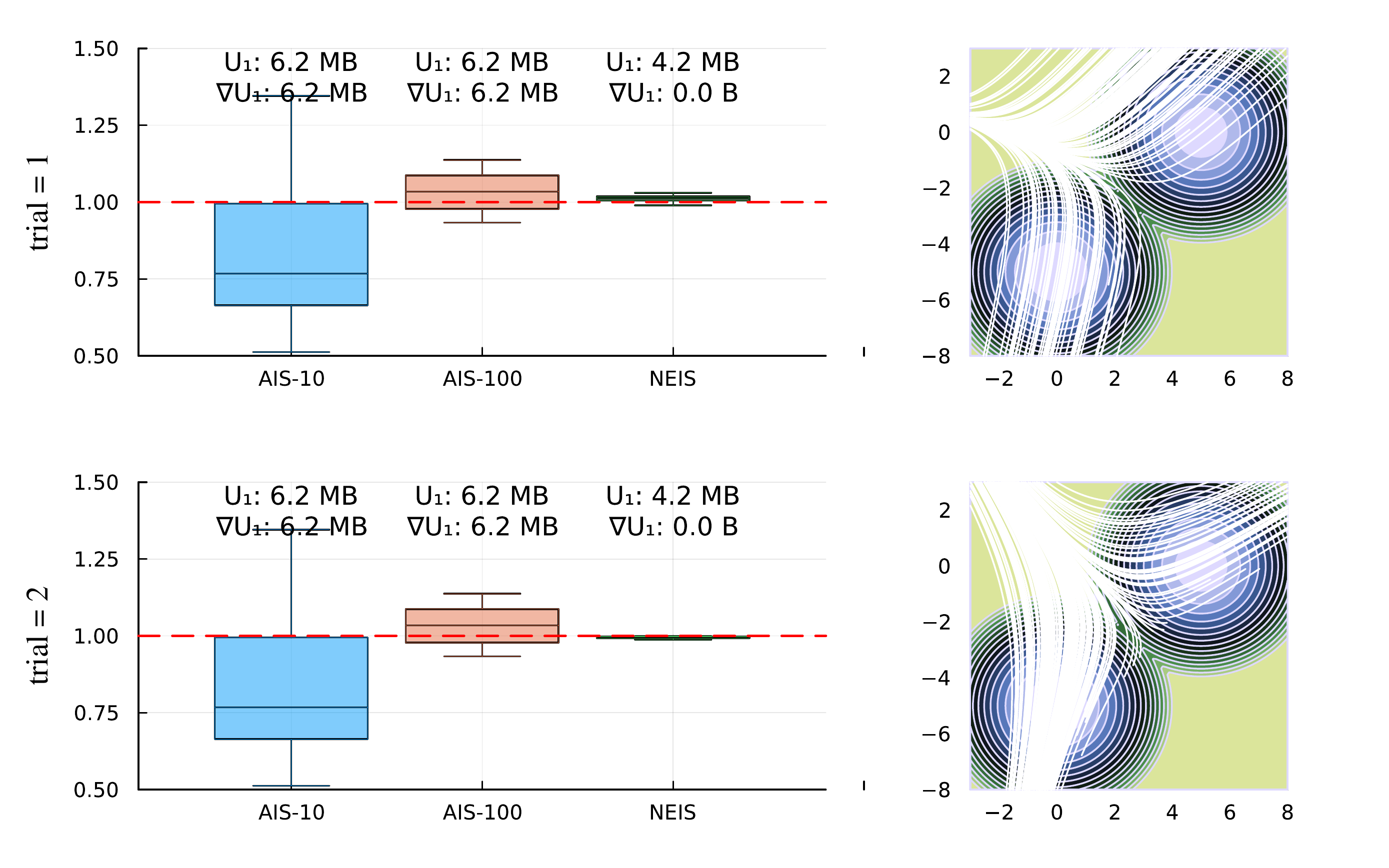}}
	
	\subfigure[Generic ansatz, $\ell=2$, $\mode=30$]{\includegraphics[width=0.9\textwidth]{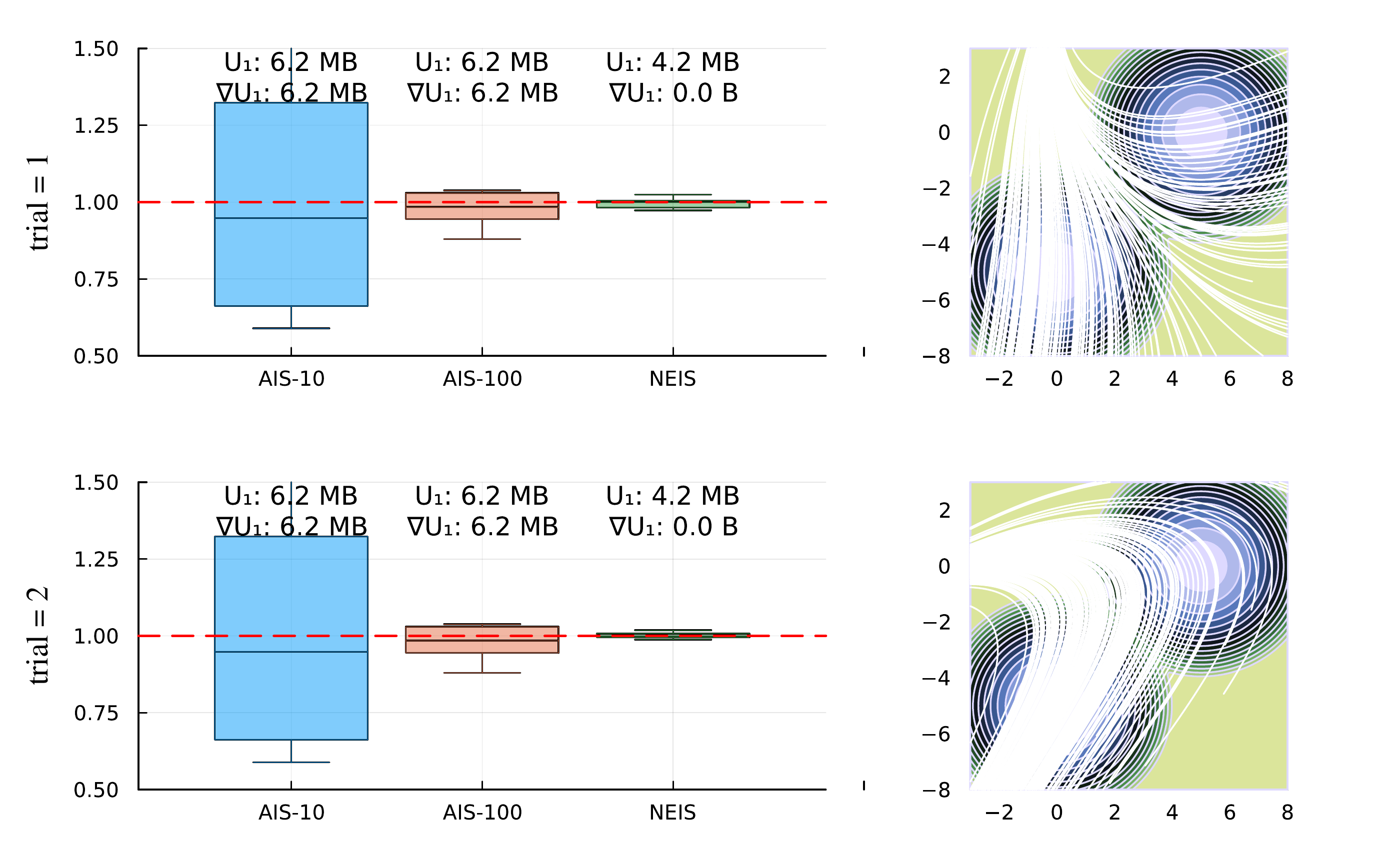}}
	\caption{Asymmetric $2$-mode Gaussian mixture in 2D: we consider the \textbf{generic ansatz} with $\ell = 2$ as the architecture for training. 
		In the left part, we show a comparison of NEIS using optimized flow with AIS, under fixed query budget: we estimate $\ratio$ using the above mentioned query budget for each method and then repeat the experiment $10$ times; we show a boxplot of these $10$ independent estimates for each method.
		In the right part, we plot a contour of $U_1$ together with streamlines of optimized flows.
		For each architecture, we use two random initializations for training, which refer to two trials above; the estimates using AIS are reused within each panel.
	}
\label{fig::cmp::1::generic}
\end{figure}

\begin{figure}
	\centering
	\subfigure[Gradient ansatz, $\ell=2$, $\mode=20$]{\includegraphics[width=0.9\textwidth]{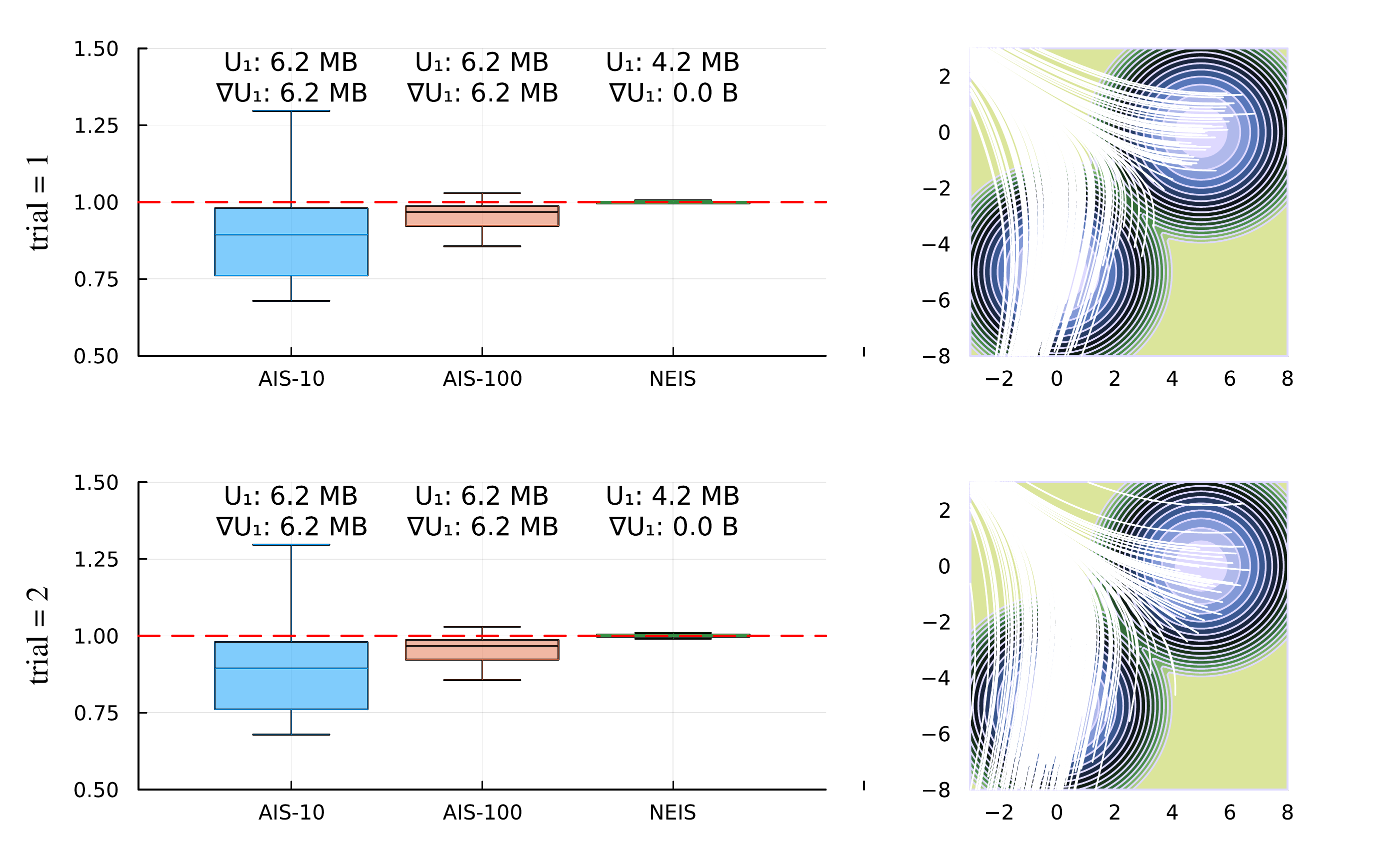}}
	
	\subfigure[Gradient ansatz, $\ell=2$, $\mode=30$]{\includegraphics[width=0.9\textwidth]{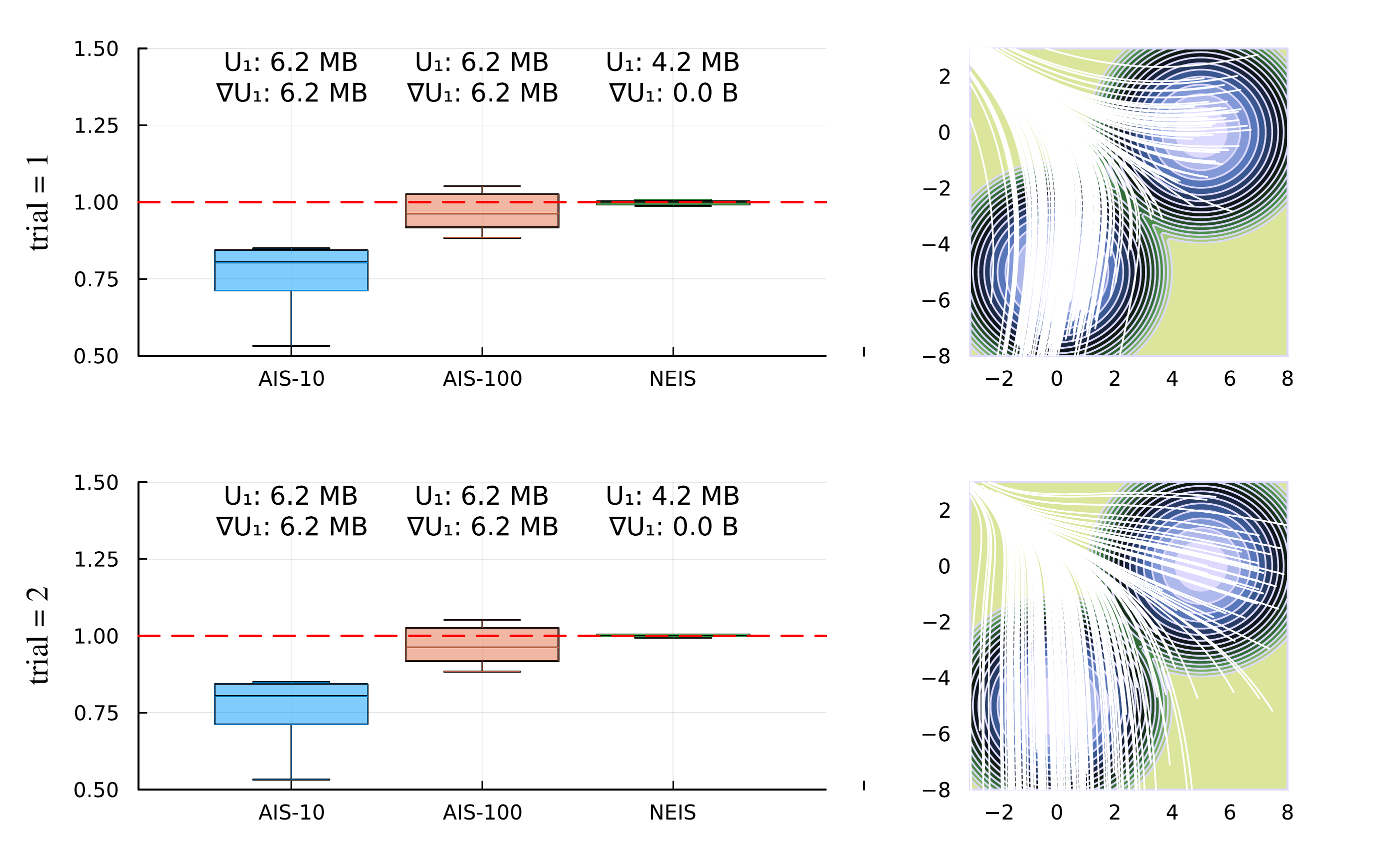}}
	\caption{Asymmetric $2$-mode Gaussian mixture in 2D: we consider the \textbf{gradient ansatz} with $\ell = 2$ as the architecture for training. 
		In the left part, we show a comparison of NEIS using optimized flow with AIS, under fixed query budget: we estimate $\ratio$ using the above mentioned query budget for each method and then repeat the experiment $10$ times; we show a boxplot of these $10$ independent estimates for each method.
		In the right part, we plot a contour of $U_1$ together with streamlines of optimized flows.
		For each architecture, we use two random initializations for training, which refer to two trials above; the estimates using AIS are reused within each panel.
	}
\label{fig::cmp::1::grad}
\end{figure}

\begin{figure}
	\centering
	\subfigure[Gradient ansatz, $\ell=2$, $\mode=30$]{\includegraphics[width=0.9\textwidth]{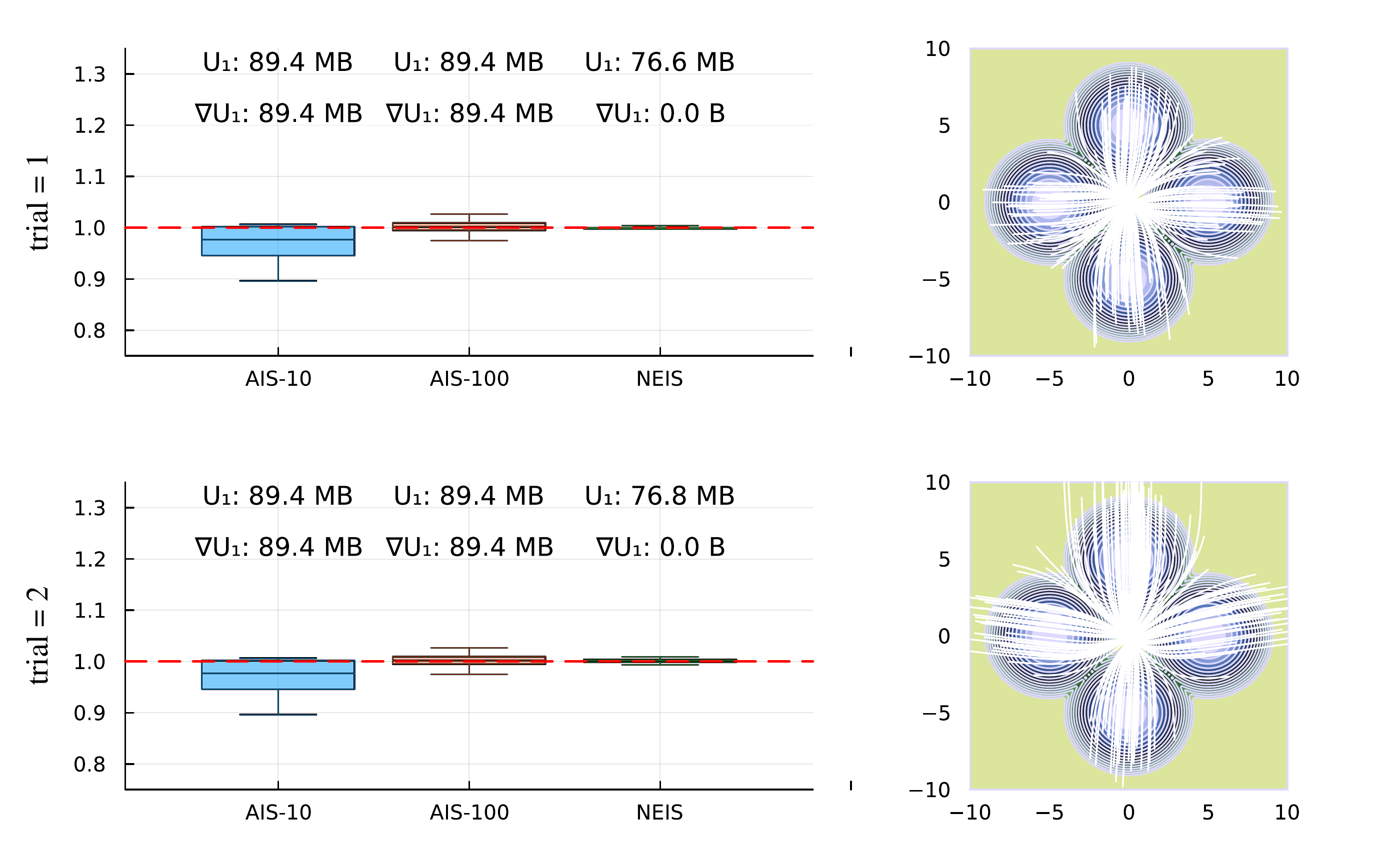}}
	
	\subfigure[Gradient ansatz, $\ell=2$, $\mode=40$]{\includegraphics[width=0.9\textwidth]{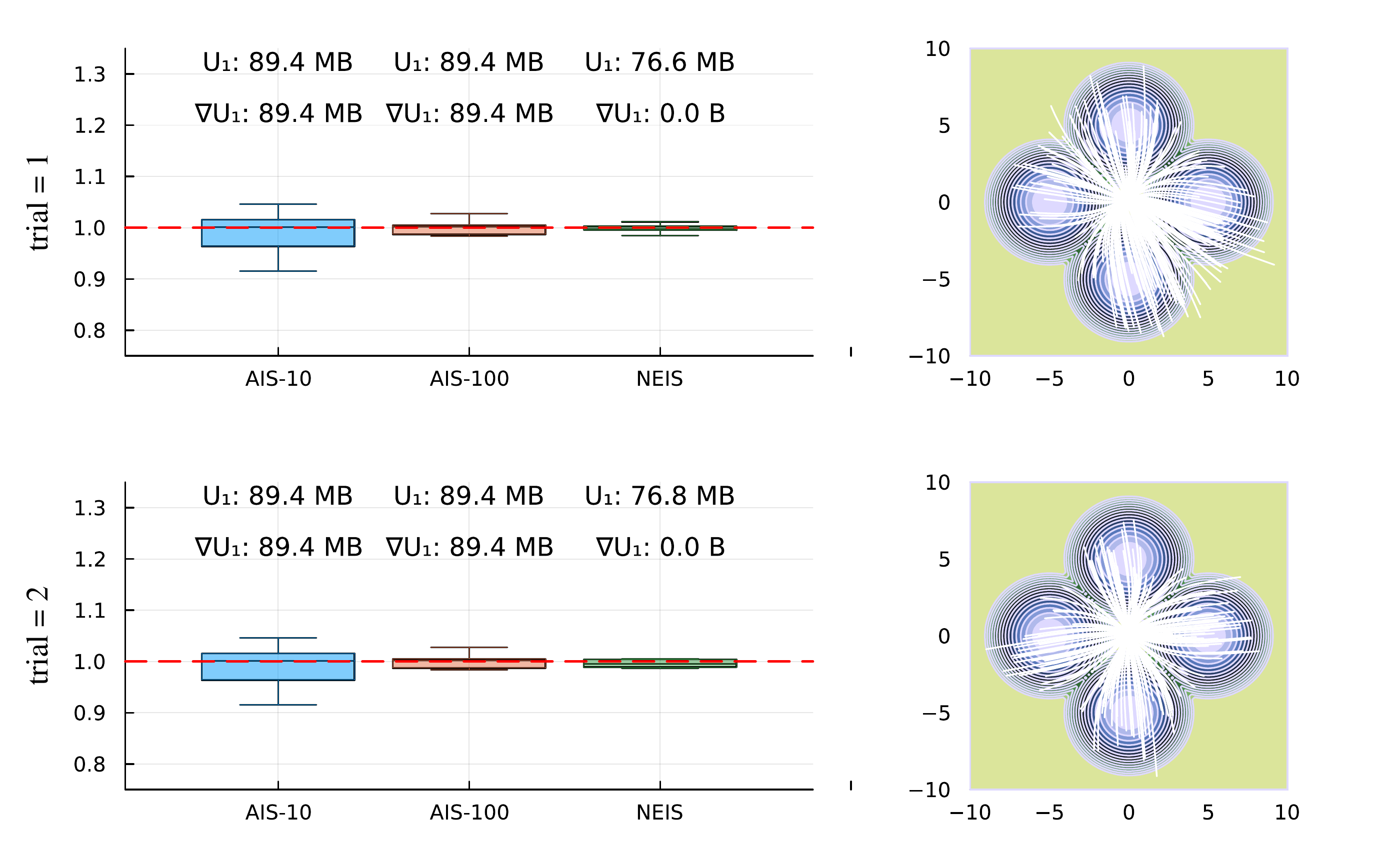}}
	\caption{Symmetric $4$-mode Gaussian mixture in 10D: we consider the \textbf{gradient ansatz} with $\ell = 2$ as the architecture for training. 
		In the left part, we show a comparison of NEIS using optimized flow with AIS, under fixed query budget: we estimate $\ratio$ using the above mentioned query budget for each method and then repeat the experiment $10$ times; we show a boxplot of these $10$ independent estimates for each method.
		In the right part, we plot a contour of projected $U_1$ (more specifically, the function $(x_1,x_2)\to U_1\big(\begin{bsmallmatrix}x_1 & x_2 & 0 \cdots \end{bsmallmatrix}\big)$) together with streamlines of optimized flows projected to the $x_1$-$x_2$ plane.
		For each architecture, we use two random initializations for training, which refer to two trials above; the estimates using AIS are reused within each panel.
	}
	\label{fig::cmp::2::grad}
\end{figure}

\begin{figure}
	\centering
	\subfigure[Linear ansatz \eqref{eqn::ln} ]{\includegraphics[width=0.9\textwidth]{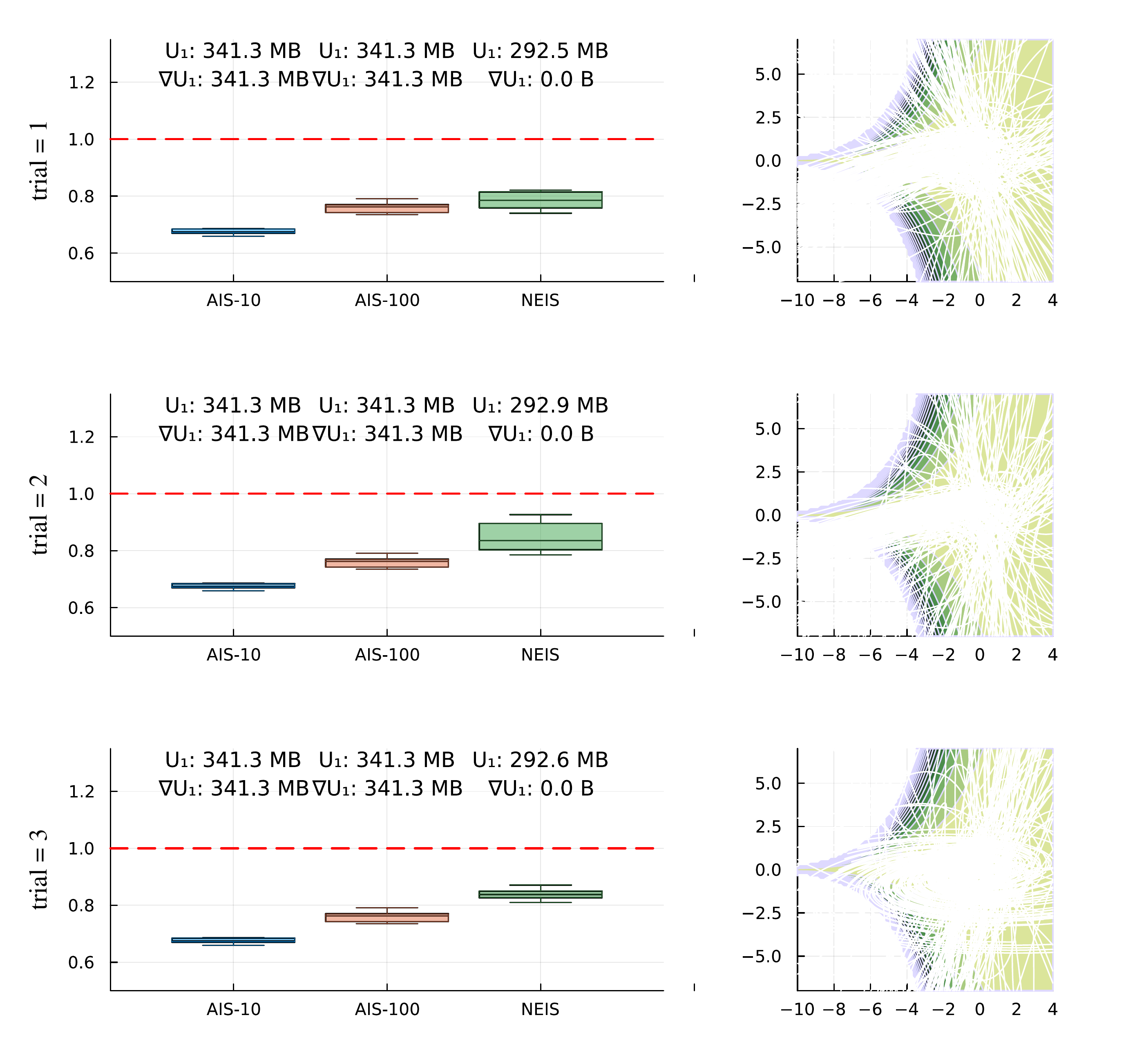}}
	
	\subfigure[two-parametric ansatz \eqref{eqn::ln_ansatz} ]{\includegraphics[width=0.9\textwidth]{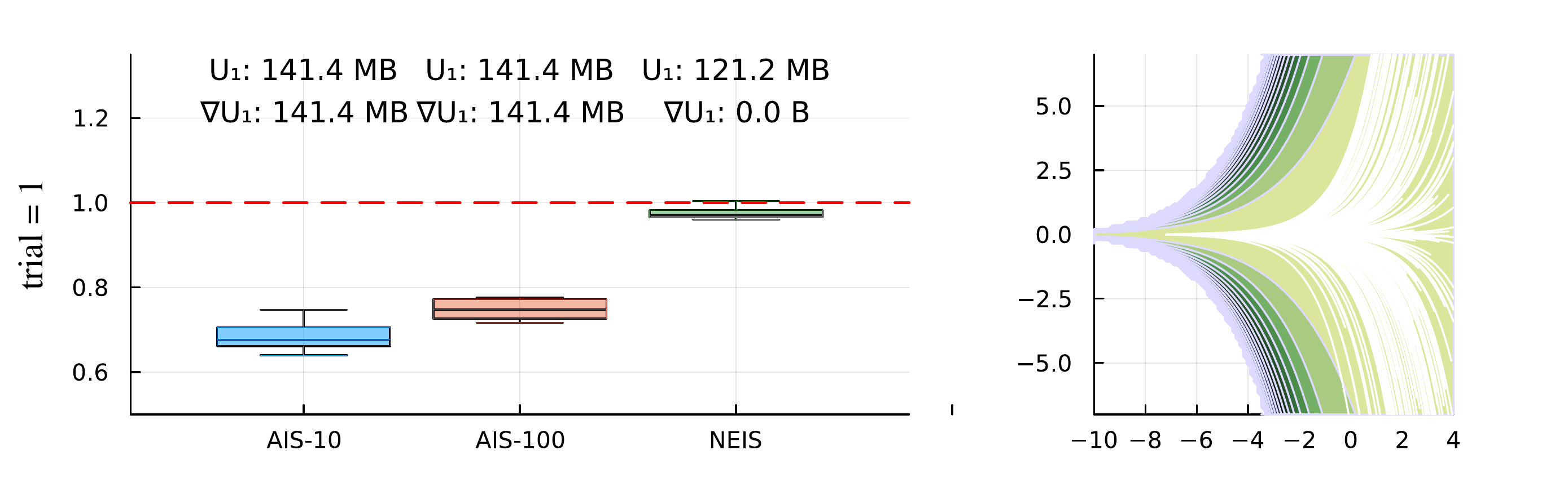}}
	
\caption{Funnel distribution in 10D: we consider the \textbf{generic linear ansatz} and a \textbf{two-parametric ansatz} as the architecture for training. 
In the left part, we show a comparison of NEIS using optimized flow with AIS, under fixed query budget: we estimate $\ratio$ using the above mentioned query budget for each method and then repeat the experiment $10$ times; we show a boxplot of these $10$ independent estimates for each method.
In the right part, we plot the contour of projected $\rho_1$ in log10 scale (more specifically, we plot the function $(x_1,x_2)\to \log_{10}\big(\rho_1(\begin{bsmallmatrix}x_1 & x_2 & 0 \cdots\end{bsmallmatrix})\big)$) together with streamlines of optimized flows projected to the $x_1$-$x_2$ plane.
For the linear ansatz, we use three random initializations for training, which refer to three trials above;
for the two-parametric ansatz, we only use a particular initialization (see the paragraph below \eqref{eqn::ln_ansatz}), which refers to the only trial in part (b); 
the estimates using AIS are reused within each panel.
}
	\label{fig::cmp::3::ln}
\end{figure}
\clearpage
}

 \end{document}